\documentclass{amsart}

\usepackage{amsfonts, amsmath, amsthm, amssymb, xypic, xcolor, hyperref, color}
\hypersetup{colorlinks=false}
\usepackage{tikz-cd}
\usepackage[shortlabels]{enumitem}

\newtheorem{thm}{Theorem}[section]
\newtheorem{theo}[thm]{Theorem}
\newtheorem{lem}[thm]{Lemma}
\newtheorem{prop}[thm]{Proposition}
\newtheorem{cor}[thm]{Corollary}
\newtheorem{ex}[thm]{Example}
\newtheorem{lemma}[thm]{Lemma}

\theoremstyle{remark}
\newtheorem{rmk}[thm]{Remark}

\newtheorem{remark}[thm]{Remark}

\theoremstyle{definition}
\newtheorem{dff}[thm]{Definition}
\newtheorem{defi}[thm]{Definition}

\newcommand{\Spec}{\operatorname{Spec}}
\newcommand{\Fil}{\operatorname{Fil}}
\newcommand{\Gal}{\operatorname{Gal}} 

\newcommand{\Aut}{\operatorname{Aut}}

\newcommand{\codim}{\operatorname{codim}}

\newcommand{\Hom}{\operatorname{Hom}}

\newcommand{\Gsimp}{\mathbf{H}}
\newcommand{\Gzero}{\mathbf{G}^0}
\newcommand{\Gsemi}{\mathbf{G}}
\newcommand{\Gzeroe}{\mathbf{G}^0_{et}}
\newcommand{\Gzerod}{\mathbf{G}^0_{dR}}
\newcommand{\Gsemie}{\mathbf{G}_{et}}
\newcommand{\Gsemid}{\mathbf{G}_{dR}}

\newcommand{\Gmon}{\mathbf{G}_{mon}}

\newcommand{\Goned}{\mathbf{G}^1_{dR}}
\newcommand{\rord}{r}
\newcommand{\pell}{p}
\newcommand{\Qpell}{\overline{\mathbb Q}_\pell}

\newcommand{\torusdim}{t}

\newcommand{\Res}[2]{{\operatorname{Res}^{#1}_{#2}}}

\newcommand{\Frob}{\operatorname{Frob}}

\begin{document}

\title{The Shafarevich conjecture for hypersurfaces in abelian varieties}

\author{Brian Lawrence}

\author{Will Sawin}

\newpage

\maketitle

\begin{abstract} Faltings proved that there are finitely many abelian varieties of genus $g$ over a number field $K$, with good reduction outside a finite set of primes $S$. Fixing one of these abelian varieties $A$, we prove that there are finitely many smooth hypersurfaces in $A$, with good reduction outside $S$, representing a given ample class in the N\'eron-Severi group of $A$, up to translation, as long as the dimension of $A$ is at least $4$. Our approach builds on the approach of \cite{LV} which studies $p$-adic variations of Hodge structure to turn finiteness results for $p$-adic Galois representations into geometric finiteness statements. A key new ingredient is an approach to proving big monodromy for the variations of Hodge structure arising from the middle cohomology of these hypersurfaces using the Tannakian theory of sheaf convolution on abelian varieties. \end{abstract}

\section{Introduction}

Fix a number field $K$ with ring of integers $\mathcal{O}_K$, and let $S$ be a finite set of primes of $\mathcal{O}_K$.
Fix an abelian variety $A$, defined over $K$, with good reduction at all primes outside $S$. We say a hypersurface $H \subseteq A$ has good reduction at $p \notin S$ if the closure of $H$ in the unique smooth projective model of $A$ over $\mathcal O_K[1/S]$ is smooth at $p$.
Our main result is the following. 
\begin{thm}[Theorem \ref{intro_1}] 
Suppose $\dim A \geq 4$. Fix an ample class $\phi$ in the N\'eron-Severi group of $A$. 
There are, up to translation, only finitely many smooth hypersurfaces $H \subseteq A$ representing $\phi$, with good reduction outside $S$. 
\end{thm} 

\vspace{5pt}

If we fix a Picard class $\psi$, rather than a N\'eron-Severi class, this theorem becomes a finiteness result for a Diophantine equation, in principle concrete. 
The theorem is equivalent to the statement that there are only finitely many $H$ representing a given Picard class $\psi$, because only finitely many translates of a given $H$ will represent $\psi$. 
The hypersurfaces in a given Picard class form a projective space, and the singular ones form an irreducible divisor as soon as $\psi$ is very ample, by a classical result (e.g.\ \cite[Theorem 1.18]{ProjDual}) which uses the fact that $A$ is not ruled by projective spaces. Thus, the singular hypersurfaces are the vanishing locus of some discriminant polynomial $\Delta(x_1,\dots,x_N)$ in the homogenous coordinates of that projective space. 
Theorem \ref{intro_1} is equivalent to the statement that, for  $u$ any $S$-unit in $\mathcal O_K$, there are only finitely many solutions of the equation $\Delta (x_1, ..., x_N) =u$ with all $x_i \in \mathcal O_K[1/S]$.

\vspace{10pt}

For $\dim A=3$ there are additional combinatorial difficulties, leading to a more complicated result.
Let $a(i)$ be the sequence \[ 1,5,20, 76, 285, 1065,  \dots \] satisfying \[ a(1)=1, a(2)=5, a(i+2) = 4 a(i+1) +1 - a(i) \]

Let $d(i)$ be the sequence \[ d(i) = { a(i)+a(i+1) \choose a(i) } \] so that \[ d(1) =6, d(2) = 53130, d(3) = 216182590635135019896, d(4) = 2.5\ldots \times 10^{79} ,\cdots\]

\begin{thm}[Theorem \ref{intro_2}]
Suppose $\dim A=3$. Fix an ample class $\phi$ in the N\'eron-Severi group of $A$. Assume that the intersection number $\phi \cdot \phi \cdot \phi$ is not divisible by $d(i)$ for any $i \geq 2$. There are only finitely many smooth hypersurfaces $H \subseteq A$ representing $\phi$, with good reduction outside $S$, up to translation.   \end{thm}

Since $a(i)$ increases exponentially, $d(i)$ increases superexponentially. Because of this rapid rate of increase, and because $d(2)$ is already large, a very small proportion of possible intersection numbers are not covered by Theorem \ref{intro_2}.  

\vspace{10pt}

If $\dim A = 2$, then hypersurfaces in $A$ are curves, and the analogue of Theorems \ref{intro_1} and \ref{intro_2} follows from the Shafarevich conjecture for curves; see Theorem \ref{intro_3}. 

\vspace{10pt}

Our result is analogous to the Shafarevich conjecture for curves, now a theorem of Faltings \cite{Faltings}, but (except in dimension 2) it doesn't seem to follow from Faltings's work; 
we'll say more about the relationship below.
Instead, the proof uses a study of variation of Galois representations
based on the work of one of the authors (B.L.) and Venkatesh \cite{LV},
and the sheaf convolution formalism of Kr\"{a}mer and Weissauer~\cite{KramerWeissauer}.

The original Shafarevich conjecture (proved by Faltings) says that there are only finitely many isomorphism classes of curves of fixed genus $g$, defined over $K$, and having good reduction outside $S$.
Similar results are now known for various families of varieties: abelian varieties (\cite{Faltings}), K3 surfaces (\cite{Andre_Sha} and \cite{She_K3}), 
del Pezzo surfaces (\cite{Scholl}), flag varieties (\cite{JL2}), complete intersections of Hodge level at most 1 (\cite{JL1}), 
surfaces fibered smoothly over a curve (\cite{Javanpeykar}), Fano threefolds (\cite{JL3}), and some general type surfaces (\cite{JL18}).

As a consequence of a hyperbolicity result of Zuo \cite{Zuo},
Javanpeykar and Loughran have suggested that the Shafarevich conjecture should hold in broad generality (see for example \cite[Conj.\ 1.4]{JL1});
the present result is further evidence in this direction.
They show that the Lang--Vojta conjecture implies the Shafarevich conjecture for certain families of complete intersections \cite[Thm.\ 1.5]{JL1}.
One expects the implication to hold for still more general families of varieties:
for any family that gives rise to a locally injective period map, 
Zuo's theorem shows that the base must be hyperbolic,
and the argument of \cite{JL1} applies.
Indeed, in our proof we use a big monodromy statement (Corollary \ref{torsion_char})
that may be seen as a strong form of injectivity of the period map. In fact, we show that this big monodromy statement implies the quasi-finiteness of a certain period map in Proposition \ref{quasi-finite-period-map} below.
%

\vspace{10pt}

 To understand the relationship between our work and previous work, it is helpful to compare and contrast with two previous finiteness theorems, both due to Faltings, involving abelian varieties. The first is the Shafarevich conjecture for abelian varieties \cite{Faltings}, i.e.\ the result that there are only finitely many isomorphism classes of abelian varieties of dimension $n$ over $K$ with good reduction outside $S$. The second is the result, in \cite{Faltings2}, that any closed subvariety of an abelian variety defined over $K$ that does not contain a positive-dimensional translate of  an abelian subvariety contains only finitely many $K$-rational points. 
 
 Both of these have been very useful for proving further arithmetic finiteness theorems. 
 The result of \cite{Faltings} was applied, using the natural maps from the moduli space of curves, certain moduli spaces of K3 surfaces, and moduli spaces of complete intersections of Hodge level 1 to the moduli space of abelian varieties, to prove most of the Shafarevich-type statements discussed above. Similarly, finiteness results for points on curves over number fields of fixed degree are proven using \cite{Faltings2} and the maps from symmetric powers of a curve to the Jacobian variety.

There does not seem to be any logical relation between our work and these two finiteness theorems. There is no reason to believe that there exists a nonconstant map from the moduli space of smooth hypersurfaces $H \subseteq A$ to any moduli space of abelian varieties 
(except when $\dim A = 2$).
Thus, our result does not seem to follow from \cite{Faltings}. There does exist a map from the moduli space of hypersurfaces to an abelian variety -- in fact $A^\vee$ -- by sending each hypersurface to its Picard class, but this is surjective so \cite{Faltings2} is not helpful. Instead, this map can be used to reduce the finiteness problem to the moduli space of smooth hypersurfaces in a given Picard class, which is an open subset of projective space. Because an open subset of projective space does not have a nonconstant map to any abelian variety, \cite{Faltings2} cannot be applied at this point.

Indeed, our main result seems to be synergistic with prior finiteness results in abelian varieties. Faltings proved that there are only finitely many abelian varieties $A$ of a given dimension with good reduction outside $S$. One can check that each of these abelian varieties has only finitely many ample N\'eron-Severi classes of a given intersection number, up to automorphism. We have proven that each of these ample classes contains only finitely many smooth hypersurfaces with good reduction outside $S$, up to translation. Finally Faltings proved that each of these hypersurfaces contains only finitely many $K$-rational points, outside of finitely many translates of abelian subvarieties. 

\vspace{10pt}

The present work uses general machinery introduced by B.L.\ and Venkatesh in \cite{LV}
to study period maps and Galois representations
applicable to cohomology in arbitrary degree.
Significant work is required to apply this machinery in our setting.
We develop a version of the sheaf convolution Tannakian category,
and use it to prove a uniform big monodromy statement.
We extend the methods of \cite{LV} to non-connected reductive groups.
Finally, we need to do some difficult combinatorial calculations related to Hodge numbers.
All of this will be explained in more detail after we recall some general ideas from \cite{LV}.

The paper \cite{LV} introduces a method to bound integral points on a variety $X$, assuming one can find a family over $X$ whose cohomology has big monodromy.
Suppose $Y \rightarrow X$ is a smooth proper family of varieties, 
extending to a smooth proper $S$-integral model 
$\mathcal{Y} \rightarrow \mathcal{X}$ over $\mathcal{O}_{K, S}$.
Then for every integral point $x \in \mathcal{X}(\mathcal{O}_{K, S})$ with associated geometric point $\overline{x}\in X$, 
the \'etale cohomology of the fiber $Y_{\overline{x}}$ gives rise to a global Galois representation 
\[\rho_x \colon \Gal_K \rightarrow \Aut( H^i_{et}({Y}_{\overline{x}}, \mathbb{Q}_p)).\]
A lemma of Faltings shows that there are only finitely many possibilities for $\rho_x$, up to semisimplification.
In various settings, it is possible to show that the representation $\rho_x$ varies $p$-adically in $x$,
and deduce that the $S$-integral points $\mathcal{X}(\mathcal{O}_{K, S})$ are not Zariski dense in $X$.

A key input to the methods of \cite{LV} is control on the image of the monodromy representation
\[ \pi_1(X, x_0) \rightarrow \Aut(H^i_{sing}(\mathcal{Y}_{x_0} )). \]
(The idea that big monodromy statements might have interesting Diophantine consequences
goes back at least to Deligne's proof of the Weil conjectures \cite{DeligneWeilII}.)
In order to show that a certain period map has big image, 
we need to know that
the Zariski closure of the image of monodromy is ``big'' in a certain sense. In particular, in the case studied in this paper the image of monodromy is sufficiently big if its Zariski closure includes one of the classical groups $SL_N, Sp_N, SO_N$. Because this is the sufficient condition we use in our argument, one can think of ``bigness" in terms of classical groups, but the precise condition in \cite{LV} is substantially more flexible, which might prove useful elsewhere.
For example, Theorem \ref{LV_thm} requires that the monodromy group be ``strongly $c$-balanced''
in the sense of Definition \ref{cbalanced_dff}, as well as two numerical conditions that are more easily satisfied when the monodromy group is larger.

A major technical difficulty in this present work 
is the need to prove a big monodromy statement
that applies uniformly to all positive-dimensional subvarieties of 
the moduli space $\mathrm{Hilb}$ of hypersurfaces in $A$.
For the monodromy groups of the universal family over $\mathrm{Hilb}$ itself, there are multiple geometric and topological arguments that could demonstrate that the monodromy contains a classical group. 
This would be sufficient to prove Zariski nondensity of the integral points $\mathrm{Hilb}(\mathcal O_{K,S})$.  Then, one hopes to improve from Zariski nondensity of $\mathrm{Hilb}(\mathcal{O}_{K, S})$ to finiteness by passing to a subvariety.
(This idea was suggested in \cite[Sec.\ 10.2]{LV}.)
Specifically, we may take $X$ an irreducible component of the Zariski closure of $\mathrm{Hilb}(\mathcal{O}_{K, S})$ in $X$. If, under the assumption that $X$ is positive dimensional modulo translation, we can show that $\mathcal {X} (\mathcal O_{K,S})$ is not Zariski dense in $X$, we obtain a contradiction. We thus deduce that $X$ is zero-dimensional modulo translation, and thus contains only finitely many distinct hypersurfaces modulo translation, and hence  $\mathrm{Hilb}(\mathcal{O}_{K, S})$ contains only finitely many hypersurfaces up to translation. However, this requires us to prove large monodromy, not just over $\mathrm{Hilb}$, but for \emph{every} positive-dimensional-modulo-translation subvariety $X \subseteq \mathrm{Hilb}$. 

For most families of varieties, such as hypersurfaces in projective space, this problem would seem totally insurmountable. Either we know almost nothing about the monodromy groups of arbitrary subvarieties, or, as in the universal family of abelian varieties, we can construct explicit subfamilies with too-small monodromy, e.g.\ low-dimensional Shimura subvarieties. However, working with a family of subvarieties of a fixed abelian variety $A$ provides us with a way out. The inverse image of $H$ under the multiplication-by-$\ell^n$ map $A \to A$ has good reduction everywhere $H$ does, except possibly at $\ell$, and we can run the argument with its \'{e}tale cohomology. The $\ell^n$-torsion points $A[\ell^n]$ act on this inverse image, and thus on its cohomology; this action splits the cohomology into a sum of eigenspaces, each with its own monodromy representation. It suffices for our purposes to show that one of these representations has big monodromy. 

This additional freedom allows a new type of argument, based on the Tannakian theory of sheaf convolution developed by Kr\"{a}mer and Weissauer~\cite{KramerWeissauer}. They defined a group, the ``Tannakian monodromy group", associated to a subvariety in an abelian variety (and in fact to much more general objects). Its definition is subtler than the definition of the usual monodromy group, but it is a better tool to work with because it depends only on a single hypersurface in the family, whose geometry can be controlled, rather than an arbitrary family of hypersurfaces, whose geometry is far murkier.
We prove a group-theoretic relationship (Theorem \ref{thm_fin_many_subtori}) between the usual monodromy group of a typical $A[\ell^n]$-eigenspace in the cohomology of a family of hypersurfaces and the Tannakian monodromy group of a typical member of the family of hypersurfaces. One can think of this as analogous to the relationship between the monodromy groups of the generic horizontal and vertical fibers of a family of varieties over (an open subset of) a product $X_1 \times X_2$. Using purely geometric arguments involving the results of \cite{Microlocal1} and \cite{Microlocal2}, we show that the Tannakian monodromy group contains a classical group, and then using Theorem \ref{thm_fin_many_subtori}, we show that the usual monodromy group does as well.

\vspace{10pt}

We believe the problem of proving big monodromy for the restriction of a local system to an arbitrary subvariety to be very difficult without this Tannakian method, but owing to its arithmetic applications, it would be very interesting to look for new examples where this can be established by a different method. 
The following vague toy problem illustrates the sort of difficulty that arises. 
Suppose given a variety $X$ of dimension at least two and a smooth family $Y \rightarrow X$.
One wants to obtain a strong lower bound, 
over all pointed curves $(Z, z) \subseteq X$, 
on the dimension of the Zariski closure of the image of monodromy
\[ \pi_1(Z, z) \rightarrow \Aut H^1(Y_z). \]
This seems difficult in all but a few special cases
(such as products of curves, 
where finiteness already follows immediately from Faltings's theorem),
but see the preprint \cite{Urbanik} for a promising approach.

\vspace{10pt}

The methods of this paper can likely be applied to many different classes of subvarieties of abelian varieties, beyond hypersurfaces. To make this generalization, the additional inputs needed are a result giving some control on the Tannakian monodromy group associated to the subvariety and the verification of a certain inequality involving the Hodge numbers and this group (see Lemma \ref{numerical_verification}).

\subsection{Outline of the proof}

The argument of \cite{LV} derives bounds on $\mathcal{X}(\mathcal{O}_{K, S})$ from a family $f \colon Y \rightarrow X$,
through a study of various cohomology objects $R^i f_*( - )$ on $X$.
The \'etale local system $R^i f_*( \mathbb{Q}_p )$ gives rise to the global Galois representations to which Faltings's lemma is applied;
a filtered $F$-isocrystal coming from crystalline cohomology is used to study the $p$-adic variation of these Galois representations;
and a variation of Hodge structure allows one to relate a $p$-adic period map to topological monodromy.
The method allows one to conclude that $\mathcal{X}(\mathcal{O}_{K, S})$ is not Zariski dense in $X$.

In the present setting, we will apply these results to $R^i f_*(\mathsf{L})$,
where $\mathsf{L}$ is a nontrivial local system on $Y$.
We now outline the main construction; details are in the proof of Theorem \ref{main_thm}.
For technical reasons, 
\footnote{Most importantly, Lemma \ref{balanced} holds over $\mathbb{Q}$ but not over an arbitrary number field.
Additionally, $p$-adic Hodge theory plays well with tensor categories over $\mathbb{Q}_p$. 
Since passing to extensions (even unramified extensions) gives rise to semilinear operators, 
we need to set up the Tannakian formalism over $\mathbb{Q}_p$.}
we prefer to work with $X$ a variety over $\mathbb{Q}$,
so we take $X$ to be a Zariski-closed subset of the Weil restriction from $K$ to $\mathbb{Q}$
of the moduli space of smooth hypersurfaces in $A$.
In fact, we simply take $X$ to be an irreducible component the Zariski closure of the set of integral points.
We take $Y$ the universal family of smooth hypersurfaces over $X$;
that is, points $x \in X(\mathbb{Q})$ are in bijection with smooth hypersurfaces $Y_x \subseteq A$ over $K$.
To conclude Theorem \ref{main_thm}, we need to show that $Y$ is a translate of a constant family over $X$.

Let $n = \dim A$.  We will study cohomology objects $R^i f_*(\mathsf{L})$, where $\mathsf{L}$ is a local system on $Y$,
given as a direct sum of characters.
Any finite-order character $\chi$ on $\pi_1(A)$ defines a local system on $Y$, by pullback via $Y \hookrightarrow A$.
If $\chi$ is defined over some field $L$, the same is true of the corresponding local system $\mathsf{L}_{\chi}$.
We construct a local system $\mathsf{L}$ defined over $\mathbb{Q}$
by descent, as a sum of Galois conjugates of $\mathsf{L}_{\chi}$;
and we take $\mathsf{V} = R^{n-1} f_*(\mathsf{L})$.
The construction is given in Lemma \ref{construct_v} in \S \ref{locsys}.

We think of $\mathsf{V}$ and $\mathsf{L}$ as ``motives'' with various realizations satisfying some compatibilities,
as in Deligne \cite{Deligne_motives}.
The precise realizations and compatibilities we need are formalized in the notion 
of a Hodge--Deligne system (Definition \ref{HD_def}).



The \'etale realization of $\mathsf{V} = R^{n-1} f_*(\mathsf{L})$ gives, for every $x \in \mathcal{X}(\mathbb{Z}[1/S])$,
a Galois representation
\[ \rho_x \colon \Gal_{\mathbb{Q}} \rightarrow \Aut H^k(Y_x, \mathbb{Q}_p). \]
By Faltings's lemma (Lemma \ref{falt_finite}), 
there are only finitely many possibilities for the semisimplification $\rho_x^{ss}$ of $\rho_x$,
as $x$ varies over $\mathcal{X}(\mathbb{Z}[1/S])$.

As in \cite{LV}, we want to show that the fibers of the map
\[ x \mapsto \rho_x^{ss} \]
are not Zariski dense.
To do this, we consider the map that takes a $p$-adic point $x \in X(\mathbb{Q}_p)$
to a \emph{local} Galois representation.
By $p$-adic Hodge theory, the local Galois representation
\[ \rho_{x, p} \colon \Gal_{\mathbb{Q}_p} \rightarrow \Aut H^k(Y_x, \mathbb{Q}_p) \]
determines the filtered $\phi$-module
\[ (V_x, \phi_x, F_x) = H^k_{cris}(Y_x). \]

We recall from \cite[\S 3]{LV} some facts about the variation of $(V_x, \phi_x, F_x)$ with $x$.
For $x$ in a fixed mod-$p$ residue disk $\Omega$, the pair $(V_x, \phi_x)$ is constant: 
these spaces are canonically identified with the crystalline cohomology of the mod-$p$ reduction of $Y_x$.
The filtration $F_x$ varies with $x$.
The assignment $x \mapsto F_x$ defines a $p$-adic period map
\[ \Phi_p \colon \Omega \rightarrow \mathcal{H} \]
to a certain flag variety.
The $p$-adic period map is analogous to the classical complex-analytic period map of Hodge theory,
and indeed the two maps are closely related, a fact we will exploit in Section \ref{sec:per}
(see for example the proof of Lemma \ref{PV_per}).

The global semisimplification $\rho_x \mapsto \rho_x^{ss}$ causes substantial technical difficulties.
Before our main argument in Section \ref{sec:int_thm}, 
we give (Section \ref{sec:int_thm_ss})
an alternative, simpler proof under the additional assumption that every relevant representation $\rho_x$ is semisimple. 
For this sketch, to illustrate ideas, let us make the same assumption; that is, let us imagine that every global representation $\rho_x$ is semisimple.
Then there are literally only finitely many possibilities for the isomorphism class $\rho_x$,
so (restricting to the local Galois representation and applying the crystalline Dieudonn\'e functor)
there are only finitely many possibilities (up to isomorphism) for the filtered $\phi$-module $(V_x, \phi_x, F_x)$,
as $x$ ranges over all integral points.
In this simplified setting, we need only show that
\[ \{ x \in X(\mathbb{Q}_p) \mid (V_x, \phi_x, F_x) \cong (V_0, \phi_0, F_0) \} \]
is contained in a positive-codimension algebraic subset of $X$.

Isomorphism classes of triples $(V_x, \phi_x, F_x)$ correspond to orbits of the Frobenius centralizer $Z(\phi)$ on $\mathcal{H}$,
so we want to control $\Phi_p^{-1}(Z)$, where $Z$ is an orbit of the Frobenius centralizer.
We'll have the result we want if we can prove precise versions of the following two conditions.
\begin{enumerate}[(a)]
\item \label{condition_a} The Frobenius centralizer is small.
\item \label{condition_b} The image of $\Phi_p$ is not contained in a small algebraic set.
\end{enumerate}

In fact, since we don't know that the global Galois representations are semisimple,
we need a stronger form of \ref{condition_a}.
\begin{enumerate}[(a')]
\item \label{condition_a'} 
(See Lemma \ref{codim_final_estimate}.)
Fix a $\phi$-module $(V, \phi)$ and a semisimple global Galois representation $\rho^{ss}$.
Consider all global Galois representations $\rho$ whose semisimplification is $\rho^{ss}$, 
and such that $D_{cris}(\rho|_{\mathbb{Q}_p}) \cong (V, \phi, F)$, for some filtration $F$ on $V$.

The $F$ that arise in this way all lie in a subvariety $Z \subseteq \mathcal{H}$ of low dimension.
\end{enumerate}

Once we have items $\ref{condition_a'}$ and $\ref{condition_b}$,
we know that $X(\mathbb{Z}[1/S])$
is contained in $\Phi_p^{-1}(Z)$.
A $p$-adic version of the Bakker--Tsimerman transcendence theorem (Theorem \ref{bt})
will imply that $X(\mathbb{Z}[1/S])$ is not Zariski dense.

Condition \ref{condition_a'} comes from two ingredients.
First, the semilinearity of Frobenius gives an upper bound on its centralizer
(Lemma \ref{semilinear_dim}).
Second, the possible subrepresentations of a global Galois representation are constrained by purity
(Lemma \ref{balanced}),
which restricts the structure of \emph{local} Galois representations coming from \emph{global} $\rho$ having a given semisimplification.
It is this latter result that requires us to work over $\mathbb{Q}$
(or at least a number field that has no CM subfield).
As mentioned above, we can always pass to this situation by restriction of scalars.

We introduce ``$H^0$-algebras'' (\S \ref{H0_subsection}) to package the cohomological data that arise in this situation.
The Hodge--Deligne system $\mathsf{V}$ will have the structure of module 
over a certain algebra object $\mathsf{E}$ in the category of Hodge--Deligne systems;
this module structure allows us to keep track of the Galois actions on embeddings of the field $K$, 
isomorphism classes of local systems $\mathsf{L}$, and on the coefficient field of the local system, in a uniform and convenient way.

Condition \ref{condition_b} is a question about the monodromy of
the variation of Hodge structure given by $\mathsf{V}$.
As mentioned above, we only need a very weak lower bound on the Zariski closure of the monodromy group. We call the relevant condition ``strongly $c$-balanced" (Definition \ref{cbalanced_dff};
see Corollary \ref{torsion_char} and Lemma \ref{cbalanced_monodromy} for precise statements). It depends on a parameter $c$ which must be taken sufficiently large.
The technical difficulty in Corollary \ref{torsion_char} is that it applies uniformly to any family of hypersurfaces in an abelian variety, as is required to prove finiteness.

It is now crucial that, in our case, $Y$ is a subvariety of $A \times X$, with the map $f$ the restriction of the projection map $A \times X$ to $X$. The abelian variety $A$ has many rank-one local systems $\mathsf{L}$, each of which we can pull back to $Y$, push forward to $X$, and apply this machinery to. These local systems are associated to characters of the fundamental group $\pi_1(A)$.

To apply the $p$-adic Hodge theory argument described above, it suffices to have a local system $\mathsf{L}$ on $Y$ such that $R^{n-1} f_*(\mathsf{L})$ has big monodromy in our sense. (There are some additional technical conditions that we suppress here to focus on the main difficulty.) In fact we will show big monodromy for almost all rank one local systems $\mathsf{L}$, in a precise sense (Theorem \ref{Tannakian-group-calculation} and Corollary \ref{torsion_char}). To do this, it is necessary to have a framework in which the vector spaces $R^{n-1} f_*(\mathsf{L})_x$ for different local systems $\mathsf{L}$ can be studied all at once. This is accomplished by the Tannakian theory of sheaf convolution~\cite{KramerWeissauer}.

The fundamental objects of the Tannakian theory of sheaf convolution are perverse sheaves. The fundamental perverse sheaf for us is the constant sheaf on $Y_x$, pushed forward to $A$, and placed in degree $1-n$. The vector space $R^{n-1} f_*(\mathsf{L})_x$ can be recovered from this by applying the functor $K \mapsto H^0 (A, K \otimes \mathsf{L})$. The theory of~\cite{KramerWeissauer} views (a slight modification of) the category of perverse sheaves on $A$ as the category of representations of a certain group; the functors $K \mapsto H^0 (A, K \otimes \mathsf{L})$ are almost all isomorphic to the functor taking a representation to the underlying vector space. The image of this group on the representation associated to a perverse sheaf $K$ is the convolution monodromy group.

We show that, if the convolution monodromy group of the constant sheaf on $Y_x$ contains a classical group for some $x \in X$, and if the family $Y$ over $X$ is not equal to a translate of the constant family, then for almost all $\mathsf{L}$, the monodromy groups of $R^{n-1} f_*(\mathsf{L})_x$ contain a classical group (Theorem \ref{thm_fin_many_subtori}, Corollary \ref{torsion_char}). To check the condition that the convolution monodromy group of the smooth hypersurface $Y_x$ contains a classical group, we use recent results of Kr\"{a}mer~\cite{Microlocal1, Microlocal2}, to reduce to a small number of cases---essentially, the simple algebraic groups acting by their minuscule representations---and then  some intricate but elementary combinatorics involving Hodge numbers to eliminate the non-classical cases.

We now conclude the argument by taking $\mathrm{Hilb}$ to be the moduli space of hypersurfaces in a given N\'{e}ron-Severi class in $A$ and $X$ to be an irreducible component of the Zariski closure of $ \mathrm{Hilb} ( \mathcal O_{K,S})$. Assuming that the universal family over $X$ is not equal to a translate of the constant family, we find a sheaf $R^{n-1} f_*(\mathsf{L})_x$ with big monodromy and, using $p$-adic Hodge theory, show the integral points of $X$ are not Zariski dense. This contradicts the definition of $X$ as an irreducible component, so we conclude the universal family over each component is equal to a translate of the constant family. It follows that all the fibers $Y_x$ contained in each irreducible component are identical up to translation, so because there are finitely many irreducible components, there are finitely many hypersurfaces up to translation -- our desired conclusion. 

\subsection{Sheaf convolution and uniform big monodromy}

Given an abelian variety $A$ over an algebraically closed field, Kr\"amer and Weissauer \cite{KramerWeissauer} 
construct a Tannakian category as a quotient of the category of perverse sheaves on $A$.
A perverse sheaf $N$ on $A$ is said to be \emph{negligible} if its Euler characteristic is zero;
the negligible sheaves form a thick subcategory, 
and the sheaf convolution category is defined as the quotient of the category of all perverse sheaves by the negligible sheaves.
The convolution of two perverse sheaves has negligible perverse homology in nonzero degrees;
in other words, it is ``perverse up to negligible sheaves,''
and convolution defines a tensor structure on this quotient category.

One original motivation for this construction was the Schottky problem \cite{KWTheta}.
Given a principally polarized abelian variety $A$ (say of dimension $g$) with theta divisor $\Theta$,
one wants to know whether $A$ is isomorphic to a Jacobian, say $\operatorname{Jac} C$.
In this case, $\Theta$ would be the $(g-1)$-st convolution power of $C$.
Informally, the role of the Tannakian formalism here is to determine whether $\Theta$ is ``a $(g-1)$-st convolution power of something.''

An alternate motivation for the sheaf convolution theory comes from work of Katz. This time, one works with an abelian variety $A$ over a finite field $\mathbb F_q$. A perverse sheaf $K$ on $A$ has a trace function $f_K$ on $A(\mathbb F_q)$. Associated to a character $\chi$ of $A(\mathbb F_q)$ is the character sum $\sum_{x \in A(\mathbb F_q)} f_K(x) \chi(x)$. Katz showed (in unpublished work analogous to \cite{KatzCaE}) that the distribution of $\sum_{x \in A(\mathbb F_q)} f_K (x) \chi(x)$, viewed as a random variable for uniformly random $\chi$, converges to a distribution determined by the convolution monodromy group, in the limit as $q$ goes to $\infty$. More precisely, the distribution is like the trace of a random element in the maximal compact subgroup of the convolution monodromy group. To gain some intuition for this, note that given representations $V_1$, $V_2$, we have $\operatorname{tr}(g, V_1 \otimes V_2) = \operatorname{tr}(g,V_1) \operatorname{tr}(g,V_2)$; that is, taking the tensor product of representations has the effect of multiplying the traces. For the character sums $\sum_{x \in A(\mathbb F_q)} f_K(x) \chi(x)$, convolution has the same effect:
\[ \sum_{x \in A(\mathbb F_q)} (f_{K_1} * f_{K_2}) (x) \chi(x)= \left( \sum_{x \in A(\mathbb F_q)} f_{K_1} (x) \chi(x)\right) \left( \sum_{x \in A(\mathbb F_q)} f_{K_2} (x) \chi(x)\right).\]
In other words, convolution of these functions $f_K$ has a similar effect on this sum as tensor product of the representations $V$ has on the trace. It stands to reason that a framework where perverse sheaves correspond to representations, and convolution of sheaves correspond to tensor product of representations, would have relevance to the distribution of the trace. In particular, this should be plausible if one is familiar with Deligne's equidistribution theorem \cite[Theorem 3.5.3]{DeligneWeilII}, whose proof is similar to the argument Katz uses to establish the relationship between the distribution and the convolution monodromy group \cite[Corollary 7.4]{KatzCaE}. 

For non-algebraically closed fields, such as finite fields, we can construct a Tannakian category in almost the same way as Kr\"amer and Weissauer did---again defining negligible sheaves as those with zero Euler characteristic. The key facts (for example, that the convolution of two perverse sheaves has negligible perverse cohomology in nonzero degrees) hold over the base field once checked over its algebraic closure. 

To relate these two categories, it is convenient to restrict attention to geometrically semisimple perverse sheaves on $A_k$, and to perverse sheaves on $A_{\overline{k}}$ which are summands of the pullback from $A_k$ to $A_{\overline{k}}$ of geometrically semisimple perverse sheaves. Having done this, we obtain (in Lemma \ref{tannakian-quotient}) an exact sequence of pro-algebraic groups \begin{equation}\label{eq-tan-ex} 1\to G_{\overline{k}} \to G_k \to \operatorname{Gal}_k \to 1\end{equation} where $G_k$ is the Tannakian group of a suitable category of perverse sheaves on $A_k$, $G_{\overline{k}} $ is the Tannakian group of a suitable category of perverse sheaves on $A_{\overline{k}}$, and $\operatorname{Gal}_k $ is the Tannakian group of the category of $\ell$-adic $\Gal(\overline{k}/k)$-representations -- in other words, the Zariski closure of $\Gal(\overline{k}/k)$ in the product of the general linear groups of all its finite-dimensional $\ell$-adic representations. We think of this as a close analogue of the usual exact sequence \[ 1 \to \pi_1^{geom}(X) \to \pi_1^{arith}(X) \to \Gal(\overline{k}/k ) \to 1\] for a variety $X$ over a field $k$. 

Just like this usual exact sequence, \eqref{eq-tan-ex} often has splittings. In our case, splittings arise from certain local systems $\mathsf{L}$ on $A$ defined over $k$, as the cohomology of a perverse sheaf twisted by a local system is a Galois representation, 
on which $ \operatorname{Gal}_k $ acts, and we can check that this action factors through the Tannakian group $G_k$, giving the splitting.

Fix now a subvariety $X$ of the moduli space of smooth hypersurfaces in an abelian variety $A$. Let $k$ be the field of functions on the generic point of $X$. Let $H$ be the universal hypersurface in $A$, defined over $k$. Let $K$ be the constant sheaf on $H$, pushed forward to $A$, placed in degree $1-n$; this is our perverse sheaf of interest. Associated to $k$ is a representation of $G_k$. The action of $G_{\overline{k}}$ on this representation is a purely geometric object. By geometric methods, we will show in Theorem \ref{Tannakian-group-calculation} that the image of $G_{\overline{k}}$ acting on this representation contains $SL_N, SO_N$, or $Sp_N$ as a normal subgroup. So the image of $G_k$ on the representation associated to $k$ contains the same classical group as a normal subgroup. Because the action of $\Gal_k$ in this setting matches the action of the fundamental group, it will suffice for our big monodromy theorem to show that the action of $\Gal_k$ also contains (as a normal subgroup) the same classical group.

To do this, we construct in Lemma \ref{representation-detection} a battery of tests, each consisting of a representation 
of the normalizer of the classical group, such that any subgroup of the normalizer contains the classical group if and only if it has no invariants on any of these representations. Associated to each of these representations is a perverse sheaf on $A_k$. We prove Lemma \ref{specialization-lemma} showing that the action of $ \operatorname{Gal}_k $ on the cohomology of a perverse sheaf, defined using a generic local system $\mathsf{L}$, has invariants if and only if the perverse sheaf has a very special form. Using Lemma \ref{non-constancy}, we check that the relevant perverse sheaves do not have this special form unless the family of hypersurfaces over $X$ is constant, up to translation by a section of $A$, completing the proof of Theorem \ref{thm_fin_many_subtori}. 

Next we describe how we check in Theorem \ref{Tannakian-group-calculation} that the image of the  $G_{\overline{k}}$-action on the representation associated to a smooth hypersurface in $A$ contains a classical group acting by the standard representation as a normal subgroup. This proceeds in two steps. The first step shows (in Lemmas \ref{lie-irreducible} and \ref{Tannakian-group-simple}) that the commutator of the identity component of this image group is a simple algebraic group acting by a minuscule representation. (Recall that a \emph{minuscule representation} is one where the eigenvalues of the maximal torus action are conjugate under the Weyl group.) The second step eliminates (in Lemmas \ref{minuscule_case_2} and \ref{reduction-to-combinatorics} and Proposition \ref{combinatorics-main-statement})  all such pairs of a group and a representation except the standard representations of the classical groups. The first step is a conceptual proof using sophisticated machinery from \cite{Microlocal1,Microlocal2}, while the second uses no additional machinery (except a bit of Hodge theory) but involves an intricate combinatorial argument.

For the first step, we apply results of Kr\"{a}mer that study the characteristic cycle of a perverse sheaf. This is a fundamental invariant of any perverse sheaf on a smooth variety, defined as an algebraic cycle on the cotangent bundle of that variety. (For abelian varieties, the cotangent bundle is a trivial vector bundle.) By examining how the characteristic cycle of a perverse sheaf changes when it is convolved with another perverse sheaf, Kr\"{a}mer was able to relate the convolution monodromy group to the characteristic cycle. In particular, he gave criteria for the commutator subgroup of the identity component to be a simple group, and for the representation of it to be minuscule. The fact that our hypersurface is smooth makes its characteristic cycle relatively simple---it is simply the conormal bundle to the hypersurface. This makes Kr\"{a}mer's minisculeness criterion straightforward to check, but to check simplicity we must make a small modification to Kr\"{a}mer's argument. The reason for this is that Kr\"{a}mer, motivated by the theta divisor and the Schottky problem, assumed that a hypersurface in $A$ was invariant under the inversion map, while we do not wish to assume this. 

For the second step, the exceptional groups and spin groups are not too hard to eliminate, as they only occur for representations of very specific dimensions. The Tannakian dimension in our setting is the topological Euler characteristic of the hypersurface, which we have an explicit formula for. Comparing these, we can see in Lemma \ref{minuscule_case_2} that the problematic cases only occur for curves in an abelian surface, which are excluded by the assumption $\dim (A) \geq 3$. The only remaining case, except for the good classical cases, is the case of a special linear group acting by a wedge power representation. For this representation, the Euler characteristic formula is not sufficient, but we are eventually able to rule this case out using a more sophisticated numerical invariant, the Hodge numbers. If the convolution monodromy group acts on the representation associated to $H$ by the $k$-th wedge power of an $m$-dimensional representation, we might expect that the Hodge structure on the cohomology of $H$, or the cohomology of $H$ twisted by a rank one local system, is the $k$-th wedge power of an $m$-dimensional Hodge structure. This would place some restrictions on the Hodge numbers. We don't prove this, but instead prove in Lemma \ref{reduction-to-combinatorics} a $p$-adic Hodge-theoretic analogue, using the $\Gal_k$-action discussed earlier. On the other hand, we can calculate the Hodge numbers of the cohomology of $H$ twisted by a rank one local system using the Hirzebruch-Riemann-Roch formula. Working this out gives a complicated set of combinatorial relations between the Hodge numbers of the original $m$-dimensional Hodge structure. By a lengthy combinatorial argument in Appendix \ref{combinatorics}, we find all solutions of these relations, noting in particular that they occur only for abelian varieties of dimension less than four. This completes the proof.

\subsection{Outline of the paper}

The argument proceeds in three parts.  

First, we use the sheaf convolution formalism to prove a big monodromy result for families of hypersurfaces.
In Section \ref{sec:sheaf} we introduce the sheaf convolution category, 
a Tannakian category of perverse sheaves on an abelian variety.
In Section \ref{sec:tan} we investigate the convolution monodromy group of a hypersurface;
we show in many cases that this group must be as big as possible. 
In Section \ref{sec:mon} we relate the convolution monodromy group to the geometric monodromy group,
which gives the big monodromy statement we need.

Sections \ref{sec:hds}--\ref{sec:int_thm} explain how to deduce non-density of integral points,
following the strategy in \cite{LV}.
Section \ref{sec:hds} contains some technical preliminaries.
We introduce the notion of Hodge--Deligne system, which is closely related to Deligne's ``system of realizations'' of a motive, although we include only the realizations that are relevant for our argument.
We discuss ``$H^0$-algebras'', roughly, algebra objects in the category of Artin motives with rational coefficients, which we need to express the semilinearity of Frobenius. We also recall some facts from the theory of not-necessarily-connected reductive groups.
Section \ref{sec:per} relates the big monodromy statement from Section \ref{sec:mon} to the $p$-adic period map, 
via the theorem of Bakker and Tsimerman (\cite{BT}).
In Section \ref{sec:int_thm_ss}, we deduce the non-density of integral points,
under the simplifying assumption that all the global representations that arise are semisimple.
In Section \ref{sec:int_thm}, we prove the theorem in full generality. 
The argument used to handle the global semisimplification involves combinatorics on reductive groups, analogous to \cite[\S 11]{LV}.
We conclude with Theorem \ref{LV_thm}, which is analogous to Lemma 4.2, Prop.\ 5.3, and Thm.\ 10.1 in \cite{LV}.

Finally, we wrap up the proof of our main theorem in Section \ref{sec:main_thm}.

Appendices \ref{sec:numerical}, \ref{combinatorics}, and \ref{sec:eulerian_ineq} contain some purely combinatorial calculations involving Eulerian numbers. Appendix  \ref{sec:numerical} verifies the two numerical conditions in the hypotheses of Theorem \ref{LV_thm}. Appendix \ref{combinatorics} is devoted to the proof of Prop.\ \ref{combinatorics-main-statement},
which is used to show that the representation of the Tannakian group associated to a smooth hypersurface is not the wedge power of a smaller-dimensional representation---the last remaining case where the Tannakian group could be too small, and Appendix \ref{sec:eulerian_ineq} contains inequalities that are used in Appendix \ref{combinatorics}.

\subsection{Acknowledgements}
We would like to thank Johan de Jong, Matthew Emerton, Sergey Gorchinskiy, Ariyan Javanpeykar, Caleb Ji, Shizhang Li, Benjamin Martin, Bjorn Poonen, Akshay Venkatesh, Thomas Kr\"amer, and Marco Maculan
for interesting discussions related to this project. We would like to thank the three anonymous referees for numerous helpful comments.

This work was conducted while Will Sawin served as a Clay Research Fellow, and, later, was supported by NSF grant DMS-2101491.
Brian Lawrence would like to acknowledge support from the National Science Foundation.
We met to work on this project at the Oberwolfach Research Institute for Mathematics, Columbia University, and the University of Chicago;
we would like to thank these institutions for their hospitality.

\newpage

\tableofcontents


\section{Sheaf convolution over a field}
\label{sec:sheaf}

A \emph{Tannakian category} over a field $F$ of characteristic $0$ is a rigid symmetric monoidal $F$-linear abelian category with a faithful exact tensor functor to the category of vector spaces over $F$.  The point of these conditions is that Tannakian categories are necessarily equivalent to the category of representations of some pro-algebraic group (the group of automorphisms of the functor), together with the forgetful functor to the category of vector spaces. Thus, associated to each object is some representation of this pro-algebraic group. For such a representation $V$, we refer to the image of the Tannakian group inside $GL(V)$ as the \emph{Tannakian monodromy group}.

Kr\"{a}mer and Weissauer \cite{KramerWeissauer} constructed a Tannakian category as a quotient of the category of perverse sheaves on an abelian variety over an algebraically closed field (initially of characteristic zero, but Weissauer \cite{Weissauer_VF}  later extended it to characteristic $p$), where the tensor operation is \emph{sheaf convolution}. We will use the Tannakian monodromy groups from their theory, which we call the \emph{convolution monodromy groups}, to control usual monodromy groups.

In this section, we check that these convolution monodromy groups behave similarly to the usual monodromy groups with respect to the distinction between the geometric and arithmetic monodromy groups. In the setting of the \'{e}tale fundamental group, we can define both geometric and arithmetic monodromy groups, with the geometric a normal subgroup of the arithmetic. We will check that the same works here. The Tannakian group of the category defined by Kr\"{a}mer and Weissauer will function as the geometric group, and we will define a Tannakian category of perverse sheaves over a non-algebraically closed field whose Tannakian monodromy group will function as the arithmetic group. We will verify that the geometric group is a normal subgroup of the arithmetic group.

Our construction of the Tannakian category over a non-algebraically closed field will follow a version of the strategy of Kr\"{a}mer and Weissauer, and thus will also serve as a very brief review of their construction.
 
Let $A$ be an abelian variety over a field $k$ of characteristic zero. Fix a prime $\pell$. Let $D^b_c(A, \Qpell)$ be the derived category of bounded complexes of $\pell$-adic sheaves on $A$ with constructible cohomology. Define a \emph{sheaf convolution} functor $*\colon D^b_c(A, \Qpell) \times D^b_c(A, \Qpell) \to D^b_c(A, \Qpell)$ that sends complexes $K_1, K_2$ to \[ K_1 * K_2  =R a_* ( K_1 \boxtimes K_2) \] for $a\colon A \times A \to A$ the group law.

\begin{lemma}\label{rigsymmon} $( D^b_c(A,\Qpell), *)$ is a rigid symmetric monoidal category, where the unit object is the skyscraper sheaf at $0$, and the dual of a complex $K$ is \[K^\vee= [-1]^* DK\] where $D$ is Verdier duality and $[-1]\colon A\to A$ is the inversion map. \end{lemma}

\begin{proof}  These were proved in \cite[\S2.1]{Weissauer_BN_1} (the symmetric monoidality and unit) and \cite[Proposition]{Weissauer_BN_2} (the rigidity and description of the dual). These results are stated in the case where $k$ is an algebraically closed field, but they proceed without modification in the general case.
 \end{proof}

Let $\mathcal P$ be the category of perverse sheaves on $A$ with $\Qpell$-coefficients. Let $\mathcal N$ be the subcategory of perverse sheaves with Euler characteristic zero. We similarly write $\mathcal P_{\overline{k}}$ and $\mathcal N_{\overline{k}}$ for the category of perverse sheaves on $A_{\overline{k}}$ and its subcategory of objects with Euler characteristic zero, respectively.  Let $D^b(\mathcal N)$ be the category of complexes in $D^b_c(A, \Qpell)$ whose perverse homology objects lie in $\mathcal N$. 

The Tannakian category will be constructed by combining this rigid symmetric monoidal structure with the abelian structure on the category of perverse sheaves. This requires modifying the category of perverse sheaves slightly because it is not quite stable under convolution. Instead one verifies that it is stable under convolution ``up to" $\mathcal N$, i.e.\ that the convolution of two perverse sheaves has all perverse homology objects in nonzero degrees lying in $\mathcal N$. This lets us give $\mathcal P/\mathcal N$ the structure of a rigid symmetric monoidal $\Qpell$-linear abelian category.

\begin{lemma}\label{categorical-facts}
\begin{enumerate}

\item Perverse sheaves on $A$ have nonnegative Euler characteristics.

\item $\mathcal N$ is a thick subcategory of $\mathcal P$ (i.e.\ it is stable under taking subobjects, quotients, and extensions).

\item $D^b(\mathcal N)$ is a thick subcategory of $D^b_c( A, \Qpell)$ (i.e.\ for any distinguished triangle with two objects in $D^b(\mathcal N)$, the third one is in $D^b(\mathcal N)$ as well).

\item For $K_1, K_2 \in D^b_c( A, \Qpell)$, if $K_1$ or $K_2$ lies in $D^b(\mathcal N)$, then $K_1 * K_2$ lies in $D^b(\mathcal N)$.

\item For $K_1, K_2 \in \mathcal P$,  ${}^p \mathcal H^i( K_1 * K_2) \in \mathcal N$ if $i \neq 0$.

\item Convolution descends to a functor \[D^b_c( A, \Qpell) /D^b(\mathcal N) \times D^b_c( A, \Qpell) /D^b(\mathcal N)\to D^b_c( A, \Qpell) /D^b(\mathcal N).\]

\item The essential image of $\mathcal P$ in $D^b_c( A, \Qpell) /D^b(\mathcal N)$ is equivalent to $\mathcal P/N$.

\item The essential image of $\mathcal P$ in $D^b_c( A, \Qpell) /D^b(\mathcal N)$ is stable under convolution.

\item  $(\mathcal P/\mathcal N, *)$ is a rigid symmetric monoidal $\Qpell$-linear abelian category.

\end{enumerate}

\end{lemma}

\begin{proof} It suffices to check the first five statements after passing to $A_{\overline{k}}$, where they were checked in \cite[Corollary 1.4]{FraneckiKapranov}, \cite[Prop \ 10.1 and preceding paragraph]{KramerWeissauer}, and \cite[Lemma 13.1]{KramerWeissauer}. The remainder follow from the first five by the arguments in, e.g., \cite[p. 90, Theorem 5.1, Theorem 5.2]{KramerSummary}.

\end{proof}

We will work with lisse rank-one sheaves on an abelian variety. It will be convenient to parametrize them by representations of the fundamental group.

\begin{defi} Let $A$ be an abelian variety over a field $k$. Fix a continuous character $\chi\colon \pi_1^{et}(A_{\overline{k}}) \to \overline{\mathbb Q}_\pell^\times$. We define the \emph{character sheaf} $\mathcal L_\chi$ to be the unique rank-one sheaf on $A_{\overline{k}}$ whose monodromy representation is $\chi$.\end{defi}

We also have a canonical way to descend these sheaves to $A_k$:

\begin{defi} Let $A$ be an abelian variety over a field $k$. Let $\chi$ be a character of $\pi_1^{et}(A_{\overline{k}})$ that is $\Gal(\overline{k}|k)$-invariant. We define the \emph{character sheaf} $\mathcal L_\chi$ to be the unique lisse rank-one sheaf on $A_k$ whose associated representation of the fundamental group restricts to $\chi$ on $\pi_1^{et}(A_{\overline{k}})$ and whose stalk at the identity has trivial Galois action.

In other words, we take the splitting of the exact sequence $1 \to \pi_1^{et}(A_{\overline{k}}) \to \pi_1^{et}(A_k) \to \Gal(\overline{k}|k) \to 1$ induced by the identity at $1$, and use it to extend $\chi$ from $\pi_1^{et}(A_{\overline{k}}) $ to $\pi_1^{et}(A_k)$.   \end{defi}

For $\chi$ a character of $\pi_1^{et}(A_{\overline{k}})$, 
let $\mathcal P^\chi$ be the subcategory of $\mathcal P$ consisting of perverse sheaves $K$ with $H^i(A_{\overline{k}}, Q\otimes \mathcal L_\chi)=0$ for all $i \neq 0$ and $Q$ a subquotient of $K_{\overline{k}}$, and let  $\mathcal N^\chi= \mathcal P^\chi \cap \mathcal N$. 

\begin{lemma}

\begin{enumerate} 

\item The essential image of $\mathcal P^\chi$ in $\mathcal P/\mathcal N$ is equivalent to $\mathcal P^{\chi}/ \mathcal N^{\chi}$.

\item $\mathcal P^{\chi}/ \mathcal N^{\chi}$ contains the unit and is stable under convolution and duality.

\item $K \mapsto H^0(A_{\overline{k}}, K \otimes \mathcal L_\chi)$ is an exact tensor functor from $\mathcal P^\chi/ \mathcal N^\chi$ to $\Qpell$-vector spaces.

\item The category $\mathcal P^{\chi}/ \mathcal N^{\chi}$, convolution, and the functor $H^0(A_{\overline{k}}, K \otimes \mathcal L_\chi)$ are a rigid symmetric monoidal $\Qpell$-linear abelian category with a faithful exact tensor functor to $\Qpell$-vector spaces.

\end{enumerate}

\end{lemma}

\begin{proof}

(1) follows from \cite[Lemma 12.3]{KramerWeissauer} and the fact that $\mathcal P^\chi$, by construction, is a thick subcategory.

The claims in (2) may be checked after passing to an algebraically closed field. To check that it contains the unit, we must check that the skyscraper sheaf at zero has cohomology only in degree zero, which is obvious. To check that it is closed under duality, it suffices to observe that 
\begin{eqnarray*} H^i ( A_{\overline{k}}, [-1]^* DQ  \otimes \mathcal L_\chi) &=& H^i ( A_{\overline{k}},  DQ  \otimes [-1]_* \mathcal L_\chi) = H^i ( A_{\overline{k}},  DQ  \otimes \mathcal L_\chi^{-1} ) \\ 
&=& H^i ( A_{\overline{k}},  D(Q  \otimes  \mathcal L_\chi)) = H^{-i} ( A_{\overline{k}}, Q \otimes \mathcal L_\chi)^\vee \end{eqnarray*}
so if one vanishes for all $i \neq 0$ the other does. That it is closed under convolution is checked in \cite[Theorem 13.2]{KramerWeissauer}.

The claims in (3) may be checked after passing to an algebraically closed field, where they are proved in \cite[Theorem 13.2]{KramerWeissauer}. Specifically, \cite[Proposition 4.1]{KramerWeissauer} reduces this to the case where $\chi$ is trivial. In this case, exactness follows from the long exact sequence of cohomology, which reduces to a short exact sequence because higher and lower cohomology groups vanish, and tensorness follows from the K\"{u}nneth formula, which gives \[H^0(A_{\overline{k}} , a_* (K_1 \boxtimes K_2) ) = H^0(A_{\overline{k}} \times A_{\overline{k}}, K_1 \boxtimes K_2) = H^0 ( A_{\overline{k}}, K_1) \otimes H^0( A_{\overline{k}}, K_2) \] again using the vanishing of higher and lower cohomology. 

(4) The category $\mathcal P^{\chi}/ \mathcal N^{\chi}$ is rigid symmetric monoidal by part (2) and Lemma \ref{categorical-facts}(4). It is $\Qpell$-linear abelian because it is the quotient of a $\Qpell$-linear abelian category by a thick subcategory. The functor is an exact tensor functor by part (3), and is faithful since exact $\Qpell$-linear tensor functors between rigid abelian $\Qpell$-linear tensor categories are automatically faithful if the endomorphisms of the unit are $\Qpell$.
\end{proof}

By \cite[Theorem 1.1]{KramerWeissauer}, for any $K \in \mathcal{P}$, there exists $\chi$ such that $K \in \mathcal{P}^{\chi}$. 
We use the fiber functor $K \mapsto H^0(A_{\overline{k}}, K \otimes \mathcal L_\chi)$ on $\mathcal{P}^{\chi}$
to define the ``Tannakian group'' of $K$.  This group is independent of the choice of $\chi$, since
any two fiber functors on the same Tannakian category over an algebraically closed field are equivalent \cite[Theorem 3.2]{DeligneMilne}, and give rise to equivalent Tannnakian groups.

For $A$ an abelian variety over a field $k$ with algebraic closure $\overline{k}$, we say that a perverse sheaf $K$ on $A$ is \emph{geometrically semisimple} if its pullback to $A_{\overline{k}}$ is a sum of irreducible perverse sheaves. 

\begin{lemma}\label{maps-of-semisimples-are-easy} Let $K_1$ and $K_2$ be geometrically semisimple perverse sheaves on $A$.  Then $\operatorname{Hom}_{\mathcal P/\mathcal N} (K_1,K_2)$ is the quotient of the space of homorphisms $K_1 \to K_2$ by the subspace of homomorphisms factoring through  an element of $\mathcal N$. \end{lemma}

\begin{proof} Without loss of generality, we may assume that $K_1$ and $K_2$ are indecomposable.

We first check that the set of isomorphism classes of irreducible components of the pullback $K_{1,\overline{k}}$ of $K_1$ to $\overline{k}$ forms a single $\Gal(\overline{k}|k)$-orbit. Suppose not; then we can fix a $\Gal(\overline{k}|k)$-orbit and consider an endomorphism of $K_{1,\overline{k}}$ defined as the idempotent projector onto the sum of all irreducible components in that orbit. This endomorphism is, by construction, stable under $\Gal(\overline{k} | k)$. Because $K_1$ is perverse, $R{\mathcal H}om(K_1,K_1)$ is concentrated in degrees $\geq 0$ and thus  \[\Hom (K_1, K_1) = H^0 ( A, R{ \mathcal H}om (K_1,K_1) ) = H^0 ( A_{\overline{k}}, R{ \mathcal H}om (K_1,K_1))^{\Gal(\overline{k}|k)} = \Hom (K_{1,\overline{k}}, K_{1,\overline{k}} )^{\Gal(\overline{k}|k)}\] and hence this endomorphism arises from a nontrivial idempotent endomorphism of $K_1$, contradicting the irreducibility of $K_1$. 


It follows that either all irreducible components of $K_{1,\overline{k}}$ are in $\mathcal N_{\overline{k}}$ or none of them are. The same is true for $K_{2,\overline{k}}$ by the same argument.


If all irreducible components of $K_{1,\overline{k}}$ or $K_{2,\overline{k}}$ are in $\mathcal N_{\overline{k}}$, then $K_1$ or $K_2$ is in $\mathcal N$, so maps in the quotient category are zero and all maps factor through elements of $\mathcal N$, and the statement holds.

Thus, we may assume that no irreducible components of $K_1$ and $K_2$ are in $\mathcal N$. By definition, $\operatorname{Hom}_{\mathcal P/\mathcal N} (K_1,K_2)$ is the limit of $\Hom (K_1', K_2')$ where $K_1'$ is a subobject of $K_1$ whose quotient lies in $\mathcal N$ and $K_2'$ is a quotient of $K_2$ by an object in $\mathcal N$. By assumption, we must have $K_1'= K_1$ and $K_2'=K_2$, so $\operatorname{Hom}_{\mathcal P/\mathcal N} (K_1,K_2)= \operatorname{Hom}(K_1,K_2)$. Again because no irreducible components of $K_1$ and $K_2$ lie in $\mathcal N$, no nonzero map from $K_1$ to $K_2$ factors through an object in $\mathcal N$, so the statement holds in this case as well. \end{proof}

\begin{lemma}\label{atc} \begin{enumerate}

\item The full subcategory of $\mathcal P^\chi/ \mathcal N^\chi$ consisting of geometrically semisimple perverse sheaves is a Tannakian subcategory of $\mathcal P^\chi/ \mathcal N^\chi$.

\item  The full subcategory of $\mathcal P_{\overline{k} }^\chi/\mathcal N_{\overline{k}}^\chi$ consisting of summands of the pullbacks to $A_{\overline{k}}$ of geometrically semisimple elements of $\mathcal P^\chi/ \mathcal N^\chi$ on $A_{k}$ is a Tannakian subcategory of $\mathcal P^\chi_{\overline{k} }/ \mathcal N^\chi_{\overline{k} }$.

\end{enumerate}

\end{lemma}

\begin{proof} To prove part (1), we must check that this subcategory contains the unit, and is closed under kernels, cokernels, direct sums, convolution, duals. The unit, direct sum, and dual steps are straightforward. For kernels and cokernels, by Lemma \ref{maps-of-semisimples-are-easy} it suffices to check that kernels and cokernels of morphisms between geometrically semisimple sheaves are geometrically semisimple, which is clear. For convolution, this follows from Kashiwara's conjecture, proven in \cite{DrinfeldKashiwara} and \cite{GaitsgorydeJong}.

To prove part (2), the argument for the units, direct sum, convolution, and dual steps is straightforward, again using Kashiwara's conjecture. For the kernel and cokernel, the key is that summands of the pullback of geometrically semisimple perverse sheaves on $A_k$ remain semisimple. Semisimplicity allows us to apply Lemma \ref{maps-of-semisimples-are-easy} to reduce to kernels and cokernels of honest morphisms and then shows that those kernels and cokernels are themselves summands. \end{proof}
 
Fix an abelian variety $A$ over $k$ and a character $\chi$ of $\pi_1^{et}(A_{\overline{k}})$.

Let $G_k$ be the Tannakian fundamental group of the full subcategory of $\mathcal P^\chi/ \mathcal N^\chi$ consisting of geometrically semisimple perverse sheaves. 

Let $G_{\overline{k}}$ be the Tannakian fundamental group of the full subcategory of $\mathcal P_{\overline{k} }^\chi/\mathcal N_{\overline{k}}^\chi$ consisting of summands of the pullbacks to $A_{\overline{k}}$ of geometrically semisimple perverse sheaves on $A_{k}$.

Let $\operatorname{Gal}_k$ be the Tannakian group of the category of $\pell$-adic Galois representations over $k$.

\begin{lemma}\label{tannakian-quotient} The group $G_{\overline{k}}$ is a normal subgroup of $G_k$, with quotient $\operatorname{Gal}_k$. \end{lemma}

\begin{proof} 
There is a functor from the Tannakian category of Galois representations over $k$ to the geometrically semisimple objects of $\mathcal P^{\chi}/\mathcal N^\chi$ that sends a Galois representation to the corresponding skyscraper sheaf at the identity. 

There is a functor from geometrically semisimple perverse sheaves on $A_k$ to summands of pullbacks of geometrically semisimple perverse sheaves to $A_{\overline{k}}$, given by pullback to $A_{\overline{k}}$.

Because these functors are both exact tensor functors, they define homomorphisms $G_{\overline{k}} \to G_k \to \operatorname{Gal}_k$. We wish to show that this is an exact sequence of groups, i.e.\ that $G_{\overline{k}}$ is a normal subgroup of $G_k$ whose quotient is $\operatorname{Gal}_k$. To do this, we check the criteria of \cite[Theorem A.1]{EsnaultHaiSun} (which incorporate earlier results of \cite[Proposition 2.21]{DeligneMilne}).

First, to check that $G_k \to \operatorname{Gal}_k$ is surjective, it suffices to check that the functor from Galois representations to skyscraper sheaves at the origin is full, and that a subquotient of a skyscraper sheaf at the origin is a skyscraper sheaf at the origin. These are both easy to check.

Second, to check that $G_{\overline{k}} \to G_k$ is a closed immersion, we must check that every representation of  $G_{\overline{k}}$ is a subquotient of a pullback to $G_{\overline{k}}$ of a representation of $G_k$. This is automatic, as the Tannakian category of representations of $G_{\overline{k}}$ is defined to consist of perverse sheaves that are summands of pullbacks of perverse, geometrically semisimple sheaves on $A_k$ that lie in $P^{\chi}$, which by definition are representations of $G_k$, and because all summands are subquotients.

Third we must check that a perverse sheaf on $A_k$ is a skyscraper sheaf at the origin if and only if is trivial when pulled back to $A_{\overline{k}}$. This is obvious.

Fourth, we must check that given a geometrically semisimple perverse sheaf on $A_k$, its maximal trivial subobject over $A_{\overline{k}}$ (i.e.\ the maximal sub-perverse sheaf that is a skyscraper sheaf at the origin) is a subobject over $A_k$. By duality, it is equivalent to check this with quotient objects, where the maximal trivial quotient is simply the stalk at zero of the zeroth homology and hence is certainly defined over $k$.

The fifth condition is simply the second condition with ``subquotient" replaced with ``subobject". This follows again because summands are subobjects.\end{proof}

It is likely possible to prove the analogous theorem, without the ``geometrically semisimple" conditions in the definitions of the key Tannakian categories, by a similar but more complicated argument. However, this additional level of generality is not needed for our paper, and so we did not pursue this.

Using the fact that $G_{\overline{k}}$ is a normal subgroup of $G_k$, we will show that the Galois action on our fiber functor normalizes the Tannakian monodromy. 

\begin{lemma}\label{normalizer-lemma} Let $A$ be an abelian variety over a field $k$. Let $\chi$ be a character of $\pi_1^{et}(A_{\overline{k}})$ that is $\Gal(\overline{k}|k)$-invariant. Let $\mathcal L_\chi$ be the associated character sheaf.

Let $K$ be a geometrically semisimple perverse sheaf on $A$ such that $H^i( A_{\overline{k}}, K \otimes \mathcal L_\chi)$ vanishes for $i\neq 0$. Then the action of $\Gal(\overline{k}|k)$ on $H^0( A_{\overline{k}}, K \otimes \mathcal L_\chi)$ normalizes the commutator subgroup of the identity component of the geometric convolution monodromy group of $K$.


\end{lemma}

\begin{proof} 
We first prove that the action of $\Gal(\overline{k}|k)$  normalizes the geometric convolution monodromy group of $K$. For this, note that $\Gal(\overline{k}|k)$  acts by automorphisms of the fiber functor $H^i_c( A_{\overline{k}}, K \otimes \mathcal L_\chi)$ of the arithmetic Tannakian category, 
giving a homomorphism $\Gal(\overline{k}|k) \rightarrow G_k$. Since the geometric convolution monodromy group is a normal subgroup of the arithmetic Tannakian group, it follows that  $\Gal(\overline{k}|k)$ normalizes the geometric convolution monodromy group.

It follows that $\Gal(\overline{k}|k)$ also normalizes the commutator subgroup of the identity component, since the commutator subgroup of the identity component is a characteristic subgroup.
%
 \end{proof}

\newpage

\section{The convolution monodromy group of a hypersurface}
\label{sec:tan}

In this section, we fix an abelian variety $A$ of dimension $n$ and a smooth hypersurface $H$ in $A$, which we will take except where noted to be defined over the complex numbers. Let $i\colon H\to A$ be the inclusion.  Let $G_H$ be the convolution monodromy group of the perverse sheaf $i_* \mathbb Q [n-1]$. The main goal of this section is to compute $G_H$.

To begin, we compute various Euler characteristics of $H$ - its arithmetic Euler characteristic, its topological Euler characteristic, and the Euler characteristics of the wedge powers of its cotangent sheaf. Using these, we will calculate the dimension, and Hodge numbers, of the cohomology of $H$ with coefficients in a rank-one lisse sheaf. These Hodge numbers will be used to compute the convolution monodromy group, and also used in later sections.

\begin{lemma}\label{Euler-characteristic-line-bundle}  Let $L$ be a line bundle on $A$.  We have
\[ \chi(A, {L}) = c_1(\mathcal{L})^{n} / {n !}  \]
\end{lemma} 
\begin{proof} By Hirzebruch-Riemann-Roch, the Euler characteristic of the coherent sheaf $ L$ is the integral of its Chern character against the Todd class. By definition, the Chern character of $L$ is $e^{ c_1[L]}= \sum_{k=0}^n c_1(L)^k/k!$. Because the tangent bundle of $A$ is trivial, its Todd class is $1$. Integrating is equivalent to taking the degree $n$ term, which is $c_1({L})^{n} / {n !}$.
\end{proof} 

\begin{lemma}\label{arithmetic-Euler-formula} The arithmetic Euler characteristic $\chi(H, \mathcal O_H)$ of $H$ is $(-1)^{n-1} [H]^{n} / n!$. \end{lemma}

\begin{proof}  Using the exact sequence $0 \to \mathcal O_A (-H) \to \mathcal O_A \to \mathcal O_H \to 0 $, we observe that
\[ \chi(H, \mathcal O_H) = \chi(A, \mathcal O_A) - \chi(A, \mathcal O_A(-H)) = 0 -  (- [H])^n /n! = (-1)^{n-1} [H]^n/n!\]
by Lemma \ref{Euler-characteristic-line-bundle}. \end{proof} 

\begin{lemma}\label{topological-Euler-formula} The topological Euler characteristic of $H$ is $(-1)^{n-1}  [H]^{n}$. \end{lemma}

\begin{proof} The topological Euler characteristic of $H$ is the top Chern class of the tangent bundle of $H$. Using the exact sequence $0 \to \mathcal O( - [H]) \to \Omega^1_A \to \Omega^1_H \to 0$, and the fact that all Chern classes of $\Omega^1_A$ vanish, we see that the top Chern class of $\Omega^1_H$ is $(-1)^{n-1}  [H]^{n}$.  \end{proof}  

Motivated by Lemma \ref{arithmetic-Euler-formula}, we define the \emph{degree} $d$ of $H$ to be $\frac{ [H]^n}{n!}$, which is always a positive integer for $H$ an ample hypersurface.

The Hodge numbers of $H$ can be computed in terms of Eulerian numbers.  For a general reference on Eulerian numbers in combinatorics, see \cite[Chap.\ 1]{Petersen}.

\begin{lemma}\label{lem-Euler-formula}We have
\begin{equation}\label{hodge-number-Euler-formula} \chi( H, \Omega^i_H) = (-1)^{n-1-i} d  A (n, i)  \end{equation}
where $A(n,i)$ is the Eulerian number. 
\end{lemma}

\begin{proof}
By \cite[p.\ 272, \emph{Some recursion formulas}]{Weissauer_TC}, we have
\[ \chi ( H, \Omega^i_H) = (-1)^n \sum_{j=0}^i {n \choose i-j} (-1)^j  ( j^n - (j+1) ^n) \chi(\mathcal L) \] \[ =(-1)^n  \sum_{j=0}^{i+1} (-1)^j  \left( {n \choose i-j} + {n \choose i-j+1}\right)  j^n \chi(\mathcal L)=(-1)^n  \sum_{j=0}^{i+1} (-1)^j   {n+1 \choose i+1-j}   j^n \chi(\mathcal L) \]
We now use the combinatorial identity (\cite[Cor.\ 1.3]{Petersen}) 
\[ (-1)^n  \sum_{j=0}^{i+1} (-1)^j   {n+1 \choose i+1-j}  j^n  = (-1)^{ n-1 -i} A (n, i) \]
and Lemma \ref{Euler-characteristic-line-bundle}
\[\chi(\mathcal L) =d\]
to derive \eqref{hodge-number-Euler-formula}. 
\end{proof}

Let $N = (n!)d= [H]^n$.

Recall from the introduction that $a(i)$ is the sequence \[ 1,5,20, 76, 285, 1065,  \dots \] satisfying \[ a(1)=1, a(2)=5, a(i+2) = 4 a(i+1) +1 - a(i) \] (OEIS  A061278).

\begin{theo}\label{Tannakian-group-calculation}

Assume that $H$ is not equal to the translate of $H$ by any nontrivial point of $A$ and $n>2$. Assume also that neither of the following holds:

\begin{enumerate}

%

\item $n=2$ and $d=28$.

\item $n=3$ and $d ={ a(i)+a(i+1) \choose a(i) } / 6$ for some $i \geq 2$.

\end{enumerate}

Then $G_H$ contains as a normal subgroup either $SL_{N}, Sp_{N}$, or $SO_{N}$. If $H$ is not equal to any translate of $[-1]^* H$ then the $SL_{N}$ case holds. If $H$ is equal to such a translate, then if $n$ is even $ Sp_{N}$ holds and if $n$ is odd $SO_{N }$ holds. \end{theo}

\begin{remark} In the exceptional cases, there are only a few other possibilities for $G_H$.
Suppose $G_H$ does not contain $SL_{N}$, $Sp_{N}$, or $SO_{N}$.
Then in case (1)
, $G_H$ contains $E_7$ acting by its $56$-dimensional representation. In (2), $G_H$ contains as a normal subgroup $SL_{a(i)+a(i+1)}$ acting by the representation $\wedge^{a(i)}$.  \end{remark} 

The proof occupies the remainder of this section.

We will call the representation of $G_H$ associated to the object $i_* \mathbb Q [n-1]$ the \emph{distinguished} representation.

\begin{lemma}\label{Tannakian-dimension-formula} The dimension of the distinguished representation of $G_H$ is $N$. \end{lemma}

\begin{proof} By construction, the dimension of the representation associated to any object in the Tannakian category is the Euler characteristic of the corresponding perverse sheaf, which is $(-1)^{n-1}$ times the topological Euler characteristic of $H$. This now follows from Lemma \ref{topological-Euler-formula}.\end{proof} 

\begin{lemma}\label{lie-irreducible} Assume that $H$ is not equal to the translate of $H$ by any nontrivial point of $A$. Then the distinguished representation of $G_H$ is an irreducible minuscule representation of the Lie algebra of $H$. \end{lemma}

\begin{proof}Because $H$ is smooth, the characteristic cycle of  $i_{*} \mathbb Q [ n]$ is simply the conormal bundle of $H$ with multiplicity $1$, hence has a single irreducible component, with multiplicity one.  The representation is then irreducible minuscule by \cite[Corollary 1.0]{Microlocal1}
\end{proof}

%
%

\begin{lemma}\label{Tannakian-group-simple} Assume $H$ is not translation-invariant by any nontrivial point of $A$ and $n>2$. Then the identity component of $G_H$ is simple modulo its center.  \end{lemma}

\begin{proof} We follow the argument of \cite[Theorem 6.2.1]{Microlocal2}, with minor modifications. That theorem is not directly applicable because it assumes that $H$ is symmetric (i.e.\ stable under inversion), which is not necessarily true here. Thus we restate the proof, which we can also simplify somewhat because our assumption that $H$ is smooth is stronger than the analogous assumption in \cite{Microlocal2}. First, we review some notation and terminology from \cite{Microlocal2}.

The proof relies on the notion of the characteristic cycle of a perverse sheaf. Classically, the characteristic cycle of a perverse sheaf on a variety $X$ of dimension $n$ is an effective Lagrangian cycle in the cotangent bundle $T^* X$ of $X$. In other words, it is a nonnegative-integer-weighted sum of irreducible $n$-dimensional subvarieties of $T^* X$ whose tangent space at a generic point is isotropic for the natural symplectic form on $T^* X$. Any such subvariety is automatically the conormal bundle to an irreducible subvariety of $X$, of arbitrary dimension, i.e. its support. 

For $A$ an abelian variety, because $T^* A$ is a trivial bundle, we can express it as a product $A \times H^0 ( A, \Omega^1_A)$. The projection onto the second factor is called the Gauss map. Kr\"{a}mer considers an irreducible component \emph{negligible} if its image under the Gauss map is not dense, and a cycle \emph{clean} if none of its components are negligible \cite[Definition 1.2.2]{Microlocal2}. He defines $cc(K)$ for a perverse sheaf $K$ to be the usual characteristic cycle but ignoring any negligible components, making it automatically clean \cite[Definition 2.1.1]{Microlocal2}.

The degree of a cycle is the degree of the Gauss map restricted to that cycle. It is manifestly a sum over components of the degree of the Gauss map on that component, which vanishes if and only if the component is negligible.

A clean cycle is determined by its restriction to any open set $ U \subseteq H^0 ( A, \Omega^1_A)$. In particular, given two such cycles $\Lambda_1 ,\Lambda_2$, because \[\dim \Lambda_1 = \dim \Lambda_2 = n =\dim H^0 ( A, \Omega^1_A),\] one can find an open set over which both $\Lambda_1$ and $\Lambda_2$ are finite. The fiber product $\Lambda_1 \times_U \Lambda_2$ then maps to $U$ by the obvious projection and to $A$ by composing the two projections $\Lambda_1 \to A, \Lambda_2 \to A$ with the multiplication map $A \times A \to A$. Hence $\Lambda_1 \times_U \Lambda_2$ maps to $A \times U$. Its image has a unique clean extension to $A \times H^0 ( A, \Omega^1_A)$.
Kr\"{a}mer defines this $\Lambda_1 \circ \Lambda_2$ to be this extension \cite[Example 1.3.2]{Microlocal2}.

A key property of this convolution product is that $\deg(\Lambda_1 \circ \Lambda_2 ) = \deg(\Lambda_1) \deg(\Lambda_2)$; as a consequence, if $\Lambda_1,\Lambda_2$ are clean and nonzero, then  $\Lambda_1 \circ \Lambda_2$ is nonzero as well.

We are now ready to begin the argument.

Assume for contradiction that the identity component of $G_H$ is not simple modulo its center. Then its Lie algebra is not simple modulo its center. By \cite[Proposition 6.1.1]{Microlocal2}, it follows from this that there exist $m \in \mathbb N$ and effective clean cycles $\Lambda_1, \Lambda_2$ on $A$, with $\deg (\Lambda_i) > 1$, such that
\[ [m]_* cc(i_* \mathbb Q[n]) = \Lambda_1 \circ \Lambda_2 \]
where  $[m]$ is the multiplication-by-$n$ map. Because  $H$ is smooth, the characteristic cycle of $i_* \mathbb Q[n]$ is simply the conormal bundle $\Lambda_H$ of $H$, which is irreducible. Its degree is $d \cdot n!$, because the degree of the Gauss map of the conormal bundle to $H$ is the sum of the multiplicities of vanishing of a general $1$-form on $H$, which is the Euler characteristic of $H$, which is $d \cdot n!$. In particular, this degree is nonzero, so $cc(i_* \mathbb Q[n]) = \Lambda_H$.

Because $H$ is not translation-invariant, the map from $H$ to its image under $[m]$ is generically one-to-one, and so $[m]_*\Lambda_H $ is an irreducible cycle with multiplicity one in the cotangent bundle of $A$. This implies $\Lambda_1$ and $\Lambda_2$ are irreducible: If not, say if $\Lambda_1 = \Lambda_1^a + \Lambda_1^b$, we would have
\[ \Lambda_H = \Lambda_1 \circ \Lambda_2 = \Lambda_1^a \circ \Lambda_2 + \Lambda_1^b \circ \Lambda_2 \] with both $\Lambda_1^a \circ \Lambda_2 $ and $ \Lambda_1^b \circ \Lambda_2 $ nonzero, contradicting irreducibility.

It follows that $\Lambda_1$ and $\Lambda_2$  must be the conormal bundles $\Lambda_{Z_1}$ and $\Lambda_{Z_2}$ of varieties $Z_1$ and $Z_2$.  Because $\deg (\Lambda_i) >1$, neither $Z_1$ nor $Z_2$ can be a point, as the conormal bundle to a point is simply an affine space, and its Gauss map is an isomorphism, and thus has degree $1$.

Let $Y$ be the image of $H$ under $[m]$. By \cite[Lemma 5.2.2]{Microlocal2} there is a dominant rational map from $Y$ to $Z_1$ (say), and thus a dominant rational map from $H$ to $Z_1$. Because $H$ is smooth and $Z_1$ is a subvariety of an abelian variety, this dominant rational map automatically extends to a surjective morphism \cite[Theorem 4.4.1]{Neron}. 
Moreover, by the Lefschetz hyperplane theorem (since $n>2$), $A$ is the Albanese of $H$, so the surjective morphism $f_1 \colon H \to Z_1 \subseteq A$ extends to a homomorphism $g_1\colon  A \to A$, giving a commutative diagram \[ \begin{tikzcd} H \arrow[r, "f_1",two heads] \arrow[d, "i", hookrightarrow] & Z_1\arrow[d,hookrightarrow] \\ A \arrow[r,"g_1"] &A \end{tikzcd}\] Let $B_1$ be the image of $g_1$. 
Because $f_1$ is surjective, $Z_1 \subseteq B_1$.

If $Z_1 = B_1$ then $Z_1$ is an abelian variety, and the conormal bundle to any nontrivial abelian subvariety of $A$ has Gauss map of degree $0$, contradicting $\deg (\Lambda_i) > 1$. Otherwise, by commutativity of the diagram, $H \subseteq g_1^* Z_1$. Because $H$ is a hypersurface, it is a maximal proper subvariety of $A$, so $H = g_1^* Z_1$. This contradicts ampleness of $H$ unless $g_1$ is finite, and contradicts $H$ not being translation-invariant unless $g_i$ is an isomorphism. This means $Z_1 $ and $H$ are isomorphic as subvarieties of an abelian variety. Thus $\deg (\Lambda_{Z_1} ) = \deg ( T^*_{Z_1} Z_1) = \deg(T^*_H H)$ and so because $\deg (\Lambda_{Z_1} ) \deg(\Lambda_{Z_2} ) =\deg (T^*_H H)$, we have $\deg (\Lambda_{Z_2}=1)$, contradicting $\deg (\Lambda_i) > 1$. 

Because we have a contradiction in every case, we have shown that $G$ is simple modulo its center.\end{proof}

\vspace{10pt}

\begin{lemma}\label{list-of-groups}   Assume that $H$ is not equal to the translate of $H$ by any nontrivial point of $A$. Then the commutator subgroup of the identity component of $G_H$, together with its distinguished representation, is one of the following:

\begin{enumerate}

\item $SL_N, Sp_N$ or $SO_N$, with its standard representation distinguished.

\item $SO_m$ with one of its spin representations distinguished, or $E_6$ or $E_7$ with one of its lowest-dimensional nontrivial representations distingusihed.

\item $SL_m$ with the representation $\wedge ^k$ distinguished for some $2 \leq k \leq m/2$. \end{enumerate}

\end{lemma}

\begin{proof} It follows from the Lemma \ref{Tannakian-group-simple} that the commutator subgroup of the identity component of $G_H$ is a simple Lie group. Furthermore, from Lemma \ref{lie-irreducible} its distinguished representation must be irreducible and minuscule. But the above is an exhaustive list of minuscule representations of simple Lie groups (see e.g.\ \cite[p. 7]{Microlocal1}).\end{proof}

\begin{lemma}\label{duality-lemma} Assume that $G_H$ contains as a normal subgroup one of $SL_N$, $Sp_N$, or $SO_N$. Then it contains $SL_N$ only if $H$ is not equal to any translate of $[-1]^* H $, it contains $Sp_N$ only if $n$ is even, and it contains $SO_N$ only if $n$ is odd. \end{lemma}

\begin{proof} 

Note first that the distinguished representation of any subgroup of $GL_N$ which contains $Sp_N$ or $SO_N$ as a normal subgroup is equal to the tensor product of its dual representation with a one-dimensional representation, since it is contained in the normalizer $GSp_N$ or $GO_N$ respectively. Conversely, if $N>2$ then the distinguished representation of any subgroup of $GL_N$ which contains $SL_N$ as a normal subgroup is not equal to the tensor product of its dual representation with any one-dimensional representation.

Translating into the language of the Tannakian category, we see that under this assumption, $G_H$ contains $SL_N$ as a normal subgroup if, and only if, the perverse sheaf $i_* \mathbb Q [n-1]$ is not isomorphic, up to negligible factors, to the convolution of its dual $[-1]^* D i_* \mathbb Q [n-1]= [-1]^* i_* \mathbb Q [n-1]$ with any perverse sheaf corresponding to a one-dimensional representation. Now perverse sheaves corresponding to a one-dimensional representation are always skyscraper sheaves \cite[Proposition 10.1]{KramerWeissauer}, and convolution with a skyscraper sheaf is equivalent to translation, so it is equivalent to say that $i_* \mathbb Q[n-1]$ is not isomorphic, up to negligible factors, to any translate of $[-1]^* i_* \mathbb Q[n-1]$. Because $i_* \mathbb Q[n-1]$ and $[-1]^* i_* \mathbb Q[n-1]$ are both irreducible perverse sheaves,  there can be no negligible factors, and so this happens if and only if they are isomorphic (up to translation). Because $i_* \mathbb Q[n-1]$ and $[-1]^* i_* \mathbb Q[n-1]$ are each constant sheaves on their support, they are isomorphic (up to translation) if and only if their supports are equal (up to translation), which happens exactly when $H$ is equal to a translate of $[-1]^* H$. This handles the $SL_N$ case.

The argument to distinguish $Sp_N$ and $SO_N$ is identical to \cite[Lemma 2.1]{KWTheta}, which is stated only in the case where $H$ is a theta divisor, but the assumption is never used in the proof, except that they write $g!$ instead of $N$ since, for a theta-divisor, $N=g!$.
\end{proof}

To prove the main theorem, it remains to give a complete list of $n,d$ for which $G_H$ can contain as a normal subgroup one of the groups in Lemma \ref{list-of-groups}, cases (2) or (3).

Case (2) is relatively easy as for these groups the dimensions have a special form.

\begin{lemma}
\label{minuscule_case_2}
If $n>2$, we cannot have the commutator subgroup of the identity component of $G_H$ be  $SO_m$ with a spin representation distinguished, or $E_6$ or $E_7$ with one of their lowest-dimensional nontrivial irreducible representations distinguished. 
\end{lemma}
\begin{proof} For $n>2$, $(n!)d$ is always a multiple of $3!=6$. However, the stated representations cannot have dimension a multiple of $6$, contradicting Lemma \ref{Tannakian-dimension-formula}. Indeed, the spin representations have dimension a power of $2$, while the lowest-dimensional representations of $E_6$ and $E_7$ have dimension $27$ and $56$ respectively, and none of these are multiples of $6$.
\end{proof}

The remainder of the section is devoted to restricting the case of wedge powers. We will obtain further numerical obstructions by introducing a Hodge torus into our convolution monodromy group. The action of the Hodge torus is obtained using the Galois action on $H^{n-1}(X_{\overline{\kappa}}, \mathcal L_\chi)$ for a field $\kappa$ and $p$-adic Hodge theory, and thus relies on Lemma \ref{normalizer-lemma} and our earlier construction of a Tannakian category of perverse sheaves over a non-algebraically closed field, as well as a calculation of Hodge-Tate weights.

An alternate approach should be possible, using classical Hodge theory and a Tannakian category of mixed Hodge modules on $A_{\mathbb C}$, as suggested by \cite[Example 5.2]{KramerWeissauer}, but we did not take this approach as we found the Galois action useful elsewhere in the argument.

We first review some $p$-adic Hodge theory. For $V$ a representation of the Galois group of a $p$-adic field $K$ with coefficients in $K$, $\mathbb C_p$ the completion of the algebraic closure of $K$, and $\mathbb C_p(q)$ the Tate twist by $q$, the \emph{Hodge-Tate weights} of $V$ are the integers $q$ such that $( V \otimes_{ K} \mathbb C_p)^{ \Gal( \overline{K} /K)} \neq 0$, and the \emph{multiplicity} of the Hodge-Tate weight $q$ is the dimension of $( V \otimes_{ \overline{\mathbb Q_p} }\mathbb C_p)^{ \Gal( \overline{K} /K)} $ over $K$.

We say $V$ is a \emph{Hodge-Tate representation} if the sum of the Hodge-Tate weights is $\dim V$.  The map $D$ sending $V$ to $\bigoplus_{q\in \mathbb Z} ( V \otimes_{ K} \mathbb C_p)^{ \Gal( \overline{K} /K)} $ gives a functor from the category of Hodge-Tate representations to graded vector spaces. The category of Hodge-Tate representations is stable under direct sums, tensor products, duals, subobjects, and quotients, and $D$ is a faithful exact tensor functor. It follows that the category of Hodge-Tate representations is Tannakian, and $D$ corresponds to a map from the Tannakian group $\mathbb G_m$ of the category of graded vector spaces to the Tannakian group of the category of Hodge-Tate representations.

\begin{lemma}\label{hodge-tate-calculation} Let $H$ be a smooth hypersurface, defined over a $p$-adic field $k$, in an abelian variety $A$ of dimension $n$. Let $\chi\colon  \pi_1( A_{\overline{k}}) \to \overline{\mathbb Q_p}^\times$ be a finite-order character such that $H^i(H_{\overline{k} } , \mathcal L_\chi) =0 $ for $i \neq n-1$. Let $k'$ be a finite extension of $k$ containing the coefficient field and field of definition of $\chi$.

Then the Hodge-Tate weights of the $\Gal(\overline{k}'|k')$ action on $H^{n-1}(H_{\overline{k}'}, \mathcal L_\chi)$ are $0,\dots, n-1$, where the multiplicity of the weight $q$ is $ d A(n,q)$.\end{lemma}


\begin{proof} As a finite-order character of $\pi_1(A)$, $\chi$ factors through $A[m]$ for some $m$. Let $H'$ be the inverse image of $H$ under the multiplication-by-$m$ map of $A$. Then we can express $ H^*(H_{\overline{k}}, \mathcal L_\chi)$ as the part of the \'{e}tale cohomology $H^*(  H'_{\overline{k}}, \mathbb Q_p)$ where $A[m]$ acts by the character $\chi$. Applying $p$-adic Hodge theory, we see that the dimension of the Hodge-Tate weight $q$ subspace of $H^{p+q} (  H_{\overline{k}},\mathcal L_\chi)$ is equal to the dimension of the subspace of $H^p(H'_{k'},  \Omega^q_{H'})$ on which $A[m]$ acts by the character $\chi$. By descent, this is the dimension of $H^p (H_{k'}, \Omega^q_H \otimes L)$ for a torsion line bundle $L$ on $H$. Because $H^{p+q} (  H_{\overline{k}'},\mathcal L_\chi)$  vanishes for $p \neq n-1-q$, the dimension of $H^p (H, \Omega^q_H \otimes L)$ for $p= n-1-q$  is equal to $(-1)^{n-1-q} $ times the Euler characteristic $\chi(H_{k'}, \Omega^q_H \otimes L ) =\chi(H, \Omega^q_H )$ because, as $L$ is torsion, its Chern class vanishes. By Lemma \ref{lem-Euler-formula}, this Euler characteristic is $(-1)^{n-1-q} d A(n,q)$, so the dimension and  Hodge-Tate multiplicity are both $d A(n,q)$ . \end{proof}

\begin{lemma}\label{reduction-to-combinatorics} Let $H$ be a hypersurface on on abelian variety $A$ of dimension $n$ and let $d= [H]^n/n!$. Suppose that the commutator subgroup of the identity component of $G_H$ is $ SL_m$, with distinguished representation $\wedge^k$. Then there exists a function $m_H$ from the integers to the natural numbers and an integer $s$  such that $\sum_{i} m_H(i) =m $ and such that, for all $q \in \mathbb Z$,
\begin{equation}\label{eq:wedge-Euler}  \sum_{ \substack{ m_S \colon \mathbb Z \to \mathbb Z \\ 0 \leq m_S(i) \leq m_H (i) \\ \sum_{i} m_S(i) =k \\ \sum_i i m_S(i) = s+q }} \prod_i { m_H(i) \choose m_S(i) }  = d A (n, q) . \end{equation} 

Here we use the convention that $A(n,q)=0$ unless $0 \leq q<n$. \end{lemma} 

\begin{proof} Fix a finite-order character $\chi$ of the geometric fundamental group of $A$ such that the cohomology of $H$ with coefficients in $\mathcal L_\chi$ is concentrated in degree $0$. Let $\kappa$ be a finitely generated field over which $A, H$, and $\chi$ may all be defined.

  By our assumption on the convolution monodromy group $G_H$ and Lemma \ref{normalizer-lemma}(2), it follows that the $\Gal(\overline{\kappa}/\kappa)$ action on $H^{n-1}(X_{\overline{k}'}, \mathcal L_\chi)$ lies in the normalizer of $SL_m$ inside the general linear group of the representation $\wedge^k$. It suffices to check that this condition implies the existence of $m_H$ and $s$. For this we may spread out $H$ to a family over a variety defined over a number field and specialize to a closed point. Since the property that the image of Galois is contained in the normalizer is preserved under specialization, we may assume $\kappa$ is a number field.

 This normalizer of $SL_m$ inside the general linear group is $GL_m/\mu_k$ if $k\neq m/2$ and $GL_m/\mu_k \ltimes (\mathbb Z/2)$ if $k=m/2$.

The category of Galois representations generated by $H^{n-1}(X_{\overline{\kappa}}, \mathcal L_\chi)$  under tensor products, direct sums, subobjects, quotients, and duals is a Tannakian category isomorphic to the category of representations of the Zariski closure of the image of $\Gal(\overline{\kappa}|\kappa)$ acting on $H^{n-1}(X_{\overline{\kappa}}, \mathcal L_\chi)$ (a quotient of the group $\Gal_\kappa$ defined in the previous section). Because $H^{n-1}(X_{\overline{\kappa}}, \mathcal L_\chi)$ is Hodge-Tate, this  category is a Tannakian subcategory of the category of Hodge-Tate representations, so the group is a quotient of the Tannakian group of the category of Hodge-Tate representations. Thus, the map from $\mathbb G_m$ constructed via the associated graded vector space functor gives a map from $\mathbb G_m$ to the Zariski closure of the image of $\Gal(\overline{\kappa}|\kappa)$, whose weights are the Hodge-Tate weights of $H^{n-1}(X_{\overline{\kappa}}, \mathcal L_\chi)$.

 Thus we have a homomorphism from $\mathbb G_m$ to $GL_m /\mu_k$ (or its semidirect product with $\mathbb Z/2$). Because  $\mathbb G_m$ is connected, this defines a homomorphism $\mathbb G_m \to GL_m / \mu_k$. Its weight multiplicities on the representation $\wedge^k$ must equal the Hodge-Tate multiplicities of $H^{n-1}(X_{\overline{\kappa}}, \mathcal L_\chi)$. 

Homomorphisms $\mathbb G_m \to GL_m$ are parameterized by their weights, an $m$-tuple of integers with $w_1 \leq \dots \leq w_m$.  Any homomorphism  $\alpha \colon \mathbb G_m \to GL_m / \mu_k$ has a unique lift $\overline{\alpha} \colon \mathbb G_m \to GL_m$ forming a commutative diagram \[\begin{tikzcd} \mathbb G_m \arrow[d,"x \mapsto x^k " ] \arrow[r, "\overline{\alpha}"] & GL_m \arrow[d]\\ \mathbb G_m \arrow[r,"\alpha"] & GL_m/\mu_k \end{tikzcd} \]
If we let $\overline{w}_1 \leq \dots \leq \overline{w}_m$ be the weights of $\overline{\alpha}$, then $\overline{w}_i \equiv \overline{w}_j \mod k$ since their difference is a weight of the adjoint representation, which factors through $GL_m/\mu_k$ and thus factors through $x\mapsto x^k$. 

Thus $ \frac{ \overline{w}_1 }{k},\dots, \frac{ \overline{w}_m}{k}$ are rational numbers with the same fractional part, equal to the ``weights" of the standard representation of $GL_m$. We let $w_i = \lfloor \frac{\overline{w}_i}{k} \rfloor $ and let $s = kw_i - \overline{w}_i $, which is independent of $i$, so that $\frac{\overline{w}_i }{k}=w_i -\frac{s}{k}$ .

Then the weights of the representation $\wedge^k$ are exactly $ (\sum_{i \in S} w_i )-s$ for all $S \subseteq \{1,\dots m\}$ with $|S|=k$. 

It follows that the multiplicity of the weight $q$ inside $\wedge^k$ of the standard representation of $GL_m$ is the number of subsets $S\subseteq \{1,\dots , m\} $ with $|S|=k $ and $\sum_{i \in S} w_i =s +q$.

To calculate this, let $m_H(j)$ be the number of $i$ such that $w_i=j$. For $S \subseteq\{1,\dots,m\}$, let $m_S(j)$ be the number of $i \in S$ such that $w_i=j$. Then $|S| = \sum_j m_S(j)$ and $\sum_{i \in S} w_i = \sum_j j m_S(j)$. Furthermore the number of sets $S$ attaining a given function $m_S$ is $\prod_i { m_H(i) \choose m_S(i) } $. 

The stated identities then follow from Lemma \ref{hodge-tate-calculation}.

\end{proof}

To complete the argument, we have the following purely combinatorial proposition, which is proven in Appendix \ref{combinatorics}.

\begin{prop}[Appendix \ref{combinatorics}]\label{combinatorics-main-statement} Suppose that there exists a natural number $k$, function $m_H$ from the integers to the natural numbers and an integer $s$ such that $ 1< k < -1+ \sum_{i} m_H(i) $ and Equation \eqref{eq:wedge-Euler} is satisfied for all $q \in \mathbb Z$. Then we have one of the cases

\begin{enumerate}

\item $m=4$ and $k=2$, 

\item $n=2$ and $d = {2k-1 \choose k}$ for some $k>2$, or

\item $n=3$ and $d ={ a(i)+a(i+1) \choose a(i) } / 6$ for some $i \geq 2$.
(Here $a(i)$ is defined by
\[ a(1)=1, a(2)=5, a(i+2) = 4 a(i+1) +1 - a(i).) \]

\end{enumerate}

\end{prop}

We are now ready to prove the main theorem of this section:

\begin{proof}[Proof of Theorem \ref{Tannakian-group-calculation}]
The commutator subgroup of the identity component of $G_H$ must be one of the groups listed in the three cases of Lemma \ref{list-of-groups}.
In case (1), we know immediately that $G_H$ contains as a normal subgroup one of $SL_N$, $Sp_N$ or $SO_N$, and we conclude by Lemma \ref{duality-lemma}.

We rule out case (2) by Lemma \ref{minuscule_case_2} and case (3) by Proposition \ref{combinatorics-main-statement}.  Note that Proposition \ref{combinatorics-main-statement} does not rule out the case $m=4,k=2$, but $\wedge^2 SL_4$ is $SO_6$ so $G_H$ contains $SO_6$ as a normal subgroup in this case. 
\end{proof}

\newpage

\section{Big monodromy from big convolution monodromy}
\label{sec:mon}

Let $X$ be a variety, $A$ an abelian variety, and $Y \subseteq X \times A$ a family over $X$ of smooth hypersurfaces in $A$, with $X$, $A$, $Y$ all defined over $\mathbb C$. Let $n=\dim A$, so that $n-1$ is the relative dimension of $Y$ over $X$.  Let $\eta$ be the generic point of $X$ and $\overline{\eta}$ a geometric generic point. 

Let $i\colon Y_{\overline{\eta}} \to A_{\overline{\eta}}$ be the inclusion, and let $K$ be the perverse sheaf $K = i_* \Qpell [n-1]$ on $A_{\overline{\eta}}$. Let $G$ be the convolution monodromy group of $K$. 
Let $G^* $ be the commutator subgroup of the identity component of $G$. We continue to call the \emph{distinguished representation} of $G$ the representation arising from the object $K$, and let its restriction to $G^*$ be the distinguished representation of $G^*$.

\begin{lemma}\label{non-constancy} Assume that $G^* $ is a simple algebraic group with irreducible distinguished representation, and that $Y$ is not equal to a constant family of hypersurfaces translated by a section of $A$.  

Let $K'$ be an irreducible perverse sheaf on $A_{\overline{\eta}}$ in the Tannakian category generated by $K$. Assume that $K'$ is a pullback from $A_{\mathbb C}$ to $A_{\overline{\eta}}$ of a perverse sheaf on $A_{\mathbb C}$. Then $G^*$ acts trivially on the irreducible representation of $G$ corresponding to $K'$. \end{lemma}

\begin{proof} Let $\overline{\eta \times \eta}$ be a geometric generic point of $X \times X$. Let $pr_1, pr_2\colon A_{\overline{\eta \times \eta}} \to A_{\overline{\eta}}$ be the maps induced by the two projections $\overline{\eta \times \eta} \to \overline{\eta}$. The convolution monodromy groups of $pr_1^* K$ and $pr_2^* K$ are both isomorphic to $G$, so the convolution monodromy group of $pr_1^* K \oplus pr_2^* K $ is isomorphic to a subgroup of $G \times G$ whose projection onto each factor is surjective. By Goursat's lemma, there are normal subgroups $H_1, H_2$ in $G$ and an isomorphism $a\colon (G/H_1) \to (G/H_2)$ such that the convolution monodromy group of $pr_1^* K \oplus pr_2^* K$ is isomorphic to \[\{ (g_1,g_2) \in G\times G | a(g_1) =g_2 \mod H_2 \}.\] 

Note that $G$ has a unique factor in its Jordan-H\"{o}lder decomposition which is a nonabelian connected simple group. Hence this factor appears either in $H_1$ and $H_2$ or in  $G/H_1$ and $G/H_2$.

In the first case, we must have $G^* \subseteq H_1$ and $G^* \subseteq H_2$. This is because, if the nonabelian connected simple factor appears in $H_i$, then it must appear in $H_i \cap G^*$ as $H_i / (H_i \cap G^*) \subseteq G/ G^*$ which, modulo scalars, is contained in the outer automorphism group of $G^*$ and thus is virtually abelian and cannot contain a nonabelian connected simple factor. Furthermore $H_i \cap G^*$ is a normal subgroup of $G^*$, and since it cannot be a finite group, it must be $G^*$.

Now, using the fact that $G^* \subseteq H_1$ and $G^* \subseteq H_2$, we will show that $G^*$ acts trivially on the irreducible representation corresponding to $K'$. To do this, observe that $pr_1^* K'$ and $pr_2^* K'$ are isomorphic because $K'$ is a pullback from $A_{\mathbb C}$. These correspond to two representations of the convolution monodromy group of $pr_1^* K \oplus pr_2^* K $ that factor through the projection onto the first and second factors respectively. Because $G^*$ lies in $H_1$ and $H_2$, the convolution monodromy group of  $pr_1^* K \oplus pr_2^* K $  contains two copies of $G^*$. The first copy of $G^*$ acts trivially on $pr_2^* K'$, so it must act trivially on $pr_1^* K'$, so $G^*$ acts trivially on $K'$, as desired.

In the second case, $H_1$ and $H_2$ must both be contained in the scalars. To see this, because the scalars are the centralizers of $G^*$, it suffices to show that the image of $H_i$ in the automorphisms of $G^*$ vanishes. Equivalently, we must show that the image of $H_i$ in the automorphisms of the Lie algebra of $G^*$ vanishes. This automorphism group is an extension of the finite outer automorphism group of $G^*$ by $G^*$ mod its center. Because the image of $H_i$ in the automorphism group is normalized by $G^*$, it either contains $G^*$ or is finite, and it cannot contain $G^*$, so it is finite. Because it is finite and normalized by $G^*$, it commutes with $G^*$. Because the Lie algebra of $G^*$ is an irreducible representation of $G^*$, this forces the image of $H_i$ to act as scalars. But there are no nontrivial scalar automorphisms of a nonabelian Lie algebra, as they would never preserve any equation $[x,y]=z$, and so the image of $H_i$ is trivial, as desired.

Now, using the fact that $H_1$ and $H_2$ are both contained in the scalars, we will derive a contradiction. Thus the convolution monodromy group $pr_1^* K \oplus pr_2^* K $ is contained in the set $\{ (g_1,g_2) \in G^2 | a(g_1) = g_2 \}$ for some automorphism $a$ of $G$ mod scalars. Let us first check that the automorphism $a$ is inner. To do this, let $\overline{\eta}_m$ be the geometric generic point of $X^m$. Let $pr_1,\dots, pr_m$ be the projections to $\eta$. For each $i$ we have a homomorphism $\rho_i$ from the convolution monodromy group of $\bigoplus_{i=1}^m pr_i^* K$  to the convolution monodromy group of $pr_i^*K $ modulo scalars, and let $\rho^1$ and $\rho^2$ be the similarly-defined homomorphisms from the convolution monodromy group of $pr_1^* K \oplus pr_2^*K$ to the convolution monodromy groups of $pr_1^*K$ and $pr_2^*K$ modulo scalars. Then because for any pair $i,j$ there is a projection $X^m \to X^2$ onto the $i$th and $j$th copies of $X$, and $pr_i^* K \oplus pr_j^* K$ is isomorphic to the pullback of $pr_1^*K \oplus pr_2^*K$ along this projection, the convolution monodromy group of $pr_1^*K \oplus pr_2^*K$  is isomorphic to to the convolution monodromy group of $pr_i^* K \oplus pr_j^* K$. Hence for any pair $i$ and $j$, there exists an automorphism $\sigma_{ij}$ of $G$ modulo scalars that sends $\rho_i$ to $\rho_j$. Since $\rho_i$ and $\rho_j$ are surjective, this automorphism is unique, and so $\sigma_{ik} =\sigma_{jk} \sigma_{ij} $.  Hence if $m$ is greater than the order of the outer automorphism group of $G$ mod scalars, there are $i$ and $j$ such that $\sigma_{ij}$ is an inner automorphism, say conjugation by $g\in G$.  Because $pr_i^* K \oplus pr_j^* K$ is isomorphic to the pullback of $pr_1^*K \oplus pr_2^*K$, conjugation by $g$ sends $\rho^1$ to $\rho^2$. Because the convolution monodromy group is well-defined only up to inner automorphisms in the first place, we may assume $\rho^1 =\rho^2$. It follows that the map from the convolution monodromy group of $pr_1^*K \oplus pr_2^*K$ to $G \times G$ has image consisting of pairs  $(g_1,g_2)$ where $g_2= \lambda g_1$ for a scalar $\lambda$.

Now the representation associated to $pr_1^*K  * [-1]^* Dpr_2^* K$ is the standard representation of the first $G$ tensored with the dual of the standard representation of the second $G$. Because the convolution monodromy group consists of pairs $(g_1,g_2)$ where $g_2= \lambda g_1$ for a scalar $\lambda$, this tensor product contains a one-dimensional subrepresentation $\chi$. (Viewing this tensor product as the space of endomorphisms of the standard representation, the one-dimensional subrepresentation consists of scalar endomorphisms). Because $\chi$ admits a nontrivial homomorphism to $\operatorname{std}_1 \otimes \operatorname{std}_2^\vee$, we have a natural map $\chi \otimes \operatorname{std}_2 \to \operatorname{std}_1$, which must be an isomorphism because both sides are irreducible.   Any one-dimensional representation of the convolution monodromy  group must be a skyscraper sheaf $\delta_x$ \cite[Proposition 10.1]{KramerWeissauer}, so we have an isomorphism $K_2 * \delta_x = K_1$ for some $x \in A( \overline{\eta \times \eta})$. Considering the support, we see that the translation of $Y_{\eta_2}$ by $x$ is $Y_{\eta_1}$. Spreading out this identity and then specializing $\eta_2$ to a sufficiently general point, we see that $Y$ is generically the translation of a constant variety by a section $x$ of $A$. We can extend this section to some open set, and then $Y$ over some open set will be the translation of a constant variety by a section, and then because $Y$ is a smooth proper family this will be true globally, contradicting the assumption. \end{proof}

\begin{lemma}\label{pair-group-calculation} Let $A_1,A_2$ be abelian varieties. For $j\in \{1,2\}$, let $z_j\colon A_j \to A_1 \times A_2$ be the inclusion map and let $K_j$ be a perverse sheaf on $A_j$ with sheaf convolution group $G_j$. Then the sheaf convolution group of $z_{1* }A_1 \oplus z_{2*} A_2$ is $G_1 \times G_2$. \end{lemma}

\begin{proof} Since $z_{1*} A_1$ and $z_{2*}A_2$ have Tannakian groups $G_1$ and $G_2$, the Tannakian group of their sum is a subgroup of $G_1 \times G_2$ whose projection onto both factors is surjective, so by Goursat's lemma there exists a group $H$ and quotient maps $q_1 \colon G_1\to H$ and $q_2 \colon G_2\to H$ such the sheaf convolution group of $z_{1* }A_1 \oplus z_{2*} A_2$ is $\{(a_1,a_2) \in G_1\times G_2 \mid q_1(a_1)=q_2(a_2)\}$. For $L_1, L_2$ perverse sheaves corresponding to a faithful representation of $H$, this implies $z_{1*}L_1\cong z_{2*}L_2$, absurd unless $L_1$ and $L_2$ are both skyscraper sheaves at the identity, implying $H$ is trivial and so the Tannakian group is a product, as desired.\end{proof}

\begin{cor}\label{multiple-non-constancy} Assume that $G^* $ is a simple algebraic group with irreducible distinguished representation, and that $Y$ is not equal to a constant family of hypersurfaces translated by a section of $A$. 

Let $c$ be a positive integer, and let $i_1,\dots, i_c$ be the inclusions of $A$ into $A^c$ that send $A$ to one of the $c$ coordinate axes. 

Then the convolution monodromy group of $\bigoplus_{j=1}^c  i_{j*} K$ is $G^c$ and thus contains $(G^*)^c$ as a normal subgroup. The representation associated to $\bigoplus_{j=1}^c  i_{j*} K$ is isomorphic to the sums of the distinguished representations of the $c$ factors of $G$.  Finally, this normal subgroup acts trivially on any representation of the convolution monodromy group corresponding to a perverse sheaf that is pulled back from $A^c_{\mathbb C}$ to $A^c_{\overline{\eta}}$.  \end{cor}

\begin{proof} 

The convolution monodromy group is $G^c$ by induction on Lemma \ref{pair-group-calculation}.


If any irreducible representation corresponding to a perverse sheaf which is a pullback of a perverse sheaf from $A^c_{\mathbb C}$ to $A^c_{\overline{\eta}}$ is nontrivial on $(G^*)^c$, then it is nontrivial on the $i$th copy of $G^*$ for some $i$. Tensor product with $\mathcal L_\chi$ for generic $\chi$ and then pushforward to $A$ gives a faithful exact tensor functor, whose associated map on Tannakian groups is the $i$th inclusion of $G$ into $G^c$, so the image under this map is nontrivial, which contradicts Lemma \ref{non-constancy} and the fact that the pushforward to $A_{\overline{\eta}}$ of a pullback from $A^c_{\mathbb C}$ is itself a pullback from $A_\mathbb C$ by proper base change.\end{proof}



We write $\Pi(A)$ for the set of continuous characters $\pi_1^{et}(A) \to \overline{\mathbb Q}_\pell^\times$. We call $\Pi(A)$ \emph{the dual torus of $\pi_1^{et}(A)$}.

We call the set of characters of $\pi_1(A)$ trivial on the fundamental group of a nontrivial abelian subvariety of $A$ a \emph{proper subtorus} of the dual torus of $\pi_1^{et}(A)$. 


The generic vanishing theorem of  Kr\"amer and Weissauer \cite[Lemma 11.2]{KramerWeissauer} states that for $K$ a perverse sheaf on $A$,   $H^i(A_\mathbb C,  K \otimes \mathcal L_\chi)=0$ for $i\neq 0$ for $\chi$ outside a finite set of torsion translates proper subtori of $\Pi(A)$. (In fact it states this for characters of the topological fundamental group, but since every character of the e\'tale fundamental group can be restricted to the topological one, that statement is stronger.)  Building on this theorem, we prove a lemma that combines that statement with some useful information about $H^0$:

\begin{lemma}\label{specialization-lemma} Let $K \in D^b_c( A_{\mathbb C(\eta)}, \Qpell)$ be a perverse sheaf of geometric origin. If no irreducible component of $K$ is a pullback from $A_\mathbb C$, then for all characters $\chi\colon \pi_1^{et}(A_{\overline{\mathbb C(\eta)} })\to \overline{\mathbb Q}_{\pell}^{\times}$ outside a finite set of torsion translates of proper subtori of $\Pi(A)$, we have $H^i(A_{\overline{\mathbb C(\eta)}}, K  \otimes \mathcal L_\chi) =0$ for $i \neq 0$ and $\left( H^0(A_{\overline{\mathbb C(\eta)}}, K  \otimes \mathcal L_\chi) \right)^{ \Gal( \overline{ \mathbb C(\eta)} | \mathbb C(\eta))}=0$. \end{lemma}

\begin{proof}  The first claim follows from \cite[Lemma 11.2]{KramerWeissauer}, which is stated for varieties over $\mathbb C$ and singular cohomology, but we may embed $\mathbb C(\eta)$ into $\mathbb C$ and then base change from the \'{e}tale to the analytic site.

The second claim follows from the same theorem, but indirectly. By restricting to an open subset of $X$, we may assume $X$ is smooth. Let $m$ be the dimension of $X$. We may spread $K$ out (using the fact that it is of geometric origin)  to a sheaf $K'$ over $A \times X$ such that $K'[m]$ is perverse.  Let $\pi\colon A \times X \to A$ and $\rho\colon A \times X \to X$ be the projections.

We will prove the second claim by contradiction.  We will first assume that $\left( H^0(A_{\overline{\mathbb C(\eta)}}, K  \otimes \mathcal L_\chi) \right)^{ \Gal( \overline{ \mathbb C(\eta)} | \mathbb C(\eta))}\neq 0$ for a particular $\chi$ such that $H^i(A_{\overline{\mathbb C(\eta)}}, K  \otimes \mathcal L_\chi) =0$ for $i \neq 0$, and derive some conclusions from this. We will then define a finite set of torsion translates of proper subtori of $\Pi(A)$, assume that this nonvanishing holds for some $\chi$ outside their union, and derive a contradiction from that.

Let us first see how to interpret the nonvanishing of monodromy invariants in terms of the perverse sheaf $K'[m]$.  This will essentially be the usual observation that sheaves with monodromy invariants have global sections, and thus have nonzero $H^0$. Additional care must be taken because $R\rho_* (K'\otimes  \mathcal L_\chi)$ is a complex of perverse sheaves, but the decomposition theorem will give us exactly what is needed.

 The stalk of $R\rho_* (K' \otimes \mathcal L_\chi)$ at the generic point is the complex $H^* ( A_{\overline{\mathbb C(\eta)}}, K  \otimes \mathcal L_\chi)$. By the assumption that this cohomology group vanishes in degree $\neq 0$, the stalk of  $R\rho_* ( K' \otimes \mathcal L_\chi)$ at the generic point is supported in degree $0$. There is some open subset of $Y$ over which $R\rho_* ( K' \otimes \mathcal L_\chi)$ remains a lisse sheaf in degree $0$, and the Galois action matches the monodromy action on that open subset. By the decomposition theorem, $R\rho_* ( K' \otimes \mathcal L_\chi)$ is a sum of shifts of irreducible perverse sheaves. In particular, this monodromy action is semisimple.

Now we assume that the Galois invariants are nonzero. It follows that the monodromy invariants are nonzero, and thus, by semisimplicity, there is a rank-one monodromy-invariant summand. Equivalently, there is a summand of $R\rho_* ( K' \otimes \mathcal L_\chi)$, restricted to this open set, that is isomorphic to the constant sheaf $\Qpell$. Because  $R\rho_* ( K' \otimes \mathcal L_\chi)$ is a sum of shifts of irreducible perverse sheaves, this irreducible summand $\Qpell$ on an open set must extend to a shift of an irreducible perverse sheaf on the whole space. Because $X$ is smooth, the unique irreducible extension of the constant sheaf $\Qpell$ from an open subset to all of $X$ is $\Qpell$, which is a perverse sheaf shifted by $m$. (In general, it would be the IC sheaf of $X$.)

Because $H^{0}(X, \Qpell) \neq 0 $, and $\Qpell$ is a summand of $R\rho_* ( K' \otimes \mathcal L_\chi)$, it follows that $H^{0} ( X, R \rho_* (K' \otimes \mathcal L_\chi)) \neq 0$.   

Now that we have interpreted the existence of nontrivial monodromy invariants cohomologically, we can re-express the cohomology group in terms of shaves on $A$, which will enable us to understand its dependence on $\chi$ using the generic vanishing theorem. It follows from the Leray spectral sequence and the projection formula that
\[ 0 \neq H^{0}( X, R \rho_* ( K' \otimes \mathcal L_\chi))= H^{0}(A \times X, K' \otimes \mathcal L_\chi) = H^{0} (A,R\pi_* K' \otimes \mathcal L_\chi) .\]  

Now we choose our finite union of torsion translates of subtori. We apply the generic vanishing theorem to every perverse sheaf in sight. For all $j \in \mathbb Z$ with ${}^p \mathcal H^j ( R\pi_* K') \neq 0 $, we take from \cite[Lemma 11.2]{KramerWeissauer} a finite set of torsion translates of subtori such that $H^i ( A, {}^p \mathcal H^j ( R\pi_* K')\otimes \mathcal L_\chi)=0$ for all $\chi$ not in this set and all $i \neq 0$. Because there are only finitely many $j$ where ${}^p \mathcal H^j ( R\pi_* K') \neq 0 $, the union of all of these is again a finite set.

Assume that $\left( H^0(A_{\overline{\mathbb C(\eta)}}, K  \otimes \mathcal L_\chi) \right)^{ \Gal( \overline{ \mathbb C(\eta)} | \mathbb C(\eta))}\neq 0$ for a particular $\chi$ not in this set.  The vanishing of $H^i ( A, {}^p \mathcal H^j ( R\pi_*  K')\otimes \mathcal L_\chi)$ for all $j$ and all $i\neq 0$ forces the spectral sequence
\[ H^i ( A, ^{p} \mathcal H^j ( R\pi_*  K')\otimes \mathcal L_\chi)\mapsto  H^{i+j} (A,R\pi_*  K' \otimes \mathcal L_\chi)\]
to degenerate, giving
\[ 0 \neq H^{0} (A, R\pi_*  K' \otimes \mathcal L_\chi) = H^0( A, ^{p} \mathcal H^{0} ( R\pi_* K')\mathcal L_\chi).\]

In particular, we can conclude that 
\[ ^{p} \mathcal H^{0} (R \pi_* K') \neq 0.\]
We will now derive a contradiction from this simpler statement, which notably is independent of $\chi$.

 Note first that because $K'[m]$ is perverse, and $\pi$ has fibers of dimension at most $m$, we have  ${}^p\mathcal H^{i}(R\pi_* K' ) = 0$ for $i <  0$ by \cite[4.2.4]{BBDG}. 
 Hence there is a natural map \[  {}^p\mathcal H^{0 }( R\pi_* K'  )\to R \pi_* K'\] arising from the perverse $t$-structure. Because ${}^p\mathcal H^{0}(R\pi_* K' )$ is nonzero, this map must be nonzero. Applying adjunction, we obtain a nonzero map $ \pi^* {}^p\mathcal H^{0}(R\pi_*  K' )  \to K'$, and and shifting by $m$, a nonzero map \[  \pi^* {}^p\mathcal H^{0}(R\pi_*  K' )[ m]  \to K'[m] .\]  Because this is a nonzero map between perverse sheaves, some irreducible component of the source is equal to some irreducible component of the target. This cannot happen because by assumption no irreducible component of $K$ is a pullback from $A_{\mathbb C}$, giving a contradiction.\end{proof}

\begin{lemma}\label{representation-detection}  Let $G^*$ be a simple algebraic group and $c$ a natural number. Fix an irreducible representation of $G^*$. Let $N(G^*)$ be the normalizer of $G^*$ inside the group of automorphisms of this representation. Then there is a finite list of irreducible representations of $N(G^*)^c$ such that a reductive subgroup $B$ of $N(G^*)^c$ contains $(G^*)^c$ if and only if $B$ has no invariants on any of these representations.

\end{lemma}

\begin{proof}  Let $m$ be such that any subgroup of $N(G^*)$ which is finite modulo scalars contains an abelian subgroup of index $m$. Such $m$ exists by Jordan-Schur. Let $V$ be an irreducible representation of $N(G^*)$ of dimension $>m$ which remains irreducible on restriction to $G^*$.  Let $\pi_i$ be the $i$-th projection $(N(G^*)^c) \to N(G^*)$.

We take our list to be, for each $i$ from $1$ to $c$, all irreducible $N(G^*)$-subrepresentations of $ (V \otimes V^\vee)/\mathbb C$ and $ (\operatorname{ad} G^* \otimes \operatorname{ad} G^{*\vee})/\mathbb C$ composed with  $\pi_i$, together with, for $1\leq i < j \leq c$,  the representation $\operatorname{ad} G^* \circ \pi_i \otimes \operatorname{ad} G^{{*\vee}} \circ \pi_j$. It is straightforward to check that $(G^*)^c$ has no invariants on these representations, so if $B$ contains $(G^*)^c$ then $B$ has no invariants. We must check the converse.

If $B$ has no invariants, then $\pi_i(B)$ acts irreducibly on $V$ and $\operatorname{ad} (G^*)$. The Lie algebra of $\pi_i(B)$ is an invariant subspace of $\operatorname{ad} (N(G^*)) \cong \operatorname{ad} (G^*) \oplus \mathbb C$, thus either contains $\operatorname{ad} (G^*)$, in which case $\pi_i(B)$ contains $G^*$, or is contained in $\mathbb C$, in which case $\pi_i(B)$ is finite modulo scalars, hence has an abelian subgroup of index $\leq m$, thus cannot act irreducibly on $V$, a contradiction. (Compare \cite[Theorem 2.2.2]{KatzMMP}, due originally to Larsen.)

Since $\pi_i(B)$ contains $\pi_i$ for all $i$,  we have $\pi_i ( [B^0,B^0])=G^*$. By \cite[Theorem on p.\ 1152]{Kolchin68}, if $[B^0,B^0] \neq (G^*)^c$, then there exists $i,j$ with $1 \leq i< j \leq c$ and an isomorphism $G^* \to G^*$ that sends $\pi_i$ to $\pi_j$. This isomorphism is unique and thus $B$-invariant, so $\operatorname{ad} G^* \circ \pi_i $ and  $\operatorname{ad} G^{{*}} \circ \pi_j$ are isomorphic as representations of $B$. Thus $\operatorname{ad} G^* \circ \pi_i \otimes \operatorname{ad} G^{{*\vee}} \circ \pi_j$ has invariants, contradicting our assumption and showing $ (G^*)^c = [B^0,B^0]\subseteq B$. \end{proof}

\begin{remark} Lists of representations satisfying the condition of Lemma \ref{representation-detection} with smaller dimension follow from Larsen's conjecture \cite[Theorem 1.4]{LarsensConjecture} in the case that $G^*$ is a classical group, but these depend on the classification of finite simple groups. For some applications of results of this type, optimizing the dimension of the representations is relevant, but not here.\end{remark}

Let $f \colon Y \to X$ and $g \colon Y \to A$ be the projection maps associated to $Y \subseteq X \times A$ a family of smooth projective hypersurfaces parameterized by $X$. 

Recall that the \emph{geometric monodromy group} of a constructible $\pell$-adic sheaf on an irreducible scheme is the Zariski closure of the image of the natural map from the geometric \'{e}tale fundamental group of the largest open set on which the sheaf is lisse to the general linear group of the stalk at the generic point.

The following theorem is the analogue of Pink's specialization theorem \cite[Theorem 8.18.2]{Katz-ESDE}, which shows, given any sheaf on the total space of a family of schemes, for ``most" schemes in the family (i.e.\ for the fibers over a dense open subset), the monodromy of the restricted sheaf is equal to a generic monodromy group. Our analogue shows that for ``most" characters $\chi$, the monodromy of $\bigoplus_{i=1}^c  R^{n-1} f_* ( g^* \mathcal L_{\chi_i})$ is (roughly) equal to a generic convolution monodromy group. This will connect our earlier investigations of the convolution monodromy group to our later arguments, which require control on monodromy groups.

\begin{theo} 
\label{thm_fin_many_subtori}
Assume that $Y_{\overline{\eta}}$ is not translation-invariant by any nonzero element of $A$, that $G^* $ is a simple algebraic group with irreducible distinguished representation, and that $Y$ is not equal to a constant family of hypersurfaces translated by a section of $A$. Fix $c\in \mathbb N$.

Then for $\chi_1,\dots, \chi_c$ characters of $\pi_1^{et}(A )$, with $(\chi_1,\dots, \chi_c)$ avoiding some finite set of torsion translates of proper subtori of the dual torus $\Pi(A)^c$ to $\pi_1^{et}(A^c)$, the following conditions are satisfied:
\begin{itemize}
\item $R^{k} f_* ( g^* \mathcal L_{\chi_i})=0$ for $k \neq n-1$, and 
\item the geometric monodromy group of $\bigoplus_{i=1}^c  R^{n-1} f_* ( g^* \mathcal L_{\chi_i})$ contains $(G^*)^c$ as a normal subgroup, where the representation associated to $\bigoplus_{i=1}^c  R^{n-1} f_* ( g^* \mathcal L_{\chi_i})$, restricted to $(G^*)^c$, is isomorphic to a sum of $c$ copies of the distinguished representation of $G^*$.
\end{itemize}
 \end{theo}

\begin{proof} First, we use the generic vanishing theorem to find a finite set of torsion translates of proper subtori of $\Pi(A)$ such that, for $\chi$ avoiding them, $R^{k} f_* ( g^* \mathcal L_{\chi})=0$ for $k \neq n-1$. We will then take the inverse images of these subtori under the duals of the $c$ projections $\Pi(A^c) \to \Pi(A) $ to be in our finite set of torsion translates of subtori of $\Pi(A)^c$.

Next, to calculate the monodromy, we will use the fact that the monodromy group is equal to the Zariski closure of the image of $\Gal( \overline{\mathbb C(\eta)} | \mathbb C(\eta))$ acting on the stalk at the geometric generic point.

 Let us first check that the geometric monodromy group of $\bigoplus_{i=1}^c  R^{n-1} f_* ( g^* \mathcal L_{\chi_i})$ is contained in $N(G^*)^c$. It suffices to show for each $i$ that the geometric monodromy group of $R^{n-1} f_* ( g^* \mathcal L_{\chi_i})$ is contained in $N(G^*)$. This follows from Lemma \ref{normalizer-lemma}(2).

Now by Lemma \ref{representation-detection} we can find an explicit list of representations of $ N(G^*)^c$ such that any reductive subgroup of $N(G^*)^c$ contains $(G^*)^c$ if and only if its action on all these representations has no invariants. By Deligne's theorem, the monodromy group of  $\bigoplus_{i=1}^c  R^{n-1} f_* ( g^* \mathcal L_{\chi_i})$ is reductive, and so we can apply this lemma.

 By Lemma \ref{tannakian-quotient}, each representation from the list of Lemma \ref{representation-detection} corresponds to a perverse sheaf on $A^c_{\eta}$ in the Tannakian subcategory generated by $\bigoplus_{j=1}^c  i_{j*} i_* \Qpell [n-1]$ inside the arithmetic Tannakian category constructed in Lemma \ref{atc} of perverse sheaves on $A^c_\eta$ modulo negligible sheaves.  (We have to check that $\bigoplus_{j=1}^c  i_{j*} i_* \Qpell [n-1]$ is geometrically semisimple, but this is clear as the constant sheaf on any closed subvariety is semisimple. It follows that the action of $G_k$ on the representation associated to this complex factors through the normalizer of the action of $G_{k'}$, and thus factors through $N(G^*)^c$, and so all representations of $N(G^*)^c$ correspond to geometrically semisimple perverse sheaves.)

 Because $(G^*)^c$ acts nontrivially on all representations from Lemma \ref{representation-detection}, by Lemma \ref{multiple-non-constancy} none of these perverse sheaves is a pullback from $A^c_{\mathbb C}$, so by Lemma \ref{specialization-lemma}, outside some finite set of torsion translates of proper subtori of $\Pi(A)^c$, the Galois group has no invariants for these representations. Thus, outside some finite set of torsion translates, the Galois group contains $(G^*)^c$.
\end{proof}

We now specialize to the case that $A$ is defined over a number field $K$. Fix an algebraic closure $\overline{\mathbb Q}$ of $\mathbb Q$. Let $\Pi^{K/\mathbb{Q}}(A)$ be the set of pairs of an embedding $\iota$ of $A$ into $\overline{\mathbb Q}$ and a character $\chi$ of $\pi_1(A_\iota)$ where $A_\iota = A\times_{K,\iota} \overline{\mathbb Q}$.  Then $\Pi^{K/\mathbb{Q}}(A)$ naturally admits an action of $\Gal_{\mathbb Q} \times \Gal_{\mathbb Q^{\textrm{cyc}}/\mathbb Q}$, where $\Gal_{\mathbb Q}$ acts on $\overline{\mathbb Q}$ and $\Gal_{\mathbb Q^{\textrm{cyc}}/\mathbb Q}$ acts on the values of the character $\chi$ (which necessarily lie in $\mathbb Q^{\textrm{cyc}}$.  This action extends the more straightforward action of $\Gal_{ K} \times \Gal_{\mathbb Q^{\textrm{cyc}}/\mathbb Q}$ on $\Pi(A)$. 

Fix inside the group $\Gal_{\mathbb Q} \times \Gal_{\mathbb Q^{\textrm{cyc}}/\mathbb Q}$ an element $\Frob_p$ of elements projecting to a lift of Frobenius inside the decomposition group at $p$ in both $\Gal_{\mathbb Q} \times \Gal_{\mathbb Q^{\textrm{cyc}}/\mathbb Q}$. 

For applications later in the paper, we will be interested in proving our large monodromy statement uniformly in an orbit of $\Gal_{\mathbb Q} \times \Gal_{\mathbb Q^{\textrm{cyc}}/\mathbb Q}$ and this statement will involve a monodromy group of a direct sum of sheaves produced by $c$ characters in the orbit of $\Frob_p$. We therefore prove a series of lemmas that enable us to obtain a monodromy statement of this type.

\begin{lemma} 
Suppose $A$ is defined over a number field $K$,
and let $p$ be a prime at which $A$ has good reduction.  Let $m$ be a natural number such that $\Frob_p^m$ fixes the Galois closure of $K$. 
For any positive integer $c$ and set $S$ of torsion translates of proper subtori of the dual torus $\Pi(A)^c$ to $\pi_1(A^c)$, there is a finite set $S'$ of torsion translates of proper subtori of $\Pi(A)$ such that for any $\chi$ outside the union of $S'$, there exist $e_1,\dots, e_c \in \mathbb Z$ such that the tuple $(\Frob_p^{m e_1} (\chi),\dots, \Frob_p^{m e_c} (\chi))$ does not lie in the union of $S$. \end{lemma}

\begin{proof}We can freely replace $m$ by any multiple. By choosing a suitable multiple, we may assume $\Frob_p$ fixes all elements of $S$. 

Let $\sigma= \Frob_p^m$. We will prove this lemma as a consequence of the fact that the action of $\sigma$ on $\Pi(A)$ is by invertible linear transformations, with no roots of unity as eigenvalues. (The action of $\Frob_p^m \in  \Gal_{\mathbb Q^{\textrm{cyc}}/\mathbb Q}$ is by multiplication by $p^m$, while the action of $\Frob_{p}^m \in \Gal_{ K}  \subseteq \Gal_{\mathbb Q}$ is invertible, with eigenvalues the inverses of Weil numbers of absolute value $p^{m/2}$. Hence their product has eigenalues Weil numbers of absolute value $q^{m/2}$, which are in particular not roots of unity.)

It suffices to show for each torsion translate of a torus $T+\xi$ in $S$, for all $\chi$ outside some finite set of torsion transltaes of proper subtori of $\Pi(A)$, the number of tuples $e_1,\dots, e_c\in \{0,\dots, |S|\}$ such that the tuple $(\sigma^{ e_1} (\chi),\dots, \sigma^{ e_c} (\chi))$ does not lie in $T+\xi$ is at most $(|S|+1)^{c-1}$. Indeed, we can then take $S'$ to be the union of all these finite sets, and then the number of tuples $e_1,\dots, e_c\in \{0,\dots, |S|\}$ such that $(\sigma^{ e_1} (\chi),\dots, \sigma^{e_c} (\chi))$ does not lie in the union of $S$ will be at least \[(|S|+1)^{c} - |S| (|S|+1)^{c-1} >0.\]

For each $r$ from $1$ to $c$, let $T^r$ be the inverse image of $T$ under the $r$'th inclusion map $\Pi(A) \to \Pi(A^c)$. Each $T^r$ is a subtorus of $\Pi(A)$ and, since $T$ is proper and the images of the inclusion maps $\Pi(A) \to \Pi(A^c)$ generate $\pi(A^c)$  at least one of the $T^r$ must be proper. Fix this $r$. We will first check that
\[ \{ \chi \mid \sigma^e(\chi) \sigma^{e'}(\chi)^{-1}  \in T^r \textrm{ for some }0 \leq e < e' \leq |S|  \}\] is a finite union of torsion translates of $T^r$ and then check that for $\chi$ not in this set, the number of tuples $e_1,\dots, e_c\in \{0,\dots, |S|\}$ such that the tuple $(\sigma^{ e_1} (\chi),\dots, \sigma^{ e_c} (\chi))$ lies in $T+\xi$ is at most $(|S|+1)^{c-1}$

For the first claim, $T^r$ consists of characters restricting to the trivial character on some abelian subvariety $B$. If $\sigma^e(\chi) \sigma^{e'}(\chi)^{-1}  \in T^r$ then $\sigma^e(\chi)$ and $\sigma^{e'}(\chi)$ agree on restriction to $B$. Since the eigenvalues of $\sigma^{e'-e}$ are algebraic numbers but not roots of unity, the fixed points of $\sigma^{e'-e}$  on $B$ are torsion and finite in number. The same is true for their inverse image under $\sigma^e$, i.e. the elements of $\Pi(B)$ whose images under $\sigma^e$ and $\sigma^{e'}$ agree. Characters of $\pi_1(A)$ whose restriction to $\pi_1(B)$ take a fixed torsion value form a torsion translate of $T$, completing the proof of the first claim.

For the second claim, the number of choices of $e_i$ for all $i\neq r$ is $(|S|+1)^{c-1}$ and these choices, together with the condition that $(\sigma^{ e_1} (\chi),\dots, \sigma^{ e_c} (\chi))$ lies in $T+\xi$, determine $e_r$, since if both $(\sigma^{e_1} (\chi),\dots, \sigma^{ e_r}(\chi),\dots \sigma^{ e_c} (\chi))$ and $(\sigma^{ e_1} (\chi),\dots,\sigma^{ e_r'}(\chi),\dots \sigma^{e_c} (\chi))$ lie in $T+\xi$ by dividing, without loss of generality with $e_r<e_r'$, it would follow that $(1,\dots,  1,\sigma^{e_r} (\chi) \sigma^{e_r'}(\chi)^{-1} ,1,\dots,1)$ lies in $T$ and hence $\sigma^{e_r} (\chi) \sigma^{e_r'}(\chi)^{-1}$ lies in $T^r$, contradicting the assumption on $\chi$.
\end{proof}

\begin{lemma}  
Suppose $A$ is defined over a number field $K$. For each tuple indexed by embeddings $\iota \colon K \to \overline{\mathbb Q}$ of finite subsets $S_\iota$ of torsion translates of proper subtori of $\Pi(A_\iota)$, there exists an embedding $\iota$ of $K$ into $\overline{\mathbb Q}$ and a torsion character $\chi$ of $\pi_1(A_\iota )$ such that for no $\Gal_\mathbb Q \times \Gal_{\mathbb Q^{\textrm{cyc}/\mathbb Q} }$-conjugate $(\iota',\chi')$ of $(\iota,\chi)$ does $\chi'$ lie in the union of $S_{\iota'}$. Furthermore, we can arrange that $\chi$ is of order a power of $\ell$, for any given prime $\ell$. \end{lemma}

\begin{proof}For $\chi$ to have no Galois conjugate in any element of any $S_{\iota'}$, it suffices that $\chi$ not lie in any Galois conjugate of these translates of $S_{\iota'}$.

By definition, each proper subtorus of $\Pi(A)$ corresponds to some abelian subvariety, which must be defined over a number field. Since every torsion point is defined over a number field, every torsion translate of a proper subtorus is defined over a number field.  Hence they have finitely many conjugates under $\Gal_\mathbb Q $. The action of $\Gal_{\mathbb Q^{\textrm{cyc}/\mathbb Q}} $ fixes each subtorus and gives each torsion point finitely many conjuates, so the total number of conjugates under both actions is finite.  Thus the union of all these conjugates is a finite union of torsion translates of proper subtori.

Each proper subtorus can contain at most $\ell^{ 2k(n-1)} $ $\ell^{k}$-torsion characters, and so any translate of a proper subtorus can contain at most $\ell^{ 2k (n-1)}$ $\ell^{k}$-torsion characters, while there are $\ell^{2kn}$ $\ell^{2k}$-torsion characters in total, so as soon as $\ell^{2k}$ is greater than this finite number of tori, there will be an $\ell^k$-torsion character not in any of them.
\end{proof}

For the remainder of this section, we'll suppose $A$ is an abelian variety over a number field $K$, 
$X$ is a smooth scheme over $\mathbb{Q}$,
and 
\[Y \subseteq X \times_{\mathbb{Q}} A = X_K \times_K A \]
is a family of hypersurfaces in $A$, smooth, proper and flat over $X_K$.
For every embedding $\iota \colon K \rightarrow \mathbb{C}$,
we can form schemes 
\[ A_{\iota} = A \otimes_{\Spec K, \iota} \Spec \mathbb{C} \]
and
\[ Y_{\iota} = Y \otimes_{\Spec K, \iota} \Spec \mathbb{C}; \]
these are both schemes over $\mathbb{C}$, and $Y_{\iota}$ has a projection to
\[ X_{\mathbb{C}} = X \otimes_{\Spec \mathbb{Q}} \Spec \mathbb{C}. \]
Let $f \colon Y_{\iota} \rightarrow X_{\mathbb{C}}$ and $g \colon Y_{\iota} \rightarrow A_{\iota}$ be the projections;
for every torsion character $\chi$ of $\pi_1^{et}(A_{\iota})$, let $\mathcal{L}_{(\iota, \chi)}$ be
the corresponding character sheaf on $A_{\iota}$.

The next result will be what we use from this section later in the paper. This gives a big monodromy statement which holds simultaneously for multiple characters -- as many as desired -- in a Galois orbit. The big monodromy will be an input into the Lawrence-Venkatesh method, and being able to work with multiple characters in a Galois orbit helps to control the variation of a global Galois representation.
(Specifically, in order to apply Lemma \ref{codim_final_estimate_ss} or Lemma \ref{codim_final_estimate}, it is advantageous to be able to take $c$ as large as desired.) 

\begin{cor} \label{torsion_char}
Assume that $Y_{\overline{\eta}}$ is not translation-invariant by any nonzero element of $A$, that $G^* $ is a simple algebraic group with irreducible distinguished representation, and that $Y$ is not equal to a constant family of hypersurfaces translated by a section of $A$.

Let $m$ be a natural number such that $\Frob_p^m$ fixes the Galois closure of $K$. 

Then for any prime $\ell$, positive integer $c$, and prime $p$ where $A$ has good reduction, there exist an embedding $\iota \colon K \rightarrow \mathbb C$ and a torsion character $\chi$ of $\pi_1^{et}(A_{\iota})$, 
of order a power of $\ell$, such that for every conjugate $(\iota', \chi')$ of $(\iota, \chi)$ by an element of $\Gal(\mathbb Q) \times \Gal_{\mathbb Q^{\textrm{cyc}/\mathbb Q}}$:
\begin{itemize}
\item for $k \neq n-1$, we have $R^{k} {f_{\iota  '}}_* ( g_{\iota '}^* \mathcal L_{(\iota ', \chi ') }) = 0$, and
\item there exist $e_1,\dots, e_c\in \mathbb Z $ such that the monodromy group of $\bigoplus_{i=1}^c  R^{n-1} {f_{\iota '}}_* ( g_{\iota '}^* \mathcal L_{\Frob_p^{me_i} (\iota', \chi') })$ contains $(G^*)^c$ as a normal subgroup, where the representation associated to $\bigoplus_{i=1}^c  R^{n-1} f_* ( g^* \mathcal L_{\Frob_p^{me_i} (\iota', \chi')})$, restricted to $(G^*)^c$, is isomorphic to a sum of $c$ copies of the distinguished representation of $G^*$. 
\end{itemize}
 
\end{cor}

We will eventually apply this result with the parameter $c$ taken sufficiently large to ensure the inequalities stated in Theorem \ref{LV_thm} are satisfied. For this reason, the parameter $c$ will depend on the varieties $A$, $X$, and $Y$ involved, and then Corollary \ref{torsion_char} will give us a character $\chi$ depending on $c$. However, until we choose the parameter $c$ in this way, all our results will be valid for any positive integer $c$.

\begin{proof} This follows from 
the previous 3 results, applying Theorem \ref{thm_fin_many_subtori} to each of the finitely many pairs $(Y_{\iota}, A_{\iota})$. 
\end{proof}

Using Theorem \ref{thm_fin_many_subtori}, we can also prove a result on the period maps of certain variations of Hodge structures associated to families of hypersurfaces in an abelian variety. This is not used anywhere in this paper. Instead, it provides a different perspective on our main result, showing that it is compatible with the ``Shafarevich'' philosophy that varieties with a quasi-finite period map should have finitely many $\mathcal O_K[1/S]$-points for any number field $K$ and set $S$ of prime ideals.

\begin{prop}\label{quasi-finite-period-map}
Let $A$ be an abelian variety of dimension $n\geq 2$ over $\mathbb C$. Let $\phi$ be an ample class in the Picard group of $A$. For a positive integer $m$, let $[m]\colon A \to A$ be the multiplication-by-$m$ map.

There exists a positive integer $m$ such that the natural period map from the moduli space of smooth hypersurfaces $H$ in $A$ of class $\phi$ to a period domain, which sends $H$ to the Hodge structure on $H^{n-1} ( [m]^{-1} H, \mathbb Q)$, is quasi-finite.

\end{prop}

\begin{proof} Let $\mathcal M$ be this moduli space. Suppose that, for some $m$ (to be chosen below), the period map is not quasi-finite. Then its fiber over some point must contain a positive dimensional analytic subvariety $W_m$. 

Consider the variation of Hodge structures over $\mathcal M \times \mathcal M$ whose fiber over a pair of hypersurfaces $H_1, H_2$ is $H^{n-1} ( [m]^{-1} H_{1} , \mathbb Q) \otimes H^{n-1} ( [m]^{-1} H_{2}, \mathbb Q)^\vee$. Over the diagonal in $\mathcal M \times \mathcal M$, this variation of Hodge structures has a Hodge class representing the identity isomorphism between the two Hodge structures. Let $Z_m \subset \mathcal M \times \mathcal M$ be the projection from the universal cover of the locus where this cohomology class is Hodge. 

By \cite[Corollary 1.3]{CDK}, $Z_m$ is a Zariski closed subset.  If $m_1$ divides $m_2$ then we have $Z_{m_2} \subseteq Z_{m_1}$. Because $\mathcal M \times \mathcal M$ is Noetherian, it follows that there exists $m$ such that $Z_{m'} = Z_m$  whenever $m'$ is a multiple of $m$. Fix such an $m$.

 $Z_m$ certainly contains the square of the positive-dimensional analytic subvariety $W_m$ discussed earlier. Fix $x \in W_m$. It follows that the fiber of $Z_m$ over $x$ (for the second projection $Z_m \subseteq \mathcal M \times \mathcal M \rightarrow \mathcal M$) has positive dimension. Let $X$ be the smooth locus of some irreducible component of the fiber of $Z_m$ over $x$. Let $f\colon Y \to X $ be the universal family of hypersurfaces $H_2$ over $X$ and $g\colon Y \to A$ the projection map. By assumption, for all multiples $m'$ of $m$, for all $y \in X$, we have an isomorphism of Hodge structures \[ H^{n-1} ( [m']^{-1} H_{1, x} , \mathbb Q) = H^{n-1} ( [m']^{-1} H_{1, y } , \mathbb Q)  \cong H^{n-1} ( [m]^{-1} H_{2, y }, \mathbb Q). \]
Hence the variation of Hodge structures $H^{n-1} ( [m']^{-1} H_{2, y }, \mathbb Q)$ is constant, and thus has finite monodromy. This is the sum, over characters $\chi$ of $\pi_1(A)$ of order dividing $m'$, of $R^{n-1} f^* (g^* \mathcal L_\chi) $, and so all these individual summands have finite monodromy. 

The family $Y \to X$ is a family of smooth hypersurfaces. Because we have fixed an ample class in the Picard group, and there are only finitely many translates of a given hypersurface in a given Picard class, and because $X$ is a positive-dimensional subvariety of the moduli space $\mathcal M$, it is not the constant family up to translation. From this fact, and the finiteness of the mondromy of  $R^{n-1} f^* (g^* \mathcal L_\chi) $ for all torsion characters $\chi$, we will derive a contradiction.

Before proceeding, we consider the case where $Y_{\overline{\eta}}$ is translation-invariant by a nonzero element of $A$, for $\overline{\eta}$ the generic point of $X$. It follows that the whole family is invariant under the same element. In this case, we consider the subgroup of all such elements and quotient $A$ by it. The family $Y$ is then a pullback from a family $Y'$ of hypersurfaces in this quotient $A'$ of $A$, and the pushforward from $Y'$ of $g^{'*} \mathcal L_{\chi'}$ is a summand of the pushforward from $Y$ of $g^* \mathcal L_\chi$, where $\chi$ is the composition of $\chi'$ with the map $A \to A'$, so our finite monodromy assumption remains true for $Y'$. Hence we may assume that $Y_{\overline{\eta}}$ is not translation-invariant by a nonzero element of $A$.

Let $G$ be the convolution monodromy group of $Y_{\overline{\eta}}$ and let $G^*$ be the commutator subgroup of the identity component of $G$. By Lemmas \ref{lie-irreducible} and \ref{Tannakian-group-simple}, $G^*$ is a simple algebraic group acting by an irreducible representation. We have thus verified all the assumptions of Theorem \ref{thm_fin_many_subtori}. It follows that for all $\chi$ outside some finite set of proper subtori $\Pi(A)$, which necessarily includes at least one torsion character, the geometric monodromy group of $R^{n-1} f_* (g^* \mathcal L_\chi)$ contains $G^*$, contradicting our assumption that it is finite, as desired. \end{proof}

\section{Hodge--Deligne systems}
\label{sec:hds}

The goal of the next few sections is to prove Theorem \ref{LV_thm}, which is analogous to Lemma 4.2, Prop.\ 5.3, and Thm.\ 10.1 in \cite{LV}.
Roughly, the theorem says that, if a smooth variety $X$ over $\mathbb{Q}$ admits a Hodge--Deligne system
that has big monodromy and satisfies two numerical conditions, 
then the integral points of $X$ are not Zariski dense.
We follow the same strategy as \cite{LV}, but we'll need to work in greater generality.
First, \cite{LV} works only with the primitive cohomology of a family of varieties,
but we'll need to work with the cohomology with coefficients in a local system.
Second, we'll need to work with Galois representations valued in a disconnected reductive group.
Finally, we are unable to precisely identify the Zariski closure of the image of monodromy;
we only know that it is a $c$-balanced subgroup of $\Gsemid$ (Definition \ref{cbalanced_dff}).

We'll begin by defining the notion of ``Hodge--Deligne system'', 
which will figure in our statement of Theorem \ref{LV_thm}.
Let $X$ be a variety over a number field $K$ (which will eventually be taken to be $\mathbb{Q}$).
A smooth, projective family of varieties over $X$
gives rise to various cohomology objects on $X$.
The argument of \cite{LV} relies on the interplay among several of these objects:
a complex period map, a $p$-adic period map, and
a family of $p$-adic global Galois representations on $X$.
Deligne has called the collection of these cohomology objects a ``system of realizations''
for a motive \cite{Deligne_motives};
our notion of ``Hodge--Deligne system'' will be closely related to Deligne's systems of realizations.

\subsection{Summary of constructions and notation}
\label{constr_notation}
In this section we define terms like ``Hodge-Deligne system" and ``$H^0$-algebra" in a level of generality which is natural but is greater than what we need for the rest of the argument. Here, we briefly review the specific setup we will use in the proof so that one can have a concrete case in mind when reading the definitions. Thus, the meaning of this summary should not be completely clear before the definitions are read. 


We begin (see Section \ref{locsys}) with $A$ an abelian variety of dimension $n$ over a number field $K$.
Let $X$ be an arbitrary smooth variety over $\mathbb{Q}$,
and $X_K$ its base change to $K$.
Let 
\[Y \subseteq X \times_{\mathbb{Q}} A = X_K \times_K A \]
be a subscheme, smooth, proper and flat over $X_K$.
(In our application, we will take $X$ to be the Weil restriction, from $K$ to $\mathbb{Q}$, 
of a subvariety of the moduli space of hypersurfaces on $A$, and $Y$ the universal hypersurface over $X$.
See the proof of Theorem \ref{main_thm}.)

We can choose finite sets $S$ of places of $K$, and $S'$ of places of $\mathbb{Q}$, and spread everything out to a family
\[ f \colon \mathcal{Y} \subseteq \mathcal{X} \times_{\mathbb{Z}[1/S']} \mathcal{A}, \]
in such a way that $\mathcal{O}_{K, S}$ is finite \'etale over $\mathbb{Z}[1/S']$, 
$\mathcal{A}$ is a smooth abelian scheme over $\mathcal{O}_{K, S}$,
$\mathcal{X}$ is smooth over $\mathbb{Z}[1/S']$,
and $\mathcal{Y} \rightarrow \mathcal{X}_{\mathcal{O}_K, S}$ is smooth, proper, and flat.

Fix a prime $p$. 
Let $L$ be a field, containing $K$, Galois over $\mathbb{Q}$, and over which $A[\rord]$ splits.

Fix some embedding $\iota_0 \colon K \rightarrow L$. 
Fix a natural number $c$.  
(The choice of $c$ will be made in the proof of Theorem \ref{main_thm}, 
depending only on the numerics of certain Hodge numbers; everything we do until then is independent of the choice of $c$.)
Corollary \ref{torsion_char} gives a torsion character $\chi_0$ of $\pi_1^{et}(A_{\iota_{0}})$,
of some order $\rord$,
satisfying a big monodromy condition.

In Lemma \ref{construct_v}, we construct
an $H^0$-algebra $\mathsf{E}_I$ on $\mathcal{O}_{K, S}$ and an $\mathsf{E}_I$-module $\mathsf{V}_I$ on $\mathcal{X}$,
where $I$ is the full $(\Gal_{\mathbb{Q}} \times \Gal_{\mathbb{Q}^{cyc} / \mathbb{Q}})$-orbit containing $(\iota_0, \chi_0)$.
The construction is roughly as follows.
Any character $\chi$ of $\pi_1^{et}(A_{\iota})$, defined over $L$, defines a local system
$\mathsf{L}_{\chi}$ on $\mathcal{A}_{L}$.
By definition, $\mathcal{Y}$ is a subvariety of $\mathcal{X} \rightarrow \mathcal{A}$; let
\[ g \colon \mathcal{Y} \rightarrow \mathcal{A} \]
be the second projection.
By Galois descent,
\[ \bigoplus_{(\iota, \chi) \in I} R^k {f_{\iota}}_* g_{\iota}^* \mathsf{L}_{\chi} \]
descends to a Hodge--Deligne system $\mathsf{V}_I$ on $\mathcal{X}$, which is a module for the algebra \[ \bigoplus_{(\iota, \chi) \in I} \mathbb Q \] which descends to an $H^0$-algebra $\mathsf{E}_I$ on $\mathcal{X}$.  We'll fix $I$, and suppress the subscript $I$ from $\mathsf{E}_I$ and $\mathsf{V}_I$.

In Section \ref{Emodules} we elaborate the structure of $\mathsf{E}$ and $\mathsf{V}$.
Let $E_0 = \mathbb{Q}_p$, and let $E$ be the $\mathbb{Q}_p$-algebra underlying either $\mathsf{E}_{et}$ or $\mathsf{E}_{dR}$.
In Section \ref{semilinear_section}, we define $\Gsimp$ to be one of the groups $GL_N$, $GSp_N$, or $GO_N$,
viewed as an algebraic group over $E$.
We take $\Gzero$ to be the Weil restriction $\Gzero = \Res{E}{E_0} \Gsimp$.
This group $\Gzero$ has an action on a free $E$-module $V$ (coming from the standard representation of $GL_N$, $GSp_N$, or $GO_N$),
and we take $\Gsemi$ to be the normalizer of $\Gzero$ in the group of $E_0$-linear automorphisms of $V$. Whether the de Rham or \'{e}tale version is meant is devoted by subscripts, as in $\Gsemie$ and $\Gsemid$.

\subsection{Hodge--Deligne systems}

\begin{dff}
Let $k$ be an integer, and $q$ a prime power.
A \emph{rational $q$-Weil number} of weight $k$ 
is an algebraic number
\footnote{It is important that we allow Weil numbers that are not algebraic integers, 
since we want Hodge--Deligne systems to form a Tannakian category.
In particular, we want the dual of a Hodge--Deligne system to again be a Hodge--Deligne system.},
all of whose conjugates have complex absolute value $q^{k/2}$.

An \emph{integral $q$-Weil number} is a rational $q$-Weil number that is an algebraic integer.

When $\ell$ is a prime of $\mathcal{O}_K$, we write $q_{\ell}$ for the cardinality of the residue field at $\ell$.
\end{dff}

\begin{dff}
\label{HD_def}
Suppose given a number field $K$
with a chosen embedding $K \rightarrow \mathbb{C}$.
Let $X$ be a smooth variety over $K$. 
Let $S$ be a finite set of primes of $K$,
and let $\mathcal{X}$ be a smooth model of $X$ over $\mathcal{O}_K \left [\frac{1}{S} \right ]$. 
Let $p$ be a prime of $\mathbb Q$ not lying below any place of $S$, such that $K$ is unramified over $p$; 
let $v$ be a place of $K$ lying over $p$.

A \emph{Hodge--Deligne system} 
\footnote{The name is meant to evoke variations of Hodge structure and Deligne's systems of realizations.}
on $\mathcal X$ at $v$
consists of the following structures:
\begin{itemize}
\item A singular local system $\mathsf{V}_{Sing}$ of $\mathbb Q$-vector spaces on $X_{\mathbb C}$.
\item  An \'etale local system $\mathsf{V}_{et}$ of $\mathbb{Q}_p$-vector spaces on $\mathcal{X}_{\text{et}} \times_{\Spec \mathbb{Z}} \Spec \mathbb{Z} [1/p]$.
\item A vector bundle $\mathsf{V}_{dR}$ on $X$, an integrable connection $\nabla$ on $\mathsf{V}_{dR}$, and a descending filtration 
$\Fil^i \mathsf{V}_{dR}$ of $\mathsf{V}_{dR}$ by subbundles 
\[ \mathsf{V}_{dR} = \Fil^{-M} \mathsf{V}_{dR} \supseteq \Fil^{-M+1} \mathsf{V}_{dR} \supseteq \cdots \supseteq \Fil^M \mathsf{V}_{dR} = 0 \]
 (not necessarily $\nabla$-stable), each of which is locally a direct summand of $\mathsf{V}_{dR}$.
 \item A filtered $F$-isocrystal $\mathsf{V}_{cris}$ on $X_{K_v}$ (see, for example, \cite[end of \S 3.1]{Tan_Tong}),
 \end{itemize}
 with the following isomorphisms:
\begin{enumerate} 
\item An isomorphism on $X_{\mathbb C, an}$ between $\mathsf{V}_{Sing} \otimes_{\mathbb Q } \mathbb Q_p$ and the pullback of $\mathsf{V}_{et}$ to $X_{\mathbb C, an }$.
\footnote{We do not use the \'etale-singular comparison; we could have left it out.}
\item An isomorphism on $X_{\mathbb C, an}$  between $\mathsf{V}_{Sing} \otimes_{\mathbb Q} \mathcal  \mathcal O_{X_{\mathbb{C}, an}}$ and $\mathsf{V}_{dR} \otimes_{\mathcal O_X}  \mathcal O_{X_{\mathbb{C}, an}} $.
 \item An isomorphism on an open neighborhood of the rigid analytic generic fiber of $X_{K_v}$ between the underlying vector bundle to $\mathsf{V}_{cris}$ and the pullback of $\mathsf{V}_{dR}$.
 \item An isomorphism on $X_{K_v, proet}$ between the $\mathcal O\mathbb B_{cris}$-modules $\mathsf{V}_{cris} \otimes_{\mathcal O_{X_{K_v}}} \mathcal O \mathbb B_{cris}$ and $\mathsf{V}_{et} \otimes_{\mathbb{Q}_p}  \mathcal O \mathbb B_{cris}$.  
 
 \end{enumerate}
and an increasing filtration $W_i$ of all four objects, compatible with all the isomorphisms, such that all this data satisfies the axioms:
 
 \begin{itemize}
 \item $\Fil^i \mathsf{V}_{dR}$ and $\nabla$ satisfy Griffiths transversality.
 \item The connection $\nabla$ is induced under the isomorphism (2) by the trivial connection on $\mathcal O_X$.
 \item  For each point of $X_{\mathbb{C}}$, the $i$-th associated graded under $W_i$ of the stalk of $( \mathsf{V}_{Sing}, \mathsf{V}_{DR} \otimes_K \mathbb{C}, \Fil^i, (2))$ at that point is a pure Hodge structure of weight $i$.
 \item The $i$-th associated graded under $W_i$ of $\mathsf{V}_{et}$ is pure of weight $i$, i.e.\ for each closed point $x$ of $\mathcal X \times_{\Spec \mathbb{Z}} \Spec \mathbb{Z} [1/p]$ with residue field $\kappa_x$, the eigenvalues of $\operatorname{Frob}_{\kappa_x}$ on the $i$-th associated graded of $\mathsf{V}_{et,x}$ are $|\kappa_x|$-Weil numbers of weight $i$.
 \item The $i$-th associated graded of $\mathsf{V}_{cris}$ under $W_i$ is pure of weight $i$, i.e.\ for each closed point $x$ of $\mathcal X$ lying over $p$ with residue field $\kappa_x$, the eigenvalues of Frobenius on the $i$-th associated graded of $\mathsf{V}_{cris,x}$ are $|\kappa_x|$-Weil numbers of weight $i$.
\item The connection $\nabla$ has regular singularities in a smooth simple normal crossings compactification of $X_K$. 
\item The isomorphism (3) is compatible with the connection.
\item The isomorphism (4) is compatible with connection, filtration, and Frobenius. 
 
 \end{itemize}
 
 \end{dff}

We note that the isomorphism (2) and the first three axioms make up the definition of a variation of Hodge structure;
we will denote by $\mathsf{V}_H$ the variation of Hodge structure given by $\mathsf{V}_{Sing}$, $\mathsf{V}_{dR}$, $\Fil^i \mathsf{V}_{dR}$ and $\nabla$.
The isomorphism (4) and the last axiom make up the definition of a crystalline local system \cite[\S 1]{Tan_Tong}.  (Faltings calls these objects ``dual-crystalline sheaves'' \cite[Theorem 2.6]{Faltings_crys}, at least in the situation where the Hodge--Tate weights are bounded between $0$ and $p-2$.)

We say a Hodge-Deligne system is \emph{pure} of weight $w$ if $W_{w-1}$ vanishes and $W_w$ is the whole system.

The \emph{rank} of a Hodge--Deligne system $\mathsf{V}$ is the rank of the local system $\mathsf{V}_{sing}$ of $\mathbb{Q}$-vector spaces.  By the various isomorphisms, this is equal to the ranks of $\mathsf{V}_{et}$, $\mathsf{V}_{dR}$, and $\mathsf{V}_{cris}$.

We will also need to work with polarized and integral variations of Hodge structure, and for that we need the following slight modifications of the notion of Hodge--Deligne system.

\begin{dff}
Let $K$, $X$, $S$, $\mathcal{X}$, $v$ be as above.
An \emph{integral Hodge--Deligne system} on $\mathcal{X}$
consists of a Hodge-Deligne system on $\mathcal{X}$ together with an integral structure on $\mathsf{V}_{sing}$ (i.e.\ a singular local system  $\mathsf{V}_{int}$ of free $\mathbb{Z}$-modules on $X_{\mathbb{C}}$ together with an isomorphism $\mathsf{V}_{int} \otimes_{\mathbb{Z}} \mathbb{Q} \cong \mathsf{V}_{sing}$.)
\end{dff}

\begin{dff}Let $K$, $X$, $S$, $\mathcal{X}$, $v$ be as above. A \emph{polarized Hodge-Deligne system} on $\mathcal{X}$ consists of a Hodge-Deligne system on $\mathcal{X}$, pure of some weight, together with a polarization of the variation of Hodge structures $(\mathsf{V}_{sing}, \mathsf{V}_{dR},(2))$ (i.e.\ a morphism of local systems $\mathsf{V}_{sing} \otimes \mathsf{V}_{sing} \to \mathbb Q$ which restricted to the stalk at any point of $X(\mathbb C)$ defines a polarization of the pure Hodge structure at that point.)

\end{dff}

\begin{dff}
\label{dff:int-frob-evals}
We say that a Hodge--Deligne system $\mathsf{V}$ \emph{has integral Frobenius eigenvalues} if
the Weil numbers appearing as eigenvalues of Frobenius on 
$\mathsf{V}_{et,x}$ and $\mathsf{V}_{cris,x}$ are integral, 
for all closed points $x \in \mathcal X \times_{\Spec \mathbb{Z}} \Spec \mathbb{Z} [1/p]$
and all closed points $x$ of $\mathcal X$ lying over $p$, respectively. 
\end{dff}

\begin{dff}
The \emph{differential Galois group} of a Hodge--Deligne system $\mathsf{V}$ 
is the differential Galois group of the underlying vector bundle with connection $\mathsf{V}_{dR}$.
(For the definition of differential Galois group, see \cite[\S\S 1.2 -- 1.4]{vdPS}. 
A vector bundle with connection gives rise to a linear differential equation;
by the differential Galois group of the vector bundle with connection,
we mean the differential Galois group of a Picard--Vessiot ring
of the corresponding differential equation.)
\end{dff}

The differential Galois group is the Zariski closure of the monodromy group 
of the variation of Hodge structure $\mathsf{V}_H$;
this follows from the Riemann--Hilbert correspondence, 
and the fact that the period map has regular singularities along a smooth normal crossings compactification. 

\begin{rmk}
\label{rmk:cris}
Let $k_v$ be the residue field of $K$ at $v$.
The ``Frobenius automorphism'' of $K_v$ is the element of $\Gal_{K_v / \mathbb{Q}_p}$
that acts as the $p$-th power map on $k_v$.

A filtered $F$-isocrystal $\mathsf{V}_{cris}$ gives,
for every $\overline{x} \in \mathcal{X}(k_v)$,
a pair 
\[(V_{\overline{x}}, \phi_{\overline{x}}),\]
where $V_{\overline{x}}$ is a $K_v$-vector space, 
and $\phi_{\overline{x}}$ is an endomorphism of $V_{\overline{x}}$, semilinear over Frobenius.
Furthermore, for every $x \in \mathcal{X}(\mathcal{O}_{K_v})$ belonging to the residue class of $\overline{x}$, 
the object $\mathsf{V}_{cris}$ 
defines a filtration on $V_{\overline{x}}$,
with an isomorphism to the filtered vector space $\mathsf{V}_{dR} \otimes K_v$.
We'll call the resulting data 
\[\mathsf{V}_{cris, x} = (V_{cris, x}, \phi_{cris, x}, F_{cris, x}).\]
\end{rmk}

\begin{ex}
\label{triv_example}
(The trivial Hodge--Deligne system.) 

Let $K$ and $E$ be number fields, and let $S$ be a finite set of places of $K$.
Take $\mathcal{X} = \operatorname{Spec} \mathcal{O}_{K, S}$, and define the trivial Hodge--Deligne system $\mathsf{O}_E$ on $\operatorname{Spec} \mathcal{O}_{K, S}$ by:
\begin{itemize}
\item $\mathsf{O}_{E, sing} = E$.
\item $\mathsf{O}_{E, et} = E \otimes_{\mathbb{Q}} \mathbb{Q}_p$, with the trivial Galois action.
\item $\mathsf{O}_{E, dR} = E \otimes_{\mathbb{Q}} K$, with trivial connection and filtration (i.e.\ $\Fil^0 \mathsf{O}_{E, dR} = \mathsf{O}_{E, dR}$, and $\Fil^1 \mathsf{O}_{E, dR} = 0$).
\item $\mathsf{O}_{E, cris}$ is determined by $\mathsf{V}_{dR}$ and the requirement that Frobenius act on $E \otimes_{\mathbb{Q}_p} K_v$ through the trivial action on $E$ and the Frobenius automorphism of $K_v$.

\item The weight filtration is such that $\mathsf{O}_E$ is concentrated in weight zero.
\end{itemize}

When $X$ is an arbitrary smooth $K$-variety, we define the system $\mathsf{O}_E$ on $X$ by pullback from $\Spec K$.
\end{ex}

\begin{rmk}
In general, to give $\mathsf{O}_{E, cris}$ the structure of filtered $F$-isocrystal on $X_{K_v}$, 
it is enough to give a ``Frobenius'' isomorphism between the vector bundle with connection $(\mathsf{O}_{E, dR}, \nabla)$ and its pullback under a lift of Frobenius to $X_{K_v}$.

In Example \ref{triv_example}, with $\mathsf{O}_{E, dR}$ constant and $\nabla = 0$, the Frobenius is simply given as an automorphism of $E \otimes_{\mathbb{Q}_p} K_v$.
\end{rmk}

In general, Hodge--Deligne systems will come from families of varieties by taking cohomology.

\begin{ex}
\label{pushforward_example}
(Pushforward of Hodge--Deligne systems.)

Let $X$ be a smooth variety over a number field $K$, 
and let $\mathcal{X}$ be a smooth model for $X$ over $\mathcal{O}_{K, S}$, for some finite set $S$ of places of $K$.
Let $\pi \colon \mathcal{Y} \rightarrow \mathcal{X}$ be a smooth, projective family of relative dimension $n$;
let $Y$ be the base change of $\mathcal{Y}$ to $K$.
Let $\mathsf{V}$ be a Hodge--Deligne system on $\mathcal{Y}$, and choose some $k$ with $0 \leq k \leq n$.
Let $p$ be a prime of $\mathbb Q$ not lying below any place of $S$, such that $K$ is unramified over $p$, and let $v$ be a place of $K$ lying over $p$.
We define a Hodge--Deligne system $\mathsf{W} = \mathsf{R}^k \pi_*(\mathsf{V})$ on $\mathcal{X}$ at $v$, as follows.

\begin{itemize}
\item Take $\mathsf{W}_{sing} = R^k \pi_* \mathsf{V}_{sing}$, with the pushforward taken in the analytic topology on $X$ and $Y$.
(This is again a local system, by Ehresmann's theorem.)
\item Take $\mathsf{W}_{et} = R^k \pi_* \mathsf{V}_{et}$, with the pushforward taken in the \'etale topology.
(See \cite{FK} for an introduction to the \'etale topology.)
\item Take $\mathsf{W}_{dR}$ to be the relative de Rham cohomology of $\mathsf{V}_{dR}$ over $X$, i.e.\ the pushforward of $\mathsf{V}_{dR}$ as a $D$-module, with its Hodge filtration $\Fil^i$.
This is a filtered vector bundle by Hodge theory.
\item Take $\mathsf{W}_{cris}$ to be the $v$-adic crystalline cohomology of $Y$. 
This is a filtered $F$-isocrystal by $p$-adic Hodge theory (\cite[Thm.\ 8.8]{Scholze}).
(See \cite{Berthelot} or \cite{Berthelot_Ogus} for the construction of crystalline cohomology, and \cite[\S 1]{Ogus} for its structure as filtered $F$-isocrystal.)
\item The isomorphisms (1), (2), (3), (4) follow from Artin's comparison theorem, Hodge theory, the de Rham-crystalline comparison, and relative $p$-adic Hodge theory (\cite[Thm.\ 8.8]{Scholze}, \cite[Thm.\ 5.5]{Tan_Tong}) respectively.
\item The filtration $W_i$ is induced from the filtration of $\mathsf{V}$ after shifting by $k$. In particular, if $\mathsf{V}$ is pure of weight $w$, then $\mathsf{W}$ is pure of weight $w+k$.
\end{itemize}
\end{ex}

\begin{dff}
If $\mathsf{V}$ and $\mathsf{W}$ are two Hodge--Deligne systems on $\mathcal{X}$,
a morphism from $\mathsf{V}$ to $\mathsf{W}$ consists of:
\begin{itemize}
\item A map of analytic local systems $\mathsf{V}_{sing} \rightarrow \mathsf{W}_{sing}$,
\item A map of \'etale local systems $\mathsf{V}_{et} \rightarrow \mathsf{W}_{et}$,
\item A map of vector bundles $\mathsf{V}_{dR} \rightarrow \mathsf{W}_{dR}$, flat with respect to the connections on $\mathsf{V}_{dR}$ and $\mathsf{W}_{dR}$, and respecting the filtrations $\Fil^{*} \mathsf{V}_{dR}$ and $\Fil^{*} \mathsf{W}_{dR}$, and 
\item A map of filtered F-isocrystals $\mathsf{V}_{cris} \rightarrow \mathsf{W}_{cris}$,
\end{itemize}
compatible with all the comparison isomorphisms (1), (2), (3), (4).
\end{dff}

\begin{lem} The Hodge--Deligne systems on $\mathcal{X}$ at $v$ form a Tannakian category with fiber functor given by $\mathsf{V}_{sing, x}$ for some $x\in \mathcal{X}(\mathbb{C})$.\end{lem}

In this Tannakian category, the tensor product of two systems will be defined by separately tensoring the individual objects $\mathsf{V}_{sing},\mathsf{V}_{et}, \mathsf{V}_{dR},$ and $ \mathsf{V}_{cris}$, and similarly for the dual of a system.

\begin{proof}  
Most of the argument is standard.  Only two points require special attention.

The first is existence of a cokernel for $f_{dR}$.  
In general the cokernel of a morphism of vector bundles need not be a vector bundle.
But if  $f \colon \mathsf{V} \rightarrow \mathsf{W}$ is a morphism of Hodge--Deligne systems, 
the cokernel of $f_{dR} \colon \mathsf{V}_{dR} \rightarrow \mathsf{W}_{dR}$ 
must be a vector bundle because $f_{dR}$ is a flat map of vector bundles with connection.

The second is the equality of images and coimages.
A priori, the image of $f_{dR}$
has two possibly different filtrations, the image filtration and the coimage filtration.
But these filtrations agree because variations of Hodge structure form an abelian category.

The remaining verifications are tedious but routine.
\end{proof}

If $f \colon \mathcal{X}' \rightarrow \mathcal{X}$ is a morphism, then for any system $\mathsf{V}$ on $\mathcal{X}$,
we can define the \emph{pullback} $f^* \mathsf{V}$ on $\mathcal{X}'$
by pulling back the four components separately.
When the map $f$ is clear, we'll sometimes write $\mathsf{V} | _{\mathcal{X}'}$ instead of $f^* \mathsf{V}$.

\begin{dff}
Let $\mathcal{X}$ be a smooth $\mathcal{O}_{K, S}$-scheme.  A \emph{constant} Hodge--Deligne system on $\mathcal{X}$ is a system of the form $f^* \mathsf{V}$,
where $\mathsf{V}$ is a Hodge--Deligne system on $\Spec \mathcal{O}_{K, S}$.
\end{dff}

Constant Hodge-Deligne systems will be much more general than most notions of motives. For instance nothing prevents us from combining the \'etale and crystalline cohomology of one variety with the Hodge structure of a different variety, as long as they have the same Hodge numbers. Despite this, the notion of Hodge-Deligne system is strong enough for the arguments that we will make.

\subsection{$H^0$-algebras}
\label{H0_subsection}

Throughout this section, $K$ will denote a number field, and $S$ a finite set of places of $K$.

In order to make the arguments of \cite{LV} work, we need bounds on the centralizer of Frobenius.

The paper \cite{LV} works with Hodge--Deligne systems of the form $\mathsf{R}^k \pi_*(\mathsf{O}_{\mathbb{Q}})$,
for $\pi \colon \mathcal{Y} \rightarrow \mathcal{X}$ a family with not necessarily geometrically connected fibers.
In this context, the zeroth cohomology $H^0(\mathcal{Y}_x)$ of a fiber (over any $x \in \mathcal{X}(K)$) has nontrivial Galois structure.
The action of $H^0(\mathcal{Y}_x)$ on $H^k(\mathcal{Y}_x)$
gives rise to the bounds we need on the Frobenius centralizer, by means of the semilinearity of Frobenius. 
Specifically, $H^k_{cris}(\mathcal{Y}_x)$ has a natural structure of $H^0_{cris}(\mathcal{Y}_x)$-module,
and the Frobenius endomorphism of $H^k_{cris}(\mathcal{Y}_x)$
is semilinear over the Frobenius endomorphism of $H^0_{cris}(\mathcal{Y}_x)$.

We need an analogous statement the Frobenius centralizer in our situation.
Let $\mathcal{A}$ be a smooth proper model of $A$ over $\mathcal{O}_{K, S}$.
Suppose $\mathcal{O}_{K, S}$ is finite \'etale over some $\mathbb{Z}[1/S']$, 
and $\mathcal{X}$ is smooth over $\mathbb{Z}[1/S']$.
Suppose $\mathcal{Y} \subseteq \mathcal{X} \times_{\mathbb{Q}} \mathcal{A}$ is a smooth, proper, and flat over $\mathcal{X}$;
in the application, each fiber $\mathcal{Y}_x$ will be a hypersurface in $A$.
An order-$r$ character $\chi$ of $\pi_1(A)$, defined over some field $L$,
gives rise to a Hodge--Deligne system $\mathsf{L}_{\chi}$ on the base change of $\mathcal{Y}_x$ to $\mathcal{O}_{L, S}$
(Lemma \ref{fin_ord_char});
considering conjugates of $\chi$, we can descend from $L$ to $\mathbb{Q}$ (Lemma \ref{construct_v}).
Taking cohomology, we will study a Hodge--Deligne system $\mathsf{V}$ on $\mathcal{X}$,
whose fiber over a point is (a descent to $\mathbb{Q}$ of) $H^k(\mathcal{Y}_x, \mathsf{L}_{\chi})$.
This will be an algebra over $H^0(\mathcal{Y}_x, \mathsf{L}_{\chi})$;
we will study this structure in detail.

The object $H^0(\mathcal{Y}_x, \mathsf{L}_{\chi})$ 
is the cohomology of a motive with coefficients in $\mathbb{Q}[\mu_r]$,
defined over $K$ and having an algebra structure coming from the group scheme $A[r]$.
The cohomology $H^k(\mathcal{Y}_x, \mathsf{L}_{\chi})$ is naturally a module over the stalk of $ \mathsf{L}_{\chi}$ at the identity thanks to the compatibility of $ \mathsf{L}_{\chi}$ with convolution ;
the purpose of this section and the next is to clarify the structure of the stalk of  $ \mathsf{L}_{\chi}$ and modules over it.

To this end, we will define a general notion of ``$H^0$-algebras.''
Loosely speaking, an $H^0$-algebra is a weight-zero algebra object in the category of Hodge--Deligne systems.
We will soon see that
motives over extensions of $K$ (Example \ref{field_pushforward_h0}), 
motives with coefficients in a number field $E$ (Examples \ref{coeff_algebra} and \ref{coeff_group_module}), 
and motives with an action of a finite abelian group (Examples \ref{group_algebra} and \ref{coeff_group_module})
are all modules over various $H^0$-algebras. 
Our construction (Lemma \ref{construct_v})
of Hodge--Deligne systems coming from families of hypersurfaces on abelian varieties
will combine these ideas.

\begin{dff}
A commutative $H^0$-algebra is a Hodge--Deligne system $\mathsf{E}$
on a smooth $\mathcal{O}_{K, S}$-scheme $\mathcal{X}$,
equipped with morphisms
\[e \colon \mathsf{O}_{\mathbb{Q}} \rightarrow \mathsf{E}\]
and
\[m \colon \mathsf{E} \otimes \mathsf{E} \rightarrow \mathsf{E},\]
satisfying the following properties.
\begin{itemize}
\item $\mathsf{E}$ is pure of weight $0$.
\item The filtration on $\mathsf{E}_{dR}$ is trivial: $\Fil^0 \mathsf{E}_{dR} = \mathsf{E}_{dR}$ and $\Fil^1 \mathsf{E}_{dR} = 0$.
\item The morphisms $e$ and $m$ make $\mathsf{E}$ into a commutative algebra object.
\end{itemize}
When we say ``$H^0$-algebra'', we will mean ``commutative $H^0$-algebra''.
\end{dff}

\begin{ex}
\label{coeff_algebra}
(Trivial Hodge--Deligne system $\mathsf{O}_E$.)

Let $K$ and $E$ be number fields, and let $S$ be a finite set of places of $K$.
The trivial Hodge--Deligne system $\mathsf{O}_E$ of Example \ref{triv_example} has an $H^0$-algebra structure coming functorially from the algebra structure on $E$.
\end{ex}

\begin{dff}
If $\mathcal{X} = \Spec \mathcal{O}_{K, S}$, we say that $\mathsf{E}$ is \emph{\'etale} if $\mathsf{E}_{sing}$ is an \'etale $\mathbb{Q}$-algebra.
\end{dff}

\begin{ex}
\label{H0_alg_str}
(\'Etale $H^0$-algebras over a field.)

Let $\mathcal{X} = \Spec \mathcal{O}_{K, S}$, with $K$ a number field and $S$ a finite set of places of $K$.  
In this setting we can give a concrete description of \'etale $H^0$-algebras $\mathsf{E}$ over $\Spec \mathcal{O}_{K, S}$.

The singular realization $E = \mathsf{E}_{sing}$ has the structure of $\mathbb{Q}$-algebra, which we assume is \'etale; $\mathsf{E}_{dR}$ is determined by
\[ \mathsf{E}_{dR} = E \otimes_{\mathbb{Q}} K \]
with trivial filtration.

The \'etale realization $\mathsf{E}_{et}$ is the $\mathbb{Q}_p$-algebra
\[ E \otimes_{\mathbb{Q}} \mathbb{Q}_p, \]
equipped with a continuous action of the Galois group $\Gal_K$.
By assumption, $E$ is an \'etale $\mathbb{Q}$-algebra, so $\Aut (E \otimes_{\mathbb{Q}} \mathbb{Q}_p)$
is a finite group.
The action of $\Gal_K$ descends to the maximal quotient $\Gal_{K, S'}$ of $\Gal_K$ 
unramified outside the union of $S$ and the set of places of $K$ lying over $p$.

Finally we turn to $\mathsf{E}_{cris}$.  The structure of $K_v$-algebra is given by an isomorphism 
\[\mathsf{E}_{cris} \cong E \otimes_{\mathbb{Q}_p} K_v.\]
The filtration is trivial, and the Frobenius (which we will notate $\sigma$) is the endomorphism 
$\sigma_1 \otimes \sigma_2$ of $E \otimes_{\mathbb{Q}_p} K_v$,
where $\sigma_1$ gives the action of $\operatorname{Frob}_v \in G_K$ on $E$, 
and $\sigma_2$ is the endomorphism of $K_v$ that acts as the $p$-th power map on residue fields.
The Frobenius $\operatorname{Frob}_v \in G_K$ is only well-defined up to conjugacy, 
but that is enough to determine $\mathsf{E}_{cris}$ up to isomorphism.
\end{ex}

\begin{ex}
\label{h0_family_algebra}
($H^0$ of a family.)

If $\pi \colon \mathcal{Y} \rightarrow \mathcal{X}$ is a proper map, 
then the degree-zero cohomology of $\mathcal{Y}$, equipped with the cup product, gives an $H^0$-algebra on $\mathcal{X}$.

The underlying Hodge-Deligne system 
is constructed as in Example \ref{pushforward_example} as $\mathsf{R}^0 \pi_* (\mathsf{O}_{\mathbb{Q}})$.
 The Hodge filtration is trivial because the Hodge filtration on $H^0_{dR}$ of any smooth scheme is trivial. The map $\mathsf{O}_{\mathbb{Q}} \to \mathsf{E}$ is the unit of the adjunction between $\pi_*$ and $\pi^*$, and the map $\mathsf{E} \otimes \mathsf{E} \to \mathsf{E}$ is given by cup product. 
\end{ex}

\begin{ex}
(Group algebra of a finite abelian group.)

For a finite abelian group $G$, define the $H^0$-algebra $\mathsf{O}[G]$ as follows. 
\begin{itemize}
\item $\mathsf{O}[G]_{sing} = \underline{\mathbb Q[G]}$.
\item $\mathsf{O}[G]_{et} = \underline{\mathbb{Q}_p[G]}$, with trivial Galois action.
\item $\mathsf{O}[G]_{dR} = K[G]$ is the trivial vector bundle, with trivial connection.
\item The filtration on $\mathsf{O}[G]_{dR}$ is $\mathsf{O}[G]_{dR} = \Fil^0 \supseteq \Fil^1 = 0$.
\item The filtered $F$-isocrystal is the constant vector bundle $\mathcal{O}[G]$ with trivial connection.
Its fiber at any point is the group algebra $K_v[G]$, with Frobenius action coming from the Frobenius on $K_v$.
\end{itemize}
\end{ex}

Note that, for a group $G$, it is natural to view the group algebra $\mathbb{Q}[G]$ as a space of measures on $G$, and thus dual to the space of functions on $G$. Then the multiplication in the group algebra corresponds to convolution of measures. This suggests the right way to generalize the group algebra to group schemes, as the dual to their ring of functions. (Of course, the trace map makes their ring of functions self-dual.)

\begin{ex}
\label{group_algebra}
(Group algebra of a finite commutative group scheme over $\mathcal{O}_{K, S}$.)

Let $G$ be a finite \'etale commutative group scheme over some $\mathcal{O}_{K, S}$. The group operation $G \times G \to G$ defines a Hopf algebra comultiplication  $\Gamma(G, \mathcal O_G) \to \Gamma(G , \mathcal O_G) \otimes \Gamma(G, \mathcal O_G)$. The dual map $\left( \Gamma(G, \mathcal O_G)\right)^\vee \otimes \left( \Gamma(G, \mathcal O_G)\right)^\vee \to \left( \Gamma(G, \mathcal O_G)\right)^\vee$ gives $\left( \Gamma(G, \mathcal O_G)\right)^\vee$ the structure of an $\mathcal{O}_{K, S}$-algebra. 

Denoting by $\pi \colon G \rightarrow \Spec \mathcal{O}_{K, S}$ the structure map, we define $\mathsf{E} = (\pi_* \mathsf{O}_{\mathbb{Q}} | _{G})^{\vee}$.  A concrete description is as follows.

\begin{itemize}
\item $\mathsf{E}_{sing} = \mathbb{Q}[G(\mathbb{C})]$ with the usual algebra structure. 
\item $\mathsf{E}_{et} = \mathbb{Q}_p[G(\overline{K})]$ with its natural Galois action and structure of $\mathbb{Q}_p$-algebra.
\item  $\mathsf{E}_{dR} = \left( \Gamma(G_K, \mathcal O_G)\right)^\vee$ with is natural algebra structure.
\item $\mathsf{E}_{cris} = \left( \Gamma(G_{K_v}, \mathcal O_{G_{K_v}})\right)^\vee$ with its natural algebra structure
and Frobenius coming from the Galois action on $G$.
\end{itemize}

After passing to an extension of $K$ over which $G$ splits, 
an element $a \in G$ gives an element of each of $\mathsf{E}_{sing}$, $\mathsf{E}_{et}$, $\mathsf{E}_{dR}$, and $\mathsf{E}_{cris}$,
and each of these four realizations is generated (as a vector space over the appropriate field) by $G$.
\end{ex}

\begin{ex}
\label{field_pushforward_h0}
Let $L/K$ be an extension of fields, and let $S$ and $S'$ be finite sets of places of $K$ and $L$, respectively, 
such that $\mathcal{O}_{L, S'}$ is finite \'etale over $\mathcal{O}_{K, S}$.
Let $\pi$ be the map of schemes $\pi \colon \Spec \mathcal{O}_{L, S'} \rightarrow \Spec \mathcal{O}_{K, S}$.  
If $\mathsf{E}$ is an $H^0$-algebra on $\mathcal{O}_{L, S'}$, then $\pi_* \mathsf{E}$ is an $H^0$-algebra on $\mathcal{O}_{K, S}$.
\end{ex}

\subsection{Modules over an $H^0$-algebra}
\label{sec:Emodules}
\begin{dff}
Let $\mathsf{E}$ be an $H^0$-algebra.  
An $\mathsf{E}$-module is an $\mathsf{E}$-module object in the category of Hodge--Deligne systems;
in other words, it is a Hodge--Deligne system $\mathsf{V}$
with a morphism $m_{\mathsf{V}} \colon \mathsf{E} \otimes \mathsf{V} \rightarrow \mathsf{V}$,
such that the composition 
\[ \mathsf{V} \cong \mathsf{O}_{\mathbb{Q}} \otimes \mathsf{V} \xrightarrow{e \otimes 1} \mathsf{E} \otimes \mathsf{V} \xrightarrow{m_{\mathsf{V}}} \mathsf{V} \]
is the identity map, and the diagram
\begin{equation}
\xymatrix{
\mathsf{E} \otimes \mathsf{E} \otimes \mathsf{V} \ar[r]^{1 \otimes m_{\mathsf{V}}} \ar[d]^{m \otimes 1} &  \mathsf{E} \otimes \mathsf{V} \ar[d]^m \\
\mathsf{E} \otimes \mathsf{V} \ar[r]^{m_{\mathsf{V}}} & \mathsf{V}
}
\end{equation}
commutes.
\end{dff}

\begin{rmk}
\label{coeff_group_module}
An $\mathsf{O}_E$-module (Example \ref{coeff_algebra}) is a motive with coefficients in $E$.
An $\mathsf{O}[G]$-module (Example \ref{group_algebra}) is a motive with an action of the group scheme $G$.
\end{rmk}

If $\mathsf{V}$ and $\mathsf{W}$ are $\mathsf{E}$-modules, then we define the tensor product
\[ \mathsf{V} \otimes_{\mathsf{E}} \mathsf{W} \]
as the coequalizer of two maps
\[ \mathsf{E} \otimes \mathsf{V} \otimes \mathsf{W} \rightrightarrows \mathsf{V} \otimes \mathsf{W}, \]
the first of which is induced from $\mathsf{E} \otimes \mathsf{V} \rightarrow \mathsf{V}$, and the second from $\mathsf{E} \otimes \mathsf{W} \rightarrow \mathsf{W}$.
(See \cite[\S 2.3]{Deligne_tensor}.)
Then $\mathsf{V} \otimes_{\mathsf{E}} \mathsf{W}$ also has the structure of $\mathsf{E}$-module.

\begin{dff}
Let $E$ be a finite \'etale algebra over a field $E_0$.
We say an $E$-module $V$ is \emph{equidimensional}
if it is free of finite rank.

Equivalently, writing $E$ as a product of fields $E_i$, 
we say that $V$ is equidimensional if 
\[\dim_{E_i} (V \otimes_E E_i)\]
is independent of $i$.

In this case, we call that dimension the \emph{rank} of $V$.
\end{dff}

\begin{dff}
Suppose $\mathsf{E}$ is a constant \'etale $H^0$-algebra, and $\mathsf{V}$ is an $\mathsf{E}$-module.
We say $\mathsf{V}$ is \emph{equidimensional} if 
the stalks of $\mathsf{V}_{sing}$ are equidimensional modules over the stalks of $\mathsf{E}_{sing}$.
In this case, the ranks of the stalks as $\mathsf{E}_{sing}$-modules are constant;
we call that rank the rank of $\mathsf{V}$.
\end{dff}

\begin{lem}
If $\mathsf{V}$ is an equidimensional $\mathsf{E}$-module of rank $N$ on $X$, then the following statements hold.
\begin{itemize}
\item The stalks of $\mathsf{V}_{et}$ are equidimensional modules of rank $N$ over the stalks of $\mathsf{E}_{et}$.
\item The stalks of $\mathsf{V}_{dR}$ are equidimensional modules of rank $N$ over the stalks of $\mathsf{E}_{dR}$. 
\item The stalks of $\mathsf{V}_{cris}$ are equidimensional modules of rank $N$ over the stalks of $\mathsf{E}_{cris}$. 
\end{itemize}
\end{lem}

\begin{proof}
Follows from the ``comparison isomorphisms'' in the definition of Hodge--Deligne system.
\end{proof}

\begin{dff}
\label{dff:gsp}
Suppose $\mathsf{V}$ is an equidimensional $\mathsf{E}$-module of rank $N$.
Then we say that $\mathsf{V}$ is an \emph{$\mathsf{E}$-module with $GL_N$-structure}, 
or \emph{has $GL_N$-structure over $\mathsf{E}$}.

Suppose additionally that there exists a Hodge-Deligne system $\mathsf{L}$ of rank 1, 
with a nondegenerate bilinear pairing
\[ \mathsf{V} \otimes_{\mathsf{E}} \mathsf{V} \rightarrow \mathsf{L}. \]
If the pairing is alternating, we say that $\mathsf{V}$ has \emph{$GSp_N$-structure over $\mathsf{E}$};
if symmetric, we say that $\mathsf{V}$ has \emph{$GO_N$-structure over $\mathsf{E}$}.
\end{dff}

In Lemma \ref{lem:H_structure}, we will see that if $Y \subseteq X \times_{\mathbb{Q}} A$ is equal to a translate of $[-1]^* Y$, 
then the cup-product pairing gives either $GSp_N$-structure or $GO_N$-structure on its middle cohomology;
here we take $\mathsf{L}$ to be the top-degree cohomology of $Y$.

Below (Definition \ref{G_str}) we will define a notion of ``object with $G$-structure'' in a Tannakian category.
The Tannakian notion (Definition \ref{G_str}) only applies to objects of an $E$-linear tensor category, with $E$ a field;
and $\mathsf{E}$ is not a field. 
We will use the words ``over $\mathsf{E}$'' to emphasize this distinction. The two notions are compatible in that, when $\mathsf{E}_{Sing}$ is a field, viewed as a constant local system, a $GSp_N$-structure on $\mathsf{V}$ over $\mathsf{E}$ gives a $GSp_n$-structure on $\mathsf{V}_{Sing}$ over $\mathsf{E}_{Sing}$, and similar statements are true for $GO_N$ and for the other realizations. 
In Section \ref{Emodules} we'll give a detailed description of these objects in Tannakian terms.


\subsection{Local systems on an abelian variety: construction of a Hodge--Deligne system}
\label{locsys}

Let $A$ be an abelian variety of dimension $n$ over a number field $K$, $X$ an arbitrary smooth variety over $\mathbb{Q}$,
and $X_K$ its base change to $K$.
Let $S$ be a finite set of primes of $\mathcal{O}_K$, and $S'$ a finite set of primes of $\mathbb{Z}$,
such that $\mathcal{O}_{K, S}$ is finite \'etale over $\mathbb{Z}[ 1 / S']$.
Let $\mathcal{A}$ be a smooth proper model for $A$ over $\mathcal{O}_{K, S}$, 
and $\mathcal{X}$ a smooth model for $X$ over $\mathbb{Z}[ 1 / S']$.
Let 
\[ \mathcal{Y} \subseteq \mathcal{X} \times_{\mathbb{Z}[ 1 / S']} \mathcal{A} = \mathcal{X}_{\mathcal{O}_{K, S}} \times_{\mathcal{O}_{K, S}} \mathcal{A} \]
be a subscheme, smooth, proper and flat over $\mathcal{X}_{\mathcal{O}_K, S}$.
Let $f \colon \mathcal{Y} \rightarrow \mathcal{X}$ and $g \colon \mathcal{Y} \rightarrow \mathcal{A}$ be the projections.
Fix a prime $p$ over which $K$ is unramified and a positive integer $\rord$ prime to $p$. 
Let $L$ be a field, containing $K$, Galois over $\mathbb{Q}$, and over which $A[\rord]$ splits;
we suppose that $S$ has been chosen so that $\mathcal{O}_{L, S}$ is unramified over $\mathcal{O}_{K, S}$.
Let $\chi$ be an order-$\rord$ character of $\pi_1^{et}(A)$; 
let $\mathcal{L}_{\chi}$ be the corresponding character sheaf on $\mathcal{A}$. 
(It is a $\mathbb{Q}_p[\mu_{\rord}]$-local system on the \'etale site of $\mathcal{A}_{\mathcal{O}_{L, S}[1/p]}$.)
Let 
\[k = n-1 = \dim A - 1 = \dim Y.\]
We want to create a Hodge--Deligne system on $\mathcal{X}$ whose base change to $\mathcal{O}_{L, S}$ has
$R^k f_* g^* \mathcal{L}_{\chi}$ as a direct summand.  

The tensor product $\mathcal{O}_{K, S} \otimes_{\mathbb{Z}[ 1 / S']} \mathcal{O}_{L, S}$ splits as a direct sum
\[ \mathcal{O}_{K, S} \otimes_{\mathbb{Z}[ 1 / S']} \mathcal{O}_{L, S} \cong \bigoplus_{\iota} \mathcal{O}_{L, S}^{(\iota)}, \]
indexed by the $[K \colon \mathbb{Q}]$ embeddings of $K$ into $L$.
Here each $\mathcal{O}_{L, S}^{(\iota)}$ is an isomorphic copy of $\mathcal{O}_{L, S}$; the superscript $(\iota)$ is merely an index.
We have the corresponding splitting 
\[ \mathcal{A} \times_{\Spec \mathbb{Z}[ 1 / S']} \Spec \mathcal{O}_{L, S} \cong \coprod_{\iota} \mathcal{A}_{\iota}, \]
where for each $\iota$, we define
\[ \mathcal{A}_{\iota} = \mathcal{A} \times_{\Spec \mathcal{O}_{K, S}; \iota} \Spec \mathcal{O}_{L, S}. \]
Similarly, define
\[ \mathcal{Y}_{\iota} = \mathcal{Y} \times_{\Spec \mathcal{O}_{K, S}; \iota} \Spec \mathcal{O}_{L, S}, \]
the base change of $\mathcal{Y}$ along $\iota$.  Then we have the Cartesian diagram
\[
\xymatrix{
\coprod_{\iota}  \mathcal{Y}_{\iota} \ar[r] \ar[d]  & \mathcal{Y} \ar[d] \\
\coprod_{\iota} \mathcal{X}_{\mathcal{O}_{L, S}} \times_{\mathcal{O}_{L, S}} \mathcal{A}_{\iota} \ar[r] \ar[d] & \mathcal{X} \times \mathcal{A} \ar[d] \\
\mathcal{X}_{\mathcal{O}_{L, S}} \ar[r] & \mathcal{X}. \\
}
\]
Let $f_{\iota} \colon \mathcal{Y}_{\iota} \rightarrow \mathcal{X}_{\mathcal{O}_{L, S}} $ and $g_{\iota} \colon \mathcal{Y}_{\iota} \rightarrow \mathcal{A}_{\iota}$ be the projections.

Let $\rord$ be prime to $p$,
and let $\Pi^{K/\mathbb{Q}}(A)[\rord]$ be the set of all pairs $(\iota, \chi)$,
where $\iota \colon K \rightarrow L$ is a $\mathbb{Q}$-linear embedding, and
$\chi$ is a character of $\pi_1(A_{\iota})$ of order dividing $\rord$. 
For fixed $\iota$, the set of characters $\chi$ is naturally identified with the set of homomorphisms
\[ A_{\iota}[\rord] \rightarrow \mu_{\rord}; \]
thus, $\Pi^{K/\mathbb{Q}}(A)[\rord]$ has a natural action of
\[ \Gal_{\mathbb{Q}} \times \Gal_{\mathbb{Q}[\mu_{\rord}] / \mathbb{Q}}, \]
where $\Gal_{\mathbb{Q}}$ acts on the pairs $(\iota, \chi)$ via its action on $L$, and $\Gal_{\mathbb{Q}[\mu_{\rord}] / \mathbb{Q}}$ acts on $mu_{\rord}$.

\begin{lem}
\label{fin_ord_char}
Let $\chi$ be character of $\pi_1^{et}(A)$ of some finite order $\rord$.
Then there exists an $\mathsf{O}_{\mathbb{Q}[\mu_{\rord}]}$-module $\mathsf{L}_{\chi}$ on $\mathcal{A}|_{\mathcal{O}_{L, S}}$
such that $(\mathsf{L}_{\chi})_{et}$ is the character sheaf associated with the character $\chi$,
and $(\mathsf{L}_{\chi})_{sing}$ is the analytic $\mathbb{Q}[\mu_{\rord}]$-local system associated with $\chi$.
\end{lem}

\begin{proof}
Descent.
\end{proof}

\begin{lem}
\label{construct_v}
(Construction of $\mathsf{E}_I$ and $\mathsf{V}_I$.)

Let $I$ be an orbit of $\Gal_{\mathbb{Q}} \times \Gal_{\mathbb{Q}[\mu_{\rord}] / \mathbb{Q}}$ on $\Pi^{K/\mathbb{Q}}(A)[\rord]$.
There exist an $H^0$-algebra $\mathsf{E}_I$ on $\mathcal{O}_{K, S}$ and an $\mathsf{E}_I$-module $\mathsf{V}_I$ on $\mathcal{X}$ with the following properties.
\begin{itemize}
\item After base change to $\mathcal{O}_{L, S}$ and extension of coefficients to $\mathbb{Q}[\mu_{\rord}]$, we have the direct sum decomposition 
\[\mathsf{E}_I |_{\mathcal{O}_{L, S}} \otimes_{\mathsf{O}_{\mathbb{Q}}} \mathsf{O}_{\mathbb{Q}[\mu_{\rord}]} \cong \bigoplus_{(\iota, \chi) \in I} \mathsf{O}_{\mathbb{Q}[\mu_{\rord}]}^{(\iota, \chi)}, \] 
where each $\mathsf{O}_{\mathbb{Q}[\mu_{\rord}]}^{(\iota, \chi)}$ is a copy of $\mathsf{O}_{\mathbb{Q}[\mu_{\rord}]}$.
\item The Galois representation $\mathsf{E}_{I, et}$ 
is compatible with the isomorphism
\[ \mathsf{E}_{I, et} \otimes_{\mathbb{Q}} \mathbb{Q}[\mu_{\rord}] \cong \bigoplus_{(\iota, \chi) \in I} \mathbb{Q}_p \otimes_{\mathbb{Q}} \mathbb{Q}[\mu_{\rord}], \]
where $\Gal_K$ acts on the right-hand side by permutation of the characters $\chi$.

Similarly, the Frobenius endomorphism of $\mathsf{E}_{I, cris}$ is compatible with the isomorphism
\[ \mathsf{E}_{I, cris} \otimes_{\mathbb{Q}} \mathbb{Q}[\mu_{\rord}] \cong \bigoplus_{(\iota, \chi) \in I} \mathbb{Q}_p \otimes_{\mathbb{Q}} \mathbb{Q}[\mu_{\rord}], \]
where Frobenius acts on the right-hand side by permuting the pairs $(\iota, \chi)$, via the $\Gal_{\mathbb{Q}}$ action, with trivial action on $\mathbb{Q}[\mu_{\rord}]$.
\item After base change to $\mathcal{O}_{L, S}$ and extension of coefficients to $\mathbb{Q}[\mu_{\rord}]$, the module $\mathsf{V}_I$ decomposes as the direct sum
\[ \mathsf{V}_I |_{\mathcal{O}_{L, S}} \otimes_{\mathsf{O}_{\mathbb{Q}}} \mathsf{O}_{\mathbb{Q}[\mu_{\rord}]} = \bigoplus_{(\iota, \chi) \in I} R^k {f_{\iota}}_* g_{\iota}^* \mathsf{L}_{\chi}. \]
Furthermore, this decomposition is compatible with the decomposition of 
\[\mathsf{E}_I |_L \otimes_{\mathsf{O}_{\mathbb{Q}}} \mathsf{O}_{\mathbb{Q}[\mu_{\rord}]}\] 
into fields. 
\item $\mathsf{V}_I$ can be made into a polarized, integral Hodge--Deligne system.
\end{itemize}
\end{lem}


\begin{proof}

We have the map 
\[id_X \times [\rord] \colon \mathcal{X} \times_{\mathbb{Z}[ 1 / S']} \mathcal{A} \rightarrow \mathcal{X} \times_{\mathbb{Z}[ 1 / S']} \mathcal{A}\] 
where $[\rord] \colon \mathcal{A} \rightarrow \mathcal{A}$ is multiplication by $\rord$; let
\[ \mathcal{Y}_{\rord} = (id_\mathcal{X} \times [\rord])^{-1}(\mathcal{Y}). \]

\[
\xymatrix{
\mathcal{Y}_r \ar[r] \ar[d] \ar@/_20pt/[dd]_{h} & \mathcal{Y} \ar[d] \ar@/^10pt/[dr]^{g} \ar@/^20pt/[dd]^(.75){f} \\
\mathcal{X} \times \mathcal{A} \ar[r]^{id_\mathcal{X} \times [r]} \ar[d] & \mathcal{X} \times \mathcal{A} \ar[d] \ar[r] & \mathcal{A} \\
\mathcal{X} \ar[r]^{=} & \mathcal{X} \\
}
\]

We have an isomorphism of Hodge--Deligne systems
\[ R^k h_* \mathsf{O}_{\mathbb{Q}} \cong R^k f_* g^* [\rord]_* \mathsf{O}_{\mathbb{Q}}. \]
Furthermore, on $\mathcal{A} \times_{\Spec \mathbb{Z}[1/S']} \Spec \mathcal{O}_{L, S}$, we have the direct sum decomposition of Hodge--Deligne systems 
\[ [\rord]_* \mathsf{O}_{\mathbb{Q}} |_{\mathcal{A} \times_{\Spec \mathbb{Z}[1/S']} \Spec \mathcal{O}_{L, S}} \otimes_{\mathsf{O}_{\mathbb{Q}}} \mathsf{O}_{\mathbb{Q}[\mu_{\rord}]} \cong \bigoplus_{(\iota, \chi) \in \Pi^{K/\mathbb{Q}}(A)[\rord]} \mathsf{L}_{\iota, \chi}, \]
where $\mathsf{L}_{\iota, \chi}$ is the system $\mathsf{L}_{\chi}$ on $\mathcal{A}_{\iota}$, and the trivial system on all other components of $\mathcal{A} \times_{\Spec \mathbb{Z}[1/S']} \Spec \mathcal{O}_{L, S}$.
This gives a decomposition on $\mathcal{X}_{\mathcal{O}_{L, S}}$
\[ (R^k h_* \mathsf{O}_{\mathbb{Q}})|_{\mathcal{O}_{L, S}} \otimes \mathsf{O}_{\mathbb{Q}}[\mu_{\rord}] = \bigoplus_{(\iota, \chi)} R^k f_* g^* \mathsf{L}_{\iota, \chi}
= \bigoplus_{(\iota, \chi)} R^k {f_{\iota}}_* g_{\iota}^* \mathsf{L}_{\chi} . \]

Let
\[ \mathsf{V} = R^k h_* \mathsf{O}_{\mathbb{Q}} \cong R^k f_* g^* [r]_* \mathsf{O}_{\mathbb{Q}}.  \]

Let $\mathsf{E}$ be the pullback to $\mathcal{X}$ of the Hodge--Deligne system on $\mathbb{Z}[1/S']$ coming from 
\[ \mathsf{E} = (\pi_* \mathsf{O}_{\mathbb{Q}} | _{A[\rord]})^{\vee}, \]
where $\pi \colon \Spec K \rightarrow \Spec \mathbb{Q}$ is the projection.
This $\mathsf{E}$ has a structure of $H^0$-algebra, coming from the structure of group scheme on $A[\rord]$
(see Examples \ref{group_algebra} and \ref{field_pushforward_h0}).
Furthermore, the group action $ \mathcal{A}[\rord] \times \mathcal{Y}_{\rord} \rightarrow \mathcal{Y}_{\rord} $ makes $\mathsf{V}$ into a module over $\mathsf{E}$.

After base change from $\mathbb{Z}[1/S']$ to $\mathcal{O}_{L, S}$ and extension of coefficients, we have a decomposition of $H^0$-algebras on $\mathcal{X}$ 
\[ \mathsf{E} |_{\mathcal{O}_{L, S}} \otimes_{\mathbb{Q}} \mathsf{O}_{\mathbb{Q}[\mu_{\rord}]} \cong 
\bigoplus_{(\iota, \chi) \in \Pi^{K/\mathbb{Q}}(A)[\rord]} \mathsf{O}_{\mathbb{Q}[\mu_{\rord}]}^{(\iota, \chi)}. \]

The set $\Pi^{K/\mathbb{Q}}(A)[\rord]$ has commuting actions of $\Gal_{\mathbb{Q}}$ and $\Gal_{\mathbb{Q}^{cyc} / \mathbb{Q}}$,
the former coming from the base field $\mathbb{Q}$, and the latter from the coefficient field $\mathbb{Q}^{cyc}$.
For each orbit $I$ of $\Gal_{\mathbb{Q}} \times \Gal_{\mathbb{Q}^{cyc} / \mathbb{Q}}$,
the direct sum $\bigoplus_{(\iota, \chi) \in I} \mathsf{O}_{\mathbb{Q}[\mu_{\rord}]}^{(\iota, \chi)}$ 
descends to an $H^0$-algebra over $\mathbb{Q}$, with $\mathbb{Q}$-coefficients.
More precisely, choose any $(\iota_0, \chi_0) \in I$,
and let $\mathsf{E}_{I}$ be the pushforward of $\mathsf{O}_{\mathbb{Q}[\mu_{\rord}]}$ from $K$ to $\mathbb{Q}$.
Then there is an isomorphism
\[ \mathsf{E}_{I} |_{\mathcal{O}_{L, S}} \otimes \mathsf{O}_{\mathbb{Q}[\mu_{\rord}]} \cong \bigoplus_{(\iota, \chi) \in I} \mathsf{O}_{\mathbb{Q}[\mu_{\rord}]}^{(\iota, \chi)}. \]
Over $\mathbb{Q}$ and with $\mathbb{Q}$-coefficients,
$\mathsf{E}$ splits as the direct sum of the algebras $\mathsf{E}_{I}$.

The $\mathsf{E}$-module $\mathsf{V}$ also splits as the direct sum of the objects $\mathsf{V}_I = \mathsf{V} \otimes_{\mathsf{E}} \mathsf{E}_I$,
and after base change and extension of coefficients we have
\[ \mathsf{V}_I |_{\mathcal{O}_{L, S}} \otimes_{\mathsf{O}_{\mathbb{Q}}} \mathsf{O}_{\mathbb{Q}[\mu_{\rord}]} = \bigoplus_{(\iota, \chi) \in I} R^k {f_{\iota}}_* g_{\iota}^* \mathsf{L}_{\chi}. \]

Finally we have to explain the polarization and integral structure on $\mathsf{V}_{I, sing}$.
On $\mathsf{V}_{sing}$ we have a standard polarization and integral structure:
the polarization comes from Poincar\'e duality, and the integral structure is simply 
the integral structure on the singular cohomology of $\mathcal{Y}_r$.
These structures induce a polarization and integral structure on $\mathsf{V}_{I, sing}$,
by restricting the polarization and intersecting the integral lattice with $\mathsf{V}_{I, sing}$.
\end{proof}

\begin{lem}
\label{lem:H_structure}
Fix notation as in Lemma \ref{construct_v}.  Then $\mathsf{V}_{I}$ is an $\mathsf{E}_{I}$-module with $GL_N$-structure, in the sense of Definition \ref{dff:gsp}.

Furthermore, if $Y$ is equal to a translate of $[-1]^* Y$, then $\mathsf{V}_{I}$ has $GSp_N$-structure over $\mathsf{E}_{I}$ if $n$ is even and $GO_N$-structure over $\mathsf{E}_{I}$ if $n$ is odd.
\end{lem}

Note that $[-1]^* Y$ is dual to $Y$ in the Tannakian formalism; see \ref{duality-lemma}.

\begin{proof}
To prove that $\mathsf{V}_{I}$ has $GL_N$-structure we only need to check that $\mathsf{V}_{I}$ is equidimensional;
this is a consequence of the transitive Galois action on the index set $I$.

Suppose $Y$ is equal to a translate of $[-1]^* Y$ (so also $\mathcal{Y}$ is equal to a translate of $[-1]^* \mathcal{Y}$).  This equality gives an involution
\[ \iota \colon \mathsf{V} \rightarrow \mathsf{V}. \]
The cup product pairing 
and the trace map compose to give a map
\[ \langle - , - \rangle \colon \mathsf{V} \otimes \mathsf{V}  = R^k h_* \mathsf{O}_{\mathbb{Q}} \otimes R^k h_* \mathsf{O}_{\mathbb{Q}} \rightarrow R^{2k} h_*  \mathsf{O}_{\mathbb{Q}} \rightarrow \mathsf{O}_{\mathbb{Q}}(-k). \]
This pairing does not factor through $\mathsf{V} \otimes_{\mathsf{E}} \mathsf{V}$, which would be equivalent to the identity \begin{equation}\label{not-really-E-linear} \langle ev,w \rangle = \langle v, ew \rangle \end{equation} of maps $\mathsf{V} \times \mathsf{E} \times \mathsf{V} \to \mathsf{O}_{\mathbb{Q}}(-k)$ (where $v,e,$ and $w$ are local sections of the sheaves underlying $\mathsf{V}, \mathsf{E}, $ and $\mathsf{V}$.)
Instead, for $a \in A[r]$, the pairing satisfies
\[ \langle v, w \rangle = \langle av, aw \rangle. \]

However, when $Y$ is equal to a translate of $[-1]^* Y$, and $\iota$ is the involution of $\mathsf{V}$ described above,
we can form the pairing
\[ (v \otimes w) \mapsto \langle v, \iota w \rangle \]
as the composition
\[ \mathsf{V} \otimes \mathsf{V} \stackrel{1 \otimes \iota}{\rightarrow} \mathsf{V} \otimes \mathsf{V} \stackrel{\langle - , - \rangle}{\rightarrow} R^{2k} f_* g^* \mathsf{O}_{\mathbb{Q}} |_Y. \]
We claim that this pairing does satisfy \eqref{not-really-E-linear}, so that it descends to a pairing
\[ \mathsf{V} \otimes_{\mathsf{E}} \mathsf{V} \rightarrow R^{2k} f_* g^* \mathsf{O}_{\mathbb{Q}} |_Y. \]

It is enough to check \eqref{not-really-E-linear} after extending both the base field and the coefficient field, 
so we can assume $\mathsf{E}$ is the group algebra of the finite abelian group $A[r]$, and that $\mathsf{E}$ splits as a direct sum over characters of that group.
Now \eqref{not-really-E-linear}  follows from
\[ \langle av, \iota w \rangle = \langle v, a^{-1} \iota w \rangle = \langle v, \iota (aw) \rangle, \]
where $a \in A[r]$.

We have \[ \langle w, \iota v \rangle = \langle \iota w, \iota^2 v \rangle = \langle \iota w, v\rangle = (-1)^k \langle v, \iota w \rangle \] so this pairing is symplectic if $n$ is even (and thus $k$ is odd) and symmetric if $n$ is odd (and thus $k$ is even).
\end{proof}

For a fixed $(\Gal_{\mathbb{Q}} \times \Gal_{\mathbb{Q}^{cyc} / \mathbb{Q}})$-orbit $I$ on $\Pi^{K/\mathbb{Q}}(A)[\rord]$, 
define the reductive group $\Gsimp$ as follows.
Let $N$ be the rank of $\mathsf{V}_{I}$.
If $Y$ is equal to a translate of $[-1]^* Y$ and $n$ is even, let $\Gsimp = GSp_N$.
If $Y$ is equal to a translate of $[-1]^* Y$ and $n$ is odd, let $\Gsimp = GO_N$.
Otherwise, let $\Gsimp = GL_N$.
Then by Lemma \ref{lem:H_structure}, $\mathsf{V}_{I}$ has $\Gsimp$-structure over $\mathsf{E}_{I}$.
In the following sections we will analyze this structure in some detail.


\subsection{An algebraic group}
\label{semilinear_section}
Let $E_0$ be a field -- we will mostly be interested in the case $E_0 = \mathbb{Q}_p$ -- and let $E$ be a finite \'etale $E_0$-algebra. In this subsection we will define some groups $\Gzero $ and $\Gsemi$. Specializing $E$ to $\mathsf{E}_{et}$, the group $\Gsemi$ will contain the monodromy group of $\mathsf{V}_{et}$, and specializing $E$ to $\mathsf{E}_{dR}$, the group $\Gzero$ will contain the differential Galois group of $\mathsf{V}_{dR}$.
 
Let $\Gsimp$ be a reductive group over $E_0$,
and choose a representation $V_{simp}$ of $\Gsimp$.  
We will assume that $\Gsimp$ is one of $GL_N$, $GSp_N$, or $GO_N$, 
and $V_{simp}$ is the standard representation. A subtlety is that the definition of $GO_N$ depends on the choice of a symmetric bilinear form on $V_{simp}$. Here, and at all future points, when we assume that $\Gsimp$ is $GO_N$, we mean that there exists a nondegenerate symmetric $E_0$-linear form on $V_{simp}$ for which $\Gsimp$ is the group of similitudes. Since this is cumbersome, and the differences between different quadratic forms are of little importance to our argument, we leave the quadratic form implicit and say simply $\Gsimp = GO_N$.

Let $\Gzero$ be the Weil restriction $\Gzero = \Res{E}{E_0} \Gsimp_E$, and let $V = E \otimes_{E_0} V_{simp}$.
By restriction of scalars, we will view $V$ as an $E_0$-vector space with actions of $E$ and $\Gzero$.
Let $\Gsemi$ be the normalizer of $\Gzero$ in the group of $E_0$-linear automorphisms of $V$.

\begin{dff}
Let $\sigma$ be a $E_0$-linear automorphism of $E$.
We say that an $E_0$-linear endomorphism $\phi \colon V \rightarrow V$ is \emph{$\sigma$-semilinear} (or \emph{semilinear over $\sigma$})
if
\[ \phi(\lambda v) = \sigma(\lambda) \phi(v) \]
for all $\lambda \in E$ and $v \in V$.
\end{dff}

\begin{lem}
\label{reductive_normalizer}
If $\Gsimp = GL_N$, then $\Gsemi$ is the algebraic group whose $E_0$-points correspond to endomorphisms $\phi$ of $V$, 
semilinear over some $E_0$-linear automorphism $\sigma$ of $E$.

If $\Gsimp$ is $GSp_N$ or $GO_N$, then $\Gzero$ preserves, up to scaling, an (alternating or symmetric) $E$-valued pairing $\langle - , - \rangle$ on $V$.
In this setting $\Gsemi$ is the group of endomorphisms $\phi$ of $V$, semilinear over some $E_0$-linear automorphism $\sigma$ of $E$, and satisfying
\begin{equation}\label{semilinear-bilinear-form} \langle \phi(v_1), \phi(v_2) \rangle = C \sigma(\langle v_1, v_2 \rangle) \end{equation}
for some $C \in E$, for all $v_1, v_2 \in V$.

In particular, we have an exact sequence of groups
\[ 1 \rightarrow \Gzero \rightarrow \Gsemi \rightarrow \Aut_{E_0} E \rightarrow 1. \]
\end{lem}
\begin{proof} Any element of the normalizer of $\Gzero$ normalizes the center of $\Gzero$, which is $E$, acting by scalar multiplication. The $E_0$-linear automorphisms of $V$ that normalize the action of $E$ are exactly the semilinear automorphisms. 

If $\Gsimp$ is $GSp_N$ or $GO_N$, then for $\phi \in \Gsemi$ which is $E$-semilinear over $\sigma$, the bilinear form $\langle \phi(v_1), \phi(v_2) \rangle$ is $E$-semilinear over $\sigma$ in both variables and is preserved up to scaling by $\Gzero$.  Thus $\sigma^{-1} \left( \langle \phi(v_1), \phi(v_2) \rangle \right)$ is $E$-bilinear and preserved up to scaling by $\Gzero$. Hence it is a scalar multiple of $\langle v_1,v_2 \rangle$, which is the unique $E$-linear $\Gzero$-equivariant form. This gives \eqref{semilinear-bilinear-form}. 

Conversely, any $E$-semilinear automorphism which, if $\Gsimp$ is $GSp_N$ or $GO_N$, satisfies \eqref{semilinear-bilinear-form}, manifestly normalizes $\Gzero$. 
\end{proof}

We need a generalization of \cite[Lemma 2.1]{LV}.
\begin{lem}
\label{semilinear_dim}
Suppose $\sigma \in \Aut_{E_0} E$ is such that $E^{\sigma} = E_0$,
and suppose $\phi \in \Gsemi$ is semilinear over $\sigma$.
Then the centralizer $Z_{\Gsemi}(\phi)$ satisfies
\[ \dim Z_{\Gsemi}(\phi) \leq \dim \Gsimp. \]
\end{lem}

\begin{proof}
The proof goes through exactly as in \cite{LV}. 
By passing to the algebraic closure of $E_0$ we may assume that $E = E_0^d$, so $\Gzero = \Gsimp^d$.
The hypothesis implies that $\sigma$ acts transitively (i.e.\ cyclically) on the $d$ factors $E_0$ of $E_0^d$.
We may assume the factors are ordered so that $\sigma$ takes the $i$-th factor to the $(i+1)$-st factor (modulo $d$).
Then $\phi \in \Gsemi$ has the off-diagonal matrix form
\[
\phi = 
 \begin{pmatrix}
0 & 0 & 0 & \dots & 0 & \phi_1 \\
\phi_2 & 0 & 0 & \cdots & 0 & 0 \\
0 & \phi_3 & 0 & \cdots & 0 & 0 \\
0 & 0 & \phi_4 & \cdots & 0 & 0 \\
\vdots & \vdots & \vdots & \ddots & \vdots & \vdots \\
0 & 0 & 0 & \cdots & \phi_d & 0 \\
\end{pmatrix}. \]

Now if $g = (g_1, \ldots, g_d) \in Z_{\Gzero} (\phi)$, then
\[ g_i \phi_i = \phi_i g_{i-1} \]
for all $i$ (indices modulo $d$).  
Thus, $g_1$ determines $g_i$ for all $i$.
In other words, the projection $Z_{\Gzero} (\phi) \rightarrow \Gsimp$ onto any single factor $\Gsimp$ of $\Gsimp^d$ is injective.
The result follows because $Z_{\Gzero} (\phi)$ has finite index in $Z_{\Gsemi}(\phi)$.
\end{proof}

\subsection{Structure of $\mathsf{E}$-modules}
\label{Emodules}
Let $K$ be a number field, $S$ a finite set of places of $K$, and $\mathcal{X}$ a smooth $\mathcal{O}_{K, S}$-scheme.
Let $\mathsf{E}$ be an \'etale $H^0$-algebra on $\Spec \mathcal{O}_{K, S}$, 
and $\mathsf{V}$ an $\mathsf{E}$-module on $\mathcal{X}$ with $\Gsimp$-structure, where $\Gsimp$ is $GL_N$, $GSp_N$, or $GO_N$.
We described the structure of $\mathsf{E}$ in Example \ref{H0_alg_str}.
Now we are ready to describe $\mathsf{V}$.

The local system $\mathsf{E}_{et}$ is given by a Galois representation on $\Spec K$, 
unramified outside $S$,
which we'll also call $\mathsf{E}_{et}$.
By definition this is a finite \'etale $\mathbb{Q}_p$-algebra
with an action of $\Gal_K$.

Let $E_0 = \mathbb{Q}_p$, and let $\Gsemie$ and $\Gzeroe$ be the groups $\Gsemi$ and $\Gzero$ of Section \ref{semilinear_section} taking $E = \mathsf{E}_{et}$.
At any $x \in X(K)$, the fiber $\mathsf{V}_{et, x}$ of the \'etale local system is a representation of the Galois group $\Gal_K$.
This $V$ has the structure of $E$-algebra, and if $\Gsimp$ is $GSp$ or $GO$ then there is a pairing
\[ \mathsf{V}_{et, x}  \otimes_{\mathsf{E}} \mathsf{V}_{et, x} \rightarrow L, \]
where $L$ is a one-dimensional $\mathbb{Q}_p$-vector space with an action of $\Gal_K$.
The action of $\Gal_K$ respects the pairing, and acts on $\mathsf{V}_{et, x}$ semilinearly over its action on $\mathsf{E}$.
It follows (Lemma \ref{reductive_normalizer}) that the representation of $\Gal_K$ on $V$ is given by a homomorphism
\[ \Gal_K \rightarrow \Gsemie, \]
and the quotient
\[ \Gal_K \rightarrow \Gsemie / \Gzeroe \rightarrow \Aut_{E_0}  \mathsf{E}_{et} \]
is exactly the representation of $\Gal_K$ on $E$ given by the structure of $\mathsf{E}_{et}$.

Now we turn to de Rham cohomology and its cousins. Similarly, $\mathsf{E}_{dR}$ is a finite $\mathbb Q_p$-algebra which is isomorphic to $\mathsf{E}_{et}$ over $B_{dR}$ (and hence also over $\overline{\mathbb Q_p})$ and thus is also \'{e}tale. Let $\Gsemid$ and $\Gzerod$ be the groups $\Gsemi$ and $\Gzero$ of Section \ref{semilinear_section} taking $E = \mathsf{E}_{dR}$.

For a point $x \in X(K)$,  Lemma \ref{reductive_normalizer} implies that $\Gzerod$ is isomorphic to the group of automorphisms of $\mathsf{V}_{dR, x}$
respecting the $\mathsf{E}_{dR}$-action and (where applicable) the bilinear pairing. 

The isomorphism between $\mathsf{E}_{dR}$ and $\mathsf{E}_{et}$ after base change to $\overline{\mathbb Q}_p$ implies that $\Gzeroe$ and $\Gzerod$ become isomorphic after base change to $\overline{\mathbb Q}_p$, and the same is true for $\Gsemie$ and $\Gsemid$.

The differential Galois group $\Gmon$ satisfies $\Gmon \subseteq \Gzerod$.
The weaker inclusion $\Gmon \subseteq \Gsemid$ is immediate from the Tannakian formalism:
the local system $\mathsf{V}_{H}$ is naturally an algebra over $\mathsf{E}_{H}$,
and if $\Gsimp$ is $GSp_N$ or $GO_N$, then $\mathsf{V}_{H}$ admits a bilinear pairing
\[ \mathsf{V}_{H} \otimes_{\mathsf{E}_{H}} \mathsf{V}_{H} \rightarrow (R^{2k} f_* g^* \mathsf{O}_{\mathbb{Q}} |_Y)_{H} \]
(see Lemma \ref{lem:H_structure}).
Since the monodromy action on $\mathsf{E}_{H}$ is trivial,
$\Gmon$ is in fact contained in the finite-index subgroup $\Gzerod$.

Lastly, we describe the filtered $\phi$-module $\mathsf{V}_{cris, x}$.
For this, we assume that $K_v = \mathbb{Q}_p$.
Recall the structure of $\mathsf{E}_{cris}$ from Example \ref{H0_alg_str}:
it is the $\mathbb{Q}_p$-algebra $E$, equipped with a Frobenius endomorphism $\sigma$.
Then $\mathsf{V}_{cris, x}$ is a vector space over $E$, and it is naturally a filtered $\phi$-module with $\Gsemid$-structure.
Its Frobenius endomorphism $\phi$ is semilinear over $\sigma \in \Aut E$.

\subsection{Disconnected reductive groups}
\label{sec:disconnected}

We need to apply $p$-adic Hodge theory in a Tannakian setting:
we'll work with Galois representations valued in the disconnected algebraic group $\Gsemie$,
and the corresponding filtered $\phi$-modules.
To this end, we need some general results about groups with reductive identity component.

We will use a notion of parabolic subgroup due to Richardson \cite{Richardson} 
(see also \cite{Martin} and the survey \cite{BMR_survey}).
Let $G$ be an algebraic group over a field of characteristic zero whose identity component $G^0$ is reductive.
\begin{dff}
\label{parabolic}
Let $f \colon \mathbb{G}_m \rightarrow G$ be a morphism of schemes.
We say that $\lim_{t \rightarrow 0} f(t)$ exists if $f$ extends to a morphism $\tilde{f} \colon \mathbb{A}^1 \rightarrow G$;
in this case we write
\[ \lim_{t \rightarrow 0} f(t) = \tilde{f}(0). \]

Let $\mu \colon \mathbb{G}_m \rightarrow G$ be a cocharacter.
Define the subgroups $P_{\mu}$, $L_{\mu}$, and $U_{\mu}$ of $G$ as follows.
\begin{itemize}
\item $P_{\mu} = \{ g \in G | \lim_{t  \rightarrow  0} \mu(t) g \mu(t)^{-1} \mbox{ exists} \}$
\item $U_{\mu} = \{ g \in G | \lim_{t  \rightarrow  0} \mu(t) g \mu(t)^{-1} = 1 \}$
\item $L_{\mu} = \{ \lim_{t  \rightarrow  0} \mu(t) g \mu(t)^{-1} | g \in P_{\mu} \}$
\end{itemize}

We say that a subgroup $P \subseteq G$ is \emph{parabolic} if it is of the form $P_{\mu}$ for some $\mu$;
in this case we say $L_{\mu}$ is a Levi subgroup associated to $P$, for any $\mu$ such that $P_{\mu} = P$.
\end{dff}

In this setting, $U_{\mu}$ is unipotent, and $P_{\mu}$ is the semidirect product of $L_{\mu}$ by $U_{\mu}$;
in particular, there is a natural projection $P_{\mu} \rightarrow L_{\mu}$, given by
\[ g \mapsto \lim_{t  \rightarrow  0} \mu(t) g \mu(t)^{-1}.\]
Furthermore, $L_{\mu}$ is the centralizer of $\mu$ in $G$ \cite[Prop.\ 5.2]{Martin}.

\begin{ex}
The purpose of this example is to show that a parabolic subgroup $P \subseteq G$ 
is not uniquely determined by the parabolic subgroup $P^0 = P \cap G^0$ of $G^0$.

Let $G$ be the normalizer in $GL_{2N}$ of $GL_N \times GL_N$.
Then $G^0 = GL_N \times GL_N$.
Let cocharacters $\mu_1, \mu_2 \colon \mathbb{G}_m \rightarrow GL_N \times GL_N$ be given by
\[ \mu_1(t) = (t^3, t^2) \]
\[ \mu_2(t) = (t^2, t^2). \]
Then $P_{\mu_1} = G^0$ but $P_{\mu_2} = G$.
\end{ex}

Now we turn to the notion of semisimplification of subgroups of $G$.
See \cite{Serre_cr} for a discussion of this notion in the connected setting; it is applied in \cite[\S 2.3]{LV}.
For the general theory of complete reducibility for disconnected reductive groups, 
see \cite{BMR}, \cite{BMRT}, \cite{BHMR}, and \cite{BMR_survey}.
We warn the reader that the theory is developed there over an arbitrary field,
and many complications arise in positive characteristic that are irrelevant to us here.

\begin{dff}
\label{dff:gcr}
We say that an algebraic subgroup $H \subseteq G$ is \emph{$G$-completely reducible} if, for every parabolic subgroup $P$ containing $H$, 
there is some Levi subgroup $L$ associated to $P$ that also contains $H$.

If $H \subseteq G$ is an arbitrary algebraic subgroup, we define its \emph{semisimplification} with respect to $G$ as follows.
Let $P$ be a parabolic subgroup of $G$, minimal containing $H$.
Choose a Levi subgroup $L$ of $P$, and let $H^{ss}$ be the image of $H$ under the projection $P \rightarrow L$.

We say that a representation valued in $G$ is \emph{semisimple} if the Zariski closure of its image is $G$-completely reducible.
\end{dff}

\begin{lem}
\label{martin_gcr}
Let $H \subseteq G$ be an algebraic subgroup.
Then $H$ is $G$-completely reducible if and only if the identity component of $H$ is reductive.
If we choose an embedding of $G$ into $GL_n$, then $H$ is $G$-completely reducible if and only if it is $GL_n$-completely reducible.

For a general algebraic subgroup $H \subseteq G$,
the $G$-semisimplification $H^{ss}$ is well-defined up to $G$-conjugacy, and it is $G$-completely reducible.
\end{lem}

\begin{proof}
This is Cor.\ 3.5 and Thm.\ 4.5 of \cite{BMR_survey}.
\end{proof}





\subsubsection{Some Tannakian formalism}
\label{tannaka}

We recall the notion of object with $G$-structure.  See \cite[Def.\ 1.3]{Milne_gstr}; a general reference for the Tannakian formalism is \cite{Saavedra}.

\begin{dff}
\label{G_str}

Let $G$ be an algebraic group over a field $E$ of characteristic zero, and let $C$ be an $E$-linear rigid abelian tensor category with fiber functor.
An \emph{object in $C$ with $G$-structure} is an exact faithful tensor functor $\omega \colon \underline{ \operatorname{Rep} }_G \rightarrow C$, 
where $ \underline{ \operatorname{Rep} }_G$ is the category of finite-dimensional $E$-linear representations of $G$, together with an isomorphism between the composition of $\omega$ with the fiber functor of $C$ and the forgetful functor from $ \underline{ \operatorname{Rep} }_G$ to the category $ \underline{ \operatorname{Vec} }_E$ of vector spaces over $E$.

Two \emph{objects in $C$ with $G$-structure} are \emph{$G$-conjugate} if they correspond to isomorphic functors $\omega$. Note that $G$-conjugate objects with $G$-structure differ exactly by an automorphism of the forgetful functor $ \underline{ \operatorname{Rep} }_ G \to  \underline{ \operatorname{Vec} }_E$, i.e. by an element of $G$.
\end{dff}

If $G_1 \rightarrow G_2$ is a morphism of algebraic groups, then an object with $G_1$-structure gives rise to an object with $G_2$-structure by functoriality.

\begin{rmk}
\label{rmk_gstr}
If $V$ is a faithful representation of $G$,
then an object $\omega$ with $G$-structure is determined by $\omega(V)$.
In practice, we will find it useful to specify $\omega$ by describing $\omega(V)$.
\end{rmk}

\subsubsection{Filtrations}
\label{subsub_filt}
(Filtrations on $G$.) 

We want to define a notion of ``filtered vector space with $G$-structure'' or ``filtration on $G$.''
Definition \ref{G_str} does not apply, because the category of filtered vector spaces is not 
abelian.\footnote{In Section \ref{subsub_weakadm} below, we will apply Definition \ref{G_str} to the category of 
weakly admissible filtered $\phi$-modules, which is abelian.}
Instead, we will use the formalism of filtrations from \cite[\S IV.2]{Saavedra}.
Throughout this section, $G$ will be an algebraic group over a field $E$ of characteristic zero.

\begin{dff}
A \emph{filtration} on $G$ (or $G$-filtration) is an exact tensor filtration
of the forgetful functor $\omega$ from $\underline{\operatorname{Rep}}_G$ to $\underline{\operatorname{Vec}}_E$,
in the sense of \cite[IV.2.1.1]{Saavedra}.

In other words, a $G$-filtration is a sequence of exact subfunctors $F^n \omega$ of $\omega$, such that:
\begin{enumerate}
\item For each object $E$ of $\underline{\operatorname{Rep}}_G$, the objects $F^n \omega(E)$ form a decreasing filtration of $\omega(E)$.
\item The associated graded functor $\operatorname{gr}_F(\omega)$, which assigns to $E$ the associated graded of the filtration $F^n \omega(E)$, is exact.
\item For every $n \in \mathbb{Z}$, and all objects $E, E'$ of $\underline{\operatorname{Rep}}_G$, we have
\[ F^n \omega(E \otimes E') = \sum_{a+b = n} F^a \omega(E) \otimes F^b \omega(E').  \]
\end{enumerate}
\end{dff}

Loosely speaking, a $G$-filtration is a choice of descending filtration on every finite-dimensional representation of $G$,
compatible with tensor products, duals, and passage to subquotients.

If $G_1 \rightarrow G_2$ is a morphism of algebraic groups, then a $G_1$-filtration gives rise to a $G_2$-filtration by functoriality.

Since the base field $E$ is of characteristic zero,
every $G$-filtration is ``scindable'' \cite[Th\'eor\`eme IV.2.4]{Saavedra} 
and hence also ``admissible'' \cite[IV.2.2.1]{Saavedra}.
In particular, every filtration comes from a gradation (in general not unique) on $G$.

Let us explain this more carefully.
A representation of $\mathbb{G}_m$ on a finite-dimensional vector space $V$
gives a decomposition $V = \bigoplus V_j$, 
where $V_j$ is the $t^j$-eigenspace of $\mathbb{G}_m$ on $V$.
A cocharacter $\mu \colon \mathbb{G}_m \rightarrow G$
defines a filtration on $G$ by defining, for all representations $V$ of $G$,
\[ \Fil^i V = \bigoplus_{j \geq i} V_j. \]
That every filtration on $G$ is ``scindable'' means that every filtration comes from some cocharacter $\mu$
(in general not unique).

Given a filtration on $G$, we can define two distinguished subgroups $P$ and $U$ of $G$ \cite[\S IV.2.1.3]{Saavedra}.
The subgroup $P \subseteq G$ is the group of elements that stabilize the filtration on every representation $V$ of $G$,
and $U \subseteq P$ is the group of elements that stabilize the filtration and furthermore act as the identity on the associated graded.
When $G$ is a group whose identity component is reductive,
and the filtration is defined by a cocharacter $\mu$,
the group $P$ coincides with the Richardson parabolic $P_{\mu}$ (Definition \ref{parabolic}), and $U$ is its unipotent radical $U_{\mu}$.

Again supposing $G$ is a group whose identity component is reductive,
define an equivalence relation on the set of cocharacters of $G$ as follows:
we say that
$\mu_1 \sim \mu_2$ if $P_{\mu_1} = P_{\mu_2}$ and $\mu_1$ is conjugate to $\mu_2$ in this parabolic.
Then $\mu_1$ and $\mu_2$ give rise to the same filtration on $G$ if and only if $\mu_1 \sim \mu_2$ \cite[\S IV.2.2.4]{Saavedra}.

If $G$ acts faithfully on $V$, then a $G$-filtration is determined by the corresponding filtration on $V$,
as in Remark \ref{rmk_gstr};
we will use this without comment. 

We mention in passing that, even if $G$ is reductive, 
the parabolic $P$ does not quite determine the filtration.
If $G$ is reductive, then $P$ determines a filtration on every representation of $G$
\emph{only up to reindexing}.

\subsubsection{$G$-filtrations vs.\ $G_0$-filtrations}
\label{disconnected_filtrations}
Let $G$ be an algebraic group over a field of characteristic zero whose identity component $G^0$ is reductive.

A $G_0$-filtration $F^0$ gives rise to a $G$-filtration $F$, by functoriality; 
in terms of cocharacters, the correspondence is given by composition
\[ \mathbb{G}_m \rightarrow G_0 \rightarrow G.\]

\begin{dff}
\label{assoc_filt}
We say that a $G_0$-filtration $F^0$ is \emph{associated} to a $G$-filtration $F$
if $F$ is induced from $F^0$ by functoriality with respect to the inclusion $G_0 \hookrightarrow G$.
\end{dff}

\begin{lem}
\label{filtration_lemma}
Being associated defines a bijection between $G_0$-filtrations and $G$-filtrations.
\end{lem}

\begin{proof}
We need to show that every filtration on $G$ comes from a unique filtration on $G_0$.

Existence is easy.  A filtration on $G$ is defined by a cocharacter $\mathbb{G}_m \rightarrow G$;
since $\mathbb{G}_m$ is connected, every such cocharacter factors through $G_0$.

Conversely, suppose 
\[ \mu_1, \mu_2 \colon \mathbb{G}_m \rightarrow G \]
are two cocharacters defining the same filtration on $G$.
We need to show that $\mu_1$ and $\mu_2$ define the same filtration on $G_0$.

We know there is some $g \in P_{\mu_1}$ such that 
\[g \mu_1 g^{-1} = \mu_2.\]
We need to show that we can in fact take 
\[g \in P^0_{\mu_1} = P_{\mu_1} \cap G_0.\]

Let $g_L$ be the image of $g$ under the projection $P_{\mu_1} \twoheadrightarrow L_{\mu_1}$.
Then $g_L$ centralizes $\mu_1$ by \cite[Prop.\ 5.2(a)]{Martin}, 
and from the limit formula
\[ g_L = \lim_{t  \rightarrow  0} \mu_1(t) g \mu_1(t)^{-1} \]
we see that $g g_L^{-1} \in G_0$.
Thus $g g_L^{-1} \in G_0$ conjugates $\mu_1$ to $\mu_2$.
\end{proof}

\subsubsection{Weakly admissible filtered $\phi$-modules with $G$-structure}
\label{subsub_weakadm}

Let $G$ be an algebraic group over $\mathbb{Q}_p$.
(It is important that we work over $\mathbb{Q}_p$, and not an extension, 
because the category of filtered $\phi$-modules is only $\mathbb{Q}_p$-linear.)
A weakly admissible filtered $\phi$-module with $G$-structure gives rise to a $G$-filtration and an element (the Frobenius endomorphism) in $G$.
In the spirit of Remark \ref{rmk_gstr},
we can describe such an object as a triple $(V, \phi, F)$,
where $V$ is a vector space on which $G$ acts faithfully, $\phi \in G$ is an automorphism of $V$, and $F$ is a $G$-filtration on $V$,
such that $(V, \phi, F)$ is weakly admissible. 
Somewhat imprecisely, we will call such objects ``filtered $\phi$-modules with $G$-structure.''

\subsubsection{$p$-adic Hodge theory}

The next result is a Tannakian form of the crystalline comparison theorem, for representations valued in an arbitrary algebraic group $G$.
In order to avoid problems with semilinearity in the target category, we restrict to representations of $\Gal_{\mathbb{Q}_p}$.

\begin{lem}
\label{pht_tannakian}
Let $G \subseteq GL_N$ be an algebraic group over $\mathbb{Q}_p$.
We say that a local Galois representation
\[ \Gal_{\mathbb{Q}_p} \rightarrow G \]
is \emph{crystalline} if the representation
\[  \Gal_{\mathbb{Q}_p} \rightarrow GL_N \]
is crystalline in the usual sense.
This property is independent of the choice of embedding $G \hookrightarrow GL_N$.

\begin{enumerate}

\item The functor $\underline{D}_{\mathrm{cris}}$ of $p$-adic Hodge theory \cite[Expose III]{Asterisque}
extends to a functor from crystalline representations $\Gal_{\mathbb{Q}_p} \rightarrow G$ 
to pairs of an inner form $G'$ of $G$ and an admissible filtered $\phi$-modules over $\mathbb{Q}_p$ with $G'$-structure.  (Here morphisms of such pairs are given by isomorphisms of inner forms of $G$ together to isomorphisms of admissible filtered $\phi$-modules compatible with the isomorphism of inner forms.)

\item A homomorphism of groups $G_1 \rightarrow G_2$ and a crystalline Galois representation $\Gal_{\mathbb{Q}_p} \rightarrow G_1$ induces a morphism of inner forms $G_1' \to G_2'$ compatible with the filtered $\phi$-modules with $G_1'$-structure and $G_2'$-structure associated to $\Gal_{\mathbb{Q}_p} \rightarrow G_1$ and $\Gal_{\mathbb{Q}_p} \rightarrow G_1 \rightarrow G_2$ respectively. The map $G_1'\to G_2'$ is the twist of $G_1\to G_2$ by some $G_1$-torsor. In particular, if $G_1 \to G_2$ has any property preserved by such twists then $G_1' \to G_2'$ does as well.

\item For $G = \Gsemie$ and a Galois representation $\rho \colon \Gal_{\mathbb{Q}_p} \rightarrow \Gsemie $ whose composition with the natural map $\Gsemie  \rightarrow \Aut_{\mathbb Q_p} \mathsf{E}_{et}$ is the usual action of Galois on $\mathsf{E}_{et}$, the relevant inner twist of $\Gsemie$ is $\Gsemid$.\end{enumerate}
\end{lem}

\begin{proof}
The general claims follow from the Tannakian formalism 
and the fact that $\underline{D}_{\mathrm{cris}}$ is a fiber functor 
on the Tannakian category of crystalline Galois representations \cite[Expos\'e III, Proposition 1.5.2]{Asterisque}.

Indeed, a Galois representation $\rho \colon \Gal_{\mathbb{Q}_p} \rightarrow G$ gives rise to an exact tensor functor from $\underline{\operatorname{Rep}}_G$ to the category of Galois representations. The crystallinity assumption implies that all representations in the image are crystalline. The composition of this functor with $\underline{D}_{\mathrm{cris}}$ then gives an exact tensor functor from representations of $G$ to admissible filtered $\phi$-modules, and the underlying vector space is a fiber functor. However, this fiber functor need not be isomorphic to the forgetful functor of  $\underline{\operatorname{Rep}}_G$, so we do not immediately obtain a filtered $\phi$-module with $G$-structure. Instead, for $G'$ the group of automorphisms of this fiber functor, \cite[Theorem 2.11(b)]{DeligneMilne} implies that the Tannakian category of representations of $G$ is equivalent to the Tannakian category of representations of $G'$. Hence this functor defines an admissible filtered $\phi$-module with $G'$ structure. Furthermore, \cite[Theorem 3.2(b)]{DeligneMilne} implies that the category of fiber functors is equivalent to the category of $G$-torsors. Since the automorphism group of any $G$-torsor is an inner form of $G$, it follows that $G'$ is an inner form of $G$.

Functoriality is the statement that if two Galois representations are $G$-conjugate, then the associated groups $G'$ are isomorphic and the associated filtered $\phi$-modules with $G'$-structure are $G'$-conjugate. This is because the two exact tensor functors from $\underline{\operatorname{Rep}}_G$ to the category of Galois representations are isomorphic, and hence the functors from $\underline{\operatorname{Rep}}_G$ to the category of admissible filtered $\phi$-modules are isomorphic. This completes tthe proof of (1).

For (2), we now have a composition of functors, first from $\underline{\operatorname{Rep}}_{G_2}$ to $\underline{\operatorname{Rep}}_{G_1}$, then to the category of Galois representations, then to the category of filtered $\phi$-modules, then to the category of vector spaces. The fiber functor of $G_2$ we use to construct $G_2'$ is given by the composition of all these functors, and the fiber functor of $G_1$ we use to construct $G_1'$ is given by the composition of all these functors but one. Since these are compatible with the functor $\underline{\operatorname{Rep}}_{G_2}\to \underline{\operatorname{Rep}}_{G_1}$, we obtain a map of Tannakian categories $\underline{\operatorname{Rep}}_{G_2'}\to \underline{\operatorname{Rep}}_{G_1'}$, hence a map of groups $G_1'\to G_2$. The fiber functor of $\underline{\operatorname{Rep}}_{G_1}$ corresponds to a $G_1$-torsor, the induced fiber functor of $\underline{\operatorname{Rep}}_{G_2}$ comes from the pushforward of that torsor from $G_1$ to $G_2$, from which it follows that $G_1'\to G_2'$ arises from twisting $G_1\to G_2$ by that $G_1$-torsor.

For (3), observe that the representation $\rho$ corresponds to a functor from the Tannakian category of representations of $\Gsemie$ to Galois representations. This Tannakian category includes the standard representation of $\Gsemie$ on $\mathsf{E}_{et} \otimes_{\mathbb Q_p} V_{simp}$, which may admit a bilinear form, as well as the representation $\mathsf{E}_{et} $. Call these objects $\mathsf{V}_{rep}$ and $\mathsf{E}_{rep}$. The assumptions imply that this functor sends $\mathsf{E}_{rep}$ to the usual Galois representation $\mathsf{E}_{et}$. This Tannakian category admits a second fiber functor $f_{dR}$ which takes each Galois representation to its associated filtered $\phi$-module and takes the underlying vector space. The fiber functor $f_{dR}$ takes $\mathsf{E}_{rep}$ to the filtered $\phi$-module associated to $\mathsf{E}_{et}$, which is $\mathsf{E}_{dR}$.

The relevant inner twist is the automorphism group of $f_{dR}$, which is clearly contained in the group of automorphisms of $f_{dR} ( \mathsf{V}_{rep})$ that are semiilinear over some automorphism of $\mathsf{E}_{rep}$ and respect the bilinear form if it appears. The group of automorphisms satisfying these semilinearity and bilinear form conditions is $\Gsemid$. After passing to an algebraically closed field, the two fiber functors are isomorphic, and so the automorphism group of $f_{dR}$ becomes equal to $\Gsemie$ and thus equal to $\Gsemid$. Thus the automorphism group of $f_{dR}$ must be all of $\Gsemid$ over the base field $\mathbb Q_p$ as well. \end{proof}

\subsection{Adjoint Hodge numbers}

\begin{dff}
\label{dff:adjoint_hodge}
(Adjoint Hodge numbers; see beginning of \S 10 of \cite{LV}.) 

Let $G$ be a reductive group, 
and suppose $F$ is a filtration on $G$.

In this setting, we define the \emph{adjoint Hodge numbers} as follows:
The filtration $F$ on $G$ gives, by definition, a filtration on every representation of $G$.
We apply this to the adjoint representation of $G$ on $\operatorname{Lie} G$, 
and call the adjoint Hodge numbers $h^a$ the dimensions of the associated graded of the resulting filtration\footnote{We use the notation $h^a$, rather than the more common $h^p$, to avoid conflict with the prime $p$ used throughout the proof.}:
\[ h^a = \dim F^a V / F^{a-1} V. \] \end{dff}

\begin{dff}\label{dff:T_function} For any \emph{real} $x \in [ 0, \sum_a h^a ]$, we define the ``sum of the topmost $x$ Hodge numbers'' $T(x)$
to be the continuous, piecewise-linear function $T \colon [ 0, \sum_a h^a ] \rightarrow \mathbb{R}$ satisfying $T(0) = 0$ and
\[ T'(x) = k \]
for $\sum_{a = k+1}^{\infty} h^a < x < \sum_{a=k}^{\infty} h^a$.
(The sums are finite because only finitely many $h^a$ are nonzero.)

When we want to emphasize the dependence on the group $G$, we will write $T_G(x)$ and $h_G^a$ instead of $T(x)$ and $h^a$.
\end{dff}

\begin{dff}
\label{uniform_hodge}
(Uniform Hodge numbers.)

Recall notation from Section \ref{semilinear_section} and \ref{Emodules}.
In particular, $\Gsimp$ is a reductive group over $\mathbb{Q}_p$, and $\Gzerod = \Res{\mathsf{E}_{dR}}{\mathbb{Q}_p} \Gsimp$, is a form of $\Gsimp^d$;
the groups $\Gzerod$ and $\Gsemid$ act on the vector space $\mathsf{V}_{dR, x}$.
Let $F$ be a $\Gzerod$-filtration on $\mathsf{V}_{dR, x}$.

After base change to $\overline{\mathbb{Q}_p}$, we obtain
\[ E \otimes \overline{\mathbb{Q}_p} \cong  \bigoplus_{\iota} (\overline{\mathbb{Q}_p})_{\iota}, \]
the direct sum taken over all embeddings $\iota \colon E \rightarrow \overline{\mathbb{Q}_p}$.
(The subscript on $(\overline{\mathbb{Q}_p})_{\iota}$ is purely for notational convenience: each $(\overline{\mathbb{Q}_p})_{\iota}$ is a copy of $\overline{\mathbb{Q}_p}$.)
Similarly, the group $\Gzero_{dR,\overline{\mathbb{Q}_p}}$ splits as a direct sum of copies of $\Gsimp_{\overline{\mathbb{Q}_p}}$, and $(V, F) \otimes \overline{\mathbb{Q}_p}$ splits as a direct sum of filtered $\overline{\mathbb{Q}_p}$-vector spaces $(V_{\iota}, F_{\iota})$, indexed by $\iota$.

For each $\iota$, let $h^a_{\iota}$ be the adjoint Hodge numbers of the $\Gsimp$-filtration $F_{\iota}$.
We say that $F$ is \emph{uniform} if the numbers $h^a_{\iota}$ are independent of $\iota$.
In this situation, we write the Hodge numbers and associated $T$-function as
\[ h^a = h^a_{\Gsimp} (V_{\iota}) \mbox{ and } T(x) = T_{\Gsimp}(x). \]
\end{dff}

\begin{ex}
Let $E_0 = \mathbb{Q}_p$ and $E = E_0^2$.
Any $E$-module $V$ can be written as a direct sum $V = V_1 \oplus V_2$,
where $V_1$ is a vector space over the first factor $E_0$, and $V_2$ a vector space over the second.
A filtration $F$ on $V$ is determined by a filtration $F_i$ on $V_i$, for $i = 1, 2$.
Then $F$ is uniform if and only if the adjoint Hodge numbers of $\Gsimp$ on $(V_1, F_1)$
are the same as those of $\Gsimp$ on $(V_2, F_2)$.
\end{ex}

\begin{lem}
\label{hodge_scaling}
In the setting of Definition \ref{uniform_hodge}, 
suppose $F$ is a $\Gzerod$-filtration on $V$ with uniform Hodge numbers, 
and let $c = [E \colon \mathbb{Q}_p]$. We have
\[ h^a_{\Gzerod} = c h^a_{\Gsimp} \]
and
\[ T_{\Gzerod}(cx) = cT_{\Gsimp}(x). \]
\end{lem}

\begin{proof}
The Hodge numbers can be computed after base change to $\overline{\mathbb{Q}_p}$.
\end{proof}

\subsection{Galois representations}

The next result is a form of the Faltings finiteness lemma.   Compare also \cite[Lemma 2.4]{LV}, which applies when $G$ is a \emph{connected} reductive group.

\begin{lem}
\label{falt_finite}
Let $E_0 = \mathbb{Q}_p$, and let $E$, $N$, $\Gsimp$, $\Gsemie$, $\Gzeroe$ be as in Section \ref{semilinear_section}.
The group $\Gsemie$ has a standard embedding into $GL_{N [E: \mathbb{Q}_p]} (\mathbb{Q}_p)$,
coming from its action by endomorphisms on a free $E$-module of rank $N$.

Fix a number field $K$, a finite set $S$ of primes of $\mathcal{O}_K$, and an integer $w$.
There are, up to $\Gsemie$-conjugacy, only finitely many Galois representations
\[ \rho \colon \Gal_K \rightarrow \Gsemie \]
satisfying the following conditions.
\begin{itemize}
\item The representation $\rho$ is semisimple (in the sense of Definition \ref{dff:gcr}).
\item The representation $\rho$ is unramified outside $S$.
\item The representation $\rho$ is \emph{pure of weight $w$} with integral Weil numbers: for every prime $\ell \not \in S$, the characteristic polynomial of $\operatorname{Frob}_{\ell}$
has all roots algebraic integers of complex absolute value $q_{\ell}^{w/2}$, where $q_{\ell}$ is the order of the residue field at $\ell$. 
\end{itemize}
\end{lem}

\begin{proof}
If $\rho$ is semisimple in the sense of Section \ref{sec:disconnected}, then it is semisimple in the usual sense (as a representation into $GL_{N [E: \mathbb{Q}_p]} (\mathbb{Q}_p)$); see Definition \ref{dff:gcr} and Lemma \ref{martin_gcr}.

Let $\Gal_L \subseteq \Gal_K$ be the kernel of the natural map from $\Gal_K$ to the component group of the image of $\rho$. By applying \cite[Lemma 2.3]{LV} to a faithful representation $\Gsemie \to GL_N$, one sees that there are only finitely many possibilities for $L$.
The representation $\rho$ restricts to a $\Gzeroe$-valued representation
\[ \rho |_L \colon \Gal_L \rightarrow \Gzeroe, \]
which is also semisimple (in the usual sense).
Now $\Gzeroe$ is a connected reductive group, so we can use \cite[Lemma 2.6]{LV} to conclude that there are only finitely many possibilities for $\rho |_L$. (Kr\"amer and Maculan pointed out an error in the proof of \cite[Lemma 2.6]{LV}: In the notation of that paper, the double cosets mentioned should be cosets for the centralizer of $\mathbf{L}(\mathbf Q_p)$ and not $\mathbf{L}(\mathbf Q_p)$ itself, and this makes the reduction to the case when $\mathbf L$ is connected fail. However, the argument is still correct in the case where the Zariski closure of the image $\mathbf L$ of $\rho \mid_L$ is connected, which we have ensured by our choice of $L$. Following a suggestion of Venkatesh, the argument of \cite[Lemma 2.6]{LV} could be repaired by the reasoning below. An alternate proof is provided in \cite[Proposition 6.6]{KM}.)

For each fixed choice of $\rho|_L$ an extension to a map $\rho$ compatible with the map $ \Gal_K \rightarrow \Gsemie / \Gzeroe$ is determined by its value at a system of coset representatives for $\Gal_K/\Gal_L$, so the set of extensions forms an algebraic variety. To show this set is finite, it suffices to show that conjugation by $Z_{\Gzeroe} ( \operatorname{Im} (\rho|_L)))$ acts transitively on each component of the variety, or equivalently that the tangent space of each point in the variety modulo the tangent space of $Z_{\Gzeroe} ( \operatorname{Im} (\rho|_L)))$ vanishes. This tangent space may be calculated to be $H^1$ of $\Gal_{L/K}$ with coefficients in the Lie algebra of $Z_{\Gzeroe} ( \operatorname{Im} (\rho|_L))$, which vanishes as it is the cohomology of a finite group with coefficients in a vector space of characteristic zero.
\end{proof}

\begin{lem}
\label{fin_levis_v1}
Let $G$ be an algebraic group over a field $E$ whose identity component is reductive.

Let $H \subseteq G$ be an algebraic subgroup. 

Then the set of Levi subgroups of $G$ containing $H$ and defined over $E$
forms finitely many orbits under conjugation by the $E$-points of the centralizer $Z_G(H)$.
\end{lem}

\begin{proof}
A Levi subgroup $L$ of $G$ is the centralizer of a cocharacter $\mu \colon \mathbb{G}_m \rightarrow G$;
the subgroup $L$ contains $H$ if and only if $\mu$ takes values in $Z_G(H)$.
Since all maximal $E$-split tori in $Z_G(H)$ are conjugate,
we can assume that $\mu$ takes values in a fixed torus $T$.

So we need to know that cocharacters $\mu$ of $G$, taking values in a given torus $T$,
define only finitely many different Levi subgroups $L$ of $G$.
It is well-known that there are only finitely many possibilities for $L^0 = L \cap G^0$.
But now we have
\[ L^0 \subseteq L \subseteq N_G(L^0), \]
and since $N_G(L^0)$ contains $L^0$ with finite index, there are only finitely many possibilities for $L$.
\end{proof} 

\begin{rmk} Even if $G^0$ is a torus, and thus has only a single Levi subgroup, there can be multiple Levi subgroups of $G$ if the component group of $G$ acts nontrivially on the cocharacters of $G^0$, because a component will be in the Levi if and only if it fixes the cocharacter $\mu$. However, there will only be finitely many. \end{rmk} 

\begin{lem}
\label{fin_levis_v2}
(only finitely many Levis to give the semisimplification)

Let 
\[ \rho_0 \colon \Gal_K \rightarrow \Gsemie \]
be a semisimple Galois representation.
Then there exists a finite collection of Levi subgroups $L \subseteq \Gsemie$ with the following property:
for any Galois representation
\[ \rho \colon \Gal_K \rightarrow \Gsemie \]
whose semisimplification is $\Gsemie$-conjugate to $\rho_0$,
there exist $g \in \Gsemie$, an $L$ from the finite collection, and a parabolic subgroup $P$ containing $L$,
such that $g \rho g^{-1}$ takes values in $P$, and the composite map
\[ \Gsemie \stackrel{g \rho g^{-1}}{\rightarrow} P \rightarrow L \]
is $\rho_0$.
\end{lem}

\begin{proof}
This is an immediate consequence of Lemma \ref{fin_levis_v1}.
\end{proof}

\section{Period maps and monodromy}
\label{sec:per}

\subsection{Compatible period maps and Bakker--Tsimerman}

The Bakker--Tsimerman theorem \cite{BT} is a strong result on the transcendence of complex period mappings.
It implies a $p$-adic analogue, which we recall here.
(See also \cite[\S 9]{LV}.)


The $p$-adic Bakker-Tsimerman theorem is an unlikely intersection statement for
the $p$-adic period map attached to a Hodge--Deligne system.
(For a detailed discussion of this period map, see \cite[\S 3.3-3.4]{LV}.)
Suppose $\mathsf{V}$ is a Hodge--Deligne system on $\mathcal{X}$, and
let $\Omega \subseteq X(K_v)$ be a $v$-adic residue disk.
For all $x \in \Omega$,
\[ \mathsf{V}_{cris, x} \]
is a filtered $\phi$-module $(V_x, \phi_x, \Fil_x)$.
The structure of $F$-isocrystal means that, for all $x \in \Omega$, the vector spaces $V_x$ are canonically identified,
in such a way that $\phi_x$ is constant on $\Omega$.
Using this identification, the filtration $\Fil_x$ varies $p$-adically in $x$.
A priori, this defines a map
\[ \Phi_p \colon \Omega \rightarrow \mathcal{H} \cong GL_N / P \]
into some flag variety,
where $N$ is the rank of $\mathsf{V}$.
We will show that the period map actually takes values in a smaller flag variety $\Gmon / P_{mon}$,
where $\Gmon$ is the differential Galois group of $\mathsf{V}_{dR}$ in the sense of \cite[\S IV]{katz-conjecture}, i.e.\ the Tannakian group of $\mathsf{V}_{dR}$ in the Tannakian category of vector bundles with flat connection.

\begin{lem}
\label{PV_per}
If $\Gmon$ is the differential Galois group of $\mathsf{V}_{dR}$, then in fact the image of $\Phi_p$ is contained in a single orbit of $\Gmon$.
\end{lem} 

\begin{proof}
(This argument was suggested by Sergey Gorchinskiy.)

Fix a basepoint $x_0 \in X(K')$, for some field $K'$ with fixed embeddings into both $K_v$ and $\mathbb{C}$.
Let $\mathcal{H}$ be the flag variety classifying filtrations with the same dimensional data as 
the filtration on $\mathsf{V}_{dR, x_0}$; we have an isomorphism $\mathcal{H} \cong GL_N / P$,
for $P$ some parabolic subgroup of $GL_N$.

Consider the complex period map $\Phi_{\mathbb{C}}$ from the universal cover $\widetilde{X^{an}}$ of $X^{an}$ to $\mathcal{H}$.
Its image lands in a single orbit of $\Gmon$.  (See \cite[III.A, item (ii) on p.\ 73]{GGK}.)

To transfer the result from $\Phi_{\mathbb{C}}$ to $\Phi_p$, we use Picard--Vessiot theory.  
(For an introduction to Picard--Vessiot theory, see \cite{vdPS}.)

To the vector bundle with connection underlying $\mathsf{V}$, 
Picard--Vessiot theory attaches an (algebraic) $\Gmon$-torsor $PV \rightarrow X_{K'}$,
whose fiber over any $L$-point $x$ classifies vector space isomorphisms 
\[\mathsf{V}_{dR, x} \cong \mathsf{V}_{dR, x_0}\]
respecting the $\Gmon$-structure on both sides (but not, in general, the filtrations).
Furthermore, we obtain a $\Gmon$-equivariant map
\[ \Phi_{PV} \colon PV \rightarrow GL_N/P \]
where a point of $PV$ over $x \in X$ gives an isomorphism $\mathsf{V}_{dR, x} \cong \mathsf{V}_{dR, x_0}$,
and we use that isomorphism to identify the filtration on $\mathsf{V}_{dR, x}$ with a point of the flag variety $\mathcal{H}$.

In the complex setting, 
fix a lift $\widetilde{x_0}$ of $x_0$ to $\widetilde{X^{an}}$.
For every $\widetilde{x} \in \widetilde{X^{an}}$ lying over some point $x \in X^{an}$,
integrating the connection from $\widetilde{x_0}$ to $\widetilde{x}$ gives an identification $\mathsf{V}_{dR, x} \cong \mathsf{V}_{dR, x_0}$,
which gives by definition a point of $PV^{an}$ lying over $x$.
Thus, we obtain a complex-analytic map
\[\iota_{\mathbb{C}} \colon \widetilde{X^{an}} \rightarrow PV^{an}\]
lifting the projection $\widetilde{X^{an}} \rightarrow X^{an}$.

Similarly, in the $p$-adic setting, integration gives a rigid-analytic section 
\[\iota_p \colon \Omega \rightarrow PV^{rig}\]
to the torsor $PV^{rig} \rightarrow \iota_p$.

By definition, the complex and $p$-adic period maps are given by $\Phi_{PV} \circ \iota_{\mathbb{C}}$ and $\Phi_{PV} \circ \iota_{p}$, respectively.

The image of $\Phi_{\mathbb{C}} = \Phi_{PV} \circ \iota_{\mathbb{C}}$ is contained in a single $\Gmon$-orbit on $\mathcal{H}$,
and the image of $\iota_{\mathbb{C}}$ intersects every $\Gmon$-orbit on $PV^{an}$.
Thus, by $\Gmon$-equivariance, the image of $\Phi_{PV}$ is itself contained in a single $\Gmon$-orbit on $\mathcal{H}$,
so the same is true of the image of $\Phi_p$.
\end{proof}

\begin{lem}
\label{padic_per_map}
The construction above defines a $p$-adic period map
\[ \Phi_p \colon \Omega \to \Gmon / P_{mon}, \]
where $P_{mon}$ is a parabolic subgroup of $\Gmon$.
\end{lem}

\begin{proof}
By Lemma \ref{PV_per}, the $p$-adic period map takes values in $\Gmon / (P \cap \Gmon)$,
where $P \subseteq GL_N$ is the parabolic subgroup determined by the Hodge cocharacter $\mu \colon \mathbb{G}_m \rightarrow GL_N$.
What remains is to show that $P \cap \Gmon$ is parabolic in $\Gmon$.

We know that $\mu$ lies in the generic Mumford-Tate group, and hence normalizes $\Gmon$ \cite[\S5 Thm.\ 1]{Andre}. 
Since the outer automorphism group of $\Gmon$ is finite and $\mathbb{G}_m$ is connected, 
the cocharacter $\mu$ acts on $\Gmon$ by inner automorphisms.
Thus the adjoint action of $\mu$ by conjugation on $\Gmon$ gives a homomorphism
\[ \mathbb{G}_m \rightarrow \operatorname{Inn} \Gmon \]
to the group of inner automorphisms of $\Gmon$.
Raising to a power if necessary, we can lift $\mu$ to a cocharacter
\[ \nu \colon \mathbb{G}_m \rightarrow \Gmon.\]
Then the cocharater $\nu$ defines the parabolic subgroup $P_{mon} = P \cap \Gmon$.
\end{proof}

From now on, the parabolic subgroup $P_{mon} \subseteq \Gmon$ will simply be called $P$.
We have defined a $p$-adic period map, valued in $\Gmon / P$.
To study this map, we'll need to recall a corollary of the complex Bakker--Tsimerman theorem \cite{BT}.

\begin{lem}
\label{complex_bt} (Complex Bakker--Tsimerman theorem).
In the above setting, suppose that $Z \subseteq (\Gmon / P)_{\mathbb{C}}$ is an algebraic subvariety, and
\[ \operatorname{codim}_{(\Gmon / P)} Z \geq \dim(X). \]
Then any irreducible component of $\Phi_{\mathbb{C}}^{-1}(Z)$ is contained inside the preimage, in $\widetilde{X^{an}}$,
of the complex points of a proper subvariety of $X_{\mathbb{C}}$.
\end{lem}

\begin{proof}
(This is a mild generalization of \cite[Corollary 9.2]{LV}.  The proof is the same; we reproduce it here for the reader's convenience.)

We will apply \cite[Theorem 1.1]{BT}. 
Let $\check{D} = (\Gmon / P)_{\mathbb{C}}$ and $V = X \times Z \subseteq X^{an} \times \check{D}$.
Let $W$ be the image of $\widetilde{X^{an}}$ under the analytic map $\widetilde{X^{an}} \mapsto X^{an} \times \check{D}$ .
Let $Q$ be an irreducible component of $\Phi_{\mathbb{C}}^{-1}(Z)$;
its image under $\widetilde{X^{an}} \mapsto X^{an} \times \check{D}$ 
is contained in some irreducible component of $W \cap V$; call this component $U$.

Now we apply \cite[Theorem 1.1]{BT}.
Note that
\[ \operatorname{codim}_{X^{an} \times \check{D}} V = \operatorname{codim}_{(\Gmon / P)} Z \geq \dim X \]
by hypothesis, while
\[ \operatorname{codim}_{X^{an} \times \check{D}} W = \dim \check{D}. \]

On the other hand,
we may as well assume $\dim U > 0$; otherwise $Q$ is a point and there is nothing to prove.
So we have the strict inequality
\[ \operatorname{codim}_{X^{an} \times \check{D}} U < \dim X + \dim \check{D} \leq \operatorname{codim}_{X^{an} \times \check{D}} V + \operatorname{codim}_{X^{an} \times \check{D}} W, \]
and we conclude by \cite[Theorem 1.1]{BT}.
\end{proof}

\begin{thm} \label{bt} ($p$-adic Bakker--Tsimerman theorem).
Let $\mathsf{V}$ be a polarized, integral Hodge--Deligne system on a scheme $\mathcal{X}$ smooth over $\mathcal{O}_{K, S}$,
and let $\Gmon$ be the differential Galois group of $\mathsf{V}$.
Let $X$ be the base change of $\mathcal{X}$ to $K$.

Choose a $p$-adic residue disk $\Omega \subseteq \mathcal{X}(\mathcal{O}_{K_v})$,
and a basepoint $x_0 \in \Omega$.
By Lemma \ref{padic_per_map}, these give rise to a period map 
\[\Phi_p \colon \Omega \rightarrow \Gmon / P,\]
where $P$ is a parabolic subgroup of $\Gmon$.

Suppose $Z \subseteq \Gmon/P$ is a closed subscheme such that
\[ \codim_{\Gmon/P} Z \geq \dim X. \]
Then $\Phi_p^{-1}(Z) \subseteq \Omega \subseteq X(\mathcal{O}_{K_v})$ is not Zariski-dense in $X$.
\end{thm}

Theorem \ref{bt} is stated over a general number field $K$ for generality,
but we will apply it with $K = \mathbb{Q}$.

\begin{proof}
This is a mild generalization of \cite[Lemma 9.3]{LV}.
Specifically, \cite[Lemma 9.3]{LV} imposes the following additional hypotheses:
\begin{itemize}
\item The Hodge--Deligne system $\mathsf{V}$ arises as the cohomology (primitive in middle dimension) of a smooth, proper family of varieties.
\item The Zariski closure $\Gmon$ of the image of monodromy coincides with the full orthogonal or symplectic group.
\item The base field is $K = \mathbb{Q}$.
\end{itemize}
The proof from loc.\ cit.\ goes through in our more general setting essentially without change;
here we present the same argument, couched in slightly different language.

First of all,
note that for any fixed $n > 0$, 
$\Omega \subseteq X(\mathcal{O}_{K_v})$ is covered by finitely many mod-$p^n$ residue disks;
we choose $n$ such that the image of each of these residue disks under $\Phi_{p}$ is contained in an affine subset $U$ of $\Gmon / P$.
There is no harm in replacing $\Omega$ with one such mod-$p^n$ residue disk, and we will do so.

Now suppose $Z \subseteq \Gmon/P$ is defined by equations $F_i = 0$, with each $F_i$ a regular function on $U$.
Let $G_i = F_i \circ \Phi_p$; this is an element of $\mathcal{O}(\Omega^{rig})$, 
i.e.\ an element of the Tate algebra $R$ of the affinoid disk $\Omega^{rig}$.

We want to show that the common vanishing locus $\Phi_p^{-1} Z \cap \Omega^{rig}$ of the functions $G_i$ on $\Omega^{rig}$ 
is contained in the zero-set of some regular function on $X$.
Now $\Phi_p^{-1} Z \cap \Omega^{rig}$ is the union of finitely many irreducible components,
so it suffices to show that any one such irreducible component (call it $W_i$) is contained in the zero-set of some regular function on $X$.
As we are only interested in $K_v$-points, we may as well assume that $W_i(K_v)$ is nonempty.

Choose a basepoint $x_0 \in W_i(K_v)$ and an isomorphism $\sigma \colon \overline{\mathbb{Q}_p} \cong \mathbb{C}$, and
repeat the construction of Lemma \ref{PV_per} with respect to the basepoint $x_0$ (taking $K' = K_v$).
In particular, we have the algebraic Picard--Vessiot torsor $PV \rightarrow X_{K_v}$ under the group $\Gmon$
and the ``Picard--Vessiot period map''
\[ \Phi_{PV} \colon PV \rightarrow GL_N/P. \]
We also have maps
\[\iota_{\mathbb{C}} \colon \widetilde{X^{an}} \rightarrow PV^{an}\]
and
\[\iota_p \colon \Omega \rightarrow PV^{rig},\]
such that the complex and $p$-adic period maps are given by
\[ \Phi_{\mathbb{C}} = \Phi_{PV} \circ \iota_{\mathbb{C}}\]
and 
\[ \Phi_{p} = \Phi_{PV} \circ \iota_{p},\]
respectively.

Recall that both $\Phi_{\mathbb{C}}$ and $\Phi_{p}$ are defined by integrating the connection on $\mathsf{V}_{dR}$.
In particular, with respect to some fixed coordinates on $PV$ and local coordinates on $X$,
both $\Phi_{\mathbb{C}}$ and $\Phi_{p}$ are given by the \emph{same} power series
(which then has positive radius of convergence both for the complex and the $p$-adic metric).

Let us rephrase this, more precisely, in geometric terms.
Let $\hat{X}$ be the formal completion of $X_{K_v}$ at the point $x_0$;  
it is noncanonically isomorphic to (the formal spectrum of) a power series ring in $\dim X$ variables over $K_v$.
The formal completions of $X^{rig}$ and $X^{an}$ at $x_0$ are related to $\hat{X}$ in the obvious way.
If we may abuse notation by making a diagram out of objects of three different categories,
we have the following inclusions:
\[
\xymatrix{
 & X^{rig} \\
 \hat{X} \ar[ur] \ar[dr] & \\
  & X^{an}.
}
\]

Over the formal scheme $\hat{X}$, the Picard--Vessiot torsor $PV$ admits a canonical section $\iota_{form}$,
given (as in the proof of Lemma \ref{PV_per}) by integrating the connection;
the sections $\iota_{p}$ and $\iota_{\mathbb{C}}$ over $X^{rig}$ and $X^{an}$, respectively,
both restrict to $\iota_{form}$ when pulled back to $\hat{X}$.
Let $\Phi_{form} = \Phi_{PV} \circ \iota_{form}$; this is the restriction to $\hat{X}$ of both the complex and $p$-adic period maps.
This is the sense in which $\Phi_{p}$ and $\Phi_{\mathbb{C}}$ have the same power series.

Returning to the main argument, we will apply Lemma \ref{complex_bt} over $X^{an}$,
then ``pull back to $\hat{X}$ and push forward to $X^{rig}$.''

Pulling $Z$ back via $\sigma$ gives a variety $Z^{\sigma} \subseteq (\Gmon / P_{mon})_{\mathbb{C}}$;
its inverse image in $\widetilde{X^{an}}$ is cut out by analytic functions $G_i^{\sigma}$.
(The functions have the same power series as the functions $G_i$, 
regarded now as complex power series via the homomorphism $\sigma$.)

Now we apply the complex Bakker--Tsimerman theorem (Lemma \ref{complex_bt})
to the variety $Z^{\sigma}$.
As a result, we obtain some regular function $H$ on the scheme $X_{\mathbb{C}}$
that vanishes (pointwise) on $\Phi_{\mathbb{C}}^{-1}(Z)$.
By the locally analytic Nullstellensatz \cite[\S 3.4]{LAGeo}, 
some power $H^m$ lies inside the ideal generated by the functions $G_i^{\sigma}$ inside the ring of germs of analytic functions at $x_0$.

Pulling back to $\hat{X}$, it is clear that $H^m$ also vanishes on the formal subscheme $\Phi_{form}^{-1}(Z)$ of $\hat{X}$.

By means of our chosen isomorphism $\sigma \colon \overline{\mathbb{Q}_p} \rightarrow \mathbb{C}$,
we obtain a regular function $(H^m)^{\sigma^{-1}}$ on $X_{\overline{\mathbb{Q}_p}}$ that again vanishes on $\Phi_{form}^{-1}(Z)$.
This $(H^m)^{\sigma^{-1}}$ must in fact be defined over some finite extension of $K_v$;
taking a norm, we may assume that it is in fact a regular function on $X_{K_v}$.
For simplicity, let us call this function $H$.

At this point we have a regular function $H$ on $X_{K_v}$, and we know that it vanishes on $\Phi_{form}^{-1}(Z) \subseteq \hat{X}$.
To conclude, we need to show that $H$ vanishes on the rigid-analytic neighborhood $W_i$ of $x_0$.
To this end, recall that $R$ was the Tate algebra of the affinoid disk $\Omega^{rig}$;
let $I \subseteq R$ be the ideal defining $W_i$.
``Clearing denominators'' if necessary in the regular function $H$, we may assume $H$ lies in the Tate algebra $R$.
Let $\mathfrak{m}$ denote the ideal of $R/I$ defining the point $x_0$.
We know that $H \in \mathfrak{m}^i (R/I)$ for every integer $i \geq 0$; 
by Krull's Intersection Theorem, we conclude that $H$ vanishes in $R/I$, that is, that $H \in I$.
Thus $W_i$ is contained in the vanishing locus of the algebraic function $H$.
\end{proof}

\subsection{Complex monodromy}

Our next goal is Lemma \ref{cbalanced_monodromy}, which shows that for any positive integer $c$, 
we can construct our $\mathsf{V}$ in such a way that its differential Galois group is strongly $c$-balanced 
(Definition \ref{cbalanced_dff} below).  
This is a ``big monodromy'' statement, analogous to \cite[Lemma 4.3]{LV} and \cite[Theorem 8.1]{LV}.

\begin{lem}
\label{goursat1}
Let $H$ be one of the algebraic groups $SL_N$, $Sp_N$, or $SO_N$.
Let $G$ be a subgroup of $H^d$ such that each of the $d$ coordinate projections $\pi_i \colon G \rightarrow H$ (for $1 \leq i \leq d$) is surjective.

Define a relation $\sim$ on the index set $\{1, \ldots, d\}$ by declaring that $i \sim j$ if and only if the projection
\[ (\pi_i, \pi_j) \colon G \rightarrow H^2 \]
is \emph{not} surjective.

\begin{enumerate}
\item The relation $\sim$ is an equivalence relation.
\item If $i_1, \ldots, i_c$ are a complete set of representatives for $\sim$,
then the map
\[ (\pi_{i_1}, \ldots, \pi_{i_c}) \colon G \rightarrow H^c \]
is surjective with finite kernel.
\end{enumerate}
\end{lem}

\begin{proof}
This is an algebraic version of Goursat's lemma.  See \cite[Lemma 2.12]{LV} and \cite[Lemma 5.2.1]{Ribet}.
 
The group $H$ has a finite center; call it $Z$.
For any two indices $i$ and $j$ (possibly equal), there are two possibilities for the image of the projection
\[ G \stackrel{(\pi_i, \pi_j)}{\rightarrow} H^2 \rightarrow (H/Z)^2. \]
Either the map is surjective, or its image is the graph of an automorphism of $H/Z$.
In the former case, $(\pi_i, \pi_j)$ must surject (onto $H^2$) as well.

An easy calculation shows that $\sim$ is an equivalence relation.
If $i_1, \ldots, i_c$ is a complete system of representatives for $\sim$, then
repeated application of Goursat's lemma shows that
\[ (\pi_{i_1}, \ldots, \pi_{i_c}) \colon G \rightarrow H^c \]
is surjective with finite kernel.
\end{proof}

\begin{dff}
\label{cbalanced_dff}
Let $H$ be one of the algebraic groups $SL_N$, $Sp_N$, or $O_N$, and $G$ a subgroup of $H^d$.
For $1 \leq i \leq d$, let $\pi_i \colon G \rightarrow H$ be the coordinate projection,
and suppose that each $\pi_i$ is surjective.
The \emph{index classes} of $G$ are the equivalence classes of the relation $\sim$ of Lemma \ref{goursat1}.

Let $c$ be a positive integer\footnote{not necessarily the same as the $c$ of Lemma \ref{goursat1}}.  We say that $G$ is \emph{$c$-balanced} (as a subgroup of $H^d$)
if its index classes are all of equal size, and there are at least $c$ of them.

Suppose now we are given a permutation $\sigma$ of the index set $\{1, \ldots, d\}$.
(In the sequel, $\sigma$ will come from Frobenius.)
We say $G$ is \emph{strongly $c$-balanced} (with respect to $\sigma$)
if it is $c$-balanced,
each orbit of $\sigma$ on $\{1, \ldots, d\}$ contains elements of at least $c$ of the index classes of $G$,
and $\sigma$ preserves the partition of $\{1, \ldots, d\}$ into index classes.

Finally, let $E_0, E, \Gsimp, \Gsemid, \Gzerod$ be as in Section \ref{semilinear_section}, and let $\Goned$ be: 
\begin{itemize}
\item in the case $\Gsimp= GL_N$, the kernel of the determinant map $\Gzerod \rightarrow \mathbb{G}_{m, E}$
\item in the case $\Gsimp=GSp_N$, the kernel of the similitude character $\Gzerod\rightarrow \mathbb{G}_{m,E}$
\item in the case $\Gsimp=GO_N$, the intersection of the kernels of the determinant map $\Gzerod \rightarrow \mathbb{G}_{m, E}$  and the  similitude character $\Gzerod \rightarrow \mathbb{G}_{m, E}$ .
\end{itemize}
Then $\Goned$ is a form of $SL_N^d$, $Sp_N^d$, or $O_N^d$.
We say that an algebraic subgroup $G \subseteq \Gzerod$ is \emph{$c$-balanced} (resp.\ \emph{strongly $c$-balanced} with respect to $\sigma$)
if $(G \cap \Goned)_{\overline{E}}$ is \emph{$c$-balanced} (resp.\ \emph{strongly $c$-balanced} with respect to $\sigma$)
as a subgroup of $SL_n^d$, $Sp_N^d$ or $O_N^d$.
\end{dff}

Note that the condition of being strongly $c$-balanced gets stronger as $c$ grows. We will later choose $c$ to be sufficiently large to satisfy some inequalities given in Theorem \ref{LV_thm}. We do that in the proof of Theorem \ref{main_thm}; until then, all our statements will be proven for an arbitrary natural number $c$.

%


\begin{ex}
Let $\Delta \subseteq H^2$ be the diagonal, and let $G = \Delta \times \Delta \times \Delta \times \Delta \subseteq H^8$.

The index classes of $G$ are $\{1, 2\}, \{3, 4\}, \{5, 6\}, \{7, 8\}$; thus $G$ is $c$-balanced for $c = 1, 2, 3, 4$.

Let $\sigma = (13)(24)(57)(68)$.  Then $G$ is strongly $c$-balanced with respect to $\sigma$ only for $c = 1, 2$.

(We thank the anonymous referee for this example.) 
\end{ex}

\begin{lem}
\label{cbalanced_monodromy}
Recall notation from Section \ref{constr_notation}.
Specifically, let $K$, $p$, $A$, $X$, $Y$, $\mathcal{A}$, $\mathcal{X}$, $\mathcal{Y}$, $L$ be as in Section \ref{locsys}, and let $v$ be a place of $K$ over $p$.
Fix some embedding $\iota_0 \colon K \rightarrow L$. Fix a natural number $c$, and
pick $\chi_0$ as in Corollary \ref{torsion_char} (depending on $c$), with $G^* = SL_N, Sp_N$, or $SO_N$,
and let $E_0, E, \Gsimp, \Gsemid, \Gzerod$ be as in Sections \ref{semilinear_section} and \ref{Emodules}.
Let $I$ 
be the full $(\Gal_{\mathbb{Q}} \times \Gal_{\mathbb{Q}^{cyc} / \mathbb{Q}})$-orbit containing $(\iota_0, \chi_0)$,
and let $\mathsf{V} = \mathsf{V}_{I}$ be the corresponding Hodge--Deligne system.

The Frobenius at $v$, acts on the set $I$ through the diagonal action as an element of $(\Gal_{\mathbb{Q}} \times \Gal_{\mathbb{Q}^{cyc} / \mathbb{Q}})$; 
call this permutation $\sigma$.

Then the differential Galois group $\Gmon$ of $\mathsf{V}$ (base-changed to $\mathbb Q_p$)
is a strongly $c$-balanced subgroup of $\Gzerod$, with respect to $\sigma$.
\end{lem}

\begin{proof}
This is little more than a restatement of Corollary \ref{torsion_char}.

After base change to $\mathbb{C}$, we have that $\Gsimp_{\mathbb{C}}$ is either $GSp$, $GO$, or $GL$; 
and $\Gzero_{dR,\mathbb{C}}$ splits as a product of copies of $\Gsimp_{\mathbb{C}}$, indexed by pairs $(\iota, \chi) \in I$.
We'll write each of these direct factors of $\Gzero_{dR,\mathbb{C}}$ as $\Gsimp_{(\iota, \chi)}$.

The differential Galois group, after base change to $\mathbb{C}$, is the Zariski closure of the monodromy of the variation of Hodge structure.
The variation of Hodge structure $\mathsf{V}_H$ splits as the direct sum
of $\mathsf{V}_{H, (\iota, \chi)} = R^k {f_{\iota}}_* g_{\iota}^* \mathcal{L}_{\chi}$.
Clearly monodromy acts trivially on the set of pairs $(\iota, \chi)$,
and when $\Gsimp$ is $GSp$ or $GO$, there is a bilinear pairing on $R^k {f_{\iota}}_* g_{\iota}^* \mathcal{L}_{\chi}$.
Thus, $\Gmon \subseteq \Gzerod$.

Now Corollary \ref{torsion_char}
implies that 
the geometric monodromy group of $\mathsf{V}_{H, (\iota_0, \chi_0)}$ is all of $\Gsimp(\mathbb{C})$.
By symmetry under the action of $(\Gal_{\mathbb{Q}} \times \Gal_{\mathbb{Q}^{cyc} / \mathbb{Q}})$, 
the same is true for all $(\iota, \chi) \in I$,
so $\Gmon$ surjects onto each factor $\Gsimp$ of $\Gzerod$.
Thus, can talk about the relation $\sim$ and the index classes of Definition \ref{cbalanced_dff}.

The action of $(\Gal_{\mathbb{Q}} \times \Gal_{\mathbb{Q}^{cyc} / \mathbb{Q}})$
respects the relation $\sim$; in particular the index classes are all of the same size, and Frobenius respects the relation $\sim$.
Corollary \ref{torsion_char} shows that every $\sigma$-orbit in $I$
contains elements of at least $c$ index classes.
Specifically, note that we chose $p$ such that this Galois action is unramified at $p$,
and thus the local Galois group is generated by Frobenius.
So for any $(\iota', \chi') \in I$, the conclusion of Corollary \ref{torsion_char} 
gives $c$ elements $(\iota_1, \chi_1), \ldots, (\iota_c, \chi_c)$ of the $\sigma$-orbit of $(\iota', \chi')$ in $I$,
such that the projection
\[ \Gmon \hookrightarrow \Gzerod \twoheadrightarrow \prod_{i=1}^c \Gsimp_{(\iota_i, \chi_i)} \]
contains $(G^*)^c$, where $G^*$ is either $SL_N$, $Sp_N$, or $O_N$.
It follows that $(\Gmon \cap \Goned)_{\mathbb{C}}$ is strongly $c$-balanced with respect to $\sigma$.
\end{proof}

\section{Hodge--Deligne systems and integral points, assuming global semisimplicity}
\label{sec:int_thm_ss}

In this section we prove Theorem \ref{LV_thm_ss},
a variant of Theorem \ref{LV_thm} that assumes the semisimplicity of certain global Galois representations.
This material is not logically needed for the main argument;
we include it to illustrate the main ideas of Section \ref{sec:int_thm}, 
without the complications coming from semisimplification.

\begin{lem}
\label{fin_many_semi_types_ss}
(Compare Lemma \ref{fin_many_semi_types}.)

Let $p$ be a prime.
A \emph{semisimple} representation
\[\rho_0 \colon G_{\mathbb{Q}} \rightarrow \Gsemie\]
of the global Galois group $G_{\mathbb{Q}}$, crystalline at $p$, such that the composition $G_{\mathbb{Q}}\rightarrow \Gsemie  \rightarrow \Aut_{\mathbb Q_p} \mathsf{E}_{et}$ agrees with the usual action of Galois on $\mathsf{E}_{et}$, gives rise by $p$-adic Hodge theory to an admissible filtered $\phi$-module $(V_0, \phi_0, F_0)$ with $\Gsemid$-structure. 
Suppose another crystalline global representation $\rho \colon G_{\mathbb Q} \rightarrow \Gsemi$ is isomorphic to $\rho_0$, and call the corresponding filtered $\phi$-module $(V, \phi, F)$.
Then there is an isomorphism of filtered $\phi$-modules
\[ (V_0, \phi_0, F_0) \cong (V, \phi, F). \]

In particular, if $(V_0, \phi_0) = (V, \phi)$, then there exists an automorphism $f \colon V \rightarrow V$ such that
\begin{itemize}
\item $f \in \Gsemid$,
\item $f$ commutes with $\phi$, and
\item $f F_0 = F$.
\end{itemize}
\end{lem}

\begin{proof}The fact that representations $\Gsemie$ structure are sent to fitlered $\phi$-modules with $\Gsemid$-structure is from Lemma \ref{pht_tannakian}.

The existence of isomorphisms is because functors of $p$-adic Hodge theory take isomorphic objects to isomorphic objects.
\end{proof}

\begin{lem}
\label{codim_final_estimate_ss}
(Compare Lemma \ref{codim_final_estimate}.)

Assume we are in the setting of Section \ref{semilinear_section}.
Fix an admissible filtered $\phi$-module with $\Gsemid$-structure $(V_0, \phi_0, F_0)$,
and another $\phi$-module $(V, \phi)$ with $\Gsemid$-structure;
suppose both $\phi$ and $\phi_0$ are semilinear over some $\sigma \in \Aut_{E_0} E$.

Let $\Gmon$ be a subgroup of $\Gzerod$, strongly $c$-balanced with respect to $\sigma$ for some positive integer $c$.

Suppose $F_0$ is uniform in the sense of Definition \ref{uniform_hodge}, and let $h^a = h^a_{simp}$ be the adjoint Hodge numbers on $\Gsimp$.
Suppose $e$ is a positive integer satisfying the following numerical condition.
\begin{itemize}
\item (Numerical condition.)
\[\sum_{a>0} h^a \geq \frac{1}{c} (e + \dim \Gsimp).\]
\end{itemize}

Let $\mathcal{H} = \Gmon / (Q^0 \cap \Gmon)$ be the flag variety parametrizing filtrations on $\Gzerod$ that are conjugate to $F_0$ under the conjugation of $\Gmon$.
Then the filtrations $F$ such that $(V, \phi, F)$ is isomorphic to $(V_0, \phi_0, F_0)$ are of codimension at least $e$ in $\mathcal{H}$.
\end{lem}

\begin{proof}

By Lemma \ref{fin_many_semi_types_ss}, the filtrations $F$ satisfying the condition described form at most one orbit under the action of $Z(\phi)$ on $\mathcal{H}$.  We will show that any orbit of $Z(\phi)$ on $\mathcal{H}$ has codimension at least $e$.
This is a question about the dimension of a variety over $E_0$; by passing to an extension, we may assume that $\Gzerod = \Gsimp^d$ is split,
and $\sigma$ acts by permuting the factors.
Call the $d$ factors $\Gsimp_1, \ldots, \Gsimp_d$.

The Frobenius element $\sigma \in \Aut_{E_0} E$ gives a permutation of the index set $\{1, \ldots, d\}$,
which we also call $\sigma$.
Semilinearity over $\sigma$ means that the map $\phi$ permutes the $d$ factors according to the permutation $\sigma$.

Let $J$ be an orbit of $\sigma$ on the index set $\{1, \ldots, d\}$.
By Definition \ref{cbalanced_dff}, $J$ must contain elements of at least $c$ distinct index classes.
Let $I \subseteq J$ be a system of representatives for the index classes appearing in $J$.
This $I$ must have at least $c$ elements, and in fact there is no harm in increasing $c$ so that $\# I = c$ exactly.

Now everything in sight splits as a direct product.
Let $\Gsimp^I = \prod_{i \in I} \Gsimp_i$; define $\Gsimp^J$ similarly.
Let $E_0^J$ be the direct summand of $E$ corresponding to the $\sigma$-orbit $J$.

Since the elements of $I$ belong to distinct index classes,
the projection 
\[ \Gmon \rightarrow (\Gsimp / Z(\Gsimp))^I \]
has image a union of connected components of the target, 
so it is smooth with equidimensional fibers, and the same is true of
\[ \mathcal{H} = \Gmon / (Q^0 \cap \Gmon) \rightarrow \Gsimp^I / (Q^{0, I} \cap \Gsimp^I) = \colon \mathcal{H}^I. \]

Let $Z_{\Gsimp^I}(\phi)$ be projection to $\Gsimp^I$ of $Z_G(\phi)$;
this is the set of elements of $\Gsimp^I$ that commute with $\phi$, when $\Gsimp^I$ is viewed as a direct summand of $\Gzerod$.
Define $Z_{\Gsimp^J}(\phi)$ similarly.

To finish the proof, it is enough to show that any orbit of $Z_{\Gsimp^I}(\phi)$ on $\mathcal{H}^I$ has codimension at least $e$.
By Lemma \ref{semilinear_dim}, applied with $E_0^J$ in place of $E$, we have
\[ \dim Z_{\Gsimp^J}(\phi) \leq \dim \Gsimp. \]
Since $Z_{\Gsimp^I}(\phi)$ is the projection of $Z_{\Gsimp^J}(\phi)$ to $\Gsimp^I$, we deduce
\[ \dim Z_{\Gsimp^I}(\phi) \leq \dim \Gsimp. \]

On the other hand,
for any reductive group $G$ and filtration $F$ on $G$, corresponding to a parabolic subgroup $Q$,
the sum $\sum_{a>0} h^a$ of the adjoint Hodge numbers
(Definition \ref{dff:adjoint_hodge})
is precisely $\dim G/Q$.
Apply this with $G = \Gsimp^I$;
since the Hodge numbers $F_0$ are uniform,
we can compute the adjoint Hodge numbers on $\mathcal{H}^I$ by Lemma \ref{hodge_scaling}; 
we find
\[ \dim \mathcal{H}^I = c \sum_{a>0} h^a. \]
The result follows.
\end{proof}

\begin{thm}
\label{LV_thm_ss}
(Compare Theorem \ref{LV_thm}.)

Let $X$ be a variety over $\mathbb{Q}$,
let $S$ be a finite set of primes of $\mathbb Z$,
and let $\mathcal{X}$ be a smooth model of $X$ over $\mathbb{Z}[1/S]$. 

Let $\mathsf{E}$ be a constant $H^0$-algebra on $\mathcal{X}$, and
let $\Gsimp$ be one of $GL_N$, $GSp_N$, or $GO_N$.
Let $\mathsf{V}$ be a polarized, integral, $\mathsf{E}$-module with $\Gsimp$-structure on $\mathcal{X}$, in the sense of Definition \ref{dff:gsp}, having integral Frobenius eigenvalues (Def.\ \ref{dff:int-frob-evals}).
Suppose the Hodge numbers of $\mathsf{V}$ are uniform in the sense of Definition \ref{uniform_hodge}, and let $h^a = h^a_{simp}$ be the adjoint Hodge numbers on $\Gsimp$.
Let $\Gzerod$, $\Gsemid$ be as in Section \ref{Emodules}.

Suppose there is a positive integer $c$ such that $\mathsf{V}$ satisfies the following conditions.
\begin{itemize}
\item (Big monodromy.) 
If $\Gmon \subseteq \Gzerod$ is the differential Galois group of $\mathsf{V}$,
then $\Gmon \subseteq \Gzerod$ is strongly $c$-balanced with respect to Frobenius. 
(The Frobenius is determined from the structure of $\mathsf{E}$; see Section \ref{Emodules}.)
\item (Numerical condition.) 
\[ \sum_{a>0} h^a \geq \frac{1}{c} (\dim X + \dim \Gsimp). \] 
\end{itemize}

Let $\mathcal{X}(\mathbb{Z}[1/S])^{ss}$ be the subset of $\mathcal{X}(\mathbb{Z}[1/S])$
consisting of those $x$ for which the global Galois representation $\mathsf{V}_{et, x}$ is semisimple.
Then the image of $\mathcal{X}(\mathbb{Z}[1/S])^{ss}$ is not Zariski dense in $X$.
\end{thm}

\begin{proof}

For every $x \in \mathcal{X}(\mathbb{Z}[1/S])$, consider the semisimple global Galois representation $\rho_x = \mathsf{V}_{et, x}$.
By Lemma \ref{falt_finite},
there are only finitely many possible isomorphism classes
for the semisimple representation $\rho_x$.
So it is enough to show, for any fixed $\rho_0$,
that the set
\[\mathcal{X}(\mathbb{Z}[1/S], \rho_0) := \{ x \in X(\mathbb Z[1/S] ) | \rho_x \cong \rho_0 \}\]
is not Zariski dense in $X$.
By Lemma \ref{fin_many_semi_types_ss}, 
it is enough, for each residue disk $\Omega \subseteq X( \mathbb Z_p) $, to show that the set
\[ X(\Omega, (V_0, \phi_0, F_0)) = \{ x \in \Omega | (V, \phi, F) \cong (V_0, \phi_0, F_0) \} \]
is not Zariski dense in $X$.

We are now in the setting of Theorem \ref{bt}.
For $x \in \Omega$, the filtered $\phi$-module $\mathsf{V}_{cris, x}$ is of the form $(V, \phi, F_x)$,
where $(V, \phi)$ is independent of $x$.
(This is a general property of $F$-isocrystals; see remarks in the proof of Theorem \ref{LV_thm}.)
The variation of $F_x$ with $x$ is classified by a $p$-adic period map 
\[ \Phi_p \colon \Omega \rightarrow \Gzerod / Q, \]
where $\Gzerod/Q$ is the flag variety classifying $\Gzerod$-filtrations on $V$.

In fact, the $p$-adic period map lands in a single $\Gmon$-orbit on $\Gzerod / Q$; we can write the orbit as $\Gmon / (Q \cap \Gmon)$.
By Lemma \ref{codim_final_estimate_ss} with $e = \dim X$, there is a Zariski-closed set $Z \subseteq \Gmon / (Q \cap \Gmon)$ of codimension at least $\dim X$
such that, if $(V, \phi, F_x) \cong (V_0, \phi_0, F_0)$,
then $x \in \Phi_p^{-1}(Z)$.
We conclude by Theorem \ref{bt}.
\end{proof}

\section{Hodge--Deligne systems and integral points}
\label{sec:int_thm}

The goal of this section is to prove Theorem \ref{LV_thm}, 
which gives Zariski non-density of integral points on $X$.
Recall the setup from Section \ref{sec:hds}:
we have a smooth variety $X$ over 
$\mathbb{Q}$\footnote{
By Weil restriction, we will assume that $X$ is defined over $\mathbb{Q}$, not a general number field;
see the proof of Theorem \ref{main_thm}.}
and a Hodge--Deligne system $\mathsf{V}$ on an integral model $\mathcal{X}$ of $X$, satisfying certain conditions.
Lemma \ref{falt_finite} tells us that, as $x$ ranges over the integral points of $X$,
there are only finitely many possibilities, 
up to semisimplification, for the global Galois representation $\mathsf{V}_{x,  et}$.
We will use this to bound the integral points of $X$.

A significant technical obstacle (both in \cite{LV} and here) is that 
Lemma \ref{falt_finite} only applies to semisimple global Galois representations;
there may be many different Galois representations arising as fibers of $\mathsf{V}_{et}$,
all of which have the same semisimplification.
If we assume the Grothendieck--Serre conjecture---that all Galois representations that arise are semisimple---this difficulty does not arise.
Under this assumption, the argument is very simple; 
see Section \ref{sec:int_thm_ss}.

To apply Lemma \ref{falt_finite} without assuming semisimplicity,
we need to recognize when two global Galois representations might have the same semisimplification.
We only have access to the local representations,
which are generally far from semisimple.
The key idea is that any filtration on the global representation,
when restricted to the local representation,
must be of a special form.

Passing from local Galois representations via $p$-adic Hodge theory to filtered $\phi$-modules,
we have the following situation.
We're given a $\phi$-module $(V, \phi)$ with varying filtration $F = F^{\bullet}$;
the variation of $F$ is described by a ``monodromy group'' $\Gmon$ (the differential Galois group attached to our Hodge--Deligne system).
We will consider the associated graded of $(V, \phi, F)$ with respect to a $\phi$-stable filtration $\mathfrak{f}$ --
eventually we will take $\mathfrak{f}$ to be a semisimplification filtration on the global Galois representation.
By Lemma \ref{balanced_filtration} below, if $\mathfrak{f}$ comes from a filtration on the global representation,
then the associated graded of $F$ with respect to $\mathfrak{f}$ is a balanced filtration (Definition \ref{dff:balanced}).
So we want to know, for how many choices of $F$ does there exist a $\phi$-stable $\mathfrak{f}$,
such that the associated graded of $(V, \phi, F)$ with respect to $\mathfrak{f}$ lies in a given balanced isomorphism class?
We will bound the dimension of such $F$ in the flag variety.
This material is very similar to the combinatorial arguments in \cite[\S 10-11]{LV}.


To start with, we'll recall some results from \cite{LV} to limit the reducibility of global representations.
The following result is the reason we need to work with $\mathbb{Q}$-varieties $X$, instead of varieties over an arbitrary number field.
The natural generalization to representations of $\Gal_K$, with $K$ a number field, is false --
for a counterexample, take $K$ a CM field, and $V$ a one-dimensional representation coming from a CM elliptic curve.

\begin{lem}
\label{balanced}
Let $p$ be a prime, and
let $V$ be a representation of $\Gal_{\mathbb{Q}}$ on a $\mathbb{Q}_p$-vector space 
which is crystalline at $p$, and such that at all primes $\ell$ outside of a finite set $S$,
the characteristic polynomial of Frobenius has algebraic coefficients and all roots rational $\ell$-Weil numbers of weight $w$. 
 
Let $V_{dR} = (V \otimes_{\mathbb{Q}_p} B_{\mathrm{cris}})^{\Gal_{\mathbb{Q}_p}}$ be the filtered
$\mathbb{Q}_p$-vector space that is associated to 
$\rho|_{\mathbb{Q}_p}$ by the $p$-adic Hodge theory functor $\underline{D}_{\mathrm{cris}}$ of \cite[Expose III]{Asterisque}. 

Then the average weight of the Hodge filtration on $V_{dR}$  equals $w/2$. 



\end{lem}

\begin{proof}
This is \cite[Lemma 2.9]{LV} applied to $\mathbb{Q}$, which has no CM subfield; the condition that $p$ be a ``friendly place'' is automatically satisfied over $\mathbb{Q}$.
\end{proof}

Our first goal is to rephrase Lemma \ref{balanced} in terms of filtrations on reductive groups;
the resulting statement is Lemma \ref{balanced_filtration} below.

We work with filtrations and semisimplifications relative to 
a group $G$ whose identity component is reductive. 
When $G$ is disconnected, recall the notions of ``parabolic subgroup,'' ``filtration'' and so forth from Section \ref{sec:disconnected}.

We use the following notation (consistent with \cite{LV}): $Q$ is the parabolic subgroup of $G$ corresponding to the Hodge filtration $F$,
while $P$ corresponds to a semisimplification filtration $\mathfrak{f}$.
The group $M$ is a Levi subgroup associated to $P$, corresponding to the associated graded of $\mathfrak{f}$.

Fix $G$, $P$, and $M$.
In the (connected) reductive case  
\cite[Lemma 11.2]{LV} defines a map from filtrations $F$ on $G$ to filtrations $F_M$ on $M$;
we need to extend this result to non-connected groups $G$ as well. 
Recall (Section \ref{subsub_filt}) that for any $G$-filtration $F$, there is a cocharacter $\mu \colon \mathbb{G}_m \rightarrow G$
defining $F$.
The substance of \cite[Lemma 11.2]{LV} is that $\mu$ can be chosen to take values in $P$.
Projecting from $P$ to the Levi subgroup $M$ gives a filtration on $M$, 
which is independent of the choice of $\mu$.

\begin{lem}
\label{filtration_cochar_lem}
Suppose $G$ is an algebraic group, whose identity component is reductive, over a field of characteristic zero.
Let $P$ be a parabolic subgroup of $G$, and $M$ a Levi subgroup associated to $P$.

Fix a filtration $F$ on $G$.  
Then there exists a cocharacter $\mu \colon \mathbb{G}_m \rightarrow P$ (with image in $P$) defining $F$.
Furthermore, if $F_M$ is the filtration on $M$ defined by the composite map
\[  \mathbb{G}_m \rightarrow P \rightarrow M, \]
then $F_M$ is independent of the choice of $\mu$.
\end{lem}

\begin{proof}
For $G$ connected reductive this is \cite[Lemma 11.2]{LV}.

In the general case, by \cite[Lemma 11.2]{LV} applied to the identity component $G^0$, we know that $\mu$ can be chosen with image in $P \cap G^0$,
and the corresponding filtration on $M \cap G^0$ is independent of the choice of $\mu$.
But the filtration defined by $\mu$ on $M \cap G^0$ determines the filtration defined by $\mu$ on $M$; see Section \ref{disconnected_filtrations}.
\end{proof}

\begin{dff}
\label{filtration_cochar_dff}
Given $G$, $P$, $M$, $F$ as in Lemma \ref{filtration_cochar_lem},
we call the filtration $F_M$ on $G$ the \emph{associated graded} filtration, and write it as $F_M = \operatorname{Gr}_M F$.
(It is well-defined by Lemma \ref{filtration_cochar_lem}.)
\end{dff}

We need a generalization of the notion of ``balanced filtration'' from \cite[\S 11.1, 11.4]{LV}.
Given a group $S$ whose identity component is reductive, we define
\[ \mathfrak{a}_S = X_*(Z_{S^{0}}) \otimes \mathbb{Q} = (X^*(S^{0}) \otimes \mathbb{Q})^{\vee}, \]
where $Z_{S^{0}}$ is the center of the identity component $S^{0}$ of $S$.
A cocharacter $\mu$ of $S$ defines a class $w(\mu) = w_{S}(\mu) \in \mathfrak{a}_S$ by
\[ (\chi \circ \mu)(t) = t^{\langle w_S(\mu) , \chi \rangle} \] for all $\chi \in X^*(S_0)$.
In other words: $\chi \circ \mu$ is an automorphism of $\mathbb{G}_m$, so it is of the form $t \mapsto t^{\alpha}$;
we choose $w_S(\mu)$ so that $\langle w_S(\mu) , \chi \rangle = \alpha$, for all $\chi$.
We call $w(\mu) = w_S(\mu)$ the \emph{weight} of $\mu$.

Let $G$, $P$, $M$ be as above. 
Then the inclusion $Z(G^0) \rightarrow Z(M^0)$ gives a map $\iota_{GM} \colon \mathfrak{a}_G \rightarrow \mathfrak{a}_M$.
Furthermore, the parabolic $P$ defines a preferred element of $(\mathfrak{a}_M)^{\vee}$, the \emph{modular character} $\gamma_P$,
defined\footnote{Since the sign convention is important, we provide an example.  If $P \subseteq GL_2$ is the group of upper triangular matrices, then -- identifying $M$ with the group of diagonal matrices -- the character $\gamma_P$ is given by
\[ \gamma_P\left(  \begin{pmatrix} a & 0 \\ 0 & b  \end{pmatrix}  \right ) = a^{-1} b. \]}
as the inverse of the determinant of the adjoint representation of $M$ on the Lie algebra of $P$.

\begin{dff}
\label{dff:balanced}
Suppose given $G$ an algebraic group whose identity component is reductive, $P$ a parabolic subgroup, and $M$ a Levi subgroup associated to $P$.  
Let $F_M$ be a filtration on $M$, given by a cocharacter $\mu \colon \mathbb{G}_m \rightarrow M$.
\begin{itemize}
\item We say that $F_M$ is \emph{balanced} with respect to $P$ if $w_M(\mu) = \iota_{GM} (w_G(\mu_G))$, where $\mu_G$ is the cocharacter $\mathbb{G}_m \stackrel{\mu}{\rightarrow} M \hookrightarrow G$.
\item We say that $F_M$ is \emph{weakly balanced} if $\gamma_P(\mu) = 0$ 
for $\gamma_P$ the modular character of $P$.
\item We say that $F_M$ is \emph{semibalanced} if $\gamma_P(\mu) \leq 0$, for $\gamma_P$ the modular character of $P$.
\end{itemize}

We say a $G$-filtration $F$ is balanced (resp.\ weakly balanced, semibalanced) with respect to $P$ (or $M$, or $\mathfrak{f}$)
if the associated graded $\operatorname{Gr}_M F$ is so.

We remark that a $G$-filtration $F$ is balanced with respect to $P$ if and only if the associated $G^0$-filtration is balanced with respect to $P^0$.
\end{dff}

\begin{rmk}
Balanced implies weakly balanced because
\[ \gamma_P( \iota_{GM} (w_G(\mu) ) ) = 0; \]
this identity boils down to the fact that the center of $G$ acts trivially through the adjoint representation on the Lie algebra of $P$.

Furthermore, weakly balanced implies semibalanced.
\end{rmk}

\begin{ex}
Let $G = GL_6$, acting on the space $V$ with standard basis vectors $e_1, \ldots, e_6$.
Let $\mathfrak{f}$ be the filtration with 
\begin{eqnarray*}
\mathfrak{f}_0 & = & V \\
\mathfrak{f}_1 & = & \operatorname{span} (e_1, e_2, e_3, e_4) \\
\mathfrak{f}_2 & = & \operatorname{span} (e_1, e_2) \\
\mathfrak{f}_3 & = & 0,
\end{eqnarray*}
let $P$ be the subgroup of $G$ that stabilizes $\mathfrak{f}$, and let 
$M$ be the Levi subgroup that fixes the subspaces $\operatorname{span} (e_1, e_2)$, $\operatorname{span} (e_3, e_4)$, and $\operatorname{span} (e_5, e_6)$.

An element of $X_*(Z_{G^{0}})$ is given by an action of $\mathbb{G}_m$ on $V$ by some integral power $t^n$
(where $t$ is the coordinate on $\mathbb{G}_m$);
tensoring with $\mathbb{Q}$, we can write elements of $\mathfrak{a}_G$ formally as scalar matrices $t^a$, with $a \in \mathbb{Q}$.
Similarly, an element of $\mathfrak{a}_M$ can be written as $(t^{a_0}, t^{a_1}, t^{a_2})$, 
where $\mathbb{G}_m$ is understood to act on $\mathfrak{f}_k / \mathfrak{f}_{k+1}$ as $t^{a_k}$.

The modular character $\gamma_P$ takes $(t^{a_0}, t^{a_1}, t^{a_2})$ to $t^{4(a_0 - a_2)}$.

We will consider filtrations $F$ such that $F_0 = V$, $\dim F_1 = 3$, and $F_2 = 0$.
For such filtration $F$, with corresponding cocharacter $\mu$, we have $w_G(\mu) = 1/2$.

Now let 
\[ a_k = \dim (F_1 \cap \mathfrak{f}_k) - \dim (F_1 \cap \mathfrak{f}_{k+1}), \] 
for $k = 0, 1, 2$.
We have $a_0 + a_1 + a_2 = 3$,
and
\[ w_M(\mu) = (t^{a_0 / 2}, t^{a_1 / 2}, t^{a_2 / 2}). \]

Thus $\mu$ is balanced if and only if $(a_0, a_1, a_2) = (1, 1, 1)$.
It is weakly balanced if and only if $a_0 = a_2$; it is semibalanced if and only if $a_0 \leq a_2$.

Note that for $F$ generic in the Grassmannian, we have $(a_0, a_1, a_2) = (2, 1, 0)$.
\end{ex}

In general, the condition that $F$ be semibalanced is a strong condition on $F$,
only satisfied for $F$ in a high-codimension subset of the flag variety.
Lemma \ref{codim_first_estimate} will give a precise bound on the codimension, 
in the context of interest to us.

The notion of a semibalanced (rather than balanced) filtration will be important in the proof of
Lemma \ref{codim_final_estimate}.
Our method requires us to work with a period domain on which the monodromy group acts transitively.
We cannot guarantee that the monodromy group is all of $\Gzero$;
we know only that it is a $c$-balanced subgroup;
thus (after changing coefficients to arrange that $\Gzero = \Gsimp^d$) we will work with a period domain of the form
\[ \Gsimp^I / (Q_0^I \cap \Gsimp^I), \]
for some index set $I \subseteq \{1, \ldots, d\}$ of cardinality $c\leq d$.

Lemma \ref{balanced_filtration} tells us that the filtered $\phi$-modules coming from global Galois representations are balanced;
this amounts to a condition on the Hodge numbers \emph{averaged} over the $d$ different indices.
In the proof of Lemma \ref{codim_first_estimate}, we will pass to a subset $I$ of cardinality $c$,
on which the filtration is semibalanced.

\begin{lem}
\label{balanced_filtration}
(Filtered $\phi$-modules coming from global representations are balanced.)

Let $G$ be an algebraic group over $\mathbb{Q}_p$ whose identity component is reductive, 
and fix an embedding $G \hookrightarrow GL_N$.
Let
\[ \rho \colon \Gal_{\mathbb{Q}} \rightarrow G(\mathbb{Q}_p) \subseteq GL_N(\mathbb{Q}_p) \]
be a representation satisfying the hypotheses of Lemma \ref{balanced}.
Suppose $\rho$ has image contained in some parabolic subgroup $P(\mathbb{Q}_p) \subseteq G(\mathbb{Q}_p)$,
and let $M$ be a Levi subgroup associated to $P$.

Let $(V, \phi, F)$ be the filtered $\phi$-module over
$\mathbb{Q}_p$ that is associated to the local representation
$\rho|_{\mathbb{Q}_p}$.
Then $F$ is a $G$-filtration on $V$,
and the associated graded $F_M = \operatorname{Gr}_M(F)$ is a balanced filtration on $M$.  
\end{lem}

\begin{proof}
(Compare \cite[Prop.\ 10.6(b), \S 11.4, \S 11.6]{LV}.)

That $F$ is a $G$-filtration on $V$ is a consequence of the Tannakian formalism (see Lemma \ref{pht_tannakian}).

To see that $F_M$ is balanced,
let $\mu \colon \mathbb{G}_m \rightarrow P$ be a cocharacter defining $F$.
It is enough to show that every character $\chi \colon P \rightarrow \mathbb{G}_m$
annihilating $Z(G)$ also kills $\mu$.

For every such character $\chi$,  
we will show that $\chi \circ \rho$ is pure of weight zero: the Frobenius eigenvalues at unramified primes are rational Weil numbers of weight zero.

Indeed, let $\rho' = \rho \oplus \chi$ as a representation of $P$. 
Let $\operatorname{Frob}_{\ell}^{ss}$ be the semisimplification of a Frobenius element at $\ell$ acting on $\rho'$. 
Then $\operatorname{Frob}_{\ell}^{ss}$ lies in $P$. Fix an eigenbasis of $\operatorname{Frob}_{\ell}^{ss}$ -- with the last eigenvector lying in $\chi$ -- and let $T$ be the subgroup of $P$ consisting of elements which have each element of this eigenbasis as eigenvectors. 
This is an algebraic subgroup of the torus with coordinates $\lambda_1,\dots, \lambda_{N+1}$, hence is defined by relations $\prod_i \lambda_i^{e_i}=1$ for $e_i \in \mathbb Z$;
we have chosen indices so that $\lambda_1, \ldots, \lambda_N$ are the eigenvalues on $\rho$, and $\lambda_{N+1}$ is the eigenvalue on $\chi$.
 
Since $Z(G)$ acts trivially on $\chi$, every element of $T$ whose eigenvalues $\lambda_{1},\dots,\lambda_{N}$ are all equal acts by scalars on $\rho$, hence lies in $Z(G)$, and thus acts trivially on $\chi$. It follows by restricting all relations to the subtorus where $\lambda_1=\dots =\lambda_{N}$ that there must be some relation with $e_{N+1}\neq 0$ and $\sum_{i=1}^{N} e_i = 0$. Applying this relation to the eigenvalues of $\operatorname{Frob}_\ell^{ss}$, we see that 
\[ | \chi( \operatorname{Frob}_\ell^{ss})|^{e_{N+1}} = | \chi( \operatorname{Frob}_\ell^{ss})^{e_N} | = | \prod_{i=1}^{N} \lambda_i ( \operatorname{Frob}_\ell^{ss})^{-e_i} | = \prod_{i=1}^{N} | \lambda_i( \operatorname{Frob}_\ell^{ss}
) | ^{-e_i} \] \[ = \prod_{i=1}^{N} (p^{w/2})^{-e_i} = p ^{- ( w/2) \sum_{i=1}^{N} e_i} = p^{0}=1 \]  and so $\chi$ is pure of weight $0$.

By Lemma \ref{balanced}, the weight of the corresponding filtered $\phi$-module is also zero,
which is what we needed to prove.
\end{proof}

\begin{dff}
\label{bifiltered}
Let $E_0$, $E$, $\Gzero$ and $\Gsemi$ be as in Section \ref{semilinear_section}, and fix $\sigma \in \Aut_{E_0} E$.
A \emph{$\Gsemi$-bifiltered $\phi$-module}
is a quadruple $(V, \phi, F, \mathfrak{f})$,
with $V$ as in Section \ref{semilinear_section},
$F$ and $\mathfrak{f}$ two $\Gsemi$-filtrations on $V$, and 
$\phi \in \Gsemi$ a $\sigma$-semilinear endomorphism of $V$ respecting $\mathfrak{f}$.
In this setting, let $P$ and $Q$ denote the parabolic subgroups of $\Gsemi$ corresponding to $\mathfrak{f}$ and $F$, respectively;
to say that $\phi$ respects $\mathfrak{f}$ means that $\phi \in P$. 

A \emph{graded $\Gsemi$-bifiltered $\phi$-module}
\footnote{In the notation $(V, \phi, \mathfrak{f}, F_M)$, the filtration $F_M$ comes after the semisimplification filtration $\mathfrak{f}$ to indicate that $\mathfrak{f}$ is logically prior to $F_M$.}
is a quadruple $(V, \phi, \mathfrak{f}, F_M)$,
where $V$ is as in Notation \ref{semilinear_section},
$\mathfrak{f}$ is a $\Gsemi$-filtration on $V$ with associated Levi subgroup $M$,
and $F_M$ is a filtration on $M$.

We say that two graded $\Gsemi$-bifiltered $\phi$-modules $(V_1, \phi_1, \mathfrak{f}_1, F_{M, 1})$ and $(V_2, \phi_2, \mathfrak{f}_2, F_{M, 2})$
are \emph{equivalent} if they agree up to $\Gsemi$-conjugacy. 
More precisely, let $P_1$ and $P_2$ be the parabolic subgroups attached to $\mathfrak{f}_1$ and $\mathfrak{f}_2$, and let $M_1$ and $M_2$ be Levi subgroups associated to $P_1$ and $P_2$, respectively.
Then $(V_1, \phi_1, \mathfrak{f}_1, F_{M, 1})$ and $(V_2, \phi_2, \mathfrak{f}_2, F_{M, 2})$ are equivalent
if there exists $g \in \Gsemi$ satisfying the following conditions.
\begin{itemize}
\item $g P_1 g^{-1} = P_2$.
\item The filtrations $g F_{M, 1} g^{-1}$ and $F_{M, 2}$ on $M_2$ agree.
\item The two elements $g \phi_1 g^{-1}$ and $\phi_2$ of $P_2$
project to the same element of $M_2$.
\end{itemize}

There is an obvious functor
\[ (V, \phi, F, \mathfrak{f}) \mapsto (V, \phi, \mathfrak{f}, \operatorname{Gr}_M F) \]
from $\Gsemi$-bifiltered $\phi$-modules to graded $\Gsemi$-bifiltered $\phi$-modules
(with $F \mapsto \operatorname{Gr}_M F$ given by Definition \ref{filtration_cochar_dff}).
We say that two $\Gsemi$-bifiltered $\phi$-modules are \emph{semisimply equivalent}
if the corresponding graded $\Gsemi$-bifiltered $\phi$-modules are equivalent.

We say that a $\Gsemi$-filtered $\phi$-module $(V,  \phi, F)$
is \emph{of the semisimplicity type} $(V_0, \phi_0, F_0, \mathfrak{f}_0)$ if
there exists a $\Gsemi$-filtration $\mathfrak{f}$ on $V$ such that 
$(V, \phi, F, \mathfrak{f})$ and $(V_0, \phi_0, F_0, \mathfrak{f}_0)$ are semisimply equivalent.
\end{dff}

\begin{rmk}
To illustrate ideas, consider the case where $E = E_0$ and $\Gsemi = \Gzero = GL_n$.

A $\Gsemi$-bifiltered $\phi$-module comes (by $p$-adic Hodge theory) 
from a filtered Galois representation 
(i.e.\ a Galois representation on a vector space $V$, and a Galois-stable filtration $\mathfrak{f}_i V$ on $V$).

A graded $\Gsemi$-bifiltered $\phi$-module 
comes from the associated graded to a filtered Galois representation 
(i.e.\ the Galois representation on $\bigoplus_i  ( \mathfrak{f}_i V / \mathfrak{f}_{i+1} V )$).

Two graded $\Gsemi$-bifiltered $\phi$-modules are equivalent
if  the corresponding Galois representations are isomorphic;
two $\Gsemi$-filtered $\phi$-modules are of the same semisimplicity type 
if the corresponding representations $\bigoplus_i ( \mathfrak{f}_i V / \mathfrak{f}_{i+1} V )$ are isomorphic. 
\end{rmk}

\begin{lem}
\label{fin_many_semi_types}
(Compare Lemma \ref{fin_many_semi_types_ss}.)

Let $p$ be a prime.
A representation~
\footnote{The restriction to $\mathbb{Q}$, instead of an arbitrary number field $K$, is for two reasons.  
First, in the general setting, a filtered $\phi$-module would be \emph{semilinear} over $K_v$, 
and we have not defined filtered $\phi$-modules with $G$-structure in the semilinear setting;
this restriction is inessential.
Second, we will need to apply Lemma \ref{balanced}, and for that we need $K$ to have no CM subfield.}
\[\rho_0 \colon G_{\mathbb{Q}} \rightarrow \Gsemie\]
of the global Galois group $G_{\mathbb{Q}}$, crystalline at $p$, such that the composition $G_{\mathbb{Q}}\rightarrow \Gsemie \rightarrow \Aut_{\mathbb Q_p} \mathsf{E}_{et}$ agrees with the usual action of Galois on $\mathsf{E}_{et}$, gives rise by $p$-adic Hodge theory to an admissible filtered $\phi$-module $(V_0, \phi_0, F_0)$ with $\Gsemid$-structure. 
Suppose $\rho_0$ is semisimple.
Suppose another crystalline global representation $\rho \colon G_{\mathbb{Q}} \rightarrow \Gsemie$ with the same composition with $\Gsemie  \rightarrow \Aut_{\mathbb Q_p} \mathsf{E}_{et}$ has semisimplification $\rho_0$, and call the corresponding filtered $\phi$-module $(V, \phi, F)$.
Then there exist $\Gsemid$-filtrations $\mathfrak{f}_0$ on $V_0$ and $\mathfrak{f}$ on $V$ such that $(V, \phi, F, \mathfrak{f})$ and $(V_0, \phi_0, F_0, \mathfrak{f}_0)$ 
are semisimply equivalent.

Furthermore, we can take $\mathfrak{f}_0$ one of a list of finitely many candidates, depending only on $\rho_0$,
and the filtration $F_0$ is balanced with respect to $\mathfrak{f}_0$.
\end{lem}

\begin{proof}
By restriction to $G_{\mathbb Q_p}$ and Lemma \ref{pht_tannakian},
a crystalline representation $G_{\mathbb Q} \rightarrow \Gsemie$
gives rise to a filtered $\phi$-module with $\Gsemid$-structure $(V, \phi, F)$, with $\phi$ semilinear over $\sigma$.

To say that $\rho_0$ is the semisimplification of $\rho$ means (see Definition \ref{dff:gcr} and Lemma \ref{martin_gcr}) that there exist
a parabolic subgroup $P_{et} \subseteq \Gsemie$ with associated Levi $L_{et}$, and $g\in \Gsemie$ such that $\rho_0$ takes values in $L_{ et}$, and $ g^{-1} \rho g$ takes values in $P_{et}$, 
and the composition of $g^{-1} \rho g$ with the quotient map $P_{et} \rightarrow L_{et}$ is exactly $\rho_0 $.

Lemma \ref{fin_levis_v2} shows that, given $\rho_0$, we can take $L_{et}$ to be one of finitely many possible subgroups.
For each such $L_{et}$, there are finitely many parabolic subgroups $P_{et}$ with $L_{et} \subseteq P_{et}$ as a Levi subgroup.

Since the kernel of $P_{et} \to L_{et}$ is a unipotent group, the natural map $H^1(\mathbb Q_{p}, P_{et}) \to H^1(\mathbb Q_p, L_{et})$ is a bijection, so an inner form of $P_{et}$ is determined by the corresponding inner form of $L_{et}$. 

Now
fix $\rho_0$, $P_{et}$, and $L_{et}$. Let $P_{dR}$ and $L_{dR}$ be the inner twists of $P_{et}$ and $L_{et}$ arising by Lemma \ref{pht_tannakian} from $\rho_0$. Note that $P_{dR}$ is a parabolic of $\Gsemid$ and $L_{dR}$ is the associated Levi by Lemma \ref{pht_tannakian}(2) and the fact that the property of being an inclusion of a parabolic or an inclusion of a Levi is preserved by inner twists.

Lemma \ref{pht_tannakian} produces inner twists of $P_{et}$ and $L_{et}$ associated to $g^{-1} \rho g$. But since the projection of $g^{-1} \rho g$ under $P_{et} \to L_{et}$ agrees with $\rho_0$, the inner twists of $L_{et}$ is again $L_{dR}$, and hence the inner twist of $P_{et}$ is again $P_{dR}$.


By Lemma \ref{pht_tannakian}, we find that $\rho_0$ (resp.\ $g^{-1} \rho g$)
must give rise to filtered $\phi$-modules with $P_{dR}$-structure,
such that the corresponding filtered $\phi$-modules with $L_{dR}$-structure (obtained by functoriality via the quotient map $P \rightarrow L$)
are isomorphic.

But a filtered $\phi$-module with $P_{dR}$-structure is precisely the same as a filtered $\phi$-module with $\Gsemid$-structure,
equipped with a filtration $\mathfrak{f}_0$ whose associated parabolic is $P_{dR}$. The filtered $\phi$-modules with $\Gsemid$-structure arising from $g^{-1} \rho g$ by this process is $\Gsemid$-conjugate to the filtered $\phi$-module with $\Gsemid$-structure arising from $g^{-1}\rho g $ directly by Lemma \ref{pht_tannakian}(2) and hence $\Gsemid$-conjugate to the filtered $\phi$-module with $\Gsemid$-structure arising from $\rho$ by Lemma \ref{pht_tannakian}(1) since they correspond to isomorphic Galois representations into $\Gsemie$.

So take $\mathfrak{f}_0$ a filtration on $V_0$ attached to $P_{dR}$, and let $\mathfrak{f}$ be its image under this $\Gsemid$-conjugation.
Then $(V, \phi, F, \mathfrak{f})$ and $(V_0, \phi_0, F_0, \mathfrak{f}_0)$ are semisimply equivalent.

Finally, $F_0$ is balanced with respect to $\mathfrak{f}_0$ by Lemma \ref{balanced_filtration}.
\end{proof}

\begin{rmk}
Recall (Definition \ref{assoc_filt} and Lemma \ref{filtration_lemma}) that there is an equivalence between $\Gsemid$-filtrations and $\Gzerod$-filtrations;
we will use this without comment, and work with $\Gzerod$-filtrations on $V$.
\end{rmk}

\begin{lem}
\label{codim_first_estimate}
Take notation as in Section \ref{semilinear_section}, let $c$ be a positive integer, and suppose $E = E_0^c$.

Let $P$ be a parabolic subgroup of $\Gzero = \Gsimp^c$, corresponding to a $\Gzero$-filtration $\mathfrak{f}$, 
and let $M$ be a Levi subgroup associated to $P$.

Let $Q_0$ be another parabolic subgroup of $\Gzero$, corresponding to some filtration $F_0$, so that $\Gzero / Q_0$ parametrizes filtrations $F$ that are $\Gzero$-conjugate to $F_0$.
Suppose $F_0$ is uniform in the sense of Definition \ref{uniform_hodge}, and let $h^a = h^a_{\Gsimp}$ be the adjoint Hodge numbers on $\Gsimp$.
Let $\torusdim$ be the dimension of a maximal torus in $\Gsimp$.
Suppose $e$ is a positive integer such that:
\begin{itemize}
\item (First numerical condition.) 
\[\sum_{a>0} h^a \geq \frac{e}{c}\]
and
\item (Second numerical condition.) 
\[\sum_{a>0} a h^a > T \left ( \frac{e}{c} \right ) + T \left ( \frac{1}{2} (h^0 - \torusdim) + \frac{e}{c} \right ).\]
\end{itemize}

Then for any semibalanced filtration $F_{0, M}$ on $M$,
the set of filtrations $F$ on $\Gzero$ 
that are $\Gzero$-conjugate to $F_0$ and satisfy $\operatorname{Gr}_M F =  F_{0, M}$
is of codimension at least $e$ in $\Gzero / Q_0$. 
\end{lem}

\begin{proof}
This is essentially \cite[Prop.\ 11.3]{LV}, applied to $\Gzero = \Gsimp^c$.  
Note that the Hodge numbers $h^a$ and the function $T$ used in \cite{LV} are $h_{\Gsimp^c}$ and $T_{\Gsimp^c}$.  
By Lemma \ref{hodge_scaling}, they are related to $h_{\Gsimp}$ and $T_{\Gsimp}$ by
\[ h_{\Gsimp^c}^a = c h_{\Gsimp}^a \]
and
\[T_{\Gsimp^c} (cx) = c T_{\Gsimp}(x).\]
Similarly, the dimension of a maximal torus in $\Gsimp^c$ is $c \torusdim$.

Two small modifications need to be made to the proof in \cite{LV}.
First, we have replaced $\frac{1}{2} h^0$ with $\frac{1}{2} (h^0 - \torusdim)$.
To get this stronger bound, replace the final inequality of \cite[Equation 11.15]{LV} with
\[ \dim (Q/B) + e \leq \frac{1}{2} (a_0 - \torusdim) + e.  \]
(Here $Q = Q_0$, $a_0$ the dimension of an associated Levi subgroup, and $B$ a Borel.
There is no new idea here; this bound is stronger only because the bound in \cite{LV} was not sharp.)

Second, our hypothesis is weaker: in the above-referenced proposition, $F_{0, M}$ is assumed to be balanced,
while here it is only assumed to be semibalanced.
This is not a problem, since the inequalities work in our favor.
Recall from \cite[proof of Proposition 11.3]{LV}
that $T$ is a maximal torus contained in $P$, $\Sigma$ is the set of roots of $T$ on $\Gzero$, 
$\Sigma_P \subseteq \Sigma$ is the set of roots of $T$ on $P$, 
and $\mu$ is a cocharacter defining the parabolic subgroup $Q_0$.
In our context, \cite[Equation 11.14]{LV} is replaced with the inequality
\[ \sum_{\gamma \in \Sigma - \Sigma_P} \langle w \mu, \gamma \rangle \leq 0, \]
and \cite[Equation 11.16]{LV} becomes
\[ \sum_{\beta \in X}  \langle \mu, w^{-1} \beta \rangle \leq \sum_{X'} - \langle \mu, w^{-1} \beta \rangle. \]
The rest of the proof goes through as in \cite{LV}.
\end{proof}

Our next goal (Lemma \ref{codim_second_estimate}) is a slight generalization of \cite[Equation 11.18]{LV}: the result in \cite{LV} only holds with $\phi$ contained in a connected reductive group (e.g.\ $\Gzerod$), but here $\phi$ is semilinear, so it is contained in $\Gsemid$, but not in $\Gzerod$.

\begin{lem}
\label{unip_h1}
Let $U$ be a unipotent algebraic group over a field of characteristic zero, and $\psi$ an automorphism of $U$, such that $\psi^r = \operatorname{id}_U$.
Suppose $u \in U$ is such that
\[ u \psi(u) \psi^2(u) \cdots \psi^{r-1}(u) = 1. \]
Then there exists $v \in U$ such that
\[ u = v^{-1} \psi(v). \]
\end{lem}

\begin{proof}
By induction on $\dim U$; reduce to the case where $U = \mathbb{G}_a^k$ is abelian.
\end{proof}

\begin{lem}
\label{para_levi_conj}
Let $G$ be an algebraic group whose identity component $G^0$ is reductive.
Let $P$ a parabolic subgroup of $G^0$
and $M$ an associated Levi subgroup. 
Suppose $\phi \in G$ is semisimple and normalizes $P$.

Then $\phi$ is $P$-conjugate to an element that normalizes $M$.
\end{lem}

\begin{proof}
Suppose $\phi^r \in G^0$.  Since $\phi^r$ normalizes $P$, $\phi^r \in P$. Since $\phi^r$ is semisimple, we may conjugate by an element of $P$ to assure that $\phi^r \in M$.

Since $\phi$ normalizes $P$, and all Levi subgroups in $P$ are $P$-conjugate, we can write $\phi M \phi^{-1} = u^{-1} M u$ for some $u \in P$.
Multiplying $u$ by an element of $M$ from the left, we may assume that $u$ is unipotent.
Since $(u \phi) M (u \phi)^{-1} = M$, it will suffice to show that $u \phi$ is $P$-conjugate to $\phi$.

Now $(u \phi)^r$ is an element of $U \phi^r$ that normalizes $M$, so
\[ (u \phi)^r = \phi^r. \]
We can apply Lemma \ref{unip_h1}, with
\[ \psi(x) = \phi x \phi^{-1}, \]
to conclude that
\[ u = v^{-1} \psi(v) = v^{-1} \phi v \phi^{-1} \]
for some $v \in P$.
Then $v^{-1} \phi v = u \phi$ normalizes $M$, and we are done. 
\end{proof}

\begin{lem}
\label{codim_second_estimate}
(Compare \cite[Equation 11.18]{LV}.)

Assume we are in the setting of Section \ref{semilinear_section}.
Fix a $\Gsemi$-bifiltered $\phi$-module $(V_0, \phi_0, F_0, \mathfrak{f}_0)$,
and another $\phi$-module $(V, \phi)$ with $\Gsemi$-structure;
suppose both $\phi$ and $\phi_0$ are semilinear over some $\sigma \in \Aut_{E_0} E$.

Consider all pairs $(\mathfrak{f}, F_M)$, where $\mathfrak{f}$ is a filtration on $\Gsemi$, $M$ is an associated Levi subgroup, 
and $(V, \phi, \mathfrak{f}, F_M)$ is equivalent to $(V_0, \phi_0, \mathfrak{f}_0, (F_0)_{M_0})$,
in the sense of Definition \ref{bifiltered}.
The set of such pairs has dimension at most $\dim Z_{\Gzerod}(\phi^{ss})$. 
\end{lem}

(Recall that any element $\phi$ of an algebraic group has a unique Jordan decomposition
$\phi = \phi^{ss} \phi^{unip}$, where $\phi^{ss}$ is semisimple, $\phi^{unip}$ is unipotent, and $\phi^{ss}$ and $\phi^{unip}$ commute.)


\begin{proof}
This is a question about the dimension of a variety, so we can pass to a finite extension of $K$.
In particular, we may assume $E = K^d$ and $\sigma$ acts by permutation on the factors of $E = K^d$.
We index the factors $K_1$ through $K_d$, and we regard $\sigma$ as a permutation of $\{1, \ldots, d\}$.
Write $V_i = V \otimes_E K_i$ for the $i$-th factor of $V$, and write $\Gsimp_i$ for the corresponding factor of $\Gzero = \Gsimp^d$ (it satisfies $\Gsimp_i \cong \Gsimp$).
The semilinearity condition means that $\phi$ decomposes as a sum of maps
\[ \phi_i \colon V_i \rightarrow V_{\sigma i}. \]


It follows from functoriality of the Jordan decomposition that, 
if $(V, \phi, \mathfrak{f}, F_M)$ is equivalent to $(V_0, \phi_0, \mathfrak{f}_0, (F_0)_{M_0})$,
then $(V, \phi^{ss}, \mathfrak{f}, F_M)$ is equivalent to $(V_0, \phi_0^{ss}, \mathfrak{f}_0, (F_0)_{M_0})$.
Thus, as in \cite[\S 11.6, third paragraph]{LV}, we can assume $\phi$ is semisimple.
Otherwise, replacing $\phi$ with $\phi^{ss}$ will only increase the dimension of the set of $(\mathfrak{f}, F_M)$
for which $(V, \phi, \mathfrak{f}, F_M)$ is equivalent to $(V_0, \phi_0, \mathfrak{f}_0, (F_0)_{M_0})$.

A filtration $\mathfrak{f}$ on $\Gzero$ (which is the same as a filtration on $\Gsemi$) decomposes as the product of filtrations $\mathfrak{f}_i$ on $\Gsimp_i$.
If $\mathfrak{f}$ is $\phi$-stable, then $\mathfrak{f}_i$ determines $\mathfrak{f}_{\sigma i}$.
So to specify $\mathfrak{f}$ it is enough to specify $\mathfrak{f}_i$, for a single $i$ in each $\sigma$-orbit.

Since everything in sight splits as a direct product over $\sigma$-orbits, we may as well restrict attention to a single $\sigma$-orbit;
call it $\{i_1, \ldots, i_r\}$.
A $\phi$-stable filtration $\mathfrak{f}$ on $\Gzero$ is uniquely determined by a $\phi^r$-stable filtration $\mathfrak{f}_1$ on $\Gsimp_1$.
As in \cite[\S 11.6, fourth and fifth paragraphs]{LV}, the dimension of the set of such filtrations is exactly
\[ \dim Z_{\Gzero}(\phi) - \dim Z_P(\phi), \]
where $Z_{\Gzero}(\phi) = Z_{\Gsemi}(\phi) \cap \Gzero$, and $Z_P(\phi)$ is defined similarly.

To conclude, we need to show that, given $\mathfrak{f}$ (and an associated Levi subgroup $M$),
the set of filtrations $F_{M}$ such that $(V, \phi, \mathfrak{f}, F_M)$ is equivalent to $(V_0, \phi_0, \mathfrak{f}_0, (F_0)_{M_0})$
has dimension at most $\dim Z_P(\phi)$.
The proof of this is the same as in \cite[end of \S 11.6]{LV}, using Lemma \ref{para_levi_conj} in place of \cite[Equation 2.1]{LV}.
\end{proof}

\begin{lem}
\label{codim_final_estimate}
(Compare Lemma \ref{codim_final_estimate_ss}.)

Assume we are in the setting of Section \ref{semilinear_section} and Section \ref{Emodules}.
Fix a $\Gzerod$-bifiltered $\phi$-module $(V_0, \phi_0, F_0, \mathfrak{f}_0)$,
and another $\phi$-module $(V, \phi)$ with $\Gsemid$-structure;
suppose both $\phi$ and $\phi_0$ are semilinear over some $\sigma \in \Aut_{E_0} E$,
and $F_0$ is balanced with respect to $\mathfrak{f}_0$.

Let $\Gmon$ be a subgroup of $\Gzerod$, strongly $c$-balanced with respect to $\sigma$ for some positive integer $c$.

Suppose $F_0$ is uniform in the sense of Definition \ref{uniform_hodge}, and let $h^a = h^a_{simp}$ be the adjoint Hodge numbers on $\Gsimp$.
Let $\torusdim$ be the dimension of a maximal torus in $\Gsimp$.
Suppose $e$ is a positive integer satisfying the following numerical conditions.
\begin{itemize}
\item (First numerical condition.)
\[\sum_{a>0} h^a \geq \frac{1}{c} (e + \dim \Gsimp)\]
and
\item (Second numerical condition.)
\[\sum_{a>0} a h_a > T\left (\frac{1}{c} (e + \dim \Gsimp) \right ) + T \left(  \frac{1}{2} (h^0 - \torusdim) + \frac{1}{c} (e + \dim \Gsimp) \right ).\]
\end{itemize}

Let $\mathcal{H} = \Gmon / (Q_0 \cap \Gmon)$ be the flag variety parametrizing filtrations on $\Gzerod$ that are conjugate to $F_0$ under the conjugation of $\Gmon$.
Then the filtrations $F$ such that $(V, \phi, F)$ is of semisimplicity type $(V_0, \phi_0, F_0, \mathfrak{f}_0)$ are of codimension at least $e$ in $\mathcal{H}$.
\end{lem}

\begin{proof}
Let $P$ be the parabolic of $\Gzerod$, and $M$ an associated Levi subgroup, 
associated to a hypothetical semisimplification filtration $\mathfrak{f}$.
There are only finitely many possibilities for the parabolic group $P$, up to $\Gzerod$-conjugacy, so we may as well as fix a single $P$.

The dimension in question can be calculated after base change to an extension of $K$,
so we can assume that $\Gzerod = \Gsimp^d$.
As in the proof of Lemma \ref{codim_second_estimate},
we'll call the $d$ factors $\Gsimp_1, \ldots, \Gsimp_d$.
The whole setup factors over these $d$ factors, so we can write $P$ as the direct sum of parabolics $P_i \subseteq \Gsimp_i$, and so forth.
Again, $\sigma \in \Aut_{E_0} E$ gives a permutation of the index set $\{1, \ldots, d\}$,
which we also call $\sigma$.
Semilinearity over $\sigma$ means that the map $\phi$ permutes the $d$ factors according to the permutation $\sigma$.

The strategy is as follows. We want to apply a result like Lemma \ref{codim_first_estimate}, but we don't have full monodromy group $\Gzerod$.
Instead, we know the group $\Gmon$ is $c$-balanced; 
we'll project onto $c$ of the $d$ factors, so that $\Gmon$ projects onto the full group $(\Gsimp / Z(\Gsimp))^c$.
We need the projection of $F_0$ to the $c$ factors to be semibalanced,
and we can apply Lemmas \ref{codim_first_estimate} and \ref{codim_second_estimate} to the group $\Gsimp^c$ to finish.

For any subset $I$ of the index set $\{1, \ldots, d\}$, let $\Gsimp^I = \prod_{i \in I} \Gsimp_i$; 
define $P^I$, $M^I$, and $F_0^I$ similarly.
Any filtrations $F$ and $\mathfrak{f}$ on $\Gzerod$ can be written as products of factors $F_i$ and $\mathfrak{f}_i$,
so we can define $F^I$ and $\mathfrak{f}^I$.

We claim that, for any $\mathfrak{f}$ and any $F$ that is balanced with respect to $\mathfrak{f}$, we can find $I \subseteq \{1, \ldots, d\}$ satisfying the following properties.
\begin{itemize}
\item $I$ consists of exactly $c$ elements, from $c$ distinct index classes for $\Gmon$.
\item The elements of $I$ belong to a single orbit of $\sigma$ on $\{1, \ldots, d\}$.
\item $F^I$ is semibalanced with respect to $\mathfrak{f}^I$.
\end{itemize}

To see this, let $\mu$ be a cocharacter defining $F$ as in Definition \ref{dff:balanced}.
We can write $\mu$ as a product of $\mu_i$ over $i \in \{1, \ldots, d\}$.
Similarly, the character $\gamma_P$ splits over the factors, and we have:
\[ \gamma_P(\mu) = \sum_i \gamma_{P_i} (\mu_i). \]

Hence there exists an orbit $J$ of $\sigma$ on the $d$ factors such that $F^J$ is semibalanced.
Since $\Gmon$ is strongly $c$-balanced, 
we can find a subset 
\[I \subset J \subset \{1, \ldots, d\}\]
of the index set such that $\# I = c$, 
the elements of $I$ belong to $c$ distinct index classes, and $F^I$ is semibalanced.
Since the elements of $I$ belong to distinct index classes,
the projection 
\[ \Gmon \rightarrow (\Gsimp / Z(\Gsimp))^I \]
has image a union of connected components of the target, 
so it is smooth with equidimensional fibers, and the same is true of
\[ \Gmon / (Q_0 \cap \Gmon) \rightarrow \Gsimp^I / (Q_0^I \cap \Gsimp^I). \]

We want to estimate the codimension in $\mathcal{H} = \Gmon / (Q_0 \cap \Gmon)$ of the set of $F$ such that $(V, \phi, F, \mathfrak{f})$ is semisimply equivalent to $(V_0, \phi_0, F_0, \mathfrak{f}_0)$, for some choice of $\mathfrak{f}$.
Consider the projection
\[ (V, \phi, F, \mathfrak{f}) \rightarrow (\mathfrak{f}, F_M). \]

By Lemma \ref{codim_second_estimate} applied to $\Gsimp^J$, 
the set of pairs $(\mathfrak{f}^J, F_M^J)$
such that $(V_0^J, \phi_0^J, (F_0)_{M_0}^J, \mathfrak{f}_0^J)$ is equivalent to $(V^J, \phi^J, F_M^J, \mathfrak{f}^J)$
has dimension bounded by $\dim Z_{G^J} ((\phi^J)^{ss})$; this dimension is at most $\dim \Gsimp$ by Lemma \ref{semilinear_dim}.

Now fix $(\mathfrak{f}^J, F_M^J)$.  By Lemma \ref{codim_first_estimate} applied to $\Gsimp^I$, the set of $F^I$ such that
\[ \operatorname{Gr}_M F^I = F_M^I \]
has codimension at least $e + \dim \Gsimp$
among all $\Gsimp^I$-filtrations $F^I$.
The map $F^J \mapsto F^I$ from $\Gsimp^J$-filtrations in a given $\Gmon$-conjugacy class to $\Gsimp^I$-filtrations is smooth with equidimensional fibers, 
so the set of $F^J$ such that
\[ \operatorname{Gr}_M(F^J) = F_M^J \]
again has codimension at least $e  + \dim \Gsimp$.

It follows that the set of $F$ satisfying the desired condition has codimension at least $e$.
\end{proof}

\begin{thm}
\label{LV_thm}
(Basic theorem giving non-density of integral points.  Compare Theorem \ref{LV_thm_ss}.)

Let $X$ be a variety over $\mathbb{Q}$,
let $S$ be a finite set of primes of $\mathbb{Z}$,
and let $\mathcal{X}$ be a smooth model of $X$ over $\mathbb{Z}[1/S]$. 

Let $\mathsf{E}$ be a constant $H^0$-algebra on $\mathcal{X}$, and
let $\Gsimp$ be one of $GL_N$, $GSp_N$, or $GO_N$. 
Let $\mathsf{V}$ be a polarized, integral $\mathsf{E}$-module with $\Gsimp$-structure, in the sense of Definition \ref{dff:gsp}, having integral Frobenius eigenvalues (Def.\ \ref{dff:int-frob-evals}).
Suppose the Hodge numbers of $\mathsf{V}$ are uniform in the sense of Definition \ref{uniform_hodge}, and let $h^a = h^a_{simp}$ be the adjoint Hodge numbers on $\Gsimp$.
Let $\torusdim$ be the dimension of a maximal torus in $\Gsimp$.
Let $p$ be as in Definition \ref{HD_def}.
Let $\Gzerod$ and $\Gsemid$ be as in Section \ref{Emodules}.

Suppose there is a positive integer $c$ such that $\mathsf{V}$ satisfies the following conditions.
\begin{itemize}
\item (Big monodromy.) 
If $\Gmon \subseteq \Gzerod$ is the differential Galois group of $\mathsf{V}$,
then $\Gmon \subseteq \Gzerod$ is strongly $c$-balanced with respect to Frobenius. 
(The Frobenius is determined from the structure of $\mathsf{E}$; see Section \ref{Emodules}.)
\item (First numerical condition.) 
\[ \sum_{a>0} h^a \geq \frac{1}{c} (\dim X + \dim \Gsimp) \] 
and
\item (Second numerical condition.)
\[ \sum_{a>0} a h_a > T_{G}\left (\frac{1}{c} (\dim X + \dim \Gsimp) \right ) + T_G \left(  \frac{1}{2} (h^0 - \torusdim) + \frac{1}{c} (\dim X + \dim \Gsimp) \right ). \]
\end{itemize}
Then the image of $\mathcal{X}(\mathbb{Z}[1/S])$ is not Zariski dense in $X$.

\end{thm}

\begin{proof}
This follows from Lemmas \ref{fin_many_semi_types} and \ref{codim_final_estimate} in the same way that
Theorem \ref{LV_thm_ss} follows from Lemmas \ref{fin_many_semi_types_ss} and \ref{codim_final_estimate_ss}.

For every $x \in \mathcal{X}(\mathbb{Z}[1/S])$, the fiber $\rho_x = \mathsf{V}_{et, x}$ of the $p$-adic \'etale local system
is a global Galois representation valued in $\Gsemie$, having good reduction outside $S$ 
and all Frobenius eigenvalues Weil numbers, by hypothesis.
By Faltings's finiteness lemma (in the form of Lemma \ref{falt_finite}),
there are only finitely many possible isomorphism classes
for the semisimplified representation $\rho^{ss}_x$. 
So it is enough to show, for any fixed $\rho_0$,
that the set
\[\mathcal{X}(\mathbb{Z}[1/S], \rho_0) := \{ x \in X(\mathcal{O}_{K, S}) | \rho^{ss}_x \cong \rho_0 \}\]
is not Zariski dense in $X$.

Lemma \ref{fin_many_semi_types} gives a finite list of semisimplicity types
such that, for every $x \in \mathcal{X}(\mathbb{Z}[1/S], \rho_0)$, the filtered $\phi$-module $\mathsf{V}_{cris, x}$ belongs to one of them.
Let $\Omega \subseteq \mathcal{X}(\mathbb{Z}[1/S], \rho_0)$ be a mod-$p$ residue disk,
and fix a semisimplicity type $(V_0, \phi_0, F_0, \mathfrak{f}_0)$.
By Lemma \ref{fin_many_semi_types}, $F_0$ is balanced with respect to $\mathfrak{f}_0$.
It is enough to show that the set
\[ X(\Omega, (V_0, \phi_0, F_0, \mathfrak{f}_0)) = \{ x \in \Omega | \text{$\mathsf{V}_{cris, x}$ is of semisimplicity type $(V_0, \phi_0, F_0, \mathfrak{f}_0)$} \} \]
is not Zariski dense in $X$.

We are now in the setting of Theorem \ref{bt}.
For $x \in \Omega$, the filtered $\phi$-module $\mathsf{V}_{cris, x}$ is of the form $(V, \phi, F_x)$,
where $(V, \phi)$ is independent of $x$.
(This is a property of $F$-isocrystals in general; it reflects the fact that if $(V, \phi, F)$ is the crystalline cohomology of a scheme over $\mathbb{Z}_p$, then $(V, \phi)$ can be recovered from its reduction modulo $p$.
See \cite[Section 3.3]{LV} for further discussion.)
The variation of $F_x$ with $x$ is classified by a $p$-adic period map 
\[ \Phi_p \colon \Omega \rightarrow \Gzerod / Q, \]
where $\Gzerod/Q$ is the flag variety classifying $\Gzerod$-filtrations on $V$.

In fact, the $p$-adic period map lands in a single $\Gmon$-orbit on $\Gzerod / Q$; we can write the orbit as $\Gmon / (Q \cap \Gmon)$.
By Lemma \ref{codim_final_estimate} with $e = \dim X$, there is a Zariski-closed set $Z \subseteq \Gmon / (Q \cap \Gmon)$ of codimension at least $\dim X$
such that, if $(V, \phi, F_x)$ is of semisimplicity type $(V_0, \phi_0, F_0, \mathfrak{f}_0)$,
then $x \in \Phi_p^{-1}(Z)$.
We conclude by Theorem \ref{bt}.
\end{proof}

\section{Proof of main theorem}
\label{sec:main_thm}


\begin{dff}
\label{prim}
Let $A$ be an abelian variety over a number field $K$, and $H \subseteq A$ a hypersurface.
We say that $H$ is \emph{primitive} if it is not invariant under translation by any $x \in A(\overline{\mathbb{Q}})$.
\end{dff}

Recall the sequence $a(i)$ defined in Theorem \ref{Tannakian-group-calculation}. 

\begin{thm}
\label{main_thm}
Let $A$ be an abelian variety of dimension $n$ over a number field $K$.
Let $S$ be a finite set of primes of $\mathcal{O}_K$ including all the places of bad reduction for $A$.

Let $\phi$ be an ample class in the N\'eron-Severi group of $A$, and let $d = \phi^n/n!$. 

Suppose that either $n\geq 4$ or $n=3$ and $d$ is not $\binom{a(i)+a(i+1)}{a(i+1)/6}$ for any $i \geq 2$.

Then, up to translation there are only finitely many smooth primitive hypersurfaces $H \subseteq A$ representing $\phi$, defined over $K$ with good reduction outside $S$. \end{thm}

\begin{proof}
Choose a smooth proper model $\mathcal A$ of $A$ over $\mathcal{O}_{K, S}$.
Working over $\mathcal{O}_{K, S}$, 
let $\mathrm{Hilb}$ be the Hilbert scheme of smooth hypersurfaces of class $\phi$,
and let $H_{\mathrm{univ}} \subseteq \mathrm{Hilb} \times  \mathcal A$ be the universal family over $\mathrm{Hilb}$.
Let
\[ \mathrm{Hilb} (\mathcal{O}_{K, S})^{prim} \subseteq \mathrm{Hilb} (\mathcal{O}_{K, S}) \]
be the set of \emph{primitive} hypersurfaces in $A$ (Definition \ref{prim}). 
(Note that, by definition, a hypersurface defined over $K$ with good reduction outside $S$ extends to an $\mathcal O_{K,S}$-point of $\mathrm{Hilb}$. However, we are interested in the set of hypersurfaces which are primitive over $K$, not the set of $\mathcal O_{K,S}$-points of the scheme of primitive hypersurfaces, which may be smaller.)
The group $A(K) =\mathcal A(\mathcal {O}_{K,S})$ acts on $\mathrm{Hilb}$ by translation;
we need to show that $\mathrm{Hilb} (\mathcal{O}_{K, S})^{prim}$ is contained in the union of finitely many orbits of $A(K)$.

Theorem \ref{LV_thm} only works over $\mathbb{Q}$;
we'll pass from $K$ to $\mathbb{Q}$ by Weil restriction.
(See \cite[Theorem 7.6.4]{Neron} for the notion of Weil restriction
relative to a finite \'etale extension of rings.)
Enlarging $S$ if necessary, we can arrange that $\mathcal{O}_{K, S}$ is finite \'etale over $\mathbb{Z}[ 1 / S']$,
where $S'$ is a finite set of primes of $\mathbb{Z}$.
Over the Weil restriction 
\[ \Res{\mathcal{O}_{K, S}}{\mathbb{Z}[1/S']} \mathrm{Hilb} \]
we have the family
\[ H_{\mathrm{univ}, \mathbb{Q}} = H_{\mathrm{univ}} \times_{\mathrm{Hilb}} (\Res{\mathcal{O}_{K, S}}{\mathbb{Z}[1/S']} \mathrm{Hilb} \times_{\mathbb{Z}[1/S']} \mathcal O_{K,S}),  \]
and the fibers of this family are the same as the fibers of the original family $H_{\mathrm{univ}}$.
More precisely, the sets $\mathrm{Hilb}(\mathcal{O}_{K, S})$ and 
\[\Res{\mathcal{O}_{K, S}}{\mathbb{Z}[1/S']} \mathrm{Hilb}(\mathbb{Z}[1/S'])\]
are in canonical bijection.
Let 
\[ \Res{\mathcal{O}_{K, S}}{\mathbb{Z}[1/S']} \mathrm{Hilb}(\mathbb{Z}[1/S'])^{prim} \subseteq \Res{\mathcal{O}_{K, S}}{\mathbb{Z}[1/S']} \mathrm{Hilb}(\mathbb{Z}[1/S']) \] 
be the subset corresponding to $\mathrm{Hilb} (\mathcal{O}_{K, S})^{prim}$ under this bijection.
For any $x \in \mathrm{Hilb}(\mathcal{O}_{K, S})$, call $\operatorname{Res} x$ the corresponding point of $\Res{\mathcal{O}_{K, S}}{\mathbb{Z}[1/S']} \mathrm{Hilb}(\mathbb{Z}[1/S'])$.
Then we have a canonical isomorphism of schemes
\[(H_{\mathrm{univ}}) _x \cong (H_{\mathrm{univ}, \mathbb{Q}})_{\operatorname{Res} x};\]
here $(H_{\mathrm{univ}}) _x$ is an $\mathcal{O}_{K, S}$-scheme, $(H_{\mathrm{univ}, \mathbb{Q}})_{\operatorname{Res} x}$ is a $\mathbb{Z}[1/S']$-scheme,
and the structure maps are related by the fact that the diagram
\begin{equation}
\xymatrix{
(H_{\mathrm{univ}}) _x \ar[r]^{\cong} \ar[d] & (H_{\mathrm{univ}, \mathbb{Q}})_{\operatorname{Res} x} \ar[d] \\
\Spec \mathcal{O}_{K, S} \ar[r] & \Spec \mathbb{Z}[1/S'] \\
}
\end{equation}
commutes.

Let $X_{sing}$ be an irreducible component of the Zariski closure of $(\Res{K}{\mathbb{Q}} \mathrm{Hilb}) (\mathbb{Z}[1/S])^{prim}$ in $\Res{K}{\mathbb{Q}} \mathrm{Hilb}$,
and choose a resolution of singularities $X \rightarrow X_{sing}$. 
Let $Y = H_{\mathrm{univ}, \mathbb{Q}} \times_{\Res{K}{\mathbb{Q}} \mathrm{Hilb}} X$ be the pullback of our family of hypersurfaces to $X$. Note that $Y$ is a hypersuface in $X \times_{\mathbb{Q}} A = X_K \times_K A$.

We will apply Theorem \ref{LV_thm} to show that $Y \subset X_K \times _K A$ is the translate of a constant hypersurface $H_0 \in A$ by a section $s \in A(X_K)$. In this case, all the points of $(\Res{K}{\mathbb{Q}} \mathrm{Hilb} )(\mathbb{Z}[1/S])^{prim}$ lying in $X$ would correspond to points in the orbit of $H_0$. 

The map $X \rightarrow X_{sing}$ is proper, so it spreads out to a proper map 
\[ \mathcal{X} \rightarrow \Res{\mathcal{O}_{K, S}}{\mathbb{Z}[1/S']} \mathrm{Hilb}, \]
possibly after further enlarging the finite sets $S$ and $S'$.
Possibly after further enlargement of $S$ and $S'$, the family $Y$ spreads out to a smooth proper family $\mathcal{Y} \rightarrow \mathcal{X}$, which is a hypersurface in $\mathcal{A} \times_{ \mathbb{Z}[1 /S']} \mathcal{X}$.
Now, by properness, every $S$-integral point of $\mathrm{Hilb}$ lifts to an $S$-integral point of $\mathcal{X}$,
so by construction the $S$-integral points are Zariski dense in $\mathcal{X}$.

Fix some $p \notin S'$. Assume that $Y \subseteq X_K \times_K A$ is not equal to the translate of a constant hypersurface $H_0 \in A$ by a section $s \in A(X_K)$.  Let $\eta$ be the generic point of $X_K$. The fibers of $Y$ over points in a dense subset of $X$ are primitive hypersurfaces, so $Y_{\overline{\eta}}$ is not translation-invariant by any nonzero element of $A$. Let $G$ be the Tannakian group of the constant sheaf on $Y_{\overline{\eta}}$, and let $G^*$ be the commutator subgroup of the identity component of $G$. By Theorem \ref{Tannakian-group-calculation}, because of our assumptions on $n$ and $d$, we have $G^* = SL_N, Sp_N, $ or $SO_N$, acting by its standard representation. Furthermore the $Sp_N$ case occurs exactly when $Y_{\overline{\eta}}$ is equal to a translate of $[-1]^* Y_{\eta}$ and $n$ is even and the $SO_N$ case occurs exactly when $Y_{\overline{\eta}}$ is equal to a translate of $[-1]^* Y_{\eta}$ and $n$ is odd. In particular, $G^*$ is a simple algebraic group acting by an irreducible representation. 

We can now apply Corollary \ref{torsion_char} for a positive integer $c$ to be chosen shortly. 
This gives us, for any fixed $c$, the existence of an embedding $\iota \colon K \rightarrow \mathbb{C}$ and a torsion character $\chi$ of $\pi_1(A_{\iota})$, depending on $c$, satisfying the big monodromy condition that is needed for Lemma \ref{cbalanced_monodromy}. (Note a small subtlety here. The Tannakian monodromy groups $SO$, $Sp$, or $SL$ are calculated over an algebraically closed field $\overline{\mathbb Q}_p$, but to apply them we need them over a smaller field. Thus we use the straightforward fact that any form of $Sp, SL,$ or $SO$, together with its standard representation, over any field is the split form, or in the $SO$ case, the special orthogonal group associated to some nondegenerate quadratic form.) 
In Section \ref{locsys} and Lemma \ref{construct_v} we have constructed a Hodge--Deligne system $ \mathsf{V}_{I}$ on $\mathcal{X}$
attached to the orbit $I$ containing $(\iota, \chi)$ and the family $\mathcal{Y} \subseteq \mathcal{X}_{\mathcal{O}_{K, S}} \times_{\mathcal{O}_{K, S}} \mathcal{A}$. By Lemma \ref{cbalanced_monodromy}, the differential Galois group of $\mathsf{V}_{I}$ is a strongly $c$-balanced subgroup of $\mathsf{G}^0$.
Corollary \ref{torsion_char} gives vanishing of cohomology outside degree $n-1$, 
so the Hodge numbers of $\mathsf{V}_I$ are given by Lemma \ref{hodge-tate-calculation}.

We apply Theorem \ref{LV_thm} to $\mathsf{V} = \mathsf{V}_{I}$.
That the eigenvalues of Frobenius on $\mathsf{V}$ are integral Weil numbers is a consequence of the Weil conjectures, since $\mathsf{V}$ comes from geometry;
the polarization and integral structure are given in Lemma \ref{construct_v}.
Lemma \ref{lem:H_structure} gives an $\Gsimp$-structure over $\mathsf{E}_I$, with $\Gsimp$ chosen as in Lemma \ref{lem:H_structure}. (Note that this matches the group $G^*$ by our earlier calculation, because $Y_{\overline{\eta}}$ is equal to a translate of $[-1]^* Y_{\eta}$ if and only if $Y$ is equal to a translate of $[-1]^* Y$.)
Lemma \ref{cbalanced_monodromy} gives that $\mathsf{G}_{mon}$ is strongly $c$-balanced with respect to Frobenius. 
The Hodge numbers of $\mathsf{V}$ are uniform because we have explicitly computed them, independently of $\iota$, in Lemma \ref{hodge-tate-calculation}.

The numerical conditions in the hypothesis of Theorem \ref{LV_thm} will hold for big enough $c$.  To verify this, since every term except $c$ is independent of $c$ (because the Hodge numbers calculated in Lemma \ref{hodge-tate-calculation} are independent of $\chi$ and thus of $c$), it suffices to have
\[ \sum_{a>0} h^a >0 \]
\[ \sum_{a>0} a h^a >T_G \left( \frac{1}{2} (h^0- t) \right) .\]
By Lemma \ref{numerical_verification}, we have
\[ \sum_{a>0}  h^a > \frac{1}{2} (h^0- t)  .\]
This implies the first condition because $t\leq h^0$. It implies the second condition because, by the definition of $T_G$, $T_G(x)$ is strictly increasing for $x\leq \sum_{a>0} h^a$  and $T_G ( \sum_{a>0}  h^a) =  \sum_{a>0} a h^a$.

Hence the hypotheses of Theorem \ref{LV_thm} are satisfied, showing that the integral points are not Zariski dense, contradicting our construction. We thus conclude that our assumption for contradiction is false, i.e.\ that $Y \subseteq X_K \times_K A$ is equal to the translate of a constant hypersurface $H_0 \in A$ by a section $s \in A(X_K)$. It follows that every point of $(\Res{K}{\mathbb{Q}} \mathrm{Hilb}) (\mathbb{Z}[1/S])^{prim}$ that is contained in $X^{sing}$ corresponds to a hypersurface that is a translate of $H_0$ (by the value of the section $s$ on the inverse image of that point). Since there are finitely many irreducible components of the Zariski closure, every point of  $(\Res{K}{\mathbb{Q}} \mathrm{Hilb}) (\mathbb{Z}[1/S])^{prim}$ is contained in at least one irreducible component, and our argument applies to each irreducible component, we see that there exist finitely many hypersurfaces such that every primitive hypersurface with good reduction away from $S$ is isomorphic to a translate of one of them. In other words, there are only finitely many primitive hypersurfaces, up to translation.
\end{proof}

\begin{cor}\label{main_cor}
Let $K$, $A$, $S$, $\phi$ be as in Theorem \ref{main_thm}.

Suppose that either $n\geq 4$ or $n=3$ and $d$ is not a multiple of $\binom{a(i)+a(i+1)}{a(i+1)}/6$ for any $i \geq 2$ 

There are only finitely many smooth hypersurfaces $H \subseteq A$ representing $\phi$, with good reduction outside $S$, up to translation. 

\end{cor}

\begin{proof}
Any hypersurface $H \subseteq A$ is of the form $\pi^{-1} H'$, where $\pi \colon A \rightarrow A'$ is an isogeny defined over $K$,
and $H' \subseteq A'$ is a primitive hypersurface defined over $K$. In this case $\phi$ is the pullback of an ample class $\phi'$ along $\pi$, and we have $d = \phi^n/n!= (\phi'^n/n!) \cdot \deg \pi$. Hence $\deg \pi$ is bounded, so there are only finitely many possibilities for $\pi$. For each one there is a unique $\phi'$ with $\phi =\pi^* \phi'$. Furthermore, $d(\phi') \neq \binom{a(i)+a(i+1)}{a(i+1)}/6$ for any $i\geq 2$. We conclude by applying Theorem \ref{main_thm} to each $(A', \phi')$.
\end{proof}

\begin{thm}\label{intro_1} Suppose $\dim A \geq 4$. Fix an ample class $\phi$ in the N\'eron-Severi group of $A$. There are only finitely many smooth hypersurfaces $H \subseteq A$ representing $\phi$, with good reduction outside $S$, up to translation. 
\end{thm}

\begin{proof} This is one case of Corollary \ref{main_cor} . \end{proof}

\begin{thm}\label{intro_2} Suppose $\dim A=3$. Fix an ample class $\phi$ in the N\'eron-Severi group of $A$. Assume that the intersection number $\phi \cdot \phi \cdot \phi$ is not divisible by $d(i)$ for any $i \geq 2$. There are only finitely many smooth hypersurfaces $H \subseteq A$ representing $\phi$, with good reduction outside $S$, up to translation.   \end{thm}

\begin{proof} This is one case of Corollary \ref{main_cor}, once we cancel the factor of $3!=6$ from the denominators. \end{proof}

\begin{thm} \label{intro_3}
Suppose $\dim A = 2$. Fix an ample class $\phi$ in the N\'eron-Severi group of $A$. There are only finitely many smooth curves $C \subseteq A$ representing $\phi$, with good reduction outside $S$, up to translation. 
\end{thm}

\begin{proof}
This is a consequence of the Shafarevich conjecture for curves.  By the Shafarevich conjecture, there are only finitely many possibilities for the isomorphism class of $C$.  Consider a fixed $C$; enlarging $K$ if necessary, we may assume that $C(K)$ is nonempty, and choose a basepoint $x_0 \in C(K)$.  It is enough to show that there are only finitely many maps $C \rightarrow A$ taking $x_0$ to the origin of $A$, for which the image of $C$ represents the class $\phi$.

By the Albanese property, pointed maps $(C, x_0) \rightarrow (A, 0)$ are in bijection with maps of abelian varieties $\operatorname{Jac} C \rightarrow A$; the set of such maps forms a finitely generated free abelian group.
Fix an ample class $\psi$ on $A$; define the degree of any map $f \colon \operatorname{Jac} C \rightarrow A$ by
\[ \deg f = f(C) \cdot \psi. \]
 This intersection number is a positive definite quadratic form on $\Hom (\operatorname{Jac} C, A)$, so there are only finitely many maps of given degree, and any map representing $\phi$ has degree $\phi \cdot \psi$.
\end{proof}

\appendix
\section{Verifying the numerical conditions}
\label{sec:numerical}

The goal of this section is to verify the two numerical conditions of Theorem \ref{LV_thm} for $c$ sufficiently large. 
This is very similar to the estimates in \cite[Section 10.2]{LV}.

Our Hodge-Deligne systems arise from pushforwards of character sheaves of an abelian variety along families of hypersurfaces in that abelian variety. Thus, the Hodge structure on the stalk at a point arises from the cohomology of a hypersurface, twisted by a rank-one character sheaf. 

For $A$ an abelian variety of dimension $n$, $Y$ a degree $d$ hypersurface in $A$, $\chi$ a finite-order character of $\pi_1(A)$ such that $H^i(Y, \mathcal L_\chi)$ vanishes for $i \neq n-1$, the $k$th Hodge number of the Hodge filtration on $H^{n-1}(Y, \mathcal L_\chi)$ is $d A(n, k)$,
where $A(n, k)$ is the Eulerian number, by Lemma \ref{lem-Euler-formula}. 
(The \emph{degree} of a hypersurface in an abelian variety is the degree of the corresponding polarization.)

Recall from Section \ref{locsys} that the various realizations of $H^{n-1}(Y, \mathcal L_\chi)$
form a Hodge--Deligne system with $\Gsimp$-structure, for $\Gsimp$ one of $\operatorname{GL}$, $\operatorname{GSp}$, and $\operatorname{GO}$.
This system has uniform Hodge numbers (Definition \ref{uniform_hodge}), as computed in Lemma \ref{lem-Euler-formula},
so it makes sense to talk about the Hodge filtration as a filtration on $\Gsimp$.

\begin{lem}
\label{numerical_verification}
Let $Y$ be a smooth hypersurface in an $n$-dimensional abelian variety $A$, let $\chi$ be a finite-order character of $\pi_1(A)$ such that $H^i(Y, \mathcal L_\chi)$ vanishes for $i \neq n-1$, and let $V$ be the Hodge structure on $H^{n-1}(Y, \mathcal L_\chi)$,
regarded as a filtered vector space with $\Gsimp$-structure, for $\Gsimp$ one of $\operatorname{GL}$, $\operatorname{GSp}$, $\operatorname{GO}$.
Let $h^k$ be the adjoint Hodge numbers (Definition \ref{dff:adjoint_hodge}), and let $\torusdim$ be the dimension of a maximal torus in $\Gsimp$.

If $n\geq 2$, we have $\frac{1}{2} (h^0 - \torusdim) < \sum_{k > 0} h^k$.
\end{lem}

The relevance of this inequality is that $T_G( \sum_{k>0} h^k ) = \sum_{k>0} k h^k$ by the definition of $T_G$ (Definition \ref{dff:T_function}), so by the monotonicity of $T_G$ this implies 
\[T_G( \frac{1}{2} (h^0 - \torusdim)) <T_G( \sum_{k > 0} h^k) = \sum_{k>0}  k h^k,\] 
which for $c$ sufficiently large is equivalent to the second numerical condition of Theorem \ref{LV_thm}

\begin{proof}
(See also \cite[Lemma 10.3]{LV}.)

Since the adjoint Hodge numbers satisfy $h^k = h^{-k}$, we have
\[ \sum_{k > 0} h^k = \frac{1}{2} \left ( (\sum_{k \in \mathbb{Z}} h^k) - h^0 \right ) = \frac{1}{2} \left ( \dim \Gsimp - h^0 \right ).  \]
Thus it is enough to show that
\begin{equation}\label{numerical-key-inequality} 2 h^0 < \dim \Gsimp + \torusdim. \end{equation}

In each case, we will calculate $h^0$ and $\dim \Gsimp$ in terms of Eulerian numbers $A(n, k)$, and then prove \eqref{numerical-key-inequality} by using the following inequality of Eulerian numbers, to be established later:
\begin{equation}\label{Eulerian-squared-inequality} 2 \sum_k d^2 A(n, k)^2  \leq  d^2 \left ( \sum_k A(n, k) \right )^2 \end{equation}
We adopt the convention that $A(n, k) = 0$ when $k \not \in \{0, \ldots, n-1\}$.
\begin{itemize}
\item If $\Gsimp = \operatorname{GL}$ then
\begin{eqnarray*}
h^0 & = & \sum_k d^2 A(n, k)^2 \\
\dim \Gsimp & = & d^2 \left ( \sum_k A(n, k) \right )^2\\
\torusdim & = & d \sum_k A(n, k).
\end{eqnarray*}
so \[ 2h^0 = 2 \sum_k d^2 A(n, k)^2  \leq  d^2 \left ( \sum_k A(n, k) \right )^2 <  d^2 \left ( \sum_k A(n, k) \right )^2 + d \sum_k A(n,k) = \dim \Gsimp +\torusdim.\]

\item If $\Gsimp = \operatorname{GSp}$ then 
\begin{eqnarray*}
h^0 & = & \frac{1}{2} \left[ \left ( \sum_k d^2 A(n, k)^2 \right  ) + d A(n, \frac{n-1}{2}) \right] + 1\\
\dim \Gsimp & = &\frac{1}{2} d \left ( \sum_k A(n, k) \right ) \left [  d \left ( \sum_k A(n, k) \right ) + 1 \right ] + 1\\
\torusdim & = & \frac{1}{2} d \sum_k A(n, k) + 1.
\end{eqnarray*}
so
\[2 h^0 =  \left ( \sum_k d^2 A(n, k)^2 \right  ) + d A(n, \frac{n-1}{2}) +2 \leq \frac{d^2}{2} \left ( \sum_k A(n, k) \right )^2 +  d A(n, \frac{n-1}{2}) +2 \] \[<  \frac{d^2}{2} \left ( \sum_k A(n, k) \right )^2 +  d \left( \sum_k A(n, k)\right)+2 = \dim \Gsimp + \torusdim\]
\item If $\Gsimp = \operatorname{GO}$ then
\begin{eqnarray*}
h^0 & = & \frac{1}{2} \left[ \left ( \sum_k d^2 A(n, k)^2 \right  ) - d A(n, \frac{n-1}{2}) \right] + 1\\
\dim \Gsimp & = &\frac{1}{2} d \left ( \sum_k A(n, k) \right ) \left [  d \left ( \sum_k A(n, k) \right ) - 1 \right ] + 1\\
\torusdim & = & \frac{1}{2} d \sum_k A(n, k) + 1,
\end{eqnarray*}
 (The formula for $\torusdim$ in the $\operatorname{GO}$ case holds only for even-dimensional orthogonal groups, but $d \sum_k A(n, k)  = d \cdot n!$ is even.) Thus 
 \[ 2h^0 =  \left ( \sum_k d^2 A(n, k)^2 \right  ) - d A(n, \frac{n-1}{2}) +2 <   \left ( \sum_k d^2 A(n, k)^2 \right  )+2 \leq  \frac{d^2}{2} \left ( \sum_k A(n, k) \right )^2 +2 = \dim \Gsimp +\torusdim .\]

 \end{itemize} 
 
We now prove \eqref{Eulerian-squared-inequality}. It is known that, as $n$ grows large, the Eulerian numbers 
approximate a normal distribution with variance $\sqrt{n}/12$.
This purely qualitative result implies that \eqref{Eulerian-squared-inequality} holds for sufficiently large $n$.

To get precise bounds, we'll use log concavity, together with a calculation of the second moment.
The key idea is that a sequence of numbers that (i) is log-concave, and (ii) has large second moment,
cannot be too concentrated at the middle term.
Let
\[ a_i = \frac{\sum_k A(n, k) A(n, k-i)}{(\sum_k A(n, k))^2}; \]
this is normalized so that $\sum a_i = 1$.

Now we'll prove that $a_0 \leq 1/2$ for all $n\geq 2$; this will be Lemma \ref{over2}.
This will be a consequence of log-concavity and a formula for the second moment.

\begin{lem}
The sequence $(a_i)$ is log-concave and satisfies $a_{-i} = a_i$.
\end{lem}

\begin{proof}
This is proved in the first paragraph of the proof of \cite[Lemma 10.3]{LV}.
Symmetry is elementary;
log-concavity follows from the classical fact that the Eulerian numbers are log-concave (see, for example, \cite[Problems 4.6 and 4.8]{Petersen})
and the fact that log-concavity is preserved under convolution \cite[Thms.\ 1.4, 3.3]{JG}. 
\end{proof}

\begin{lem}
\label{second_moment}
The second moment of $(a_i)$ is
\[ \sum_i i^2 a_i = \frac{n+1}{6}. \]
\end{lem}

\begin{proof}
This is \cite[Eqn.\ 10.10]{LV}.
\end{proof}

\begin{lem}\label{termbound2}
Suppose $a_0>1/2$. Then for all $k \geq 1$ we have \[ \sum_{i=k}^{\infty} a_i < \sum_{i=k}^{\infty} \frac{1}{2 \cdot 3^i}  .\]
\end{lem}

\begin{proof} Because $a_0>1/2$, by symmetry, we have \[\sum_{i=1}^{\infty} a_i = \frac{1-a_0}{2} < \frac{1}{4}  =  \sum_{i=1}^{\infty} \frac{1}{2 \cdot 3^i}.\]
so if \[  \sum_{i=k}^{\infty} a_i \geq \sum_{i=k}^{\infty} \frac{1}{2 \cdot 3^i}  \] we must have $a_j < \frac{1}{ 2 \cdot 3^j}$ for some $1 \leq j < k $.

But then by log-concavity and the fact that $a_0>\frac{1}{2} $, we have $a_i < \frac{1}{ 2 \cdot 3^j} $ for all $i>j$, and thus in particular for all $i \geq k$, so \[ \sum_{i=k}^{\infty} a_i < \sum_{i=k}^{\infty} \frac{1}{2 \cdot 3^i}  .\] \end{proof}

\begin{lem}
\label{small_middle_2}
If $a_0 > 1/2$ then $\sum_k k^2 a_k < \frac{3}{2}  $.
\end{lem}

\begin{proof} We have \[ \sum_k k^2 a_k = 2 \sum_{k=1}^{\infty} k^2 a_k = 2 \sum_{k=1}^{\infty}  (2k-1) \sum_{i=k}^{\infty} a_i  > 2 \sum_{k=1}^{\infty}  (2k-1) \sum_{i=k}^{\infty} \frac{1}{ 2\cdot 3^i}  = 2 \sum_{k=1}^{\infty} k^2 \frac{1}{2 \cdot 3^i} = \frac{3}{2} .\] Here we use symmetry, Lemma \ref{termbound2}, and the identity \[ \sum_{k > 0} k^2 \lambda^k = \frac{\lambda (1+\lambda)}{(1-\lambda)^3}.\]
\end{proof}

\begin{lem}
\label{over2}
If $n \geq 2$ then $a_0 \leq 1/2$.
\end{lem}

\begin{proof} If $n\geq 9$, this is immediate from Lemmas \ref{second_moment} and \ref{small_middle_2}.

For $2\leq n \leq 8$, we can prove this by computation. For even $n$ this follows immediately from the symmetry $A(n, k) = A(n, n-1-k)$, so it suffices to check $n=3,5,7$, which can be done by hand.
\end{proof}

This proves Lemma \ref{numerical_verification}.
\end{proof}

\label{app:combo}

\newcommand{\bluestar}[1][]{}

\section{Combinatorics involving binomial coefficients and Eulerian numbers}\label{combinatorics}

This section is devoted to proving Proposition \ref{combinatorics-main-statement}. 
Thus, throughout this section, we preserve the notation and assumptions of Proposition \ref{combinatorics-main-statement}, which we review here.

We use $A(n, q)$ for the Eulerian numbers.
We adopt the convention that $A(n,q)=0$ unless $0 \leq q<n$;
similarly, we take ${ n\choose q}$ to vanish whenever $n$ is positive but $q<0$ or $q \geq n$.

Recall from the introduction that $a(i)$ is the sequence satisfying \[ a(1)=1, a(2)=5, a(i+2) = 4 a(i+1) +1 - a(i) .\] 

\begin{prop}[Proposition \ref{combinatorics-main-statement}]\label{combinatorics-appendix-statement} Let $n \geq 2$ and $d \geq 1$ be integers.
Suppose that there exists a natural number $k$, function $m_H$ from the integers to the natural numbers and an integer $s$.
Write $m = \sum_{i} m_H(i)$, and suppose that $ 1< k < m - 1 $. 
Suppose the equation
\begin{equation}\label{eq:wedge-Euler-copy}  \sum_{ \substack{ m_S : \mathbb Z \to \mathbb Z \\ 0 \leq m_S(i) \leq m_H (i) \\ \sum_{i} m_S(i) =k \\ \sum_i i m_S(i) = s+q }} \prod_i { m_H(i) \choose m_S(i) }  = d A (n, q) . \end{equation} 
is satisfied for all $q \in \mathbb Z$.
Then we have one of the cases

\begin{enumerate}

\item $m=4$ and $k=2$

\item $n=2$ and $d = {2k-1 \choose k}$ for some $k>2$

\item $n=3$ and $d ={ a(i)+a(i+1) \choose a(i) } / 6$ for some $i \geq 2$.

\end{enumerate} \end{prop} 

A tuple $(n, d, m, k, m_H, s)$ satisfying the conditions of the proposition will be called a \emph{solution}.

We first handle the cases $n=2,3,4$ directly, then give a general argument that handles cases $n\geq 5$. Thus Proposition \ref{combinatorics-appendix-statement} will follow immediately once we have proven Lemmas \ref{n-2-classification} (the $n=2$ case), \ref{n-3-final} (the $n=3$ case), \ref{n-4} (the $n=4$ case), \ref{m1a_final_final}, \ref{k_small}, \ref{m1b_mult1}, \ref{m1b_r1_final}, and \ref{r_big_lem3} (which collectively handle the $n\geq 5$ case).

We can assume without loss of generality that $k \leq m/2$. 

Let $m_{max}$ and $m_{min}$ be the functions $m_S : \mathbb Z \to \mathbb Z $ satisfying $ 0 \leq m_S(i) \leq m_H (i) , \sum_{i} m_S(i) =k$, and maximizing, respectively, minimizing $\sum_i  i m_S(i)$. 

\begin{lemma}

There is a unique $w$ such that $m_{min} (i) = m_H(i)$ for all $i<w$,  $m_{min} (i)=0$ for all $i>w$, and $m_{min} (i)>0$ for $i=w$.

Similarly, there is a unique $w'$ such that $m_{max} (i) = m_H(i)$ for all $i>w'$, $m_{max}=0$ for all $i<w'$ and $m_{max}>0$ for $i=w'$.

Furthermore $w\leq w'$. \end{lemma}

\begin{proof} We take $w$ to be the largest $i$ such that $m_{min}(i)>0$. The last two conditions are then obvious and the fact that $m_{min} (i) = m_H(i)$ for $i<w$ follows by minimality -- if it were not so, we could increase $m_{min}(i)$ by $1$, reduce $m_{min}(w)$ by $1$, and thereby reduce $\sum_i i m_S(i)$ by $w-i$.

We take $w'$ to be the least $i$ such that $m_{max}(i)>0$, and make a symmetrical argument. 

Finally, for contradiction, assume $w>w'$. Then 
\begin{eqnarray*}
\sum_i m_H(i) & = & m \geq 2k = \sum_{i} m_{min}(i) + \sum_i m_{max}(i) \\
& > & \sum_{i <w} m_{min}(i) + \sum_{i> w'} m_{max}(i) =  \sum_{i <w} m_{H}(i) + \sum_{i> w'} m_{H}(i)   \geq  \sum_i m_H(i), 
\end{eqnarray*}
a contradiction. \end{proof}

\begin{lemma}  \label{diff_n_new}
We have

\begin{equation}\label{eq:small-n}  \sum_{i< w} (w-i) m_H(i) + \sum_{i>w'} (i-w') m_H(i) + k (w'-w)  = n-1\end{equation}

\end{lemma}

Note that all the terms on the left side of this equation are nonnegative.

\begin{proof} 

 Because the function $q \mapsto d E(n,q)$ is supported on $q$ ranging from $0$ to $n-1$, we must have 
 
 \[ n-1 = (s+n-1)  - s= \sum_i i m_{max}(i) -\sum_i  i m_{min}(i) = \sum_i (i-w') m_{max}(i) - \sum_i (i-w) m_{min}(i) + k (w'-w)\] \[ = \sum_{i< w} (w-i) m_H(i) + \sum_{i>w'} (i-w') m_H(i) + k (w'-w)  \] \end{proof}

\subsection{The case $n \leq 4$}

 \begin{lemma}\label{n-2-classification}  Suppose $n=2$. Then we must have $d=  {2k-1 \choose k}$. Furthermore if $k=2$ then $m=4$. \end{lemma}
 
 \begin{proof}  In \eqref{eq:small-n}, because the summands on the left side are nonnegative, one must be $1$ and the others must vanish. The one that is $1$ can only be the summand associated to $i= w-1$ or $i=w'+1$ as the other summands are integer multiples of something at least $2$. By symmetry, we may assume the $1$ comes from $i=w+1$.  Because the last summand vanishes, we must have $w=w'$. 
 
 This gives $m_H(i)=0$ unless $i=w$ or $w+1$ and $m_H(w+1)=1.$
 
 Then the only possible solutions to $0 \leq m_S(i) \leq m_H(i)$ and $\sum_i m_S(i)=k$ are $m_S(w)= k, m_S(w+1) =0$ and $m_S(w)=k-1, m_S(w+1)=1$. These have $\sum_i m_S(i)$ equal to $kw$ and $kw+1$ respectively, so we must have $s= kw$. This implies
 \[{ m_S(w) \choose k }  = { m_S(w) \choose k }  { 1\choose 0} = d E(2,0) = d\]
  \[{ m_S(w) \choose k-1 } ={ m_S(w) \choose k-1 } {1\choose 1} =  d E(2,1) = d\]
and thus
  \[{ m_S(w) \choose k-1 } = { m_S(w) \choose k } \] which implies $m_S(w) = 2k-1$ and thus $d = { 2k-1\choose k}. $
  
  In the $k=2$ case, we have $m= m_S(w) + m_S(w+1) =2k=4$, as desired.
  
  \end{proof}

\begin{lemma}\label{n-2-constraint} Suppose $n=3$. Then we must have $m_H(i)=0$ for $i \neq w-1,w,w+1$ and $m_H(w-1)=m_H(w+1)=1$ unless $m=4$ and $k=2$. \end{lemma}

\begin{proof} In \eqref{eq:small-n},
either two terms are $1$ and the rest zero or one term is $2$ and the rest are zero. In the first case, since the only terms that can be $1$ are $i=w-1$ and $i=w'+1$, implying in particular that $w=w'$ we have the stated conclusion. So it suffices to eliminate the case that one term is $2$. The only possibilities are $i=w-2,w-1,w'+1,w'+2$, and the last term if $k=2$. By symmetry, we are reduced to eliminating $i=w'+1$, $i=w'+2$, and the final term. In the first two cases we have $w'=w$ and in the last case we have $k=2,w'-w=1$.

If $m_H(i)=0$ for all $i$ except $w,w+1$, and $m_H(w+1)=2$, then there are three possibilities for $m_S$: $(m_S(w),m_S(w+1))$ must equal $(k,0), (k-1,1),$ or $(k-2,2)$. Using $m_H(w)=m-2$, this gives
\[ { m-2 \choose k} = {m-2 \choose k} { 2 \choose 0} = d E(3,0) = d \]
\[ 2{ m-2 \choose k-1} = {m-2 \choose k-1} { 2 \choose 1} = d E(3,1) =4 d \]
\[ { m-2 \choose k-2} = {m-2 \choose k-2} { 2 \choose 2} = d E(3,2) = d \]
and this implies
\[  \frac{1}{2} =  \frac{ {m-2 \choose k} }{ {m-2 \choose k-1}} = \frac{ m-k-1 }{ k} \]
\[  \frac{1}{2} =  \frac{ {m-2 \choose k-2} }{ {m-2 \choose k-1}} = \frac{ k-1}{ m-k} \]
so  $m-k = 2k-2$ and $k = 2m-2k-2$ giving $m=4, k=2$.

If $m_H(i)=0$ for all $i$ except $w,w+2$, then $\sum_i  i m_S(i) \equiv  k w \mod 2$ whenever $0 \leq m_S(i) \leq m_H(i)$ and $\sum_i m_S(i)=k$. Hence the left side of \eqref{eq:wedge-Euler} is nonzero only when $s+q \equiv kw\mod 2$. This contradicts the fact that $E(2,q)$ is nonzero for $q$ of both parities. 

If $m_H(i)=0$ for all $i$ except $w,w+1$ and $k=2$, then there are three possibilities for $m_S$: $(m_S(w),m_S(w+1))$ must be $(2,0), (1,1),$ or $(0,2)$. This gives
\[ \frac{ m_H(w) (m_H(w)-1 )}{2} =  {m_H(w) \choose  2} {m_H(w+1) \choose 0 } = dA (3,0) =d\]
\[  m_H(w) m_H(w+1)=  {m_H(w) \choose  1} {m_H(w+1) \choose 1 } = dA (3,1) =4d\]
\[ \frac{ m_H(w+1) (m_H(w+1)-1 )}{2} =  {m_H(w) \choose  0} {m_H(w+1) \choose 2 } = dA (3,2) =d\]

Because $d>0$ this implies $m_H(w) = m_H(w+1)$ and thus $2 m_H(w) ( m_H(w) -1) ) = m_H(w)^2$ which implies $m_H(w) = m_H(w+1)=2$ and thus $m=4$, $k=2$.
\end{proof}

\begin{lemma}\label{n-3-classification} Suppose $n=3$. Then unless $m=4$ and $k=2$, we have 
\[(m-k)^2 - 4 (m-k) k + k^2 = m.\] 
\end{lemma}

\begin{proof} By Lemma \ref{n-2-constraint}, we have $m_H(i)=0$ for $i \neq w-1,w,w+1$ and $m_H(w-1)=m_H(w+1)=1$. This means there are four possibilities for $m_S$: $(m_S(w-1), m_S(w), m_S(w+1))$ must equal $(0,k,0)$, $(0,k-1,1)$, $(1, k-1,0)$, or $(1,k-2,1)$. Using $m_H(w) = m-2$,  and ${1\choose 0} = {1\choose 1} = 1$, we get
\[ {m-2 \choose k-1} = d E(3,0) = d\]
\[ {m-2 \choose k} + {m-2 \choose k-2} = d E(3,1) = 4d \]
\[ {m-2 \choose k-1} = d E(3,2) = d\]
so in other words we have
\[ {m-2 \choose k} + {m-2 \choose k-2}= 4 {m-2 \choose k-1} \]
which dividing by $(m-2)!$ and multiplying by $k! (m-k)!$ is
\[ (m-k) (m-k-1) + k (k-1) = 4 k (m-k) \] which is exactly the stated Diophantine equation. \end{proof}

\begin{lemma}\label{n-3-diophantine} The positive integer solutions to $a^2 -4ab + b^2 = a+b$ with $a \leq b$ have the form $(a,b) = (a(i),a(i+1))$ for some $i \geq 1$.   \end{lemma}

\begin{proof} Let $a$ and $b$ be positive integers with $a^2-4ab+b^2$ and $b \geq a$. We will show that there exists $i \geq 1 $ such that $a=a(i) $ and $b= a(i+1)$. By induction, it suffices to prove that either $a=1, b=5$ or that there exists an integer $b'$ with $0 < b'<  a$, $b = 4a +1- b'$, and $(b',a)$ solving the same equation, as if $b' =a(i') , a= a(i'+1)$ then $b= a(i'+2)$. 

To do this, let $b'= 4a+1-b$. Then because we can rewrite the equation as 
\[ b (4a+1-b) =a^2 - a, \] 
as $b$ is a solution then $b'$ is a solution as  well. Furthermore, if $a>1$ then $a^2-a>0$ so $b' = \frac{a^2-a}{b} >0$, and we must have $b' <a$ because if $b' \geq a$ we have \[ a^2-a = b ( 4a+1-b) \geq a \cdot a > a^2-a,\] so if $a>1$ we always have such a $b'$. On the other hand, if $a=1$ then $b (5-b)=0$ so, because $b>0$, we have $b=5$, the base case.
\end{proof} 

\begin{lemma}\label{n-3-final} Suppose $n=3$. Then either $m=4$ and $k=2$ or $m = a(i) + a(i+1)$ and $k = a(i)$ for some $i \geq 2$. \end{lemma}
\begin{proof}This follows from Lemma \ref{n-3-classification} and \ref{n-3-diophantine} once we observe that because $k \leq m/2$, we have $k \leq m-k$, and because $k\geq 2$, the case $k=a(1), m=a(2)$ cannot occur. \end{proof}

\begin{lemma}\label{n-4} There are no solutions for $n=4$. \end{lemma}

\begin{proof} We first consider the contribution to \eqref{eq:small-n} from $k (w'-w)$. Because $k \geq 2$ and this contribution is at most $(n-1)=3$, we can only have $w'=w$ or $w'=w+1$, $k=2$ or $3$. Let us eliminate the $w'=w+1$ cases first.

In the $k=3$ case we have $m_H(i)=0$ unless $i=w$ or $w+1$. Thus we have four possibilities for $m_S$ -- we must have $(m_S(w),m_S(w+1))=(3,0),(2,1),(1,2),$ or $(0,3)$. This gives
\[ {m_H(w) \choose 3} {m_H(w+1) \choose 0} = d E(4,0)=d\]
\[ {m_H(w) \choose 2} {m_H(w+1) \choose 1} = d E(4,1)=11d\]
\[ {m_H(w) \choose 1} {m_H(w+1) \choose 2} = d E(4,2)=11d\]
\[ {m_H(w) \choose 0} {m_H(w+1) \choose 3} = d E(4,3)=d\]
Combining the first and last equations with $d>0$, we see that $m_H(w) = m_H(w+1)$. Dividing the second equation by the first, we get \[  \frac{ 3m_H( w)} { m_H(w) -2} = 11\]  which implies $m_H(w) = \frac{11}{4}$, a contradiction.

In the $k=2$ case, one more term must be $1$, which is $i=w-1$ or $i=w'+1$. Without loss of generality, it is $i=w-1$. Then we have $m_H(i)=0$ unless $i= w-1,w,w+1$ and $m_H(w-1) = 1$. Then the possible values of $(m_S(w-1), m_S(w),m_S(w+1))$ are $(1,1,0)$, $(1,0,1)$, $(0,2,0)$, $(0,1,1)$, and $(0,0,2)$. This gives (ignoring factors of the form $n\choose 0$  or $1 \choose 1$)
\[ {m_H(w) \choose 1 } = d E(4,0) = d\]
\[ {m_H(w+1) \choose 1} +{ m_H(w)\choose 2} = d E(4,1) =11d\]
\[ {m_H(w )\choose 1}{ m_H(w+1) \choose 1} = d E(4,2) = 11d\]
\[{m_H(w+1) \choose 2} =d\]
Dividing  the third equation by the first, we get $m_H(w+1)=11$. Thus by the fourth equation $d={11\choose 2}=55$, and by the first equation $m_H(w)=55$. Then the second equation gives $11+ {55 \choose 2} = 11 \cdot 55$, which is false, so there are no solutions.

We now handle the case $w=w'$. In this case, because the total sum of \eqref{eq:small-n} is $3$, and only the two terms $i=w+1,i=w-1$ can contribute a $1$, we must have one term contributing $3$ or one $2$ and one $1$. There are four terms that might contribute $3$: $i=w+3$, $i=w+1$, $i=w-1$, and $i=w-3$, and four that might contribute $2$: $i=w+2$, $i=w+1$, $i=w-1$, and $i=w-2$.  This gives a total of $10$ possibilities for $m_H$, or $5$ up to symmetry: we may assume $(m_H(w-1), m_H(w), m_H(w+1), m_H(w+2), m_H(w+3))= (0,m-1,0,0,1), (0,m-3,3,0,0), (0,m-2,1,1,0),  (1,m-2,0,1,0),$ or $(1,m-2,2,0,0)$.

The case $(0,m-1,0,0,1)$ is easy to eliminate as it implies that the left side of \eqref{eq:wedge-Euler} is nonvanishing only in a single residue class mod $3$, but we know that the right side does not have that possibility.

The case $(0,m-3,3,0,0)$ gives us four possibilities for $m_S$, implying
\[ {m-3 \choose k} { 3 \choose 0} = d E(4,0) =d\]
\[ {m-3 \choose k-1} { 3 \choose 1} = d E(4,0) =11d\]
\[ {m-3 \choose k-2} { 3 \choose 2} = d E(4,0) =11d\]
\[ {m-3 \choose k-3} { 3 \choose 3} = d E(4,0) =d\]
The first and fourth equations gives ${m-3 \choose k} ={m-3\choose k-3}$, which implies that $m-3-k=k-3$ or $m=2k$. Dividing the third equation by the fourth, we then obtain \[ 11 = \frac{ {2k-3 \choose k-2}}{ {2k-3 \choose k-3}} \frac{{3\choose 2}}{{3 \choose 3} }= \frac{k}{k-2}  \cdot 3 \] whose unique solution is $k=\frac{11}{4}$, a contradiction.

The case $(0,m-2,1,1,0)$ gives us four possibilities for $m_S$, implying
\[  {m-2 \choose k}{ 1\choose 0 }{ 1\choose 0 } = d E(4,0) = d\]
\[  {m-2 \choose k-1}{ 1\choose 1 }{ 1\choose 0 } = d E(4,1) =11 d\]
\[  {m-2 \choose k-1}{ 1\choose 0 }{ 1\choose 1 } = d E(4,2) =11d\]
\[  {m-2 \choose k-2}{ 1\choose 1 }{ 1\choose 1 } = d E(4,3) = d\]
By the first and fourth equations we have ${m-2\choose k} = {m-2\choose k-2}$ which implies $m-2-k=k-2$ or $m=2k$. Dividing the second equation by the first we get $11 = \frac{k}{k-1}$ so $k= \frac{11}{10}$, a contradiction. 

The case $(1,m-2,0,1,0)$ gives us four possibilities for $m_S$, implying
\[ {1 \choose 1} {m-2 \choose k-1}{ 1\choose 0} = dE(4,0)=d\]
\[ {1 \choose 0} {m-2 \choose k}{ 1\choose 0} = dE(4,1)=11d\]
\[ {1 \choose 1} {m-2 \choose k-2}{ 1\choose 1} = dE(4,2)=11d\]
\[ {1 \choose 0} {m-2 \choose k-1}{ 1\choose 1} = dE(4,3)=d\]
By the second and third equations, we have ${m-2 \choose k} ={m-2\choose k-2}$ which again implies $m=2k$. But  then ${2k-2 \choose k-1} > {2k-2 \choose k-2}$ which means $d>11d$, a contradiction.

The case $(1,m-2,2,0,0)$  has six possibilities for $m_S$, implying
\[ {1 \choose 1}{ m-2 \choose k-1} {2 \choose 0} = dE(4,0)=d\]
\[ {1 \choose 0}{ m-2 \choose k} {2 \choose 0} + {1 \choose 1}{ m-2 \choose k-2} {2 \choose 1}= dE(4,1)=11d\]
\[ {1 \choose 1}{ m-2 \choose k-3} {2 \choose 2}  + {1 \choose 0}{ m-2 \choose k-1} {2 \choose 1}= dE(4,2)=11d\]
\[ {1 \choose 0}{ m-2 \choose k-2} {2 \choose 2} = dE(4,3)=d\]
By the first and fourth equations, we have ${m-2 \choose k-1} = {m-2 \choose k-2}$ which implies $m=2k-1$. But then since ${2k-2 \choose k-3} \leq {2k-2\choose k-2}$ we have by the third and fourth equations 
\[ 11d = {2k-2 \choose k-3} + 2 {2k-2\choose k-1} \leq {2k-2\choose k-2} + 2{2k-2 \choose k-2} =3d\]
a contradiction. \end{proof}

\subsection{The case $n \geq 5$: general setup}
\label{n5_setup}

We'll write $m_0 = m_{min}$ for the rest of this section.  The key equality for $q=0$ gives
\begin{equation}\label{eq:d-binomial} d = { m_H(w) \choose m_0(w) } . \end{equation} 
For any $m_S : \mathbb Z \to \mathbb Z$ as in Equation \ref{eq:wedge-Euler-copy} we'll write
\[ N(m_S) = N(m_H, m_S) = \prod_i { m_H(i) \choose m_S(i) }. \]
Now the key equality (\ref{eq:wedge-Euler-copy}) becomes
\begin{equation}\label{eq:wedge-Euler-ratio}  \sum_{ \substack{ m_S : \mathbb Z \to \mathbb Z \\ 0 \leq m_S(i) \leq m_H (i) \\ \sum_{i} m_S(i) =k \\ \sum_i i m_S(i) = s+q }} 
\frac{N(m_S)}{N(m_0)}  = A (n, q). \end{equation}

We're going to get a contradiction from combinatorial considerations involving the terms associated to small $q$ in Equation \eqref{eq:wedge-Euler}. 

By abuse of notation, we'll let $[i]$ denote the function $m_{\{i\}}$
taking the value $1$ on $i$ and zero elsewhere;
so any function $\mathbb{Z} \rightarrow \mathbb{Z}$
is a linear combination of the elementary functions $[i]$.

There are at most two functions $m_1$ that contribute to the $q=1$ case in Equation \eqref{eq:wedge-Euler-ratio}. These are 
\[ m_1^a = m_0 + [w] - [w-1]\]
and
\[ m_1^b = m_0 + [w+1] - [w].\] 
We compute
\[ \frac{N(m_1^a)}{N(m_0)} = \frac{m_H(w-1) (m_H(w) - m_0(w))}{(m_0(w) + 1)} \]
\[ \frac{N(m_1^b)}{N(m_0)} = \frac{m_H(w+1) m_0(w)} {(m_H(w) - m_0(w) + 1)}; \]
the $q=1$ case of Equation \eqref{eq:wedge-Euler-ratio} gives
\[ \frac{N(m_1^a)}{N(m_0)} + \frac{N(m_1^b)}{N(m_0)} = A(n, 1). \]

Similarly, there are at most five nonzero terms in the $q=2$ case Equation \eqref{eq:wedge-Euler-ratio}:
\begin{eqnarray*}
m_2^a & = & m_0 + [w] - [w - 2] \\
m_2^b & = & m_0 + 2[w] - 2[w-1] \\
m_2^c & = & m_0 + [w+1] - [w-1] \\
m_2^d & = & m_0 + 2[w+1] - 2[w] \\
m_2^e & = & m_0 + [w+2] - [w].
\end{eqnarray*}
We have the following equalities.
\begin{eqnarray*}
\frac{N(m_2^a)}{N(m_0)} & = & \frac{m_H(w-2) (m_H(w) - m_0(w))}{(m_0(w) + 1)} \\
\frac{N(m_2^b)}{N(m_0)} & = & \frac{m_H(w-1)(m_H(w-1) - 1) (m_H(w) - m_0(w)) (m_H(w) - m_0(w) - 1)}{2 (m_0(w) + 1)(m_0(w) + 2)} \\
\frac{N(m_2^c)}{N(m_0)} & = & m_H(w-1) m_H(w+1) \\
\frac{N(m_2^d)}{N(m_0)} & = & \frac{(m_H(w+1)) (m_H(w+1) - 1) (m_0(w)) (m_0(w) - 1)}{2 (m_H(w) - m_0(w) + 1)(m_H(w) - m_0(w) + 2)} \\
\frac{N(m_2^e)}{N(m_0)} & = & \frac{m_0(w) m_H(w+2)}{m_H(w) - m_0(w) + 1}. \\
\end{eqnarray*}
Equation \eqref{eq:wedge-Euler-ratio} gives
\begin{equation}\label{sum-of-2-terms} \sum_{*  = a, b, c, d, e}   \frac{N(m_2^*)}{N(m_0)} = A(n, 2). \end{equation}


We conclude with a lemma that will be useful at several points in the argument.

\begin{lem}
\label{alpha_beta}
We have 
\[ \frac{N(m_1^a)}{N(m_0)} + \frac{N(m_1^b)}{N(m_0)} = A(n, 1)\]
and
\[ \left ( \frac{N(m_1^a)}{N(m_0)} \right ) \left ( \frac{N(m_1^b)}{N(m_0)} \right ) < m_H(w-1)m_H(w+1).\]

In particular, if 
\[ m_H(w-1)m_H(w+1) < \frac{1}{4} A(n, 1)^2, \]
let $\alpha < \beta$ be the real roots of
\[ X^2 - A(n, 1) X + m_H(w-1)m_H(w+1). \]
Then we have
\[ \min \left  ( \frac{N(m_1^a)}{N(m_0)}, \frac{N(m_1^b)}{N(m_0)} \right ) \leq \alpha  \]
and
\[ \max \left ( \frac{N(m_1^a)}{N(m_0)}, \frac{N(m_1^b)}{N(m_0)} \right  ) \geq \beta \]
\end{lem}

\begin{proof}
The first equality is the $q=1$ case of Equation \eqref{eq:wedge-Euler-ratio}.
The second follows from the explicit formulas for the two quotients $\frac{N(m_1^*)}{N(m_0)}$.

The bounds in terms of $\alpha$ and $\beta$ follow from the first two inequalities.
\end{proof}

\subsection{The case $n \geq 5$, with $N(m_1^a)$ big}

In the following sections we will use without proof a number of inequalities involving Eulerian numbers.
The proofs of such inequalities are routine; we discuss the technique in Section \ref{sec:eulerian_ineq}.

Recall notation from the beginning of Section \ref{combinatorics}, and the beginning of Section \ref{n5_setup}.

\begin{lem}
\label{m1a_w-1_small}
If $n \geq 5$ and
\begin{equation}\label{m1a_w-1_small-assumption} \frac{N(m_1^a)}{N(m_0)} \geq \frac{A(n, 1)}{2} \end{equation} 
then $m_H(w-1) = 1$.
\end{lem}

\begin{proof}
We'll estimate $N(m_2^b)$. The key point is that, combining \eqref{m1a_w-1_small-assumption} and \eqref{sum-of-2-terms}, we have 
\begin{equation}\label{m1a_w-1_small-key} \frac{N(m_2^b) N(m_0)}{N(m_1^a)^2} \leq 4 \frac{A(n, 2)}{A(n, 1)^2}. \end{equation}

We will focus on proving a lower bound for the left side, assuming $m_H(w-1) >1$, which which will give a contradiction for sufficiently large $n$. 

By hypothesis, we have
\[ \frac{A(n, 1)}{2} \leq \frac{m_H(w-1) (m_H(w) - m_0(w))}{(m_0(w) + 1)}. \]
By Lemma \ref{diff_n_new} we have $m_H(w-1) \leq n-1$.
Also, we know $m_0(w) \geq 1$, so
\[ m_H(w) - m_0(w) \geq \frac{A(n, 1)}{n-1}. \]

If $n \geq 6$, then $\frac{A(n, 1)}{n-1}\geq 9$, which gives the bound
\begin{equation}\label{8-9-bound}  m_H(w) - m_0(w) - 1 \geq \frac{8}{9}  (m_H(w) - m_0(w)) \end{equation}    
Assuming $m_H(w-1) \geq 2$, we find
\begin{align} \begin{split}  \label{eqn-line415}
& \frac{N(m_2^b) N(m_0)}{N(m_1^a)^2} \\
=&  \frac{m_H(w-1)(m_H(w-1) - 1) (m_H(w) - m_0(w)) (m_H(w) - m_0(w) - 1)}{2 (m_0(w) + 1)(m_0(w) + 2)} \cdot \frac{(m_0(w) + 1)^2}{m_H^2(w-1) (m_H(w) - m_0(w))^2}\\
= & \frac{1}{2} \cdot  \frac{ m_0(w)+1}{ m_0(w)+2} \cdot \frac{ m_H(w-1)-1}{ m_H(w-1)} \cdot \frac{m_H(w) - m_0(w) - 1}{m_H(w) - m_0(w)} \\
\geq & \frac{1}{2} \cdot \frac{2}{3} \cdot \frac{1}{2} \cdot \frac{m_H(w) - m_0(w) - 1}{m_H(w) - m_0(w)} 
=\frac{1}{6} \cdot \frac{m_H(w) - m_0(w) - 1}{m_H(w) - m_0(w)} \\
\end{split} \end{align}
By combining this with \eqref{8-9-bound}, we conclude that
\begin{equation}
\label{eqn_b} \frac{N(m_2^b) N(m_0)}{N(m_1^a)^2} \geq 4/27 
\end{equation}
which combines with \eqref{m1a_w-1_small-key} to give
\[ \frac{A(n, 2)}{A(n, 1)^2} \geq 1/27 \]
which is impossible for $n \geq 11$.
(See the discussion in Appendix \ref{sec:eulerian_ineq}, and \texttt{bound-A3a} in the Python code.) 
 \bluestar[bound_2_11f -- used as example, careful if we have to change it]

For smaller $n$, we do a more precise version of the above analysis. For $5 \leq n \leq 10$, we will improve on the bound \eqref{m1a_w-1_small-assumption} by using Lemma \ref{alpha_beta}. For $n=5$, we will also need to replace the bound \eqref{8-9-bound} by a slightly weaker one without the assumption $n\geq 6$. With these modifications, the argument will work for $n$ from $5$ to $10$.

To apply Lemma \ref{alpha_beta}, we must give an upper bound for $m_H(w-1)m_H(w+1)$, showing that the two roots $\alpha, \beta$ are far apart. Consider
\[ m_4^a = m_0 + 2 [w+1] - 2 [w-1]. \]
We have
\[ A(n, 4) \geq \frac{N(m_4^a)}{N(m_0)} = \frac{ \left [m_H(w+1) (m_H(w+1) - 1) \right ]  \left [m_H(w-1) (m_H(w-1) - 1) \right ] }{4}, \]
so (still assuming $m_H(w-1) \geq 2$) we conclude that
\[ 2 A(n, 4) \geq m_H(w+1) (m_H(w+1) - 1).  \]

For $5 \leq n \leq 10$ 
we have
\[  m_H(w+1) \leq \sqrt{2 A(n, 4)} + 1 \leq \frac{1}{48} \frac{A(n, 1)^2}{n-1} \]
(see \texttt{bound-A3b} in the Python code).  
Thus we have
\[ m_H(w-1)m_H(w+1) \leq (n-1) \cdot  \frac{1}{48} \frac{A(n, 1)^2}{n-1} =  \frac{1}{48} A(n, 1)^2, \]
so by Lemma \ref{alpha_beta},
\begin{equation}
\label{eqn_c}
 \frac{N(m_1^a)}{N(m_0)} \geq \frac{45}{46} A(n, 1). 
\end{equation}
When $6 \leq n \leq 10$ we have from above (\ref{eqn_b}) that
\[ \frac{N(m_2^b) N(m_0)}{N(m_1^a)^2} \geq 4/27, \]
so
\[ \frac{A(n, 2)}{A(n, 1)^2} \geq \frac{4}{27} \cdot \left (  \frac{45}{46} \right  )^2, \]
which does not hold for any $6 \leq n \leq 10$.

Finally when $n = 5$ we use the estimate
\[ m_H(w) - m_0(w) \geq \frac{A(n, 1)}{n-1} = \frac{13}{2} \]
to deduce
\[  m_H(w) - m_0(w) - 1 \geq \frac{11}{13}  (m_H(w) - m_0(w)). \]
Now (\ref{eqn-line415}) implies
\[  \frac{N(m_2^b) N(m_0)}{N(m_1^a)^2} \geq \frac{1}{6} \cdot \frac{11}{13}; \]
combining this with (\ref{eqn_c}), we get
\[ \frac{A(n, 2)}{A(n, 1)^2} \geq \frac{1}{6} \cdot \frac{11}{13} \cdot \left (  \frac{45}{46} \right  )^2, \]
which is not true for $n = 5$.

Thus we arrive at a contradiction in every case.
\end{proof}

\begin{lem}
\label{m1a_w-2_small}
If
\[ \frac{N(m_1^a)}{N(m_0)} \geq \frac{A(n, 1)}{2} \]
then $m_H(w-2) \leq 2$.
\end{lem}

\begin{proof}
Assume $m_H(w-2) \geq 3$, and consider
\[ m_7^a = m_0 - 3[w-2] - [w-1] + 4[w]. \]
We will show that $N(m_7^a)$ is too big.

First, note that we must have $n \geq 8$ by Lemma \ref{diff_n_new}.

By hypothesis, we have $\frac{N(m_1^a)}{N(m_0)} \geq \frac{A(n, 1)}{2}$.
From Lemma \ref{m1a_w-1_small}, we know $m_H(w-1) = 1$, so
\[  \frac{N(m_1^a)}{N(m_0)} = \frac{(m_H(w) - m_0(w))}{(m_0(w) + 1)} \]
and thus
\[ (m_H(w) - m_0(w)) \geq A(n, 1). \]
Since $n \geq 8$, this gives
\[ (m_H(w) - m_0(w)) \geq 247. \]

Now 
\[ \frac{N(m_7^a)}{N(m_0)} = C \prod_{i = 1}^{4} \frac { (m_H(w) - m_0(w) - i + 1)} {(m_0(w) + i)}, \]
where
\[ C = \frac{m_H(w-2) (m_H(w-2) - 1) (m_H(w-2) -2) (m_H(w-1))}{6} \geq 1. \]
Using the bounds $(m_H(w) - m_0(w) - 3) > .98 (m_H(w) - m_0(w))$ and
\[  \frac{(m_0(w) + 1)^4}{(m_0(w) + 1)(m_0(w) + 2)(m_0(w) + 3)(m_0(w) + 4)} \geq \frac{2}{15} , \] we deduce that
\[ \frac{N(m_7^a)}{N(m_0)} \geq \frac{.92 A(n, 1)^4}{120} \geq \frac{A(n, 1)^4}{140}.   \] 
The inequality \bluestar[bound_1111_7 verified]
\[ A(n, 7) < \frac{1}{140} A(n, 1)^4 \]  
(see Appendix \ref{sec:eulerian_ineq}, and \texttt{bound-A3c} in the Python code) 
gives a contradiction.
\end{proof}

\begin{lem}
\label{m1a_w+1_small}
If $n \geq 5$ and
\[ \frac{N(m_1^a)}{N(m_0)} \geq \frac{A(n, 1)}{2} \]
then
\[ m_H(w+1) (m_H(w+1) - 1) \leq \frac{48}{11} A(n, 3) \frac{N(m_1^a)}{N(m_0)}  \leq \frac{48}{11} A(n, 3) A(n, 1). \]

Furthermore, if $m_0(w) \geq 2$ then we have the stronger bound
\[ m_H(w+1) (m_H(w+1) - 1) \leq \frac{36}{11} A(n, 3) \frac{N(m_1^a)}{N(m_0)}  \leq \frac{36}{11} A(n, 3) A(n, 1). \]
\end{lem}

\begin{proof}
Let 
\[ m_3^a = m_0 + 2[w+1] - [w] - [w-1]. \]
The result will follow from
\[ \frac{N(m_3^a)}{N(m_0)} \leq A(n, 3). \]

We have
\[ A(n, 3) \geq \frac{N(m_3^a)}{N(m_0)} = \frac{m_H(w+1) (m_H(w+1) - 1) m_H(w-1) m_0(w)}{2 (m_H(w) - m_0(w) + 1)}. \]
On the other hand,
\[  \frac{m_H(w-1) (m_H(w) - m_0(w))}{(m_0(w) + 1)} = \frac{N(m_1^a)}{N(m_0)}  \leq A(n, 1). \] 
We know $m_H(w-1) = 1$ by Lemma \ref{m1a_w-1_small}, so $m_H(w) - m_0(w) \geq A(n, 1) \geq 11$.  
Thus we can estimate
\[ \frac{m_0(w)}{2 (m_H(w) - m_0(w) + 1)} \cdot \frac{(m_H(w) - m_0(w))}{(m_0(w) + 1)} \geq \frac{1}{4} \cdot \frac{11}{12} = \frac{11}{48}. \]
The first bound follows.  

To get the second bound, note that if $m_0(w) \geq 2$ then
\[ \frac{m_0(w)}{m_0(w) + 1} \geq \frac{2}{3}, \]
so
\[ \frac{m_0(w)}{2 (m_H(w) - m_0(w) + 1)} \cdot \frac{(m_H(w) - m_0(w))}{(m_0(w) + 1)} \geq \frac{1}{3} \cdot \frac{11}{12} = \frac{11}{36}. \]
\end{proof}

\begin{lem}
\label{m1a_w+2_small}
If $n \geq 5$ and
\[ \frac{N(m_1^a)}{N(m_0)} \geq \frac{A(n, 1)}{2} \]
then
\[ m_H(w+2) (m_H(w+2) - 1) \leq \frac{48}{11} A(n, 5) \frac{N(m_1^a)}{N(m_0)} \leq \frac{48}{11} A(n, 5) A(n, 1). \]

Furthermore, if $m_0(w) \geq 2$ then we have the stronger bound
\[ m_H(w+2) (m_H(w+2) - 1) \leq \frac{36}{11} A(n, 5) \frac{N(m_1^a)}{N(m_0)} \leq \frac{36}{11} A(n, 5) A(n, 1). \]
\end{lem}

\begin{proof}
Let 
\[ m_5^a = m_0 + 2[w+2] - [w] - [w-1]. \]
The result follows from
\[ \frac{N(m_5^a)}{N(m_0)} \leq A(n, 5); \]
the proof is exactly analogous to the proof of Lemma \ref{m1a_w+1_small}.
\end{proof}

\begin{lem}
\label{m1a_final_bound}
If
\[ \frac{N(m_1^a)}{N(m_0)} \geq \frac{A(n, 1)}{2} \]
then $n \leq 10$.
\end{lem}

\begin{proof}
Assume $n \geq 5$.
We'll estimate each $\frac{N(m_2^*)}{N(m_0)}$, and show that for large $n$ their sum is too small.

We have the following bounds.
By Lemma \ref{m1a_w-2_small}, we have $m_H(w-2) \leq 2$, so
\[ \frac{N(m_2^a)}{N(m_0)} \leq 2 \frac{N(m_1^a)}{N(m_0)}  \leq 2 A(n, 1). \]
Lemma \ref{m1a_w-1_small} tells us that $m_H(w-1) = 1$ (if $n \geq 5$),
so
\[ \frac{N(m_2^b)}{N(m_0)} = 0. \]
By Lemma \ref{m1a_w+1_small}, we know that
\[ m_H(w+1) (m_H(w+1) - 1) \leq \frac{48}{11} A(n, 3) \frac{N(m_1^a)}{N(m_0)} \leq \frac{48}{11} A(n, 3) A(n, 1), \]
so
\[ \frac{N(m_2^c)}{N(m_0)} = m_H(w-1) m_H(w+1) = m_H(w+1) \leq \sqrt{\frac{48}{11} A(n, 3) A(n, 1)} + 1 \]  
and
\begin{eqnarray*} 
\frac{N(m_2^d)}{N(m_0)} & = & \frac{(m_H(w+1)) (m_H(w+1) - 1) (m_0(w)) (m_0(w) - 1)}{2 (m_H(w) - m_0(w) + 1)(m_H(w) - m_0(w) + 2)} \\
& \leq & \frac{1}{2} \left ( \frac{48}{11} A(n, 3) \frac{N(m_1^a)}{N(m_0)} \right ) \left ( \frac{N(m_1^a)}{N(m_0)} \right ) ^{-2} \\
& \leq & \frac{48}{11} \frac{A(n, 3)}{A(n, 1)}.  \\
\end{eqnarray*}
Finally, from Lemma \ref{m1a_w+2_small}, we find
\[ m_H(w+2) \leq  \sqrt{\frac{48}{11} A(n, 5)  \frac{N(m_1^a)}{N(m_0)}} + 1 \]
so
\begin{eqnarray*}
\frac{N(m_2^e)}{N(m_0)} & = & \frac{m_0(w) m_H(w+2)}{m_H(w) - m_0(w) + 1} \\
& \leq & \left(  \frac{N(m_1^a)}{N(m_0)} \right) ^{-1} \left ( \sqrt{\frac{48}{11} A(n, 5)  \frac{N(m_1^a)}{N(m_0)}} + 1 \right) \\
& \leq & \sqrt{ \frac{96}{11} \frac{ A(n, 5)}{A(n, 1)}} + 1.
\end{eqnarray*}

These five bounds combine (see Appendix \ref{sec:eulerian_ineq}, and \texttt{bound-A3d} in the Python code) to give 
\[  \frac{N(m_2^a)}{N(m_0)} + \frac{N(m_2^b)}{N(m_0)} + \frac{N(m_2^c)}{N(m_0)} + \frac{N(m_2^d)}{N(m_0)} + \frac{N(m_2^e)}{N(m_0)} < A(n, 2) \]
for $n \geq 11$, a contradiction. \bluestar[m1a_final verified]
\end{proof}

\begin{lem}
\label{m1a_nomoresmall}
If $5 \leq n \leq 19$ and
\[ \frac{N(m_1^a)}{N(m_0)} \geq \frac{A(n, 1)}{2} \]
then $m_0(w-1) = 1$ and $m_0(w-s) = 0$ for $s \geq 2$.
\end{lem}

\begin{proof}
We have already seen in Lemma \ref{m1a_w-1_small} that $m_0(w-1) = 1$.
Suppose for a contradiction that $m_0(w-s) > 0$ for some $s \geq 2$. 
Let
\[ m_{s+3} =  m_0 + 2 [w+1] - [w-1] - [w-s]. \]
Then
\[ n! \geq A(n, s+3) \geq \frac{N(m_{s+3})}{N(m_0)} \geq \frac{m_H(w+1) (m_H(w+1) - 1)}{2}, \]
so
\[ m_H(w+1) \leq \sqrt{2 n!} + 1. \]

We will use this stronger bound to redo the estimates in Lemma \ref{m1a_final_bound}.

We have
\[ \frac{N(m_2^c)}{N(m_0)} = m_H(w+1) \leq \sqrt{2 n!} + 1\]
and
\begin{eqnarray*} 
\frac{N(m_2^d)}{N(m_0)} & = & \frac{(m_H(w+1)) (m_H(w+1) - 1) (m_0(w)) (m_0(w) - 1)}{2 (m_H(w) - m_0(w) + 1)(m_H(w) - m_0(w) + 2)} \\
& \leq & \frac{1}{2}  m_H(w+1)^2 \left ( \frac{N(m_1^a)}{N(m_0)} \right ) ^{-2} \\
& \leq & \frac{1}{2} \left ( \sqrt{2 n!} + 1  \right  )^2  \left (  \frac{A(n, 1)}{2} \right  ) ^{-2}.\\
\end{eqnarray*}

We use the same bounds on
\[ \frac{N(m_2^a)}{N(m_0)}, \frac{N(m_2^b)}{N(m_0)}, \frac{N(m_2^e)}{N(m_0)} \]
as in Lemma \ref{m1a_final_bound};
we conclude (see \texttt{bound-A3e} in the Python code) that  
\[  \frac{N(m_2^a)}{N(m_0)} + \frac{N(m_2^b)}{N(m_0)} + \frac{N(m_2^c)}{N(m_0)} + \frac{N(m_2^d)}{N(m_0)} + \frac{N(m_2^e)}{N(m_0)} < A(n, 2) \]
for $6 \leq n \leq 19$. \bluestar[m1a_nomoresmall verified]

Finally, suppose $n = 5$.
Then since $s+3 \geq 4$, we have the much stronger bound $A(n, s+3) \leq 1$,
which implies $m_H(w+1) \leq 2$.
We deduce as above
\[ \frac{N(m_2^c)}{N(m_0)} = m_H(w+1) \leq 2 \]
\[ \frac{N(m_2^d)}{N(m_0)} \leq \left ( \frac{N(m_1^a)}{N(m_0)} \right ) ^{-2} < 1 \]
and arrive at a contradiction as before.
(See \texttt{bound-A3f} in the Python code.) 
\bluestar[m1a_nomoresmall verified]
\end{proof}

\begin{lem}
\label{m1a_w_new}
If $n \geq 5$ and
\[ \frac{N(m_1^a)}{N(m_0)} \geq \frac{A(n, 1)}{2} \]
then $m_0(w) = 1$.
\end{lem}

\begin{proof}
Suppose $m_0(w) \geq 2$.  Again, we'll redo the estimates in Lemma \ref{m1a_final_bound}.

Since $n \geq 5$ and (by Lemma \ref{m1a_final_bound}) $n \leq 10$, Lemma \ref{m1a_nomoresmall} implies that
\[ \frac{N(m_2^a)}{N(m_0)} = \frac{N(m_2^b)}{N(m_0)} = 0. \]

Lemmas \ref{m1a_w+1_small} and \ref{m1a_w+2_small} give stronger bounds in case $m_0(w) \geq 2$.
As in Lemma \ref{m1a_final_bound}, we deduce in this case that:
\[ \frac{N(m_2^c)}{N(m_0)} \leq \sqrt{\frac{36}{11} A(n, 3) A(n, 1)} + 1 \] 
\[ \frac{N(m_2^d)}{N(m_0)} \leq \frac{36}{11} \frac{A(n, 3)}{A(n, 1)} \]
\[ \frac{N(m_2^e)}{N(m_0)} \leq \sqrt{ \frac{72}{11} \frac{ A(n, 5)}{A(n, 1)}} + 1. \]
Again, we conclude for $n$ in the given range that 
\[  \frac{N(m_2^a)}{N(m_0)} + \frac{N(m_2^b)}{N(m_0)} + \frac{N(m_2^c)}{N(m_0)} + \frac{N(m_2^d)}{N(m_0)} + \frac{N(m_2^e)}{N(m_0)} < A(n, 2), \]
a contradiction.
(See Appendix \ref{sec:eulerian_ineq}, and \texttt{bound-A3g} in the Python code.) 
\bluestar[m1a_w_new verified]
\end{proof}

\begin{lem}
\label{m1a_final_final}
We cannot have $n \geq 5$ and
\[ \frac{N(m_1^a)}{N(m_0)} \geq \frac{A(n, 1)}{2}. \]
\end{lem}

\begin{proof}
Again, we'll redo the estimates in Lemma \ref{m1a_final_bound}, in light of everything we now know.
By Lemma \ref{m1a_final_bound}, we may assume that $n \leq 11$.

Now Lemma \ref{m1a_nomoresmall} implies that
\[ \frac{N(m_2^a)}{N(m_0)} = \frac{N(m_2^b)}{N(m_0)} = 0, \]
while Lemma \ref{m1a_w_new} gives us
\[ \frac{N(m_2^d)}{N(m_0)} = 0. \]

Yet again (see proof of Lemma \ref{m1a_final_bound}) we have the bounds
\[ \frac{N(m_2^c)}{N(m_0)} \leq \sqrt{\frac{48}{11} A(n, 3) A(n, 1)} + 1 \] 
and
\[ \frac{N(m_2^e)}{N(m_0)} \leq \sqrt{ \frac{96}{11} \frac{ A(n, 5)}{A(n, 1)}} + 1. \]
Yet again, we conclude that
\[  \frac{N(m_2^a)}{N(m_0)} + \frac{N(m_2^b)}{N(m_0)} + \frac{N(m_2^c)}{N(m_0)} + \frac{N(m_2^d)}{N(m_0)} + \frac{N(m_2^e)}{N(m_0)} < A(n, 2), \]
a contradiction.
(See Appendix \ref{sec:eulerian_ineq}, and \texttt{bound-A3h} in the Python code.) 
\bluestar[m1a_final_final verified]
\end{proof}

\subsection{The case $n \geq 5$, with $k \geq n$}

Recall notation from the beginning of Section \ref{combinatorics}, and the beginning of Section \ref{n5_setup}.

We'll treat the case where $k \geq n$ next.
By Lemma \ref{diff_n_new}, if $k \geq n$, then $w = w'$.

\begin{lem}
\label{bigk_twocases}
If $k \geq n$ and $n \geq 5$ then one of the two ratios
\[ \frac{N(m_1^a)}{N(m_0)}, \frac{N(m_1^b)}{N(m_0)}  \]
is less than 1, and the other is greater than $A(n, 1) - 1$.
\end{lem}

\begin{proof}
From Lemma \ref{diff_n_new}, we have
\[ m_H(w-1) + m_H(w+1) \leq n-1 \]
so
\[ m_H(w-1) m_H(w+1) \leq \frac{(n-1)^2}{2}. \]
Then apply Lemma \ref{alpha_beta}, and the inequality
\[ \frac{(n-1)^2}{2} < A(n, 1) - 1 \]
(see Appendix \ref{sec:eulerian_ineq}).
\end{proof}

\begin{lem}
\label{bigk_w-1}
If $n \geq 5$ and
\[ \frac{N(m_1^a)}{N(m_0)} > A(n, 1) - 1 \]
then $m_H(w-1) \leq 1$.

If $n \geq 5$ and
\[ \frac{N(m_1^b)}{N(m_0)} > A(n, 1) - 1 \]
then $m_H(w+1) \leq 1$.
\end{lem}

\begin{proof}
The first case follows from Lemma \ref{m1a_w-1_small}; we'll prove the second.
(As an alternative to Lemma \ref{m1a_w-1_small}, the first case could be proven by an argument analagous to the argument below.)

So suppose $\frac{N(m_1^b)}{N(m_0)} > A(n, 1) - 1$ and $m_H(w+1) \geq 2$,
and consider
\[ \frac{N(m_2^d)}{N(m_1^b)} = \frac{(m_H(w+1) - 1)  (m_0(w) - 1)}{2 (m_H(w) - m_0(w) + 2)}. \]

\[ \frac{N(m_2^d) N(m_0)}{N(m_1^b)^2} = 
\frac{1}{2} \cdot \frac{(m_H(w+1) - 1)}{m_H(w+1)} \cdot \frac{(m_H(w) - m_0(w) + 1)}{(m_H(w) - m_0(w) + 2)} \cdot \frac{(m_0(w) - 1)}{m_0(w)}.\]

From 
\[ \frac{(m_0(w))}{(m_H(w) - m_0(w) + 1)} = \frac{1}{m_H(w+1)} \cdot \frac{N(m_1^b)}{N(m_0)} \geq \frac{A(n, 1) - 1}{n-1} > 6 \]
we deduce that $m_0(w) > 6$, so by integrality $m_0(w) \geq 7$, so
\[ m_0(w) - 1 \geq \frac{6}{7} m_0(w). \]
We also know that $(m_H(w) - m_0(w) + 2) \leq 2 (m_H(w) - m_0(w) + 1)$ and $(m_H(w+1) - 1) \geq \frac{1}{2} m_H(w+1)$,
so
\[ \frac{A(n, 2)}{(A(n, 1) - 1)^2} \geq \frac{N(m_2^d) N(m_0)}{N(m_1^b)^2} \geq \frac{1}{2} \cdot \frac{1}{2}  \cdot \frac{1}{2} \cdot \frac{6}{7} = \frac{3}{28}. \]
For $n \geq 5$, this contradicts
\[   A(n, 2)< \frac{3}{28} ( A(n, 1) - 1  )^2.  \] \bluestar[bound-big-k-2-11 verified]
(See Appendix \ref{sec:eulerian_ineq} and \texttt{bound-A4a} in the Python code.) 
\end{proof}

\begin{lem}
\label{k_small}
For $n \geq 5$ we cannot have $k \geq n$.
\end{lem}

\begin{proof}
Assume $k \geq n$.

We'll consider the two cases given in Lemma \ref{bigk_twocases}.
The first case is ruled out by Lemma \ref{m1a_final_final}; we'll prove the second.
(Alternatively, the first case could be proven by an argument analagous to the argument below.)

So, suppose 
\[ \frac{N(m_1^b)}{N(m_0)} > A(n, 1) - 1. \]
From
\[ \frac{m_0(w)}{m_H(w) - m_0(w) + 1} > A(n, 1) - 1 \]
and the bounds
\[ m_H(w-2), m_H(w-1) \leq n-1 < A(n, 1) - 1 \]
we deduce that
\[ \frac{N(m_2^a)}{N(m_0)} < m_H(w-2) \left( \frac{m_0(w)}{m_H(w) - m_0(w) + 1} \right)^{-1} < 1\]
and
\[ \frac{N(m_2^b)}{N(m_0)} < \frac{1}{2} m_H(w-1)^2 \left( \frac{m_0(w)}{m_H(w) - m_0(w) + 1} \right)^{-2} < 1.\]
By Lemma \ref{bigk_w-1}, we have $m_H(w+1) \leq 1$, so 
\[ \frac{N(m_2^c)}{N(m_0)} = m_H(w-1) m_H(w+1) \leq n-1. \]
Also, $m_H(w+1) \leq 1$ implies that
\[ \frac{N(m_2^d)}{N(m_0)} = 0. \]
Finally, Lemma \ref{diff_n_new} gives $m_H(w+2) \leq \frac{n-1}{2}$, so
\[ \frac{N(m_2^e)}{N(m_0)} \leq \frac{(n-1)}{2} A(n, 1). \]
Adding these bounds, we see that
\[ A(n, 2) = \sum_{*  = a, b, c, d, e}   \frac{N(m_2^*)}{N(m_0)} \leq n + 1 + \frac{(n-1)}{2} A(n, 1).\]
But this contradicts the inequality
\[ n + 1 + \frac{(n-1)}{2} A(n, 1) < A(n, 2), \]
which holds for $n \geq 5$ (see Appendix \ref{sec:eulerian_ineq} and \texttt{bound-A4b} in the Python code). 
\end{proof}

\subsection{The case $n \geq 5$, with $N(m_1^b)$ big and $k < n$.}

Recall notation from the beginning of Section \ref{combinatorics}, and the beginning of Section \ref{n5_setup}.

\begin{lem}
\label{m1b_w+1}
Suppose $k \leq n-1$ and
\[ \frac{N(m_1^b)}{N(m_0)} \geq A(n, 1)/2. \]
Then
\[ m_H(w+1) \geq \frac{A(n, 1)}{2(n-1)}. \]
In particular, 
if $n \geq 5$ then
\[ m_H(w+1) \geq 4. \]
\end{lem}

\begin{proof}
We have
\[ m_H(w+1) \geq \frac{N(m_1^b)}{N(m_0)} \cdot \frac{1}{m_0(w)} \geq \frac{A(n, 1)}{2 m_0(w)} \geq \frac{A(n, 1)}{2(n-1)}. \] \bluestar[m1b_w+1 -- Python not needed]

The ``in particular'' follows from the fact that $m_H(w+1)$ is an integer.
\end{proof}

\begin{lem}
\label{m1b_mult1a}
Suppose $n \geq 5$, $k \leq n-1$, and
\[ \frac{N(m_1^b)}{N(m_0)} \geq A(n, 1)/2. \]
Then we cannot have simultaneously $\frac{N(m_1^a)}{N(m_0)} > 0$ and $m_0(w) \geq 2$.
\end{lem}

\begin{proof}
Suppose $\frac{N(m_1^a)}{N(m_0)} > 0$ and $m_0(w) \geq 2$; the first condition implies $m_H(w-1) \geq 1$ and $m_H(w) - m_0(w) \geq 1$.  Consider
\[ m_4^b = m_0 - [w-1] - 2[w] + 3[w+1].  \]
We have
\[ \frac{N(m_4^b)}{N(m_1^b)} = \frac{m_H(w-1) (m_H(w+1) -1) (m_H(w+1) -2) (m_0(w) - 1)}{6 (m_H(w) - m_0(w) + 2)}. \]

By Lemma \ref{alpha_beta} we have
\[ \left ( \frac{N(m_1^a)}{N(m_0)} \right ) \left ( \frac{N(m_1^b)}{N(m_0)} \right ) < m_H(w-1)m_H(w+1).\]

Thus we obtain
\[ \frac{N(m_4^b)}{N(m_1^b)} > \frac{1}{6} \left ( \frac{N(m_1^a)}{N(m_0)} \right ) \left ( \frac{N(m_1^b)}{N(m_0)} \right )^2 \frac{(m_H(w+1) - 1)(m_H(w+1) - 2)}{m_H(w+1)^2} \frac{m_H(w) - m_0(w) + 1}{m_H(w) - m_0(w) + 2} \frac{m_0(w) - 1}{m_0(w)}. \]
Using $m_H(w+1) \geq 4$, $m_H(w) - m_0(w) \geq 1$, and $m_0(w) \geq 2$, the three fractional factors on the right can be bounded below by $\frac{3}{8}$, $\frac{2}{3}$, and $\frac{1}{2}$, respectively.

On the other hand, as soon as $N(m_1^a)$ is nonzero, we have
\[ \frac{N(m_1^a)}{N(m_0)} = \frac{ m_H (w-1)   ( m_H (w) - m_0 (w) ) }{ m_0 (w)+1 } \geq \frac{1}{n}. \]
We conclude that
\[ \frac{N(m_4^b)}{N(m_1^b)} > \frac{1}{6} \left ( \frac{N(m_1^a)}{N(m_0)} \right ) \left ( \frac{N(m_1^b)}{N(m_0)} \right )^2 \frac{3}{8} \frac{2}{3} \frac{1}{2} = \frac{1}{48} \left ( \frac{N(m_1^a)}{N(m_0)} \right ) \left ( A(n, 1) - \frac{N(m_1^a)}{N(m_0)} \right )^2 \]
so
\[ \frac{N(m_4^b)}{N(m_0)} > \frac{1}{48}  \left ( \frac{N(m_1^a)}{N(m_0)} \right ) \left ( A(n, 1) - \frac{N(m_1^a)}{N(m_0)} \right )^3  \]
As a function of $\frac{N(m_1^a)}{N(m_0)}$,
this right-hand side is minimized when $\frac{N(m_1^a)}{N(m_0)} = \frac{1}{n}$, so we have
\[ A(n, 4) \geq \frac{N(m_4^b)}{N(m_1^b)} > \frac{1}{48n} \left ( A(n, 1) - \frac{1}{n} \right )^3.\]
This contradicts the inequality
\[ A(n, 4) < \frac{1}{78 n} (A(n, 1) - 1)^3, \]
which is valid for all $n$.
(See Appendix \ref{sec:eulerian_ineq} and \texttt{bound-A5a} in the Python code.) 
\bluestar[m1b_mult1a]
\end{proof}

\begin{lem}
\label{m1b_mult1b}
If $n \geq 5$, $\frac{N(m_1^a)}{N(m_0)} = 0$, $k \leq n - 1$, and 
\[ \frac{N(m_1^b)}{N(m_0)} \geq A(n, 1)/2, \]
then $m_0(w) = 1$ and $N(m_2^d) = 0$. 
\end{lem}

\begin{proof}
Since $N(m_1^a) = 0$, we have
\[ \frac{N(m_1^b)}{N(m_0)} = A(n, 1). \]

Assuming $m_0(w) \geq 2$, let's look at $m_2^d$.  We have
\[ \frac{N(m_2^d)}{N(m_0)} = \frac{(m_H(w+1)) (m_H(w+1) - 1) (m_0(w)) (m_0(w) - 1)}{2 (m_H(w) - m_0(w) + 1)(m_H(w) - m_0(w) + 2)}. \]
We want to compare this with the inequality
\[ \frac{N(m_2^d)}{N(m_0)} \leq A(n, 2) < \frac{1}{10} A(n, 1)^2, \]
which is valid for all $n$.
(See Appendix \ref{sec:eulerian_ineq} and \texttt{bound-A5b} in the Python code.) 

We have certainly
\[ \frac{N(m_2^d)}{N(m_0)} \geq \frac{1}{8} \cdot \frac{(m_H(w+1) - 1)}{m_H(w+1)} \left ( \frac{N(m_1^b)}{N(m_0)} \right )^2 = \frac{1}{8} \cdot \frac{(m_H(w+1) - 1)}{m_H(w+1)} A(n, 1)^2. \]
So we will be done if $m_H(w+1) \geq 5$.

But as in the proof of Lemma \ref{m1b_w+1}, we find (using now $\frac{N(m_1^b)}{N(m_0)} = A(n, 1)$ in place of a weaker bound) that
\[ m_H(w+1) \geq \frac{A(n, 1)}{(n - 1)}, \]
and so in particular $m_H(w+1) > 5$ whenever $n \geq 5$.
\end{proof}

\begin{lem}
\label{m1b_mult1}
If $n \geq 5$, $k \leq n - 1$, and 
\[ \frac{N(m_1^b)}{N(m_0)} \geq A(n, 1)/2, \]
then $m_0(w) = 1$.

In particular, there must be some $r > 0$ with $m_0(w - r) > 0$.
\end{lem}

\begin{proof}
That $m_0(w) = 1$ follows from Lemmas \ref{m1b_mult1a} and \ref{m1b_mult1b};
the second claim follows because $k > 1$.
\end{proof}

Let $w -  r$ be the highest Hodge weight below $w$.
In other words, take $r > 0$ minimal such that $m_H(w - r) > 0$.
(Such $r$ exists by Lemma \ref{m1b_mult1}.)
For example, if $N(m_1^a) \neq 0$ then we must have $m_H(w-1) \neq 0$, so $r = 1$.

\subsubsection{Case: $m_1^b$ big and $r = 1$.}

Recall the definitions of $m_2^*$ from above ($* = a, b, c, d, e$).

\begin{lem}
\label{m1b_r1_w-1}
If $k \leq n-1$,
\[ \frac{N(m_1^b)}{N(m_0)} \geq A(n, 1)/2, \]
and $m_0(w) = 1$, then $m_0(w-1) \leq 1$.
\end{lem}

\begin{proof}
Suppose for a contradiction $m_0(w-1) \geq 2$,
and consider
\[ m_5^b = m_0 + 3[w+1] - [w] - 2[w-1]. \]
We have
\begin{eqnarray*}
\frac{N(m_5^b)}{N(m_1^b)} & = & \left ( \frac{m_H(w-1) (m_H(w-1) - 1)}{2} \right) \left ( \frac{(m_H(w+1) - 1)(m_H(w+1) - 2)}{6} \right ) \\
& \geq & \frac{(m_H(w+1) - 1)(m_H(w+1) - 2)}{6}. 
\end{eqnarray*}

From 
\[ \frac{N(m_1^b)}{N(m_0)} = \frac{m_H(w+1) m_0(w)}{m_H(w) - m_0(w) + 1}  \]
and $m_0(w) = 1$, we find that
\[ m_H(w+1) \geq A(n, 1) / 2. \]
In particular, since $n \geq 5$, we have $m_H(w+1) \geq 13$, so certainly
\[ (m_H(w+1) - 1)(m_H(w+1) - 2) \geq \frac{1}{2} m_H(w+1)^2. \]

We conclude that
\[ \frac{N(m_5^b)}{N(m_0)} = \left ( \frac{N(m_5^b)}{N(m_1^b)} \right ) \left ( \frac{N(m_1^b)}{N(m_0)} \right ) \geq \frac{1}{12} m_H(w+1)^2 \frac{A(n, 1)}{2} \geq \frac{A(n, 1)^3}{96}. \]
This contradicts the inequality
\[ A(n, 1)^3 > 341 A(n, 5), \]
which is valid for all $n$.
(See Appendix \ref{sec:eulerian_ineq} and \texttt{bound-A5c} in the Python code.) 
\bluestar[m1b_r1_w_1mod verified]
\end{proof}

\begin{lem}
\label{m1b_r1_final}
We cannot have $n \geq 5$, $k \leq n-1$, 
\[ \frac{N(m_1^b)}{N(m_0)} \geq A(n, 1)/2, \]
and $r = 1$.
\end{lem}

\begin{proof}
Assume the stated conditions hold.  We will bound the five ratios $\frac{N(m_2^*)}{N(m_0)}$, and show that their sum is too small.
Lemmas \ref{m1b_mult1} and \ref{m1b_r1_w-1} tell us that $m_0(w-1) = m_0(w) = 1$.
Hence, $N(m_2^b) = N(m_2^d) = 0$.

We have 
\[ \frac{N(m_2^a)}{N(m_0)} = \frac{m_0(w-2) (m_H(w) - 1)}{2} \]
\[ \frac{N(m_2^c)}{N(m_0)} = m_H(w+1) \]
\[ \frac{N(m_2^e)}{N(m_0)} = \frac{m_H(w+2)}{m_H(w)}. \]

Since
\[ \frac{A(n, 1)}{2} \geq \frac{N(m_1^a)}{N(m_0)} = \frac{m_H(w )- 1}{2}  \]
and $m_0(w-2) \leq n - 1$,
we have
\[ \frac{N(m_2^a)}{N(m_0)} \leq \frac{(n-1) A(n, 1)}{2}. \]

Let
\[ m_3^b = m_0 + 2[w+1] - [w] - [w-1]. \]
Dividing the bound
\[ A(n, 3) \geq \frac{N(m_3^{b})}{N(m_0)} = \frac{(m_H(w+1))(m_H(w+1)-1)}{2(m_H(w) -  m_0(w)  + 1)}. \]
by
\[ \frac{N(m_1^b)}{N(m_0)} =  \frac{m_H(w+1)m_0(w)}{(m_H(w) - m_0(w) + 1)} \geq \frac{A(n, 1)}{2}, \]
we obtain
\[ (m_H(w+1)-1) \leq \frac{4 A(n, 3)}{A(n, 1)}, \]
so
\[ \frac{N(m_2^c)}{N(m_0)} \leq \frac{4 A(n, 3)}{A(n, 1)} + 1. \]

Finally, let
\[ m_5^c = m_0 + 2[w+2] - [w] - [w-1]. \]
We have
\[ A(n, 5) \geq \frac{N(m_5^c)}{N(m_0)} = \frac{m_H(w+2) (m_H(w+2) - 1)}{2 m_H(w)} \geq \frac{1}{2} \left ( \frac{N(m_2^e)}{N(m_0)} - 1 \right )^2,  \]
so
\[  \frac{N(m_2^e)}{N(m_0)} \leq \sqrt{2 A(n, 5)} + 1.  \]

We conclude that
\[ A(n, 2) = \sum_{*  = a, b, c} \frac{N(m_2^*)}{N(m_0)} \leq \left ( \frac{n A(n, 1)}{2} \right ) + \left  ( \frac{4 A(n, 3)}{A(n, 1)} + 1 \right ) + \left ( \sqrt{2 A(n, 5)} + 1 \right ), \]
which contradicts the inequality
\[ A(n, 2) > \left ( \frac{(n-1) A(n, 1)}{2} \right ) + \left  ( \frac{4 A(n, 3)}{A(n, 1)} + 1 \right ) + \left ( \sqrt{2 A(n, 5)} + 1 \right ), \]
valid for all $n \geq 5$.
(See Appendix \ref{sec:eulerian_ineq} and \texttt{bound-A5d} in the Python code.) 
\bluestar[m1b_r1_final_mod verified]
\end{proof}

\subsubsection{Case: $N(m_1^b)$ big and $r \geq 2$.}

Suppose $n \geq 5$ and
\[ \frac{N(m_1^b)}{N(m_0)} \geq A(n, 1)/2. \]
From Lemma \ref{m1b_mult1}, we have $m_0(w) = 1$, and because $r\geq 2$, we have $m_0(w-1) = 0$.

\begin{lem}

Suppose $r \geq 2$, $m_0(w) = 1$, and
\[ \frac{N(m_1^b)}{N(m_0)} \geq A(n, 1)/2. \]

We have
\[ m_0(w) = 1,\]
\[ m_H(w + i) = d A(n, i) \]
for $0 \leq i \leq r-1$, and
\[ m_H(w + r) = d A(n, r) - \frac{d (d-1)}{2} m_H(w-r). \]
\end{lem}

\begin{proof}
We have already seen in Lemma \ref{m1b_mult1} that $m_0(w) = 1$.

For $w \leq q \leq w+r-1$, there is only one nonzero term in Equation \eqref{eq:wedge-Euler};
for $q = w+r$, there are only two.
\end{proof}

\begin{lem}
\label{r_big_lem1}
Suppose $n \geq 5$, $r \geq 2$, $m_0(w) = 1$, and
\[ \frac{N(m_1^b)}{N(m_0)} \geq A(n, 1)/2. \]
Then
\[ A(n, 3r-2) \geq \frac{A(n, r-1) (d A(n, r-1) -1)}{2}.\]
\end{lem}

\begin{proof}
Let
\[ m_{3r - 2} = m_0 + 2[w + r -  1] - [w] - [w -  r]. \]
We have
\[ A(n, 3r-2) \geq \frac{N(m_{3r - 2})}{d} = \frac{(d A(n, r - 1)) (d A(n, r-1) - 1) m_H(w-r)}{2 d (m_H(w) -  m_0(w) + 1)} \geq \frac{A(n, r-1) (d A(n, r-1) -1)}{2}.\]
\end{proof}

\begin{lem}
\label{r_big_lem2}
Suppose $n \geq 5$, $r \geq 2$, $m_0(w) = 1$, and
\[ \frac{N(m_1^b)}{N(m_0)} \geq A(n, 1)/2. \]
Then $r = 2$ and 
\[ d < \frac{3 A(n, 4)}{A(n, 1)^2}. \]
\end{lem}

\begin{proof}
By Lemma \ref{r_big_lem1}, we have
\[ A(n, 3r-2) \geq \frac{A(n, r-1) (d A(n, r-1) -1)}{2} \geq  \frac{A(n, r-1) ( A(n, r-1) -1)}{2}.\]
If $r \geq 3$, this is not possible by Lemma \ref{bound_rr_3r}. \bluestar[]

If $r = 2$ the result follows from Lemma \ref{r_big_lem1} and
\[ d A(n, r-1) -1 > \frac{2}{3} d A(n, r-1). \]
\end{proof}

\begin{lem}
\label{r_big_lem3}
We cannot have $n\geq 5$, $k\leq n-1$, $r \geq 2$, and $m_0(w) = 1$.
\end{lem}

\begin{proof}
Assume we had a solution satisfying the given conditions.
By Lemma \ref{r_big_lem2} we know that $r = 2$ and
\[ d < \frac{3 A(n, 4)}{A(n, 1)^2}. \]

This implies
\[ d < \frac{A(n, 2)}{2n}, \] 
by the inequality
\[  6n A(n, 4) < A(n, 1)^2 A(n, 2), \]
which holds for all $n \geq 3$. \bluestar[bound_112_4 verified, in fact it seems to hold with 6 replaced by 2000]
(See Appendix \ref{sec:eulerian_ineq} and \texttt{bound-A5e} in the Python code.) 
Thus, since $m_H(w-2) < n$, we have
\[ m_H(w+2) = d A(n, 2) - \frac{d (d-1)}{2} m_H(w-2) \geq \frac{3}{4} d A(n, 2). \]

Finally, taking
\[ m_6^a = m_0 + 2[w+2] - [w] - [w-2], \]
we find that
\begin{eqnarray*}
A(n, 6) & \geq & \frac{N(m_6^a)}{N(m_0)}  \\
& \geq &  \frac{m_H(w+2) (m_H(w+2) - 1) m_0(w) m_H(w-2)}{2 (m_H(w) - m_0(w) + 1)} \\
& = & \frac{m_H(w+2) (m_H(w+2) - 1) m_H(w-2)}{2d}\\ 
& \geq & \frac{1}{2} \cdot \frac{3}{4} A(n, 2) ( \frac{3}{4} d A(n, 2) - 1) \\
&  \geq & \frac{9}{32} (A(n, 2) - 1)^2.
\end{eqnarray*}
This contradicts the inequality 
\[ (A(n, 2) - 1)^2 > 100 A(n, 6), \]
which holds for all $n$.
(See Appendix \ref{sec:eulerian_ineq} and \texttt{bound-A5f} in the Python code.) 
\bluestar[bound_22_6a verified, in fact holds with 100 replaced by 500]  
\end{proof}

\newpage

\section{Collected inequalities involving Eulerian numbers}
\label{sec:eulerian_ineq}

The argument in Section \ref{combinatorics} used several dozen inequalities involving Eulerian numbers.
We will not give detailed proofs of them; aside from the inequality in Lemma \ref{bound_rr_3r},
each of the inequalities used can be proven using Lemma \ref{euler_asymptotic_bound} or Lemma \ref{patching_asymptotics} for large $n$,
and then verifying by hand the finite number of remaining cases.
The Python code used to verify these remaining cases has been posted as an ancillary file alongside the arXiv submission.
One proof is presented as Lemma \ref{ineq_example} to illustrate the method.

The reader is encouraged to verify the plausibility of such inequalities for large $n$,
using the asymptotic approximation $A(n, q) \sim (q+1)^n$,
which is valid for fixed $q$ and large $n$.

Recall our convention that $A(n, q) = 0$ if $q \geq n$.

\begin{lem}
\label{euler_asymptotic_bound}
For all $n \geq 1$ and $q \geq 0$, we have
\[ (q+1)^n - (n+1)q^n \leq A(n, q) \leq (q+1)^n. \]
\end{lem}

\begin{proof}
Recall (\cite[\S 1.3]{Petersen}) that the Eulerian number $A(n, q)$ counts the number of permutations of $\{1, 2, \ldots, n\}$ with exactly $q$ ascents.

If we label the integers $1$ to $n$ with labels $1$ through $q+1$, then we get a permutation with at most $q$ ascents by giving all the numbers of label $1$ in decreasing order, then all the numbers of label $2$ in decreasing order, and so on.  Every permutation with at most $q$ ascents arises in this way; this proves the right-hand inequality.

If a permutation constructed this way has fewer than $q$ ascents, then there must exist adjacent labels $i$ and $i+1$ where all the numbers with label $i$ are less than all the numbers with label $i+1$.
If that happens, we can record a number $j$ from $0$ to $n$ which is the number of elements in the sequence with label at most $i$, then subtract one from the labeling of everything with label greater than $i$. There are $(n+1) q^n$ possibilities of this new data, and we can recover the original labeling by adding $1$ to the label of everything that comes after the first $j$ elements in the sequence. So the number of labelings giving permutations with fewer than $q$ ascents is at most $(n+1) q^n$. 
This proves the left-hand inequality.
\end{proof}

\begin{lem}
\label{euler_factorial_bound}
For all $n \geq 1$ and $q \geq 0$, we have
\[ A(n, q) \leq n!. \]
\end{lem}

\begin{proof}
Follows from
\[ \sum_q A(n, q) = n!. \]
\end{proof}

The two bounds given patch well enough for our modest needs:
if $q \sim n / \log n$, then the bound in Lemma \ref{euler_asymptotic_bound} is close to $n!$,
at least in a power sense.
Surely more precise asymptotics are known, but
these weak bounds are enough for 
the proof of Lemma \ref{bound_rr_3r}.

\begin{lem}
\label{patching_asymptotics}
If $n \geq 2$ and 
\[ q < \frac{n} {log(n+1) + 1} - 1 \]
then
\[ (1 - 1/e) (q+1)^n \leq A(n, q) \leq (q+1)^n.  \]
\end{lem}

\begin{proof}
We have
\[ 1/(q+1) < \log \left ( \frac{q+1} {q} \right ) < 1/q \]
\[ e^{n/(q+1)} < \left ( \frac{q+1}{q} \right )^n < e^{n/q}. \]
If $n/{(q + 1)} > \log (n + 1) + 1$ then 
\[ (n+1) q^n < (q+1)^n / e. \]
Now use Lemma \ref{euler_asymptotic_bound}.
\end{proof}

\begin{lem}
\label{bound_rr_3r}
For $n$ arbitrary and $3 \leq r \leq n-2$ we have
\[ A(n, r-1) (A(n, r-1) - 1) > 2 A(n, 3r-2). \]

\end{lem}

\begin{proof}
We'll assume $n \geq 21$; for smaller $n$ there are only finitely many cases, which can be checked by hand.
(See \texttt{bound-B} in the Python code.)  
\bluestar[eulerian.py bound_rr_3r verified]

We can also assume $r \leq n/2$; otherwise, the right-hand side is zero.
We'll split into two cases: either $r < \frac{n} {log(n+1) + 1}$ or $r > \sqrt{n}$.
Our hypothesis on $n$ guarantees that there is some $r_0$ with $\sqrt{n} \leq r_0 < \frac{n} {log(n+1) + 1}$;
a fortiori, one of the two cases always holds.

If $r < \frac{n} {log(n+1) + 1}$, then Lemma \ref{patching_asymptotics} gives
\[ A(n, r-1) \geq (1 - 1/e) r^n, \]
so 
\[ A(n, r-1) (A(n, r-1) - 1) \geq (1 - 1/e)^2 r^{2n} - (1 - 1/e) r^n. \]
On the other hand, by Lemma \ref{euler_asymptotic_bound}, we can bound the right-hand side:
\[ 2 A(n, 3r-2) \leq 2 (3r-1)^n. \]  
The reader may verify that
\[ (1 - 1/e)^2 r^{2n} > 2 (3r-1)^n + (1 - 1/e) r^n \]
whenever $r \geq 3$ and $n \geq 14$; this proves the inequality we want.

Recall that we chose $r_0$ such that $\sqrt{n} \leq r_0 < \frac{n} {log(n+1) + 1}$.
Lemma \ref{patching_asymptotics} gives
\[ A(n, r_0 - 1) (A(n, r_0 - 1) - 1) > \frac{3}{10} r_0^{2n} \geq \frac{3}{10} n^n > 2 * n!.  \]
Now for any $r$ with $r_0 \leq r \leq n/2$, we have 
\[ A(n, r) \geq A(n, r_0), \]
so
\[ A(n, r - 1) (A(n, r - 1) - 1) > 2 * n! \geq 2 A(n, 3r-2) \]
by Lemma \ref{euler_factorial_bound}.
\end{proof}

We'll conclude with Lemma \ref{ineq_example}, whose proof is given merely to illustrate a routine technique.
A number of inequalities were used without proof in Appendix \ref{combinatorics}. 
They can all be proven by asymptotic estimates using Lemma \ref{patching_asymptotics} for large $n$, 
followed case-by-case verification for small $n$.
The following bound was used in the proof of Lemma \ref{m1a_w-1_small};
we give a full proof here to illustrate the general method.

\begin{lem}
\label{ineq_example}
For $n \geq 11$, we have
\[ \frac{A(n, 2)}{A(n, 1)^2} < 1/27. \]
\end{lem}

\begin{proof}
Lemma \ref{patching_asymptotics} gives us
\[ A(n, 2) \leq 3^n \]
and
\[ A(n, 1) \geq (1 - 1/e) 2^n, \]
provided $n \geq 6$.
Hence we will be done as soon as we can show that
\[ \frac{1}{(1-1/e)^2} \cdot \left  ( \frac{3}{4}  \right ) ^n < 1/27, \]
which happens for $n \geq 15$.
The cases $11 \leq n \leq 14$ must be checked separately.
\end{proof}


\newpage

\begin{thebibliography}{9}

\bibitem{Andre} Yves Andr\'{e}. \emph{Mumford-Tate groups of mixed Hodge structures
and the theorem of the fixed part}. {Comp.\ Math.}\  82 (1992), 1--24.

\bibitem{Andre_Sha} Yves Andr\'{e}. \emph{On the Shafarevich and Tate conjectures for hyperk\"ahler varieties}.  Math.\ Ann.\ 305, 205--248 (1996).

\bibitem{BT} Benjamin Bakker and Jacob Tsimerman.  \emph{The Ax-Schanuel conjecture for variations of Hodge structures}.  Invent.\ Math.\ 217 (2019), 77--94.

\bibitem{BHMR} Michael Bate, Sebastian Herpel, Benjamin Martin, and Gerhard R\"ohrle.  \emph{Cocharacter-closure and the rational Hilbert-Mumford Theorem}. Math.\ Zeit.\ 287 (2017), 39--72.

\bibitem{BMR} Michael Bate, Benjamin Martin, and Gerhard R\"ohrle.  \emph{A geometric approach to complete reducibility}.  Invent.\ Math.\ 161 (2005), 177--218.

\bibitem{BMR_survey} Michael Bate, Benjamin Martin, and Gerhard R\"ohrle.  \emph{Semisimplification for subgroups of reductive algebraic groups}.  Forum of Mathematics, Sigma, 8, E43 (2020).

\bibitem{BMRT} Michael Bate, Benjamin Martin, Gerhard R\"ohrle, and Rudolf Tange.  \emph{Closed orbits and uniform $S$-instability in invariant theory}.

\bibitem{BBDG} A.\ A.\ Beilinson, J.\ Bernstein, P.\ Deligne, and O.\ Gabber. \emph{Fascieaux pervers}, {Ast\'{e}risque},  Soci\'{e}t\'{e} Math\'{e}matique de France, Paris. {\bf 100} (1982).

\bibitem{Berthelot} Pierre Berthelot.  \emph{Cohomologie cristalline des sch\'emas de caract\'eristique $p > 0$}.  Springer, Lecture Notes in Mathematics, 1974.

\bibitem{Berthelot_survey} Pierre Berthelot.  \emph{Cohomologie rigide et cohomologie rigide \`{a} supports propres}, Premi\`{e}re partie. Preprint, 1996.

\bibitem{Berthelot_Ogus} Pierre Berthelot and Arthur Ogus.  \emph{Notes on Crystalline Cohomology}.  Princeton University Press, 1978.

\bibitem{Neron} Siegfried Bosch, Werner L\"utkebohmert, Michel Raynaud. \emph{N\'eron Models}. Springer, 1990.

\bibitem{CDK} Eduardo Cattani, Pierre Deligne, and Aroldo Kaplan, \emph{On the locus of Hodge classes}.  Journal of the AMS {\bf 8} 483--506 (1995).

\bibitem{Deligne_tensor} Pierre Deligne,  \emph{Cat\'egories tensorielles}.  Mosc. Math. J., {\bf 2} (2002), 227--248. Online at \url{https://www.math.ias.edu/files/deligne/Tensorielles.pdf}.


\bibitem{DeligneWeilII} Pierre Deligne, \emph{La conjecture de Weil: II}, IH\'{E}S Mathematical Publications, {\bf 52}, 137--252 (1980).

\bibitem{Deligne_motives} Pierre Deligne, \emph{Le groupe fondamental de la droite projective moins trois points}.  1989.

\bibitem{DeligneMilne} P.~Deligne and J.~S.~Milne, \emph{Tannakian categories}: Hodge Cycles, Motives, and Shimura Varieties. Lecture Notes in Mathematics, vol 900. Springer, Berlin, Heidelberg. \url{https://www.jmilne.org/math/xnotes/tc2018.pdf}. 

\bibitem{DrinfeldKashiwara} Vladimir Drinfeld, \emph{On a conjecture of Kashiwara}, Mathematics Research Letters, {\bf 8}, 713--728, 2001.

\bibitem{EsnaultHaiSun} H\'{e}l\`{e}ne Esnault, Ph\^{u}ng H\^{o} Hai, and Xiaotao Sun, On Nori's fundamental group scheme, In: Mikhail Kapranov, Yuri Ivanovich Manin, Pieter Moree, Sergiy Kolyada, Leonid Potyagailo, Geometry and Dynamics of Groups and Spaces. Progress in Mathematics, vol 265. Birkh\"{a}user Basel 

\bibitem{Faltings_crys} Gerd Faltings.  Crystalline cohomology and $p$-adic Galois representations.  In \emph{Algebraic analysis, geometry, and number theory: proceedings of the JAMI inaugural conference}, edited by Jun-Ichi Igusa, Johns Hopkins University Press, 1989.

\bibitem{Faltings} Gerd Faltings, Endlichkeitss\"atze f\"ur abelsche Variet\"aten \"uber Zahlk\"orpern.  Invent. math. {\bf 73} (1983), 349--366.

\bibitem{Faltings2} Gerd Faltings, Diophantine approximation on abelian varieties, Annals of Mathematics, {\bf 133} (1991), 549--576.

\bibitem{Asterisque} Jean-Marc Fontaine.
\newblock P\'{e}riodes $p$-adiques.
\newblock In {\em Ast\'{e}risque}, volume 223. Soci\'{e}t\'{e} Math\'{e}matique
  de France, 1994.

\bibitem{FraneckiKapranov} J.~Franecki and M.~Kapranov, The Gauss map and a noncompact Riemann-Roch formula for constructible sheaves on semiabelian varieties, Duke Mathematical Journal {\bf 104} (2000), 171-180.

\bibitem{FK} Eberhard Freitag and Reinhardt Kiehl.  \emph{\'Etale cohomology and the Weil conjectures}.  Springer, 1988.

\bibitem{GabberLoeser} {Ofer Gabber and Fran\c{c}ois Loeser, Faisceaux perver $\ell$-adique sur un tore}, Duke Mathematical Journal {\bf 83} (1996), 501--606.

\bibitem{GaitsgorydeJong} D.~Gaitsgory, \emph{On de Jong's conjecture}, Israel Journal of Mathematics {\bf 157}, 155--191 (2007).

\bibitem{GGK} Mark Green, Phillip A.\ Griffiths, and Matt Kerr. \emph{Mumford--Tate Groups and Domains}.  Annals of Mathematics Studies, 2012.

\bibitem{Javanpeykar} Ariyan Javanpeykar.  N\'eron models and the arithmetic Shafarevich conjecture for certain canonically polarized surfaces.  Bull. London Math. Soc. 47(1):55--64 (2015).

\bibitem{JL18} A.\ Javanpeykar and D.\ Loughran.  Arithmetic hyperbolicity and a stacky Chevalley--Weil theorem.  Journal of the LMS {\bf 101} (2021), 846-849.

\bibitem{JL1} A.\ Javanpeykar and D.\ Loughran.  Complete intersections: moduli, Torelli, and good reduction. Math. Ann. 368 (2017), no. 3--4, 1191--1225.

\bibitem{JL2}A.\ Javanpeykar and D.\ Loughran.  Good reduction of algebraic groups and flag varieties.  Arch. Math. (Basel), 104(2):133--143 (2015).

\bibitem{JL3} Ariyan Javanpeykar and Daniel Loughran.  Good reduction of Fano threefolds and sextic surfaces. Ann. Sc. Norm. Super. Pisa Cl. Sci. (5) 18 (2018), no. 2, 509--535.

\bibitem{JG}
Oliver Johnson and Christina Goldschmidt.
\newblock Preservation of log-concavity on summation.
\newblock {\em ESAIM Probab. Stat.}, 10:206--215, 2006.

\bibitem{LAGeo}
Theo de Jong and Gerhard Pfister.  \emph{Local analytic geometry}.  Advanced Lectures in Mathematics.  Viehweg and Teubner, 2000.

\bibitem{katz-conjecture} Nicholas M.~Katz. \emph{A conjecture in the arithmetic theory of differential equations}. Bull. Soc. math. France {\bf 110} (1982), 203-239.

\bibitem{KatzCaE} Nicholas M.~Katz. \emph{Convolution and Equidistribution: Sato-Tate theorems for finite field Mellin transforms}. Annals of Mathematics Studies {\bf 180}, 2012.

\bibitem{Katz-ESDE} Nicholas M.~Katz. \emph{Exponential Sums and Differential Equations}. Annals of Mathematics Studies {\bf 124}, 1991.


\bibitem{KatzMMP} Nicholas M.~Katz. \emph{Moments, Monodromy and Perversity: a Diophantine perspective}. Annals of Mathematics Studies {\bf 159}, 2006.

\bibitem{Kolchin68} E.~R.~Kolchin, Algebraic groups and algebraic dependence, American Journal of Mathematics {\bf 90} (1968), 1151-1164.

\bibitem{KramerSummary} Thomas Kr\"amer. Perverse sheaves on semiabelian varieties, Rendiconti del Seminario Matematico della Universit\`a di Padova {\bf 132} (2014), 83--102.

\bibitem{Microlocal1}
Thomas Kr\"{a}mer., Characteristic cycles and the microlocal geometry of the Gauss map, I. \url{https://arxiv.org/pdf/1604.02389v3.pdf}, 2016.

\bibitem{Microlocal2}
Thomas Kr\"{a}mer., Characteristic cycles and the microlocal geometry of the Gauss map, II. Journal f\"ur die reine und angewandte Mathematik (Crelles Journal, {\bf 2021} (2021), 53--92

\bibitem{KM} Thomas Kr\"{a}mer and Marco Maculan, Arithmetic finiteness of very regular varieties. \url{https://arxiv.org/abs/2310.08485}, 2023.

\bibitem{KramerWeissauer} Thomas Kr\"{a}mer. and Rainer Weissauer, Vanishing theorems for constructible sheaves on abelian varieties, J. Algebraic Geom. {\bf 24} (2015), 531--568 

\bibitem{KWTheta} T.~Kr\"{a}mer. and R.~Weissauer, On the Tannaka group attached to the theta divisor of a generic principally polarized abelian variety, Mathematische Zeitschrift {\bf 281} (2015), 723-745.

\bibitem{LarsensConjecture} Robert M.~Guralnick and Pham Huu Tiep, Decompositions of small tensor powers and Larsen's conjecture, Representation Theory {\bf 9} (2005), 138--208

\bibitem{LV} Brian Lawrence and Akshay Venkatesh.  \emph{Diophantine problems and $p$-adic period mappings}.  Invent. math. {\bf 221},  (2020) 893–999.

\bibitem{Martin} Benjamin M.\ S.\ Martin.  Reductive subgroups of reductive groups in nonzero characteristic. Journal of Algebra 262 (2003) 265--286.

\bibitem{Milne_gstr} James S.\ Milne, \emph{Points on Shimura varieties over finite fields: the conjecture of Langlands and Rapoport}. ArXiv preprint 0707.3173, v.\ 4.

\bibitem{NaNo} M.\ S.\ Narasimhan and M.\ V.\ Nori. Polarizations of an abelian variety. Proceeings of the Indian Academic Society (Math Society) {\bf 90} (1981), 125--128.

\bibitem{Ogus} A.\ Ogus.  $F$-crystals and Griffiths transversality.  In \emph{Proceedings of the international symposium on algebraic geometry, Kyoto, 1977 (a Taniguchi symposium)}, Kinokuniya Book-Store Co., Ltd., 1978.

\bibitem{Petersen} T.\ Kyle Petersen.  \emph{Eulerian numbers}.  Birkh\"auser, 2015.

\bibitem{PopaSchnell}
Mihnea Popa and Christian Schnell, \emph{Generic vanishing theory via Mixed Hodge modules}. For. Math. Sigma, 1 (2013).

\bibitem{vdPS} Marius van der Put and Michael F. Singer.  \emph{Galois Theory of Linear Differential Equations}. Springer, 2003.

\bibitem{Ribet} Kenneth A.\ Ribet.  \emph{Galois action on division points of abelian varieties with real multiplications}.  Amer.\ J.\ Math., 98(3): 751--804 (1976).

\bibitem{Richardson} R.\ W.\ Richardson. \emph{Conjugacy classes of $n$-tuples in Lie algebras and algebraic groups}.  Duke (1988) Vol.\ 57, No.\ 1.

\bibitem{Saavedra} Neantro Saavedra Rivano, \emph{Cat\'egories Tannakiennes}. Bulletin de la Soci\'et\'e Math\'ematique de France {\bf 100} (1972), 417--430.

\bibitem{Schmid} Wilfried Schmid.  \emph{Variation of Hodge structure: the singularities of the period mapping}.  Inv.\ math.\ 22, 211--319 (1973).

\bibitem{Scholl} A.\ J.\ Scholl.  \emph{A finiteness theorem for del Pezzo surfaces over algebraic number fields} J. London Math. Soc. (1985).

\bibitem{Scholze} Peter Scholze. \emph{$p$-Adic Hodge theory for rigid analytic varieties}. For. Math. Pi, 1 (2013).

\bibitem{Serre_galois} Jean-Pierre Serre, \emph{Cohomologie Galoisienne}. Springer, Lecture Notes in Mathematics, 1994.

\bibitem{Serre_cr} Jean-Pierre Serre, \emph{Compl\`ete R\'educibilit\'e}.  S\'eminaire Bourbaki, no.\ 932, 195--217 (2004).

\bibitem{She_K3} Yiwei She, \emph{The unpolarized Shafarevich conjecture for K3 surfaces}.  ArXiv preprint 1705.09038 (2017).

\bibitem{Tan_Tong} Fucheng Tan and Jilong Tong. \emph{Crystalline comparison isomorphisms in $p$-adic Hodge theory: the absolutely unramified case.} Alg. Num. Th.\ 13, 1509--1581 (2019).

\bibitem{ProjDual} Evgueni V.\ Tevelev. \emph{Projective Duality and Homogeneous Spaces}. Encyclopedia of Mathematical Sciences {\bf 133}, Springer-Verlag (2005).

\bibitem{Urbanik} David Urbanik. \emph{Effective methods for Diophantine finiteness}.  ArXiv preprint 2110.14289 (2021).

\bibitem{Weissauer_BN_1} Rainer Weissauer. Brill-Noether sheaves. \url{https://arxiv.org/pdf/math/0610923} (2006).

\bibitem{Weissauer_TC} Rainer Weissauer. Tannakian Categories attached to abelian varieties, in Modular Forms on Schiermonnikoog, edited by B. Edixhofen, G. van der Geer and B. Moonen, Cambridge University Press  p.267--284, (2008).

\bibitem{Weissauer_BN_2} R.\ Weissauer. A remark on rigidity of BN-sheaves. \url{https://arxiv.org/pdf/1111.6095} (2011).

\bibitem{Weissauer_VF} Rainer Weissauer. Vanishing theorems for constructible sheaves on abelian varieties over finite fields. 
Math. Ann. {\bf 365} (2016), 559--578.

\bibitem{Zuo} Kang Zuo. On the negativity of kernels of Kodaira--Spencer maps on Hodge bundles and applications. \emph{Asian J.\ Math.}, 4(1): 279--301 (2000).








\end{thebibliography}
\end{document}